\documentclass[11pt, a4paper]{article}

\usepackage[utf8]{inputenc}
\usepackage[T1]{fontenc}
\usepackage{lmodern}
\usepackage{microtype}
\usepackage[
    colorlinks,
    linkcolor=blue,citecolor=red,urlcolor=blue,
]{hyperref}

\usepackage{color}
\usepackage{ifthen} 
\usepackage{paralist}
\usepackage{graphicx}
\usepackage[all]{xy}
\usepackage{verbatim}
\usepackage{multirow}
\usepackage[tbtags]{amsmath} 
\usepackage{amssymb}
\usepackage{amsthm}
\usepackage{mathtools} 
\usepackage[colorinlistoftodos]{todonotes}

\usepackage[text={16.6cm, 24.0cm}, centering]{geometry}
\newcommand{\inv}[0]{{-1}}
\newcommand{\oo}[0]{\otimes}

\newcommand{\exoo}[1]{^{\oo #1}}  
\newcommand{\extimes}[1]{^{\times #1}}

\newcommand{\NN}{\mathbb{N}}

\newcommand{\FF}{\mathbb{F}}

\newcommand{\cala}[0]{\mathcal{A}}

\newcommand{\st}[0]{{\bf s}}
\newcommand{\ta}[0]{{\bf t}}

\newcommand{\id}[0]{\mathrm{id}}
\newcommand{\rrhd}[0]{\rhd'}
\newcommand{\llhd}[0]{\blacktriangleleft}

\newcommand{\low}[2]{{#1}_{({#2})}}

\newcommand{\hol}[0]{\mathrm{Hol} }

\newtheorem{theorem}{Theorem}[section]
\newtheorem*{theorem*}{Theorem}
\newtheorem{example}[theorem]{Example}
\newtheorem{lemma}[theorem]{Lemma}
\newtheorem{proposition}[theorem]{Proposition}
\newtheorem{corollary}[theorem]{Corollary}
\newtheorem{definition}[theorem]{Definition}
\newtheorem{remark}[theorem]{Remark}

\def\mytitle{Hopf algebra gauge theory on a ribbon graph}
\def\myauthors{C.~Meusburger D.K.~Wise}
\hypersetup{pdftitle=\mytitle, pdfauthor=\myauthors}

\begin{document}

\begin{center}
  {\huge\mytitle}

  \vspace{2em}

  {\large
   Catherine Meusburger\footnote{{\tt  catherine.meusburger@math.uni-erlangen.de}} \qquad\qquad
   Derek K.~Wise\footnote{{\tt derek.wise@fau.de}}}

  \vspace{1em}

  $\mbox{}^{1,2}$ Department Mathematik \\
  Friedrich-Alexander  Universit\"at Erlangen-N\"urnberg \\
  Cauerstra\ss e 11, 91058 Erlangen, Germany\\[+2ex]
  
  $\mbox{}^{2}$ Concordia University, Saint Paul\\
  1282 Concordia Avenue\\
  St.~Paul, Minnesota 55104, USA

  \vspace{1em}

  19 April 2016

  \vspace{2em}

  \begin{abstract}
\noindent We generalise gauge theory on a graph so that the gauge group becomes a finite-dimensional ribbon Hopf algebra, the graph becomes a ribbon graph, and gauge-theoretic concepts such as connections, gauge transformations and observables are replaced by linearised analogues.  Starting from physical considerations, we derive an axiomatic definition of Hopf-algebra gauge theory, including locality conditions under which the theory for a general ribbon graph can be assembled from local data in the neighbourhood of each vertex.  For a vertex neighbourhood with $n$ incoming edge ends, the algebra of non-commutative `functions' of connections is dual to a two-sided twist deformation of the $n$-fold tensor power of the gauge Hopf algebra.  We show these algebras assemble to give an algebra of functions and gauge-invariant subalgebra of `observables' that coincide with those obtained in the combinatorial quantisation of Chern-Simons theory, thus providing an axiomatic derivation of the latter.  We then discuss holonomy in a Hopf algebra gauge theory and show that for semisimple Hopf algebras this gives, for each path in the embedded graph, a map from connections into the gauge Hopf algebra, depending functorially on the path.  Curvatures---holonomies around the faces canonically associated to the ribbon graph---then correspond to central elements of the algebra of observables, and define a set of commuting projectors onto the subalgebra of observables on flat connections. 
The algebras of observables for all connections or for flat connections are topological invariants, depending only on the topology, respectively, of the punctured or closed surface canonically obtained by gluing annuli or discs along edges of the ribbon graph. 

  \end{abstract}
\end{center}

\section{Introduction}

The purpose of this article is to generalise lattice gauge theory from groups to Hopf algebras.  Our motivation stems from a variety of gauge theory-like models constructed from algebraic data assigned to discretisations of oriented surfaces, many of which are based on  Hopf algebras, their representation categories or higher categorical analogues. Examples of recent interest include models in condensed matter physics and quantum computation such as Kitaev \cite{Ki,Ki2,BMCA} and Lewin-Wen models \cite{LW1,LW2}, as well as models in non-commutative geometry \cite{MS}. There are also older models obtained from canonical quantisation of Chern-Simons theory  \cite{AGSI,AGSII,AS,BR,BR2,BFK} and much work on related models in 3d quantum gravity---for an overview see \cite{Carlipbook}.

These models strongly resemble lattice gauge theories, exhibiting gauge-like symmetries, topologically invariant subspaces of invariant states, and operators resembling Wilson loops. 
Some have been  related \cite{KKR, BK, BA,AS} to 3d topological field theories of  Turaev-Viro \cite{TVi} or of Reshetikhin-Turaev  type \cite{RT}, which in turn arise from quantisation of BF and Chern-Simons gauge theories. 

However, while their gauge theoretic features and origins suggest viewing these models as Hopf algebraic generalisations lattice gauge theory, no precise concept of lattice gauge theory valued in a Hopf algebra has so far been available.  Neither the physical requirements nor the appropriate mathematical structures for such a theory are obvious. Nor is it easy to conceptually unify the would-be examples of Hopf algebra gauge theory, which are often formulated ad hoc  for specific graphs, specific bases, generators or relations, and in language that obscures the general mathematical structures and their relation to classical gauge theoretic concepts such as gauge fields, connections, and observables.  With few exceptions, e.g.~\cite{FSV}, little work has sought to physically motivate the mathematical structures arising in such models.

The goal of this article is to address these issues. We provide a robust and physically-motivated general definition of local Hopf algebra gauge theory on a ribbon graph, construct such gauge theories, and relate them to established models arising from quantising Chern--Simons theory. We also clarify the gauge theoretic structure of these theories, including holonomy, curvature, and observables. 

To achieve this, we mimic group-based gauge theory and its physical interpretation as closely as possible, linearising standard gauge theory concepts and generalising to accommodate Hopf algebras not associated to a group.  For gauge theory on a ribbon graph with edge set $E$ and vertex set $V$, we can summarise the key ingredients of our definition of Hopf algebra gauge theory (Def.~\ref{def:gtheory}) in a table, which  shows the comparison to the group case:\footnote{Here, expressions such as $K^{\oo E}$ represent a tensor product $K \oo \cdots \oo K$ with one copy of $K$ for each element of $E$.} 
\begin{center}
\sf\small 
 \def\rowgap{.9em}
 \begin{tabular}{l | c | c}
  &{\bf group gauge theory} & {\bf Hopf algebra gauge theory}
  \\[\rowgap] 
  \hline
  & & \\
  gauge group/algebra & finite group $G$ & finite-dimensional Hopf algebra $K$ 
  \\[\rowgap]
  gauge transformations 
  & the group $\mathcal G = G^{\times V}$
  & the Hopf algebra $\mathcal G = K^{\oo V}$
  \\[\rowgap]
  connections &
  the set $\mathcal A =G^{\times E}$ &
  the vector space $\mathcal A = K^{\oo E}$ 
  \\[\rowgap]
  \multirow{2}{*}{algebra of functions} 
  & $\mathcal A^*={\rm Fun}(G^{\times E})$
  & 
  the vector space $\mathcal A^*\cong K^{\ast \oo E}$
 \\[0em]
  &with canonical algebra structure & 
  with {\em specified} algebra structure
  \\[\rowgap]
  \multirow{2}{1.5in}{gauge action on connections}
  & \multirow{2}{*}{a $\mathcal{G}$-set structure on $\mathcal A$}
  & \multirow{2}{*}{a $\mathcal G$-module structure on $\mathcal A$} 
  \\ && 
  \\[\rowgap]
  \multirow{2}{1.5in}{gauge action on functions}
  & \multirow{2}{*}{the dual $\mathcal{G}$-set structure on $\mathcal A^*$}
  & the dual $\mathcal G$-module structure, making $\mathcal A^*$ 
  \\
  && into a $\mathcal G$ module algebra. 
  \\[\rowgap]
  \multirow{2}{*}{observables} 
  & the gauge-invariant 
  & the gauge-invariant
  \\ &
  subalgebra of $\mathcal{A}^\ast$ & 
  subalgebra of $\mathcal{A}^\ast$
  \end{tabular}
\end{center}

This table shows a strong parallel but also some complications on the Hopf algebra side. Most importantly, for a general Hopf algebra, the algebra structure on $\mathcal{A}^\ast$ is {\em not} the canonical one induced by the multiplication in $K^*$, as one might naively expect from the group case, since this would not be compatible with `gauge symmetry' unless $K$ is cocommutative. Rather, we must specify an algebra structure on $\mathcal{A}^\ast$ that makes it a {\em module algebra} for the gauge action.  Physically speaking, this amounts to the requirement that gauge invariant functions, or observables, should form a {\em subalgebra}.

It is also worth mentioning why a Hopf algebra gauge theory requires {\em ribbon graphs} rather than ordinary graphs.  Even for groups, ribbon graphs are natural for two-dimensional gauge theory, since embedding a graph in an oriented surface gives it a ribbon structure.  However, for Hopf algebras, this structure seems essential: if the algebra is non-cocommutative, we know of no sensible way to get gauge transformations to act on connections without assuming a ribbon structure on our graphs.  We thus work with ribbon graphs throughout the article, reviewing their structure in Section \ref{sec:background}.

\vspace{-.5cm}
\subsection*{Main Results}

\vspace{-.2cm}
While our definition of Hopf algebra gauge theory is built up axiomatically, guided but the group case, our main results in support of the quality of this definition are that we can build up Hopf algebra gauge theories from local data, show topological invariance of their algebras of observables, relate these algebras to existing work, and thus clarify the gauge-theoretic foundations of other models. Here we outline these main results.  

\vspace{-.5cm}
\subsubsection*{1. For a finite-dimensional ribbon Hopf algebra, gauge theory on a `vertex neighbourhood' is essentially unique.}

\vspace{-.2cm}
By a {\em vertex neighbourhood}, we simply mean a cyclically-ordered directed rooted tree of height one: 
\[\includegraphics[height=1.2cm]{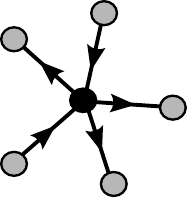}\]
but our terminology emphasises that we think of these as obtained from a larger graph by cutting all of the edges surrounding a single vertex, and gluing new vertices to the loose ends. Studying Hopf algebra gauge theory on such simple graphs (\S\ref{subsec:vertex_nb}) shows that allowing non-cocommutativity strongly constrains our choice of gauge Hopf algebra $K$.  However, we show that these constraints can be satisfied---and in an essentially unique way determined by the `locality conditions' in our definition---provided we take the Hopf algebra to be {\em quasitriangular} and require that the algebra for an $n$-valent vertex is built up from the one for a vertex with a single edge end via the braided tensor product of $K$-right module algebras.  Moreover, considering consistency with reversing the orientation of edges leads us to demand that our Hopf algebra has a {\em ribbon} element.

{\bf Theorem 1:} 
A Hopf algebra gauge theory on a vertex neighbourhood is essentially unique.
The relevant $K$-module algebra structure on $K\exoo{v}$ is related to the braided tensor product of $K$-module algebras and dual to a two-sided twist deformation of $K\exoo{v}$ with a cocycle involving multiple copies of its universal $R$-matrix.

\vspace{-.5cm}
\subsubsection*{2. Hopf algebra gauge theories on the vertex neighbourhoods of a graph induce a Hopf algebra gauge theory on the  graph.}

\vspace{-.2cm}
This result, found in \S\ref{subsec:gribbongraph}, means Hopf algebra gauge theory can be built up from local data, effectively reducing the construction of gauge theory for any ribbon graph to basic building blocks constructed on the vertex neighbourhoods.  

To show this, we split a ribbon graph into a disjoint union of vertex neighbourhoods and use comultiplication to embed the copy of $K^*$ associated with a directed edge $e$ in the original graph into the two copies of $K^*$ associated with the two {\em ends} of $e$, each associated with a vertex neighbourhood:

\[
\begin{tikzpicture}
\node {\includegraphics[height=1.5cm]{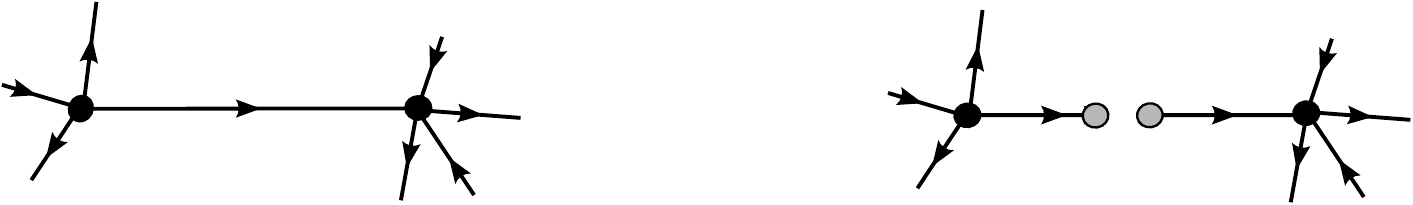}};
\node {$\to$};
\node at (-3.25,.25) {$K^*$};
\node at (2.6,.25) {$K^*$};
\node at (3.75,.25) {$K^*$};
\end{tikzpicture}
\]
Doing this simultaneously for all edges in the ribbon graph yields an injective linear map 
$$G^*\colon \cala^*\to \bigotimes_{v\in V} {\cala_v^*}$$
where $\cala_v^*$ is the algebra of functions for the vertex neighbourhood at $v$. 

{\bf Theorem 2:} The image of  $G^*$ is a subalgebra and a $K\exoo{V}$-submodule of  $\bigotimes_{v}{\cala_v^*}$.  Pulling back this structure makes $\cala^*:={K^*}\exoo{E}$ into the algebra of functions for a Hopf algebra gauge theory.

This algebra is analysed in \S\ref{subsec:algfunc} for concrete graphs; for a single loop it is related to $K^{op}$, and when $K=D(H)$ is a Drinfel'd double of a Hopf algebra $H$, the algebra for a single edge between distinct vertices is isomorphic to the {\em Heisenberg double} of $H$.   

\vspace{-.5cm}
\subsubsection*{3. Holonomy is functorial, and transforms nicely under graph operations.}

\vspace{-.3cm}
In analogy to group gauge theory, we view holonomy as a functor $$\mathrm{Hol}: \mathcal G(\Gamma)\to \mathrm{Hom}_{\FF}(K\exoo{E},K)$$ that assigns to each path in $\Gamma$---a morphism in the {\em path groupoid} $\mathcal G(\Gamma)$---a linear map $\mathrm{Hol}_p: K\exoo{E}\to K$ taking connections on $\Gamma$ to element of the Hopf algebra $K$. 

For this, the vector space $\mathrm{Hom}_{\FF}(K\exoo{E},K)$ is given the structure of a category with one object and with composition built from multiplication in $K$ and comultiplication in $K\exoo{E}$.
In principle, there are two choices of coalgebra structure on $K\exoo{E}$, coming either from the tensor product Hopf algebra structure or from the multiplication in the algebra of functions.   However, we show in \S\ref{subsec:genhols} 
that only the former gives rise to a holonomy functor in the strict sense, and only under the additional assumption that $K$ is {\em semisimple}. 

We also investigate the algebraic properties of the holonomies in \S\ref{subsec:algprop} and derive a general formula for the transformation of holonomies under graph operations and show that, for each path in $\Gamma$ that represents a simple curve on the associated surface, the holonomies form a subalgebra and a submodule of the algebra of functions of the Hopf algebra gauge theory.

\vspace{-.5cm}
\subsubsection*{4. The algebra of observables is a topological invariant of surfaces with boundaries.}

\vspace{-.2cm}
There is a well-known way of viewing any ribbon graph, defined by combinatorial data, as a surface with boundary, unique up to homeomorphism.  This is achieved by fattening out edges into `ribbons', as the name suggests, and vertices into discs, or equivalently by gluing annuli along edges according to the combinatorial data.   Moreover, ribbon graphs that induce the same surface with boundary are related by certain graph operations, reviewed in \S\ref{subsec:graphopsribbon}. In Section \ref{sec:graphops}, we show that these graph operations induce isomorphisms of Hopf algebra gauge theories---in particular, the algebras of functions and of observables depend on the choice of the ribbon graph only insofar as the ribbon graphs induce homeomorphic surfaces in this way.

{\bf Theorem 3:} If two ribbon graphs induce homeomorphic surfaces with boundary, then the algebras of observables of the associated $K$-valued Hopf algebra gauge theories are isomorphic.

\vspace{-.5cm}
\subsubsection*{5. The algebra of observables on {\em flat} connections is a topological invariant.}

\vspace{-.2cm}
Aside from viewing a ribbon graph as a surface with boundary, there is also a well-known way of viewing any ribbon graph as a {\em closed} surface, by attaching discs, rather than annuli.  By considering holonomies around these discs---faces in the resulting CW complex---we arrive at a notion of curvature in Hopf algebra gauge theory.  We show the curvature at a face defines a central element of the algebra of observables.  We show that if $K$ is semisimple, the Haar integral of $K^*$ gives rise to a set of commuting projectors associated with faces, whose image can be viewed as the algebra of observables on the set of {\em flat} connections. 

{\bf Theorem 4:}   
{If two ribbon graphs induce homeomorphic closed surfaces, then their associated algebras of observables on {flat} connections are isomorphic.}

\vspace{-.5cm}
\subsubsection*{6. Hopf algebra gauge theory provides an axiomatic derivation, and gauge-theoretic interpretation of algebras from the combinatorial quantisation of Chern-Simons theory.}

\vspace{-.2cm}
We prove that the algebra of functions obtained from our definition of Hopf algebra gauge theory and the associated subalgebras of all observables and of observables on flat connections coincide with the algebras obtained by Alekseev, Grosse and Schomerus \cite{AGSI,AGSII,AS}  and independently  by Buffenoir and Roche \cite{BR,BR2} in the canonical quantisation of Chern-Simons gauge theory. {The description of this algebra was refined further in \cite{BFK}, where the authors developed a  description that is suitable for concrete computations and naturally fits into the framework of quantum topology.  In particular, the algebra of observables,  the algebra of observables on flat connections and the results of Theorems 3 and 4 above constitute a generalisation and simplification of topological invariance results obtained earlier \cite{AGSI,AGSII, BR,BR2,BFK} in a different framework and by different methods. While the algebras themselves are thus not new, our approach adds several insights to the picture. 

First, the algebras were previously obtained by quantising the Poisson structures in \cite{FR,AM} by canonical quantisation, i.e.~replacing classical $r$-matrices by $R$-matrices and Poisson brackets by multiplication relations. While physically well-motivated, this leaves questions about uniqueness of the quantisation procedure and of the resulting algebra, especially since the underlying Poisson structure is non-canonical. Our approach addresses this by showing that the algebra can be {\em derived} from simple axioms for a local Hopf algebra gauge theory in an essentially unique way.

Second, our formulation exhibits more clearly the essential mathematical ingredients, namely, module algebras over Hopf algebras and their braided tensor products. These are physically motivated by the requirement that the gauge invariant observables form an algebra.  The resulting description is close to the viewpoint of non-commutative geometry, in which a commutative algebra of classical coordinate functions is deformed or replaced by a non-commutative analogue.

Moreover, these algebras are shown to be built up atomically from the algebras on vertex neighbourhoods, which are obtained from a two-sided twist-deformation of a Hopf algebra $K^{\oo n}$. This reflects the local nature of the Hopf algebra gauge theory and allows one to construct Hopf algebra gauge theory on a surface by gluing discs around the vertices of a graph. 
(A similar decomposition was also given in \cite{BFK}, but in a diagrammatic formulation. This has computational advantages, but  makes it  difficult to  interpret the results from a  gauge theory perspective.) 

This decomposition into vertex neighbourhoods leads to a direct and simple description of local Hopf algebra gauge theory. While the formulation in  \cite{AGSI,AGSII,AS,BR,BR2} is  highly involved and relies on specific choices of a basis of $K$, namely matrix elements in the irreducible representations, Clebsch-Gordan coefficients and intertwiners
and mainly considers a fixed ribbon graph,
the description in terms of vertex neighbourhoods allows one to give a coordinate free description for general ribbon graphs. 
In this description, the proof of topological invariance is much simpler and more direct. 
Moreover, the concept of holonomy arises in a more conceptual way, as a functor, and holonomies transform by simple and explicit rules under graph transformations.  This may also be helpful in generalising Hopf algebra gauge theory to include defects.

\vspace{-.5cm}
\subsection*{Organisation of the paper} 

\vspace{-.2cm}
Section \ref{sec:background} provides background material on graphs and their path categories (\S\ref{subsec:graphs}) and ribbon graphs (\S\ref{subsec:ribbongraph}) as well as describing the operations on ribbon graphs (\S\ref{subsec:graphopsribbon}) that  preserve the topology of the resulting surfaces. 
 
In Section \ref{sec:gtheory} we derive axioms for a Hopf algebra gauge theory on a ribbon graph. After establishing conventions (\S\ref{subsec:conventions}) and reviewing {\em group}-based gauge theory (\S\ref{subsec:groupgtheory}) we generalise it to finite-dimensional Hopf algebras (\S\ref{subsec:haxioms}) and give an axiomatic definition of a local Hopf algebra gauge theory.  We then construct such local Hopf algebra gauge theories, first on vertex neighbourhoods (\S\ref{subsec:vertex_nb}) and then on general ribbon graphs (\S\ref{subsec:gribbongraph}). For the latter
we give an explicit description of the algebra the algebra of functions (\S\ref{subsec:algfunc}), and relate to it the `lattice algebras' from the combinatorial quantisation of Chern--Simons theory \cite{AGSI,AGSII,BR,BR2}.

Section \ref{sec:graphops} investigates the dependence of the algebra of functions and its subalgebra of observables on the choice of the ribbon graph. It is shown that certain transformations of ribbon graphs induce homomorphisms of module algebras between the algebras of functions of the associated Hopf algebra gauge theories and algebra homomorphisms between their subalgebras of observables. This leads to a topological invariance result for the latter.

In Section \ref{sec:holsec} we develop other essential gauge-theoretic concepts for Hopf algebra gauge theory, namely holonomy and curvature.  Following their definition (\S\ref{subsec:genhols}), we investigate the algebraic properties of holonomies (\S\ref{subsec:algprop}) and  of curvatures and {\em flat} connections  (\S\ref{subsec:curve}). We then show that this defines  a topologically invariant algebra of observables on flat connections.}  

Appendices \ref{sec:hopfalg} and \ref{sec:modulealg} contain algebraic background material, including the  required foundational results  on Hopf algebras and module algebras over Hopf algebras.  

\section{Geometrical background: graphs and paths}
\label{sec:background}

\subsection{Graphs and paths}
\label{subsec:graphs}

In the following, we consider finite directed graphs. Unless specified otherwise,  we allow loops and  multiple edges, and vertices of  valence $\geq 1$. For a directed graph $\Gamma$, we denote by $V(\Gamma)$  and $E(\Gamma)$, respectively,  the sets of vertices and edges, and omit the argument $\Gamma$ whenever this is unambiguous. 
For an oriented edge $e\in E(\Gamma)$, we denote by $\st(e)$ the {\bf starting vertex} of $e$ and by $\ta(e)$ the {\bf target vertex} of $e$. 
An  edge $e\in E(\Gamma)$ is called a {\bf loop} if  $\st(e)=\ta(e)$.  The edge $e$ with the opposite orientation is denoted $e^\inv$, and one has $\st(e^\inv)=\ta(e)$, $\ta(e^\inv)=\st(e)$.

\begin{definition}\label{def:vertex_nb} Let $\Gamma$ be a directed graph.
\begin{compactenum}
\item  A {\bf subgraph} of $\Gamma$ is a graph $\Gamma'$ obtained from $\Gamma$ by removing edges of $\Gamma$ and any zero-valent vertices arising in the process. \\[-2ex]

\item The {\bf edge subdivision} of $\Gamma$ is the directed graph $\Gamma_{\circ}$ obtained by placing a vertex on the middle of each edge $e\in E(\Gamma)$ and equipping the resulting edges with the induced orientation, as shown in Figure \ref{fig:cutting} b).\\[-2ex]

\item For an edge $e\in E(\Gamma)$, the two corresponding edges of $\Gamma_\circ$ are called {\bf edge ends} of $e$. The edge end of $e$ that is connected to the starting  vertex  $\st(e)$ is called {\bf starting end} of $e$ and denoted $s(e)$.  The one connected to the target vertex $\ta(e)$ is called  {\bf target end} of $e$ and denoted  $t(e)$. \\[-2ex]

\item The {\bf vertex neighbourhood} $\Gamma_v$ of a vertex $v\in V(\Gamma)$ is the directed graph obtained by taking the vertex $v$, all incident edge ends at $v$ and placing a univalent vertex at the end of each edge end, as shown in Figure \ref{fig:cutting} c).  \end{compactenum}
\end{definition}

By definition, the edge subdivision  $\Gamma_\circ$ and any vertex neighbourhood $\Gamma_v$  for $v\in V(\Gamma)$  are directed graphs without loops. 
They are bipartite since  each edge of ${\Gamma}_\circ$ or $\Gamma_v$ connects a vertex  $v\in V(\Gamma)$ with a vertex 
of $\Gamma_\circ$ that is not contained in $V(\Gamma)$. 
Note also that the second edge subdivision $\Gamma_{\circ\circ}=(\Gamma_\circ)_\circ$ of a directed graph $\Gamma$ has neither  loops nor multiple edges.

\begin{figure}
\centering
\includegraphics[scale=0.35]{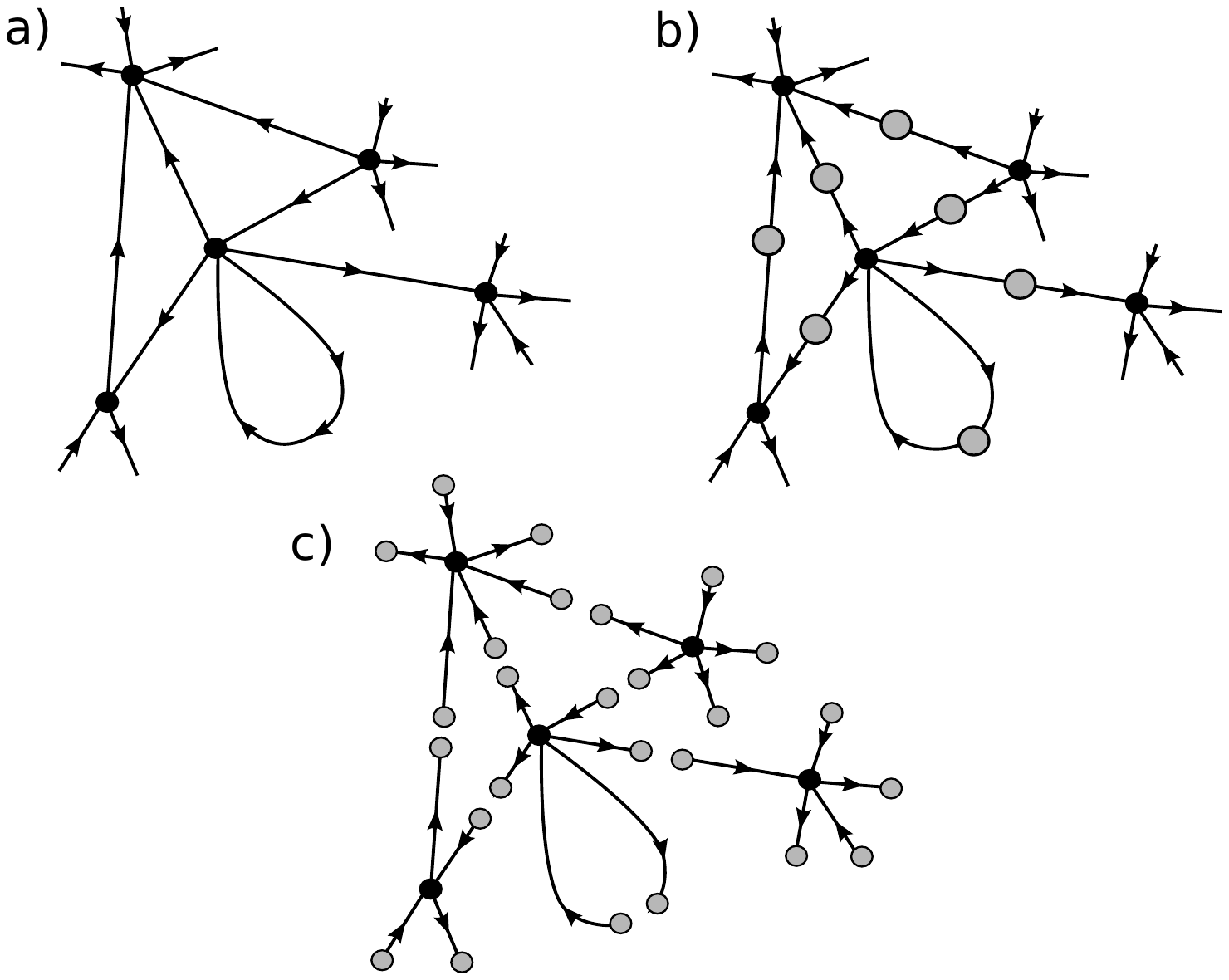}
\caption{a) Part of a directed graph $\Gamma$, b) its edge subdivision $\Gamma_\circ$,  c) the associated disjoint union of vertex neighbourhoods.}
\label{fig:cutting}
\end{figure}

Paths in a directed graph $\Gamma$ can be viewed as morphisms in the free groupoid generated by $\Gamma$.
They are described by {\bf words} in $E(\Gamma)$, which are either  finite sequences of the form  $w=((e_n, \epsilon_n),..., (e_1, \epsilon_1))$ with $n\in\NN$, $e_i\in E(\Gamma)$ and  $\epsilon_i\in\{\pm 1\}$  or empty words $\emptyset_v$ for each vertex $v\in V(\Gamma)$. 
In the following, we write  $w=e_n^{\epsilon_n}\circ \ldots\circ e_1^{\epsilon_1}$ for the former. 
A word $w$ in $E(\Gamma)$ is called {\bf composable} if  it is empty or if  $\ta(e_i^{\epsilon_i})=\st(e_{i+1}^{\epsilon_{i+1}})$ for all $i=1,...,n-1$. For a composable word $w$ we set $\st(w)=\st(e_1^{\epsilon_1})$ and $\ta(w)=\ta(e_n^{\epsilon_n})$ if $w$ is non-empty and $\st(w)=\ta(w)=v$ if $w=\emptyset_v$.
The number $n\in\NN$ is called the {\bf length} of $w$, and one sets $n=0$ for empty words $\emptyset_v$.
A word  is called {\bf reduced} if it is empty or if it is of the form $w=((e_n, \epsilon_n),..., (e_1, \epsilon_1))$ with $(e_i,\epsilon_i)\neq (e_{i+1}, -\epsilon_{i+1})$ for all $i\in\{1,...,n-1\}$. It is called {\bf cyclically reduced}
if it is reduced and $(e_1,\epsilon_1)\neq (e_n,-\epsilon_n)$ if $n\geq 1$.

 \begin{definition} \label{def:free_cat} 
 Let $\Gamma$ be a directed graph. 
 
 \begin{compactenum}
 \item  The
 {\bf path category} $\mathcal \mathcal C(\Gamma)$ is the free category generated by $E(\Gamma)\times\{\pm 1\}$. It has vertices of $\Gamma$ as objects. A morphism
 from  $u$ to $v$  is a  composable word $w$ with
 $\st(w)=u$ and  $\ta(w)=v$. Identity morphisms  are  the  trivial words $\emptyset_v$,  and the composition of morphisms is  concatenation. \\[-2ex]

\item  The
 {\bf path groupoid} $\mathcal G(\Gamma)$   is the free groupoid generated by $\Gamma$. Its objects are the  vertices of $\Gamma$.  A morphism
 from  $u$ to $v$  is an equivalence class of  composable words $w$ with
 $\st(w)=u$,  $\ta(w)=v$  with respect to the equivalence relation
  $e^{-1}\circ e\sim \emptyset_{\st(e)}$, $e\circ e^\inv=\emptyset_{\ta(e)}$ for all $e\in E(\Gamma)$. Identity morphisms  are  equivalence classes of  trivial words $\emptyset_v$,  and the composition of morphisms is induced by the concatenation. \\[-2ex]

\item  A {\bf path} $p$  in $\Gamma$ is a morphism in $\mathcal G(\Gamma)$. For a path $p$ given by a reduced word $w$, the vertex  $\st(p)=\st(w)$ the is called the {\bf starting vertex} and the vertex $\ta(p)=\ta(w)$ the {\bf target vertex}  of $p$. We  denote by 
 $p^\inv$ the reversed path given by  $\emptyset_v^\inv=\emptyset_v$ and $(e_n^{\epsilon_n}\circ \ldots\circ e_1^{\epsilon_1})^\inv=e_1^{-\epsilon_1}\circ \ldots\circ e_n^{-\epsilon_n}$. We call the path $p$ {\bf cyclically reduced} 
 if the associated reduced word $w$ is cyclically reduced.
 \end{compactenum}
  \end{definition}

\subsection{Ribbon graphs}
\label{subsec:ribbongraph}

In the article we consider directed graphs equipped with a cyclic ordering of edge ends at each vertex. These are called {\em ribbon graphs} or {\em fat graphs}, since cyclically ordering edge ends can be seen as fattening them out into ribbons; for an accessible introduction see \cite{LZ, EM}. They can be viewed as directed graphs {\em embedded} in oriented surfaces, and are thus sometimes called   {\em embedded graphs}.  

A graph embedded in an oriented surface inherits a {\em cyclic ordering} of the incident edge ends at each vertex from the orientation of the surface.    
This cyclic ordering of the edge ends equips the graph with the notion of a {\bf face}.  We say that a path $p=e_n^{\epsilon_n}\circ \ldots\circ  e_1^{\epsilon_1}$ in $\Gamma$ turns maximally right (left) at the vertex $v_i=\st(e_{i+1}^{\epsilon_{i+1}})=\ta(e_i^{\epsilon_i})$ if the starting end of  $e_{i+1}^{\epsilon_{i+1}}$ comes directly after (before) the target end of $e_i^{\epsilon_i}$ with respect to the cyclic ordering at $v_i$.
 If $p$ is closed, we say $p$ turns maximally right (left) at $v_n=\st(e_1^{\epsilon_1})=\ta(e_n^{\epsilon_n})$ if the starting end of $e_1^{\epsilon_1}$ comes directly after (before) the target end of $e_n^{\epsilon_n}$ with respect to the cyclic ordering at $v_n$. A face is then defined  as an equivalence class of closed paths  that  turn maximally left 
 at each vertex and pass any edge at most once in each direction.

\begin{definition} $\quad$
\begin{compactenum}

\item A {\bf ribbon graph}   is a directed graph  with a cyclic ordering of  the  edge ends at each vertex.  \\[-2ex]

\item  A {\bf face path} of a ribbon graph  $\Gamma$ is a closed path in $\Gamma$ which turns maximally left  
at each vertex, including the starting vertex,  and traverses an edge at most once in each direction. \\[-2ex]

\item Two face paths $f,f'$ are  {\bf equivalent} if their expressions as reduced words in  $E(\Gamma)$ are related by cyclic permutations. A {\bf face} is an equivalence class of face paths.\\[-2ex]

\item The  {\bf valence} $|v|$ {\bf of a vertex} $v$  is the number of incident edge ends at $v$, and the {\bf valence} $|f|$ {\bf of a face} $f$  is the length of its representatives as reduced words in $E(\Gamma)$.
\end{compactenum}
\end{definition}

Here and in the following we denote by $F(\Gamma)$ the set of faces of a ribbon graph $\Gamma$ and omit the argument whenever this is unambiguous. 
Given a directed graph  $\Gamma$, understood as a combinatorial graph, one obtains a graph in the topological sense, i.e.~a 1-dimensional CW-complex, by gluing intervals to the vertices according to the combinatorics specified by the edges. If additionally the graph has a ribbon graph structure, one obtains an oriented surface $\Sigma_\Gamma$
 by  gluing a disc to each face.
If $\Gamma$ is a directed graph embedded in an oriented surface $\Sigma$ and equipped with the induced ribbon graph structure, then the surface $\Sigma_\Gamma$  is homeomorphic to $\Sigma$ if and only if  each connected component of $\Sigma\setminus\Gamma$ is homeomorphic to a disc.

The gluing procedure extends to surfaces $\Sigma$ with a finite number of discs removed. In this case, one requires that each connected component of $\Sigma\setminus\Gamma$ is homeomorphic to either a disc or to an annulus, and one glues annuli instead of discs to some of the faces of $\Gamma$.
Gluing an annulus to each face of a ribbon graph $\Gamma$ yields a surface $\dot\Sigma_\Gamma$ and can be viewed as a thickening of its edges. This  motivates the term ribbon graph.

In fact, we will often need not just a cyclic ordering but a {\em linear} ordering of the incident edge ends at each vertex of a ribbon graph.  Given the cyclic order, a compatible linear order simply amounts to choosing one edge end to be the least.  
We write  $e<f$  if  $e,f\in E(\Gamma_\circ)$ are edge ends incident at a vertex $v$ and $e$ precedes $f$ with respect to the linear ordering at $v$.  
We indicate this
 linear ordering  in figures by a {\bf cilium} \cite{FR}: a short dotted line between the greatest and least edge ends and call the resulting graphs `ciliated':
\[
 \label{fig:vertex_edgeends}
 \includegraphics[scale=0.15]{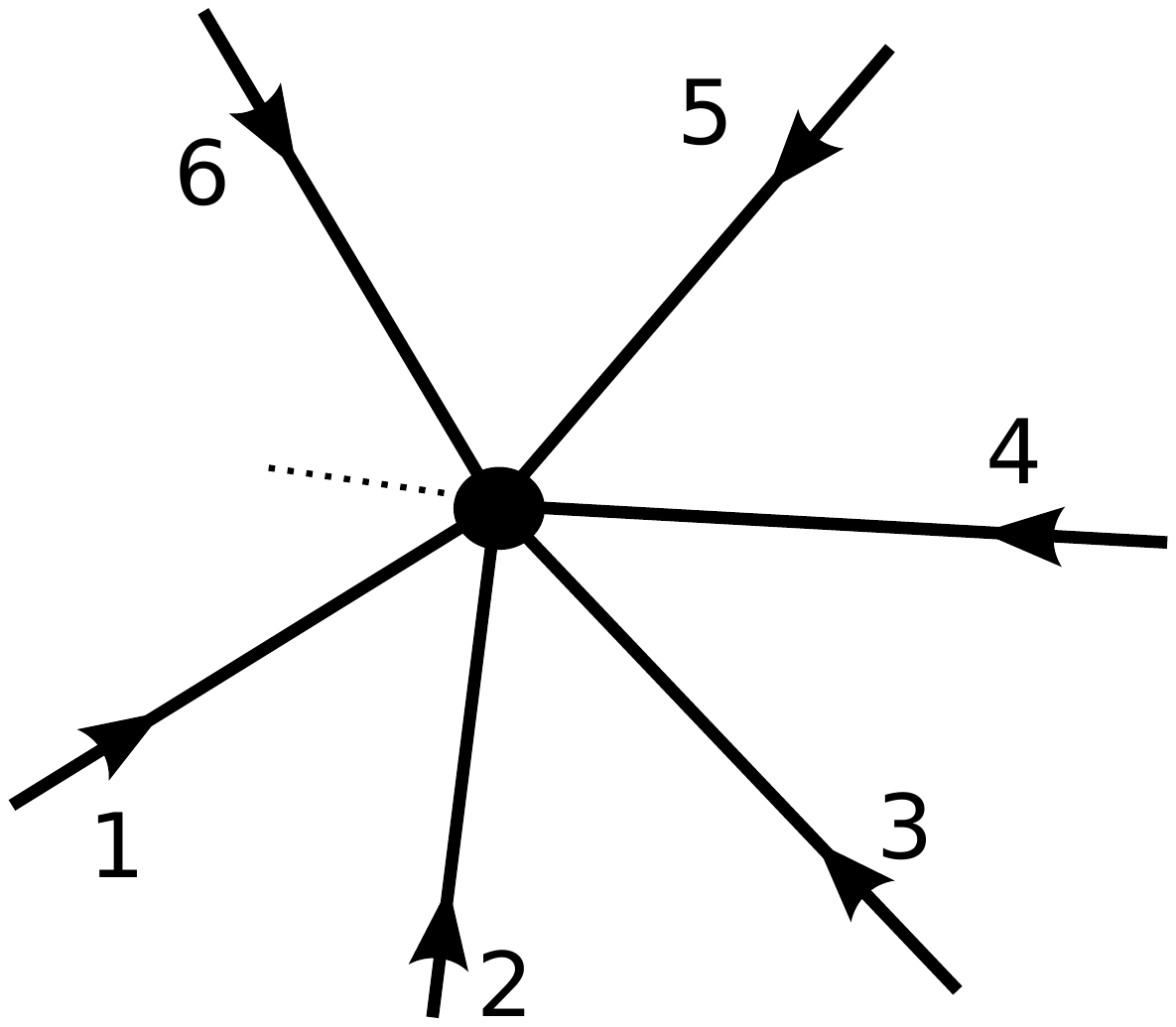}
\]

\begin{definition} \label{def:cilgraph}
A {\bf ciliated ribbon graph} $\Gamma$  is a directed graph together with a  linear ordering of the incident edge ends at each vertex.  Two edge ends  $e,f$  incident at a vertex $v\in V(\Gamma)$  are called {\bf adjacent}  if there is no edge end $g$ incident at $v$ with $e<g<f$  or $f<g<e$. 
\end{definition}

Note that  adjacency of edge ends depends on the ciliation and not just the ribbon graph structure. In particular,  the edge ends 1 and 6 in  the figure above are not adjacent. 
Given a path  in a ciliated ribbon graph $\Gamma$, we may ask if this path 
is well-behaved with respect to the ciliation, i.e.~if it is possible to thicken the graph $\Gamma$ and to draw this path in such a way on the boundary of the thickened graph that it avoids all  cilia.  If the path is a face path, it is sensible to impose that such a condition is also satisfied at the starting vertex and
that the  path on the thickened graph always remains to the left of the edges in $\Gamma$.
These compatibility conditions  can be characterised in terms of the linear ordering at each vertex of $\Gamma$.

\begin{definition} \label{def:cilpath} Let $\Gamma$ be a ciliated ribbon graph.\\[-3ex]
\begin{compactenum}
\item If $p=e_n^{\epsilon_n}\circ ...\circ e_1^{\epsilon_1}$ is a path in $\Gamma$, we say that $p$ {\bf does not traverse any cilia} if for all $i\in\{1,...,n-1\}$ the edge ends $s(e_{i+1}^{\epsilon_{i+1}})$ and $t(e_i^{\epsilon_i})$ are adjacent with respect to the linear ordering at the vertex $\st(e_{i+1}^{\epsilon_{i+1}})=\ta(e_i^{\epsilon_i})$. \\[-1ex]

\item If $f=e_n^{\epsilon_n}\circ \ldots\circ e_1^{\epsilon_i}\in \mathcal G(\Gamma)$ is a face path , we say that $f$ {\bf is compatible with the ciliation} if  $s(e_{i+1}^{\epsilon_{i+1}})<t(e_i^{\epsilon_i})$ for all $i\in\{1,...,n-1\}$ and  $s(e_1^{\epsilon_1})<t(e_n^{\epsilon_n})$
\end{compactenum}
\end{definition}
Note that a face path  that does not traverse any cilia is not necessarily compatible with the ciliation, even if $s(e_1^{\epsilon_1})< t(e_n^{\epsilon_n})$. If  a vertex $\st(e_{i+1}^{\epsilon_{i+1}})=\ta(e_i^{\epsilon_i})$  in $f=e_n^{\epsilon_n}\circ \ldots\circ e_1^{\epsilon_1}$ is bivalent, then the edge ends $s(e_{i+1}^{\epsilon_{i+1}})$ and $t(e_i^{\epsilon_i})$ are always adjacent, but the condition $s(e_{i+1}^{\epsilon_{i+1}})< t(e_i^{\epsilon_i})$ may be violated. At vertices of valence $\geq 3$ the condition that 
$s(e_{i+1}^{\epsilon_{i+1}})$ and $t(e_i^{\epsilon_i})$  are adjacent is equivalent to the $s(e_{i+1}^{\epsilon_{i+1}})< t(e_i^{\epsilon_i})$ for any face path  $f=e_n^{\epsilon_n}\circ \ldots\circ e_1^{\epsilon_1}$.

\subsection{Operations on ribbon graphs}
\label{subsec:graphopsribbon}

\begin{figure}
\centering
\includegraphics[scale=0.7]{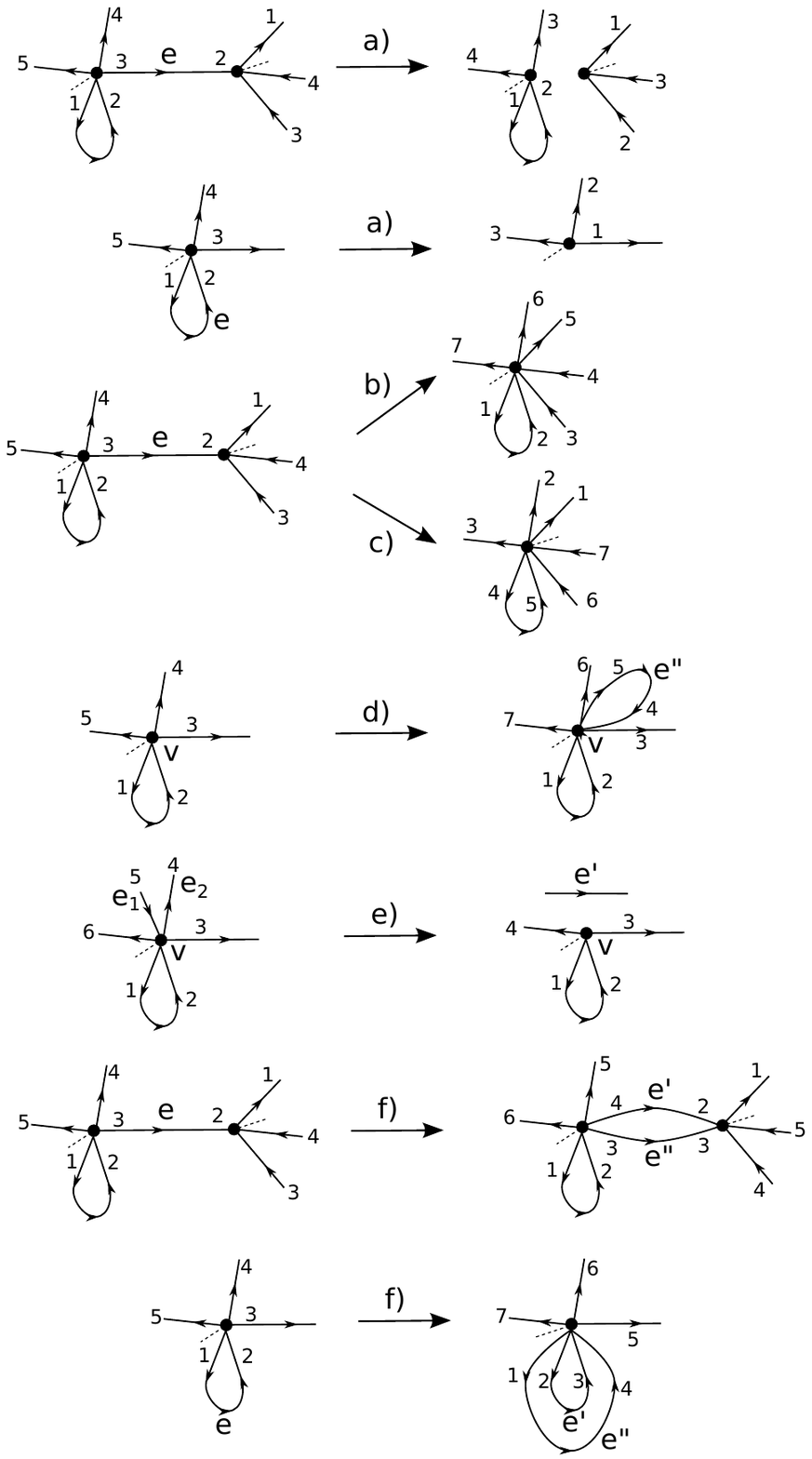}
\caption{Operations on ciliated ribbon graphs:
a) deleting an edge,
b), c) contracting an edge towards a vertex,
d) inserting a loop, e) detaching adjacent edge ends from a vertex, f) doubling an edge.
The dashed lines indicate the cilia and the numbers the linear ordering of the incident edge ends at the vertices.
}
\label{fig:graph_ops}
\end{figure}

There are a number of operations on ciliated ribbon graphs that are compatible with the ciliated ribbon graph structure.  If $\Gamma$ is a ciliated ribbon graph and $\Gamma'$ is obtained from $\Gamma$ by applying one of these operations, then $\Gamma'$ inherits a   ciliated ribbon graph structure from  $\Gamma$.  In the context of gauge theory these were  first considered in  \cite{FR} but they  are also  well-known in other contexts.

\begin{definition}\label{def:graphtrafos} Operations on ciliated ribbon graphs
\begin{compactenum}[a)]
\item {\bf Deleting an edge:}  The graph $\Gamma'$ is obtained from $\Gamma$ by deleting an edge $e\in E(\Gamma)$, as shown in Figure \ref{fig:graph_ops} a). If the starting or target vertex of $e$ is univalent or if $e$ is a loop based at a bivalent vertex, these vertices are also removed.
The orientation of all edges $e'\neq e$ and the ordering  at each vertex $v\notin\{ \st(e), \ta(e)\}$ is preserved. The ordering at the vertices $\st(e)$, $\ta(e)$ is modified as in Figure \ref{fig:graph_ops} a). 
\\[-1ex]

\item{\bf Contracting an edge towards the starting  vertex:}  Let $e\in E(\Gamma)$ be an edge of $\Gamma$ that is not a loop.  The graph $\Gamma'$ is obtained  from $\Gamma$ by  deleting the edge $e$ and its target vertex $\ta(e)$ and inserting the other edge ends incident at $\ta(e)$  between the edge ends at $\st(e)$, as shown in Figure \ref{fig:graph_ops} b). The orientation of all other edges  and the  ordering at all other vertices stays the same.  The ordering at $\st(e)$ is modified as in  Figure \ref{fig:graph_ops} b).
\\[-1ex]

\item{\bf Contracting an edge towards the target  vertex:}  Let $e\in E(\Gamma)$ be an edge of $\Gamma$ that is not a loop. 
 Then $\Gamma'$ is obtained  from $\Gamma$ by  deleting the edge $e$ and its starting vertex $\st(e)$ and inserting the other edge ends  at $\st(e)$  between the edge ends at $\ta(e)$, as shown in Figure \ref{fig:graph_ops} c). The orientation of all other edges  and the  ordering at all other vertices stays the same. The ordering at the vertex  $\ta(e)$ is modified as  in Figure \ref{fig:graph_ops} c).
\\[-1ex]

\item {\bf Inserting a loop:} The graph $\Gamma'$ is obtained from $\Gamma$ by inserting a loop $e''$ at a vertex $v\in V(\Gamma)$ in such a way that $s(e'')$ and $t(e'')$ are adjacent  with $t(e'')<s(e'')$, as shown in Figure \ref{fig:graph_ops} d). The orientation of all edges and the ordering at all vertices $w\neq v$ stays the same,  and the ordering at $v$ is modified as in Figure \ref{fig:graph_ops} d).
\\[-1ex]

 \item {\bf Detaching adjacent edge ends from a vertex:} Let $v$ be a vertex of $\Gamma$ of valence $|v|\geq 3$ and $e_1$, $e_2$ 
 two different edges of $\Gamma$  with $\st(e_2)=\ta(e_1)=v$  such that the edge ends $s(e_2)$ and $t(e_1)$ are adjacent at $v$.
  Then $\Gamma'$ is obtained from $\Gamma$ by disconnecting the edge ends $t(e_1)$ and $s(e_2)$  from $v$ and 
  combining the edges $e_2$ and $e_1$ into a single edge $e'$ with $\st(e')=\st(e_1)$ and $\ta(e')=\ta(e_2)$ with the same orientation, as shown in Figure \ref{fig:graph_ops} e). The orientation of all other edges and the ordering at all other vertices stays the same.  The  ordering at $v$ is modified as in Figure \ref{fig:graph_ops} e).\\[-1ex]

\item {\bf Doubling an edge:} Let $e$ be an edge of $\Gamma$. Then $\Gamma'$ is obtained from $\Gamma$ by replacing $e$ with a pair of edges $e',e''$ with the same orientation such that their edge ends are adjacent at the starting and target vertex with $t(e')<t(e'')$ and $s(e')>s(e'')$, as shown in Figure \ref{fig:graph_ops} f). The ordering of the edge ends at the starting and target vertex of $e$ is modified as in Figure \ref{fig:graph_ops} f). The orientation  of all other edges and the ordering at all other vertices stays the same. 
 \end{compactenum}
\end{definition}

These graph operations give rise to functors between the  path categories and path groupoids of the associated graphs. If  $\Gamma'$ is obtained from $\Gamma$ by  one of the graph operations from Definition \ref{def:graphtrafos}, then there is a canonical functor $\mathcal C(\Gamma')\to\mathcal C(\Gamma)$ associated with this transformation that induces a functor $\mathcal G(\Gamma')\to\mathcal G(\Gamma)$.

For this, note that by definition
a  functor $G:\mathcal C(\Gamma')\to\mathcal C(\Gamma)$  that induces a functor $G:\mathcal G(\Gamma')\to\mathcal G(\Gamma)$ is specified uniquely by a map $g_V: V(\Gamma')\to V(\Gamma)$ and an assignment $f'\mapsto f$ of a morphism $f\in \mathcal C(\Gamma)$ to each edge $f'\in E(\Gamma')$ such that  $\st(f)=g_V(\st(f'))$,  $\ta(f)=g_V(\ta(f'))$
and the edge $f'^\inv$ with the opposite orientation is assigned the reversed path $f^\inv$.
Conversely, any such data gives rise to a functor $G:\mathcal C(\Gamma')\to\mathcal C(\Gamma)$ that induces a functor $G:\mathcal G(\Gamma')\to\mathcal G(\Gamma)$. 

To construct these functors for the graph operations in Definition \ref{def:graphtrafos}, note   these graph operations induce  
canonical maps $g_V: V(\Gamma')\to V(\Gamma)$. The map $g_V: V(\Gamma')\to V(\Gamma)$ is a bijection or an inclusion map in case (a) (the latter if and only if vertices are removed with the edge), an inclusion map in cases (b), (c) and a bijection in cases (d)-(f). The associated functors $G:\mathcal C(\Gamma')\to\mathcal C(\Gamma)$  are then essentially  determined by the conditions  $\st(f)=g_V(\st(f'))$,  $\ta(f)=g_V(\ta(f'))$ 
 and  the condition that they map edges of $\Gamma$ that are not affected by a graph operation to the corresponding edges of $\Gamma'$. 
This leads to the following definition.

\begin{definition}\label{def:graph_functor} Suppose $\Gamma'$ is obtained from $\Gamma$ by one of the graph operations in Definition \ref{def:graphtrafos}. 
Denote for each 
edge $f\in E(\Gamma)$ that is not affected by the graph transformations  by  
$f'$ the  associated edge  
in $\Gamma'$  and label the remaining edges  as in Figure \ref{fig:graph_ops}. 
Then the associated functors $\mathcal C(\Gamma')\to\mathcal C(\Gamma)$ and $\mathcal G(\Gamma')\to\mathcal G(\Gamma)$  are given by  the following assignments of paths in $\Gamma$ to edges  $f'\in E(\Gamma')$: 
\begin{compactenum}[(a)]
\item {\bf Deleting an edge $e$:}  
\begin{align*}
D_e:\;f'\mapsto f\qquad \forall f'\in E(\Gamma').\qquad\qquad\qquad\qquad\qquad
\end{align*}  
\item {\bf Contracting an edge $e$ towards $\st(e)$:} 
\begin{flalign*}
C_{\st(e)}:\; f'\mapsto\begin{cases} f & \ta(e)\notin\{\st(f), \ta(f)\}\\
f\circ e & \st(f)=\ta(e)\neq \ta(f)\\
 e^\inv\circ f & \ta(f)=\ta(e)\neq \st(f)\\
 e^\inv\circ f\circ e & \st(f)=\ta(f)=\ta(e), f\neq e.
\end{cases} &
\end{flalign*}
\item {\bf Contracting an edge $e$ towards $\ta(e)$:}  
\begin{flalign*}
C_{\ta(e)}:\; f'\mapsto\begin{cases} f & \st(e)\notin\{\st(f), \ta(f)\}\\
f\circ e^\inv & \st(f)=\st(e)\neq \ta(f)\\
 e\circ f & \ta(f)=\st(e)\neq \st(f)\\
 e\circ f\circ e^\inv & \st(f)=\ta(f)=\st(e), f\neq e.
\end{cases}
&
\end{flalign*}
\item {\bf Adding a loop $e''$ at $v$:}  
\begin{flalign*}A_{v}:\; f'\mapsto\begin{cases} f  & f'\in E(\Gamma')\setminus\{e''\}\\
\emptyset_v & f''=e''\end{cases}\qquad\qquad\qquad
& \quad
\end{flalign*}

\item {\bf Detaching adjacent  edge ends  from $v$:}  
\begin{flalign*}  W_{e_1e_2}:\; f'\mapsto 
\begin{cases} f  & f'\in E(\Gamma')\setminus\{e'\}\\
e_2\circ e_1 & f'=e'\end{cases}\qquad\qquad\quad
\end{flalign*}

\item {\bf Doubling the edge $e$:}  
\begin{flalign*}Do_e:\;f'\mapsto\begin{cases} f & f'\in E(\Gamma')\setminus\{e',e''\}\\
e & f'\in\{e',e''\}
\end{cases}\qquad\qquad\qquad & \end{flalign*} 
\end{compactenum}
\end{definition}

A certain composite of the graph transformation functors in Definition \ref{def:graph_functor} will play a special role in the following. This  is the functor obtained by taking the edge subdivision $\Gamma_\circ$ of a ribbon graph $\Gamma$ and contracting for each edge $e\in E(\Gamma)$  one of the associated edge ends $s(e), t(e)\in E(\Gamma_\circ)$ towards a vertex in $\Gamma$. The result of this contraction procedure is the  graph $\Gamma$.  It  follows directly from  Definition \ref{def:graph_functor} that the resulting functor  depends neither on the order in which these edge contractions are performed nor on the choice of edge ends that are contracted. Hence, we obtain a unique functor $G_\Gamma:\mathcal C(\Gamma)\to\mathcal C(\Gamma_\circ)$ that  induces a functor $G_\Gamma:\mathcal G(\Gamma)\to\mathcal G(\Gamma_\circ)$.

\begin{definition}\label{def:functor_subdivision} Let $\Gamma$ be a ribbon graph, $\Gamma_
\circ$ its edge subdivision and denote for each edge $e\in E(\Gamma)$ by $s(e), t(e)\in E(\Gamma_\circ)$ the associated edge ends in $\Gamma_\circ$. Then the edge subdivision functor $G_\Gamma:\mathcal C(\Gamma)\to\mathcal C(\Gamma_\circ)$ is given by the inclusion map $\iota_V: V(\Gamma)\to V(\Gamma_\circ)$, $v\mapsto v$ and  the assignment $e\mapsto t(e)\circ s(e)$ for all $e\in E(\Gamma)$. It induces a functor $G_\Gamma:\mathcal G(\Gamma)\to\mathcal G(\Gamma_\circ)$.
\end{definition}

By making use of the functor $G_\Gamma:\mathcal C(\Gamma)\to\mathcal C(\Gamma_\circ)$, we can characterise the functors  $F:\mathcal C(\Gamma')\to\mathcal C(\Gamma)$ from Definition \ref{def:graph_functor} in terms of  functors $F_\circ: \mathcal C(\Gamma'_\circ)\to \mathcal C(\Gamma_\circ)$ between the path categories of the associated edge subdivisions.

\begin{proposition} \label{lem:graphtrafo_vertnb}For each of the functors $F:\mathcal C(\Gamma')\to \mathcal C(\Gamma)$  from Definition \ref{def:graph_functor} there is a  functor  $F_\circ:\mathcal C(\Gamma'_\circ)\to \mathcal C(\Gamma_\circ)$ such that the following diagram commutes
\begin{align*}
\xymatrix{ \mathcal C(\Gamma') \ar[r]^{F} \ar[d]^{G_{\Gamma'}} & \mathcal C(\Gamma)  \ar[d]^{G_\Gamma}\\
\mathcal C(\Gamma'_\circ) \ar[r]_{F_\circ}& \mathcal C(\Gamma_\circ).
}
\end{align*}
The functors $F_\circ$ are given by  canonical  maps  $g_{V\circ}: V(\Gamma'_\circ)\to V(\Gamma_\circ)$ and the following assignments of paths in $\Gamma_\circ$ to edge ends in $\Gamma'$:  
\begin{compactenum}[(a)]
\item {\bf Deleting an edge $e$:}  
\begin{align*}  D_{e\;\circ}:\;f'\mapsto f\qquad\forall f'\in E(\Gamma'_\circ).\qquad\qquad\qquad\qquad\qquad\qquad\qquad\qquad\quad\end{align*}  
\item {\bf Contracting an edge $e$ towards $\st(e)$:} 
$$C_{\st(e)\,\circ}:\;f'\mapsto
\begin{cases}
f & \ta(e)\notin\{\st(f),\ta(f)\}\\
f\circ t(e)\circ s(e) &f\in E(\Gamma_\circ)\setminus\{t(e), s(e)\}, \st(f)=\ta(e)\\
 s(e)^\inv\circ t(e)^\inv\circ f & f\in E(\Gamma_\circ)\setminus\{s(e), t(e)\}, \ta(f)=\ta(e)
\end{cases}
$$

\item {\bf Contracting an edge $e$ towards $\ta(e)$:}  
$$C_{\ta(e)\,\circ}:\;f'\mapsto
\begin{cases}
f & \st(e)\notin\{\st(f),\ta(f)\}\\
 f\circ s(e)^\inv\circ t(e)^\inv  &f\in E(\Gamma_\circ)\setminus\{t(e), s(e)\}, \st(f)=\st(e)\\
 t(e)\circ s(e)\circ f & f\in E(\Gamma_\circ)\setminus\{s(e), t(e)\}, \ta(f)=\st(e)
\end{cases}
$$

\item {\bf Adding a loop $e''$ at $v$:}  
$$A_{v\,\circ}:\;f'\mapsto
\begin{cases}
f & f'\in E(\Gamma'_\circ)\setminus\{s(e''), t(e'')\}\\
\emptyset_v & f'\in\{s(e''),t(e'')\}
\end{cases}\qquad\qquad\qquad\qquad\qquad\quad
$$

\item {\bf Detaching adjacent edge ends $e_1, e_2$ from $v$:} 
\begin{align*}
W_{e_1e_2\,\circ}:\;f'\mapsto 
\begin{cases}
f & f'\in E(\Gamma')\setminus \{s(e')\}\\
s(e_2)\circ t(e_1)\circ s(e_1) & f'=s(e')
\end{cases}\qquad\qquad\quad
\end{align*}

\item {\bf Doubling the edge $e$:}  
$$
Do_{e\,\circ}:\; f'\mapsto \begin{cases}
f & f'\in E(\Gamma'_\circ)\setminus\{s(e''), t(e''), s(e'), t(e')\}\\
t(e) & f'\in\{t(e''), t(e')\}\\
s(e) & f'\in\{s(e''), s(e')\}
\end{cases}\qquad\qquad\quad\;
$$
\end{compactenum}
\end{proposition}
\begin{proof} Let $\Gamma'$ be obtained from $\Gamma$ by  one of the graph operations in Definition \ref{def:graphtrafos} and
denote for each edge $f\in E(\Gamma)$ or  $f'\in E(\Gamma')$ by $m(f)$ and $m(f')$, respectively,   the bivalent vertex of  $V(\Gamma_\circ)$ or 
$V(\Gamma'_\circ)$ at the midpoint of $f$ or $f'$. Define  $g_{V_\circ}: V(\Gamma'_\circ)\to V(\Gamma_\circ)$ by $g_V(m(f'))=m(F(f'))$ for all $f'\in E(\Gamma'_\circ)$ and $g_{V_\circ}(v')=g_V(v')$ for all $v'\in V(\Gamma')$, where   $F: \mathcal C(\Gamma')\to\mathcal C(\Gamma)$ is the associated functor from Definition \ref{def:graph_functor}. Then one has $g_{V_\circ}\vert_{V(\Gamma')}=g_V$, and a  short computation shows 
 that  the expressions in Definition \ref{def:graph_functor} and Proposition \ref{lem:graphtrafo_vertnb} imply $F_\circ(t(f')\circ s(f'))=t(F(f'))\circ s(F(f'))$ for each edge $f'\in E(\Gamma')$.
\end{proof}

Note that some of the functors in Definition \ref{def:graph_functor} have (strict) right or left inverses. The contraction functors $C_{\st(e)}$ and $C_{\ta(e)}$ from Definition \ref{def:graph_functor} (b) and (c) have left inverses. For $C_{\st(e)}$, the left inverse is  given by $g: V(\Gamma)\to V(\Gamma')$, $\ta(e)\mapsto \st(e)$, $v\mapsto v$ for $v\in V(\Gamma)\setminus\{\ta(e)\}$ and the assignment $f\mapsto f'$ for all $f\in E(\Gamma)\setminus\{e\}$, $e\mapsto \emptyset_{\st(e)}$. For $C_{\ta(e)}$, it is  given by $g: V(\Gamma)\to V(\Gamma')$, $\st(e)\mapsto \ta(e)$, $v\mapsto v$ for $v\in V(\Gamma)\setminus\{\st(e)\}$  and the assignment $f\mapsto f'$ for all $f\in E(\Gamma)\setminus\{e\}$, $e\mapsto \emptyset_{\ta(e)}$.  The functors  $A_v$  and $Do_{e}$ from Definition \ref{def:graph_functor}  (d)  and (f) have a right inverses. The right inverse of  $A_v$ is the functor $D_{e''}$ that deletes the loop $e''$. The right inverses of the edge  doubling functor $Do_e$  are the functors $D_{e'}$ and $D_{e''}$ that delete the edges $e'$ or $e''$. The functor $D_e$ has a left inverse if and only if $e$ is either a loop or if $e$ is part of an edge pair as in Figure \ref{fig:graph_ops} (f), namely the functors $A_v$ or $Do_e$, respectively. 
Otherwise it has neither a left nor a right inverse.
The detaching functor $W_{e_1e_2}$  from Definition \ref{def:graph_functor} (e) has neither a right nor a left inverse.

The graph operations in Definition \ref{def:graphtrafos} and the associated functors in Definition \ref{def:graph_functor} are not independent but exhibit relations.  In addition to the relations involving their left or right inverses, these include the relation
 depicted in Figure \ref{fig:graphops_kombi} that relates the detaching of adjacent edge ends from a trivalent vertex,  a contraction towards the starting vertex   and the deleting of edges.

The graph operations in Definition \ref{def:graphtrafos} and the associated functors in Definition \ref{def:graph_functor} commute in the obvious way, whenever their composition is defined. For instance, by  contracting several  different edges of $\Gamma$ one obtains the same ribbon graph $\Gamma'$, independently of the order of the contractions.
 Similarly, edges can be removed  and loops can be added in any order, and these operations do not affect each other
as long as their composition is possible.
They  are distinguished from other possible operations insofar as they induce canonical functors between the path groupoids and that they are compatible with the ciliated ribbon graph structure: if  $\Gamma'$ is obtained from a ciliated ribbon graph $\Gamma$ by one of the graph operations in Definition \ref{def:graphtrafos},  then $\Gamma'$ inherits a ciliated ribbon graph structure from $\Gamma$, as indicated in Figure \ref{fig:graph_ops}.

In Section \ref{sec:graphops} we will show that these graph operations give rise to homomorphisms of module algebras between Hopf algebra gauge theories on $\Gamma'$ and $\Gamma$. This allows one to determine   how the Hopf algebra gauge theories depend on the choice of the  ciliated ribbon graph.
For this, it is important to note 
that graph operations in Definition \ref{def:graphtrafos}  are  complete in the sense that they allow one to relate any two ribbon graphs that can be embedded into a given surface, to describe connected sums of surfaces and to describe simple paths on the surfaces associated with a ribbon graph $\Gamma$. More specifically, the geometrical 
 applications of these graph operations are the following:
 
   \begin{figure}
\centering
\includegraphics[scale=0.4]{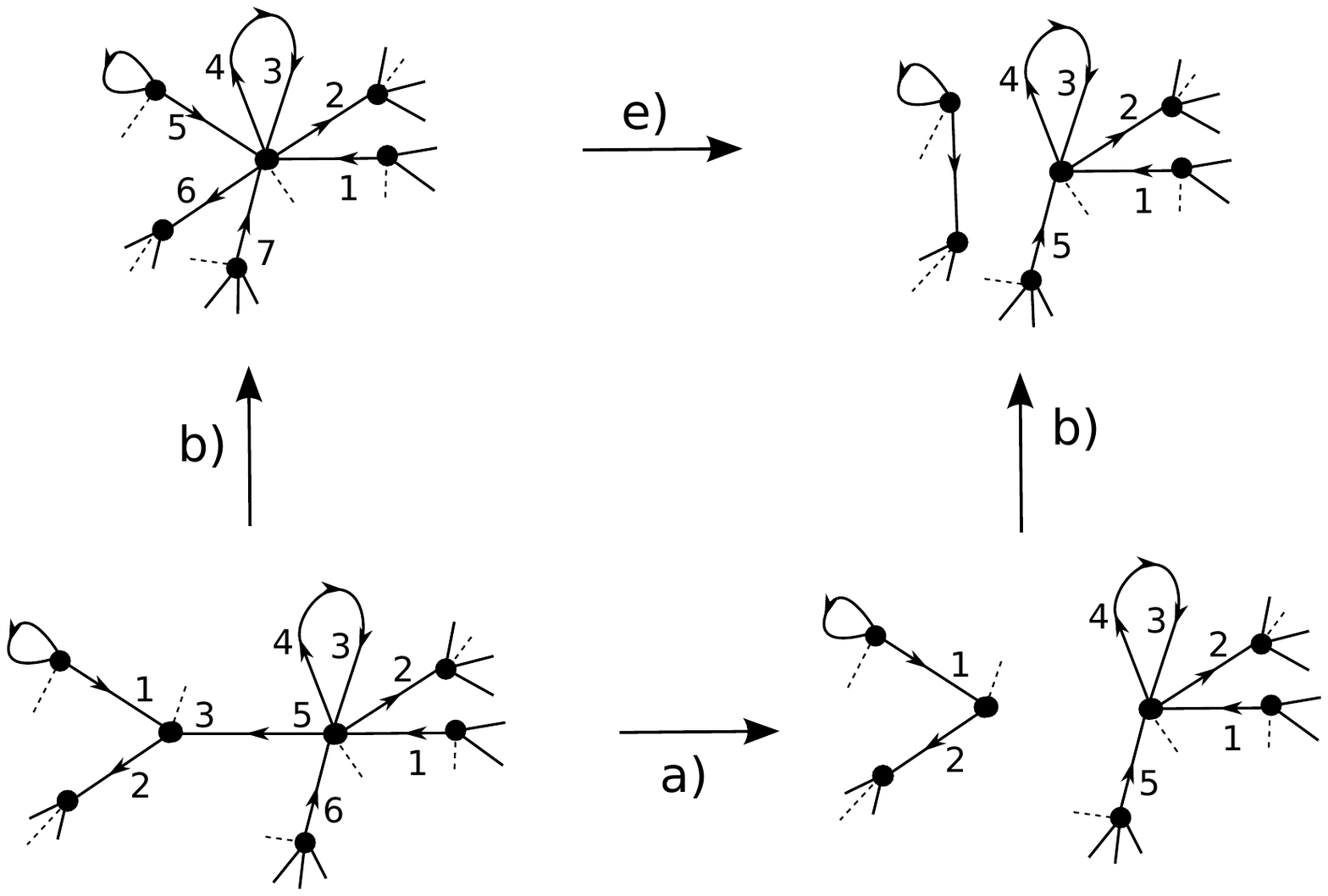}
\caption{A relation between graph operations.}
\label{fig:graphops_kombi}
\end{figure}

\begin{figure}
\centering
\includegraphics[scale=0.45]{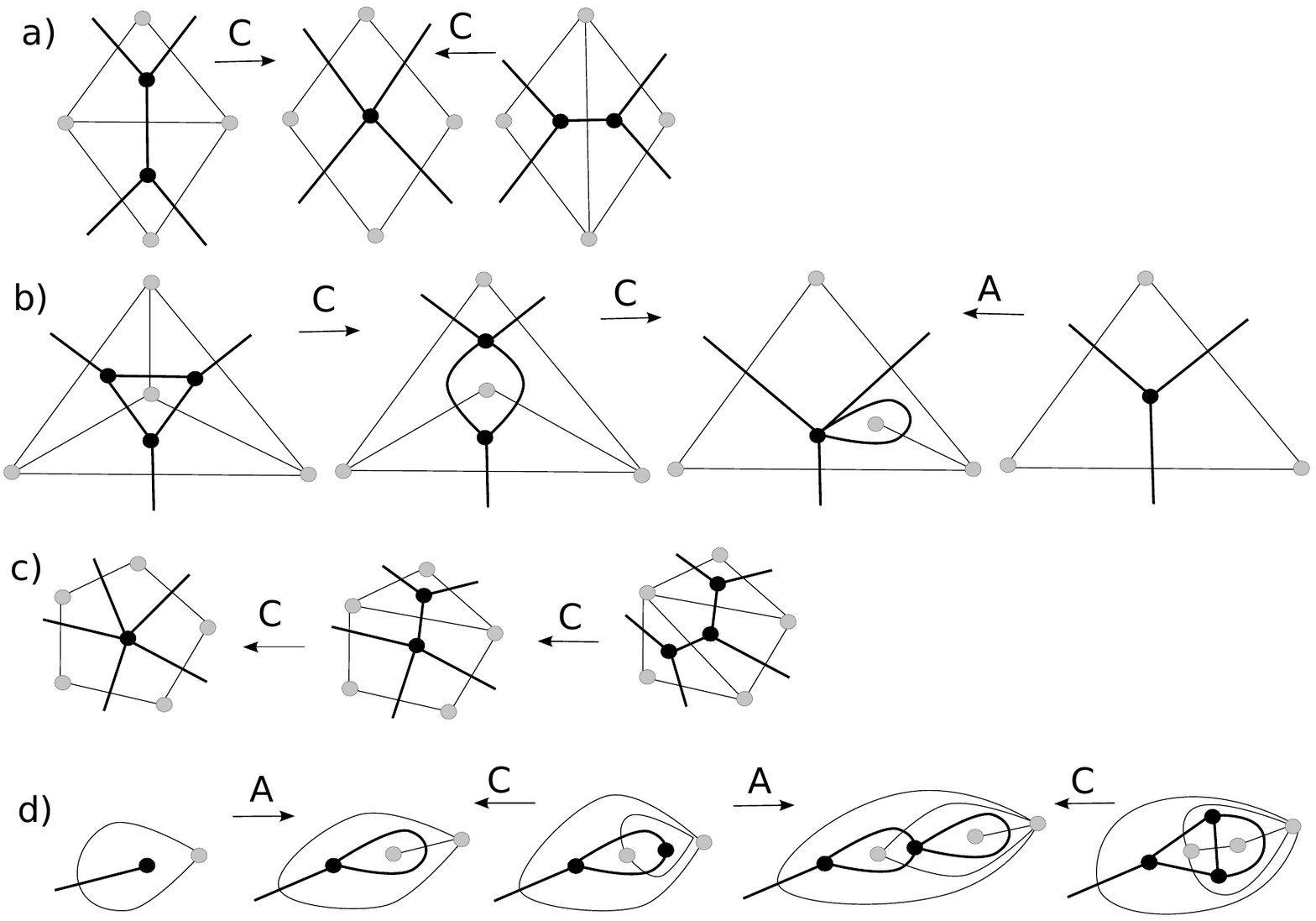}
\caption{Transforming a ribbon graph by graph transformations.\newline
a) The 2-2 Pachner move, b) the 3-1 Pachner move, c) Splitting a vertex into 3-valent vertices, d) Transforming an edge with a univalent vertex into a 3-valent graph without loops.}
\label{fig:pmoves}
\end{figure}
 
The operation of deleting edges allows one to construct subgraphs of $\Gamma$.  It is also  related to the connected sum of surfaces.
Suppose $\Gamma$ is a connected  (ciliated) ribbon graph such that erasing an edge $e\in E(\Gamma)$ yields a (ciliated) ribbon graph that is the topological sum $\Gamma'\coprod\Gamma''$ of two connected components. Then the surface $\Sigma_\Gamma$ obtained by gluing discs to the faces of $\Gamma$ is 
 the connected  sum $\Sigma_\Gamma=\Sigma_{\Gamma'}\#\Sigma_{\Gamma''}$ of the corresponding surfaces for $\Gamma'$ and $\Gamma''$.

The operation of contracting edges reduces the number of vertices in a (ciliated) ribbon graph $\Gamma$. Moreover,
if $\Gamma'$ is obtained from $\Gamma$ by an edge contraction, then the surfaces $\dot \Sigma_\Gamma$ and $\dot \Sigma_{\Gamma'}$ obtained by gluing annuli to all faces of $\Gamma$ and $\Gamma'$ are homeomorphic. 
 In particular, by selecting a rooted tree $T\subset\Gamma$ and contracting all edges of $T$, one can transform any connected (ciliated) ribbon graph $\Gamma$ into a bouquet, i.e.~a (ciliated) ribbon graph with a single vertex. 
The loops of this bouquet are a set of generators of the fundamental group $\pi_1(\dot \Sigma_\Gamma)$.
Moreover, by contracting for each edge $e\in E(\Gamma)$ one of the edge ends $s(e), t(e)\in E(\Gamma_\circ)$  in the edge subdivision $\Gamma_\circ$ towards a vertex in $\Gamma$, one obtains the graph $\Gamma$. Hence every (ciliated) ribbon graph $\Gamma$ is obtained from a (ciliated)  ribbon graph without loops or multiple edges by a finite number of edge contractions. 

\begin{figure}
\centering
\includegraphics[scale=0.4]{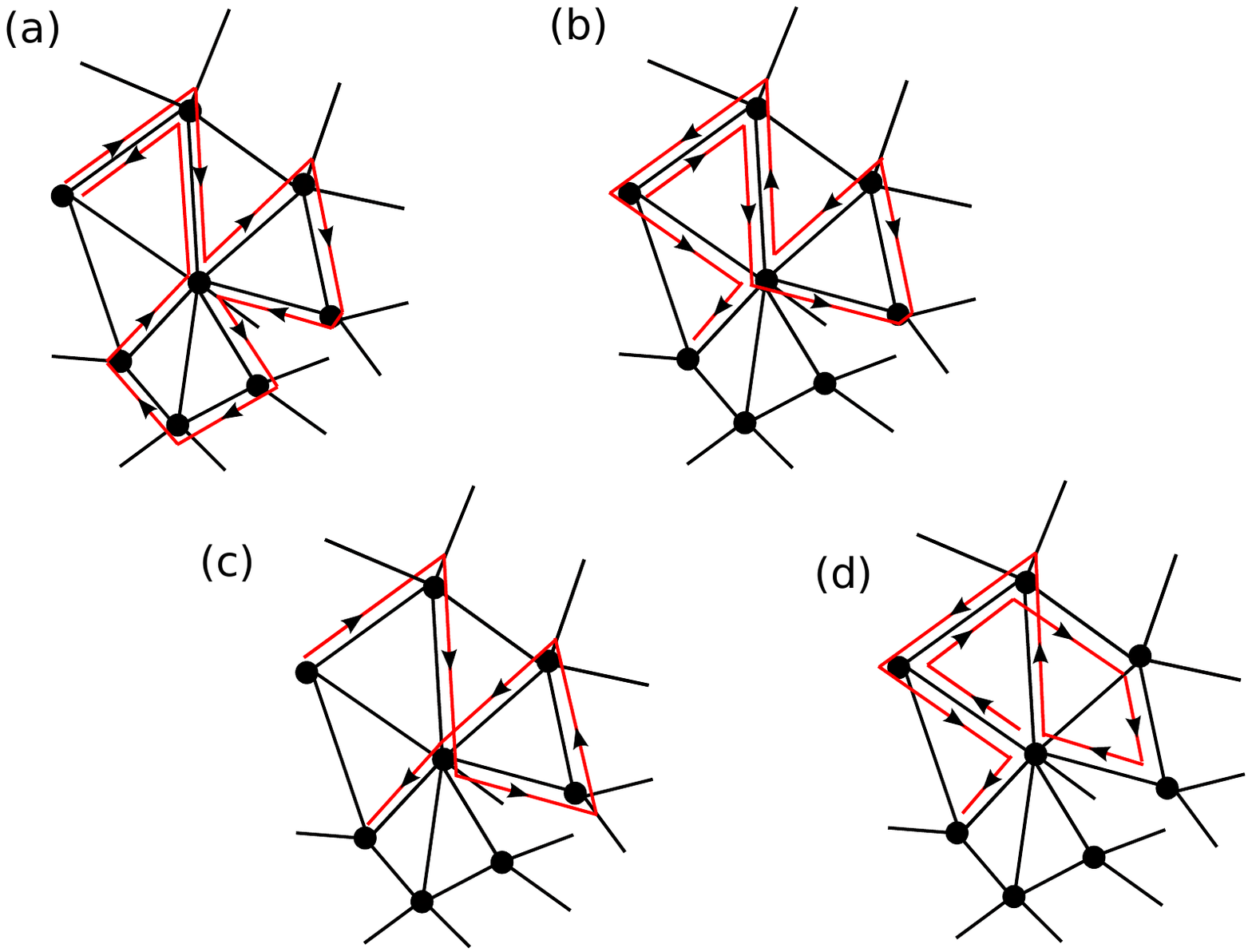}
\caption{(a),(b) regular paths, (c),(d) non-regular paths in a ribbon graph $\Gamma$. Edge orientation is omitted.}
\label{fig:nonreg_paths}
\end{figure}

\begin{figure}
\centering
\includegraphics[scale=0.4]{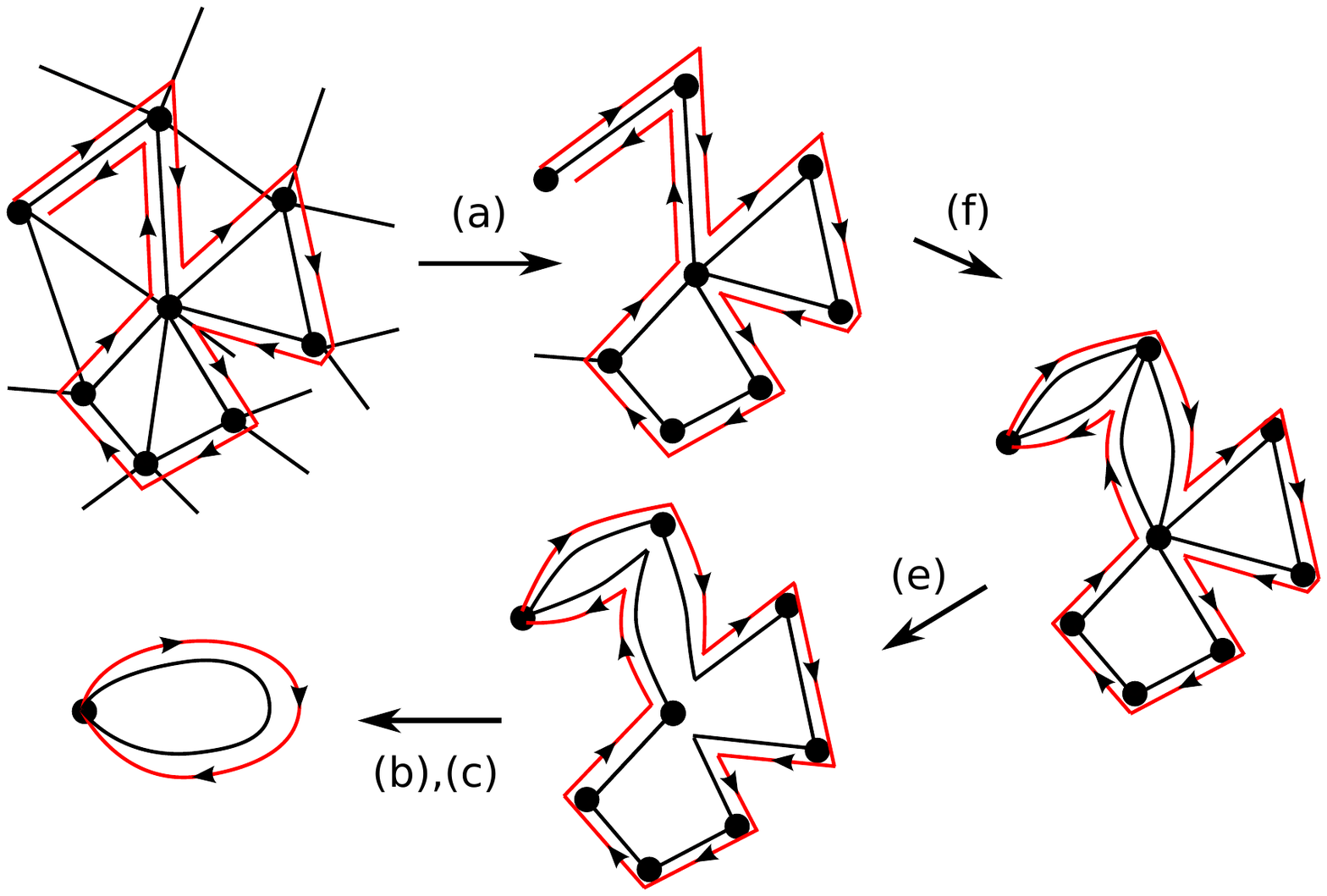}
\caption{Transforming a regular path into a path with only bivalent vertices in which each edge is traversed exactly once. Edge orientation is omitted.}
\label{fig:reg_paths}
\end{figure}

 Together, the operations of contracting edges and adding loops and their left and right inverses allow one to relate any two  ribbon graphs that can be embedded into the same compact oriented surface. 
 
 \begin{proposition}\label{prop:addcont}  $\quad$ 
 \begin{enumerate}
 \item Let $\Gamma,\Gamma'$ be  ribbon graphs such that attaching discs to all faces of $\Gamma$ and $\Gamma'$ yields homeomorphic surfaces $\Sigma_\Gamma$ and $\Sigma_{\Gamma'}$. Then $\Gamma$ and $\Gamma'$ are related by sequences of edge contractions and adding or deleting loops. 
 \item Let $\Gamma, \Gamma'$ ribbon graphs such that attaching annuli to all faces of $\Gamma$ and $\Gamma'$ yields homeomorphic surfaces $\dot \Sigma_\Gamma$ and $\dot \Sigma_{\Gamma'}$. Then $\Gamma$ and $\Gamma'$ are related by sequences of edge contractions.
 \end{enumerate}
 \end{proposition}

 \begin{proof} 
1.~By applying edge contractions as in Figure \ref{fig:pmoves} c) and contracting edges with bi- and univalent vertices,  it is possible to transform  every ribbon graph into a 3-valent   ribbon graph.  A 3-valent  ribbon graph  is dual to a (degenerate) oriented triangulation of  the associated compact surface $\Sigma_\Gamma$. It is shown in \cite{P} that any two triangulations are related by a finite sequence of the Pachner moves shown in Figure \ref{fig:pmoves} a), b). The 2-2 Pachner move in Figure \ref{fig:pmoves} a) acts on the dual ribbon graph by  contracting an edge between two 3-valent vertices and then expanding the resulting vertex. The 3-1 Pachner move  in Figure \ref{fig:pmoves} b)  acts on the dual ribbon graph  by  contracting edges and removing a loop. 
This shows that any two  ribbon graphs embedded into a given compact surface $\Sigma$ such that $\Sigma\setminus \Gamma$ and $\Sigma\setminus \Gamma'$ are disjoint unions of discs,  are related by a finite number of edge contractions and adding or removing loops. 
 
 2.~ If the surfaces $\dot\Sigma_\Gamma$ and $\dot\Sigma_{\Gamma'}$ are homeomorphic, then the associated triangulations are related by an action of the mapping class group and hence can be transformed into each other via the  Whitehead move or  2-2-Pachner move  in Figure \ref{fig:pmoves} a) (for a detailed discussion of  mapping class group actions on triangulated punctured surfaces, see for instance \cite[Chapters 1 and 5]{Pe}). As the 2-2 Pachner move involves only edge contractions, this shows that $\Gamma$ and $\Gamma'$ are related by edge contractions.
 \end{proof}

The operations of detaching adjacent edge ends from a vertex and doubling edges become important when we consider paths on the associated surfaces $\Sigma_\Gamma$ and $\dot\Sigma_\Gamma$. They allow one to construct paths in a (ciliated) ribbon graph $\Gamma$ that represent simple paths on $\Sigma_\Gamma$ and $\dot\Sigma_\Gamma$.
A closed path $p\in\mathcal G(\Gamma)$ represents the free homotopy class of a simple path on $\dot \Sigma_\Gamma$, i.e.~of an injective continuous map  $\gamma: S^1\to \dot \Sigma_\Gamma$,  if and only if it can be transformed into a path $p'\in\mathcal G(\Gamma')$ that traverses only bivalent vertices  of $\Gamma'$ and traverses each edge at most once by applying finitely many edge deletions,  edge doublings and detaching finitely many adjacent edge ends  from vertices.
This follows because any simple path $\gamma: S^1\to \dot \Sigma_\Gamma$ can be transformed into  a path that is homotopic to such a path $p'$  by enlarging the holes in the annuli of $\dot\Sigma_\Gamma$ and pushing  $\gamma$ towards $\Gamma$. Conversely, the procedures of detaching edge pairs from a vertex and doubling edges that are traversed several times by a path $p\in \mathcal G(\Gamma)$ associate to $p\in\mathcal G(\Gamma)$ a path $p'\in\mathcal G(\Gamma')$ that has the same homotopy class as $p$ in $\pi_1(\dot \Sigma_\Gamma)$.  It is also clear that a path in an embedded graph $\Gamma'$ that traverses only bivalent vertices and traverses each edge at most once cannot have any self-intersections and hence is simple.

\begin{definition} \label{def:regular}Let $\Gamma$ be a ribbon graph. A path  $p$  in $\Gamma$ is called {\bf regular} if 
there is a ribbon graph $\Gamma'$ such that each vertex of $\Gamma'$ is at most bivalent and a path $p'\in \mathcal G(\Gamma')$
with the following properties:
\begin{compactenum}[(i)]
\item $\Gamma'$ is  obtained from $\Gamma$ by deleting edges that do not occur in $p$, doubling edges in $p$ and detaching adjacent edge ends  in $p$. 
\item $F(p')=p$, where  $F:\mathcal G(\Gamma')\to\mathcal G(\Gamma)$ is the functor from Definition \ref{def:graph_functor} associated with these graph transformations and $p'$ traverses each edge of $\Gamma'$ exactly once.
\end{compactenum}
 \end{definition}

Examples of regular and non-regular paths are shown in Figure \ref{fig:nonreg_paths}. The transformation of a regular path $p\in \mathcal G(\Gamma)$ into the associated path $p'\in\mathcal G(\Gamma')$  is shown in Figure \ref{fig:reg_paths}.
Note that any face path $f=e_n^{\epsilon_n}\circ\ldots\circ e_1^{\epsilon_1}$ in a ciliated ribbon graph $\Gamma$ that is compatible with the ciliation  is regular. 
If one doubles each edge of $f$ that is traversed twice and then selects a path $f'$ in the resulting ribbon graph that always traverses the left of the two resulting edges, viewed in the direction of $f$, then the
 conditions $s(e_{i+1}^{\epsilon_{i+1}})< t(e_i^{\epsilon_i})$ and $s(e_1^{\epsilon_1})<t(e_n^{\epsilon_n})$ in Definition \ref{def:cilpath} ensure that 
two consecutive edge ends in $f'$ can always be detached from their common vertex.

\section{Hopf algebra gauge theory on a ciliated ribbon  graph}
\label{sec:gtheory}

In this section, we introduce  local Hopf algebra gauge theories on  (ciliated) ribbon graphs $\Gamma$ with values in a Hopf algebra $K$. 
We start by  characterising Hopf algebra gauge theories in terms of certain axioms. These axioms are sufficient  to obtain (i) a notion of connection or gauge field,  (ii) a notion of an algebra of functions on connections, (iii) a notion of gauge transformations acting on connections and by duality on functions, and (iv) an algebra of gauge invariant observables, all subject  to  certain 
 locality conditions. 
As  these axioms  generalise  the axioms for  lattice gauge theory with values in a group, we start with a summary of the latter in Section \ref{subsec:groupgtheory} and then generalise this description   to Hopf algebras in Sections  \ref{subsec:haxioms} to \ref{subsec:algfunc}.

\subsection{Notation and conventions}
\label{subsec:conventions}

In the following, we consider finite-dimensional Hopf algebras 
over a field $\FF$ of characteristic zero.  Some basic facts about Hopf algebras and about module algebras over Hopf algebras are collected in  appendices \ref{sec:hopfalg} and \ref{sec:modulealg}.  
Throughout the article, we  use Sweedler notation  without summation signs, writing
$\Delta(h)=\low h 1\oo\low h 2$
for the comultiplication $\Delta: H\to H\oo H$  of a Hopf algebra $H$.  We
also use this notation for other elements of $H\oo H$, e.g.~$R=\low R 1\oo\low R 2$ for an $R$-matrix. We denote by
$H^{op}$ and $H^{cop}$, respectively, the Hopf algebra  with the opposite multiplication and comultiplication and by $H^*$ the dual Hopf algebra.
Unless specified otherwise, we use Latin letters for elements of $H$ and Greek letters for  elements of $H^*$. 
 The pairing between $H$ and $H^*$ is denoted 
 $\langle\;,\;\rangle: H^*\oo H\to\FF$, $\alpha\oo h\mapsto \alpha(h)$, and 
the same notation is used for the induced pairing on tensor products.

For a  vector space $V$  and $n\in\NN$ we denote by  $V^{\oo n}$ the $n$-fold tensor product of $V$ with itself. If $V$ is also an algebra with unit $1$, then for $i_1,...,i_k\in\{1,...,n\}$ pairwise distinct and $v^1,...,v^k\in V$, we denote by
$(v^1\oo v^2\oo ... \oo v^k)_{i_1...i_k}$ the element of $V^{\oo n}$ that has the entry $v^j$ in $i_j$th  component for $i\in\{1,...,k\}$
and 1 in all other components.  Similarly, we denote by   $\iota_{i_1...i_k}$ the injective linear map 
 $\iota_{i_1...i_k}: V^{\oo k}\to V^{\oo n}$, $v^1\oo\ldots\oo v^k\mapsto (v^1\oo\ldots\oo  v^k)_{i_1...i_k}$.
Dually, if $V$ is a coalgebra with counit $\epsilon$, we denote by $\pi_{i_1...i_k}: V^{\oo n}\to V^{\oo k}$ the map that is the identity on the entries $i_1,...,i_k$ and  applies the counit  $\epsilon$ 
to all other entries.

In Section \ref{subsec:groupgtheory}, we  use analogous notation for  the $n$-fold direct product  $G^{\times n}$ of a group $G$ with itself, writing 
$(g^1, g^2, ... ,g^k)_{i_1...i_k}$
for an element of $G^{\times n}$ with the entry  $g^j\in G$ in the $i_j$th position  and the identity element in all other arguments.  We also consider the injective group homomorphisms
$\iota_{i_1...i_k}: G^{\times k}\to G^{\times n}$, $(g^1,...,g^k)\mapsto (g^1,...,g^k )_{i_1...i_k}$  and the surjective group homomorphisms \linebreak $\pi_{i_1...i_k}: G^{\times n}\to G^{\times k}$, $(g_1,...,g_n)\mapsto (g_{i_1},...,g_{i_k})$. 

As an extension of this notation, when $X$ is a set and $|X|$ its cardinality, rather than writing $V^{\oo |X|}$ we simply write $V\exoo{X}$ for the tensor power of $V$ with one factor for each element of $X$.  We do this especially when $X$ is the set $E$ of edges or the set $V$ of vertices in a ribbon graph.  When $v\in V(\Gamma)$ is a vertex in a ribbon graph, we also write $V\exoo{v}$ to denote the tensor power with factors indexed by the {\em edge ends} at $v$.  This notation is justified by identifying $v$ with its  set of edge ends.

\subsection{Graph gauge theory for a group}
\label{subsec:groupgtheory}

In its most basic version,  a   lattice gauge theory on a directed graph $\Gamma$ with values in a group $G$  involves (i) a set of  {\em connections} or {\em gauge fields}, (ii) an {\em algebra of functions} from set of connections into a field $\FF$ and (iii) a group of {\em gauge transformations}. 
Connections  are  assignments of a group element   $g_e\in G$ to each oriented edge $e\in E$ and hence can be identified with elements of the set $G\extimes{E}$. Functions on connections  are   maps
 $G\extimes{E}\to \FF$. They  form an algebra with respect to  pointwise multiplication.
Gauge transformations  are assignments 
of a group element $g_v\in G$ to each vertex $v\in V$. Their  composition  is given by the multiplication of $G$ at each vertex $v\in V$.  Hence the group of gauge transformations can be identified with the group $G\extimes{V}$.

The action of gauge transformations  on connections and functions is given by a left action  $\rhd: G\extimes{V}\times G\extimes{E}\to G\extimes{E}$ and  the associated right action $\lhd^*: \mathrm{Fun}(G\extimes{E})\times G\extimes{V} \to \mathrm{Fun}(G\extimes{E})$  defined by $( f\lhd^*h)(g)=f(h\rhd g)$ for all $h\in G\extimes{V}$, $g\in G\extimes{E}$ and $f\in \mathrm{Fun}(G\extimes{E})$.   
These group actions are required to be {\em local}
in the sense that a gauge transformation at a vertex $v\in V$ acts non-trivially only  on the group elements of  edges incident at $v$. It is given by left multiplication, right multiplication with the inverse and conjugation, respectively, for incoming edges, outgoing edges  and loops at $v$. 
Note that these requirements are consistent with a reversal of the edge orientation  if and only if  one assigns to the reversed edge  the inverse group element. Hence the reversal of edge orientation is implemented by taking the inverse in the group.

Their behaviour with respect to gauge transformations  distinguishes certain functions $f\in \mathrm{Fun}(G\extimes{E})$, namely  the
 {\em gauge invariant} functions  or {\em observables} of the theory. These are functions $f\in \mathrm{Fun}(G\extimes{E})$ with  $f\lhd^* h=f$ for all $h\in G\extimes{V}$. As one has by definition $(f\cdot f')\lhd^*h=(f\lhd^* h)\cdot (f'\lhd^* h)$, the gauge invariant functions form a subalgebra $\mathrm{Fun}_{inv}(G\extimes{E})\subset \mathrm{Fun}(G\extimes{E})$. Summarising these considerations and results, one obtains the following definition of a lattice gauge theory with values in a group $G$.

\begin{definition} \label{def:grgt}Let $\Gamma$ be a directed graph,  $G$ a group. A $G$-valued gauge theory on $\Gamma$  consists of:
\begin{compactenum}
\item The set $G\extimes{E}$ of  {\bf connections}.
\item The algebra  $\mathrm{Fun}(G\extimes{E})$ of {\bf functions} $f: G\extimes{E}\to\FF$  with the pointwise multiplication.
\item The group $G\extimes{V}$ of {\bf gauge transformations}.
\item A group action $\rhd: G\extimes{V}\times G\extimes{E}\to G\extimes{E}$ 
and the dual action $\lhd^*: \mathrm{Fun}(G\extimes{E})\times  G\extimes{V} \to \mathrm{Fun}(G\extimes{E})$ given by $(f\lhd^* h)(g)=f(h\rhd g)$ for all $g\in G\extimes{E}$ and $h\in G\extimes{V}$ such that:
\begin{align*}
&\pi_e((h)_v\rhd(g)_e)=g\quad v\notin\{\st(e),\ta(e)\}, & &\pi_e((h)_v\rhd (g)_e)=h\cdot g\cdot h^\inv \qquad \text{for } \st(e)=\ta(e)=v\\
&\pi_e((h)_{\ta(e)}\rhd (g)_e)=h\cdot g , & &\pi_e((h)_{\st(e)}\rhd (g)_e)=g\cdot h^\inv \qquad\; \text{for }\st(e)\neq \ta(e).
\end{align*}

\end{compactenum}
A function $f: G\extimes{E}\to\FF$ is called {\bf gauge invariant} if $f\lhd^* h=f$ for all $h\in G\extimes{V}$. The subalgebra $\mathrm{Fun}_{inv}(G\extimes{E})\subset \mathrm{Fun}(G\extimes{E})$ of gauge invariant functions is called {\bf algebra of observables}.
\end{definition}

In group-valued lattice  gauge theory, the holonomy assigns to each path $p\in \mathcal G(\Gamma)$ a map $\hol_p: G\extimes{E}\to G$ taking connections to group elements.
 If $p=e_n^{\epsilon_n}\circ ...\circ e_1^{\epsilon_1}$,  then  $\hol_p(g_1,...,g_{|E|})=g_{e_n}^{\epsilon_n}\circ ...\circ g_{e_1}^{\epsilon_1}$ and if  $p=\emptyset_v$ with $v\in V$, then  $\hol_{p}(g_1,...,g_{|E|})=1$. This assignment satisfies  $\hol_{q\circ p}=\hol_q\cdot\hol_p$, $\hol_{p\circ \emptyset_u}=\hol_{\emptyset_v\circ p}=\hol_p$ and $\hol_{p^\inv\circ p}=\hol_{p\circ p^\inv}=1$ for all paths $p$ from $u$ to $v$ and all paths $q$ from $v$ to $w$. In other words, if we equip the set $\mathrm{Fun}(G\extimes{E}, G)$ with a group structure by pointwise multiplication and interpret it as a groupoid with a single object, then holonomy defines a functor $F:\mathcal G(\Gamma)\to \mathrm{Fun}(G\extimes{E}, G)$.

While all structures so far  are defined for {\em directed graphs} $\Gamma$,  the notion of curvature  requires additional structure on  $\Gamma$, namely  the notion of a face.
As explained in Section \ref{subsec:graphs},  faces are defined for ribbon graphs as equivalence classes of closed paths that  turn maximally left  at all vertices, including its starting vertex, and traverse each edge at most once in each direction.  
In the associated oriented surface $\Sigma_\Gamma$, the faces of $\Gamma$ represent paths that border a disc. This  makes it natural to interpret the holonomy around a face as the curvature of the connection inside this disc.

\begin{definition}\label{def:groupflat} Let $G$ be a group,  $\Gamma$  a ribbon graph and consider a $G$-valued  gauge theory on $\Gamma$. Then for each face path $f\in \mathcal G(\Gamma)$ and connection $g\in G\extimes{E}$,  the holonomy $\hol_f(g)\in G$  is called the {\bf curvature} of $g$ at $f$. A connection $g\in G\extimes{E}$ is called {\bf flat at a face} if $\hol_f(g)=1$  for all face paths representing the face. It is called  {\bf flat} if it is flat at all faces of $\Gamma$.  
\end{definition}

Note that the flatness of a connection at a face does not depend on the choice of face path representing the face.  
Note also that  the action of a gauge transformation $h=(h_1,...,h_{|V|})\in G\extimes{V}$ on a holonomy  is given by
$\hol_p(h\rhd g)=h_{v}\cdot \hol_p(g)\cdot h_{u}^\inv$ for all paths $p$ from $u$ to $v$ and connections $g\in G\extimes{E}$. This implies in particular  that 
the set of connections that are flat at a given face of $\Gamma$ is invariant under gauge transformations.

Definitions  \ref{def:grgt} and \ref{def:groupflat}  can be modified to formulate lattice gauge theories for  topological groups or Lie groups. Connections are then identified, respectively,  with the topological space $G\extimes{E}$ or  the smooth manifold $G\extimes{E}$.
 Functions on  connections are required to be  continuous or smooth and identified with the algebras $C(G\extimes{E})$ or $C^\infty(G\extimes{E})$.  The action of gauge transformations on connections and functions must be an action of   topological groups or Lie groups. 
One can also impose that  the lattice gauge theory carries additional structures, such as a Poisson bracket in the Lie group case.  In this sense, Definitions \ref{def:grgt} and \ref{def:groupflat} contain the minimum requirements for a lattice gauge theory with values in a group.  They define a lattice gauge theory for the  category of groups, while the latter define lattice gauge theories in the categories of  topological groups or Lie groups.

\subsection{Graph gauge theory for a Hopf algebra - the axioms}
\label{subsec:haxioms}

In this section, we introduce the axioms for a Hopf algebra gauge theory on a ribbon graph $\Gamma$ in analogy to the ones for a gauge theory with values in a  group $G$. Let $\Gamma$ be a ribbon graph and $K$ a finite-dimensional Hopf algebra over $\FF$ with dual $K^*$ and
pairing $\langle\,,\rangle: K^*\oo K\to \FF$. For tensor powers of $K$ or $K^*$ we use the notation introduced in Section \ref{subsec:conventions}. 
Following the discussion in the preceding subsection, we then obtain the Hopf algebra  counterparts of connections, functions on connections and gauge transformations by linearising the corresponding structures  for groups.

\begin{compactenum}
\item {\bf Connections:} A connection with values in $K$ should replace the  assignment of a group element to each edge of the graph. Hence it  should be  viewed as an element of the vector space $K\exoo{E}$. The transformation of a connection under  orientation reversal for an edge $e\in E$ is implemented by applying an involution $T: K\to K$ to the copy of $K$ associated with $e$.\\[-1ex]

\item {\bf The algebra of functions:} The dual vector space ${K^*}\exoo{E}$ can be viewed as the Hopf algebra counterpart of the set  of functions $f: G\extimes{E}\to\FF$  in a group gauge theory, and the pairing $\langle\;,\;\rangle: {K^*}\exoo{E}\oo K\exoo{E}\to K\exoo{E}$ takes the place of the evaluation $\text{ev}: \mathrm{Fun}(G\extimes{E})\times G\extimes{E}\to\FF$, $(f,g)\mapsto f(g)$. As  the functions $f: G\extimes{E}\to \FF$ form not only a set but an {\em algebra} with respect to pointwise multiplication, we require that  the
 vector space ${K^*}\exoo{E}$ is also equipped with the structure of an associative unital algebra.  Its  unit  should  be viewed as the Hopf algebra counterpart of the constant function $f\equiv1$ on $G\extimes{E}$ and   be given by the element $1\exoo{E}\in {K^*}\exoo{E}$. \\[-1ex]

\item {\bf The Hopf algebra of gauge transformations:} A gauge transformation with values in $K$ should generalise the assignment of a group element to each vertex of $\Gamma$ and hence correspond to an element of the vector space $K\exoo{V}$. The linearised counterpart of the group structure on $G^{\times V}$ is a Hopf algebra structure on $K\exoo{V}$.
\\[-1ex]

\item {\bf The action of gauge transformations on connections and functions:} Just as in  group gauge theory, gauge transformations should act on connections and thus by duality on functions. Hence, the vector space $K\exoo{E}$ must be a  left module over the Hopf algebra $K\exoo{V}$. This  implies that ${K^*}\exoo{E}$ becomes a $K\exoo{V}$-right module
 with the dual $K\exoo{V}$-module structure (see Remark \ref{rem:dualstruct}). 
This is the Hopf algebra analogue of the identity $(f\lhd^* h^* )(g)=f(h\rhd  g)$ in the group case.\\[-1ex]

\item {\bf The subalgebra of observables:} The Hopf algebra analogue of a gauge invariant function on $G\extimes E$ is an invariant of the $K\exoo{V}$-module ${K^*}\exoo{E}$, i.~e.~an element $\alpha\in {K^*}\exoo{E}$  with
 $\alpha\lhd^* h=\epsilon(h)\,\alpha$ for all $h\in K\exoo{V}$. 
 In the group case, the gauge invariant functions form a subalgebra of the algebra $\mathrm{Fun}(G\extimes{E})$ with  the pointwise multiplication. 
 In the Hopf algebra case, this can only be achieved if the $K\exoo{V}$-module structure satisfies a certain compatibility condition with the algebra structure on ${K^*}\exoo{E}$. One must impose that  ${K^*}\exoo{E}$ is not only a $K\exoo{V}$-right module and an associative algebra, but a $K\exoo{V}$-right module {\em algebra}.  With this additional assumptions Lemma  \ref{lem:project}  ensures that the submodule of invariants is a subalgebra of ${K^*}\exoo{E}$.\\[-1ex]

\item {\bf Locality conditions:} The locality conditions for the gauge transformations in a group gauge theory  and their actions on connections and functions  have direct analogues for a Hopf algebra:
\\[-2ex]
\begin{compactenum}[(i)]
\item  The Hopf algebra structure on  $K\exoo{V}$  should be local in the sense that gauge transformations at different vertices commute.  In other words, as an algebra the Hopf algebra   $K\exoo{V}$ is the $|V|$-fold tensor product of the algebra  $K$. \\[-1ex]

\item The action of gauge transformations on functions and connections must be local in the sense that  functions 
of the form $(\alpha)_e$ with $e\in E$ span a submodule of ${K^*}\exoo{E}$  and are only affected by gauge transformations  at the starting and target vertex of $e$.  This amounts to  the conditions
 $  K\exoo{V}\rhd \iota_e(K^*)=\iota_e(K^*)$ and 
 $(k)_v\rhd (\alpha)_e=\epsilon(k)\, (\alpha)_e$ for all $\alpha\in K^*$, $k\in K$,  $v\in V\setminus\{\st(e), \ta(e)\}$. The corresponding conditions for connections  are obtained by duality.\\[-1ex]
\end{compactenum}
The  conditions on the algebra of functions are less obvious.  In a group gauge theory the algebra $\mathrm{Fun}(G\extimes{E})$ is local in the sense that the product of two functions $f,f': G\extimes{E}\to\FF$ that depend only on the copies of $G$ associated with edges $e,e'\in E$ depends only on the copies of $G$ associated with $e,e'$. Moreover, the algebra $\mathrm{Fun}(G\extimes{E})$ is commutative. 
While the first requirement can be formulated analogously for a Hopf algebra gauge theory, the second clearly is  too restrictive. In view of the locality conditions on gauge transformations it is natural to weaken it by imposing commutativity only for functions associated with edges that have no vertices in common. In other words:\\[-2ex]

\begin{compactenum}[(i)]
\setcounter{enumii}{2}
\item The algebra structure on the algebra  ${K^*}\exoo{E}$ of functions should be local in the sense that $(\alpha)_e\cdot (\beta)_e\in \iota_e(K^*)$,  $(\alpha)_e\cdot (\beta)_{e}\in \iota_{ee'}(K^*\oo K^*)$ for all $e,e'\in E$ and
$(\alpha)_e\cdot (\beta)_{e'}=(\beta)_{e'}\cdot (\alpha)_e$ for all edges $e,e'\in E$ that do not have a vertex in common.
\end{compactenum}
\end{compactenum}

These conditions impose  restrictions on the Hopf algebra structure on  $K\exoo{V}$, the  $K\exoo{V}$-module structures on 
 $K\exoo{E}$,  ${K^*}\exoo{E}$ and on  the algebra structure on $K\exoo{E}$. 
The  condition that $K\exoo{V}$ is a Hopf algebra and isomorphic to  $K\exoo{V}$ {\em as an algebra} restricts the possible coalgebra structures. In the absence of additional data or requirements, the only natural choices for the Hopf algebra structure  on $K\exoo{V}$ are the
$|V|$-fold tensor product of the Hopf algebra $K$ or $K^{cop}$.

Moreover, for each edge $e\in E$ with $\st(e)\neq \ta(e)$,  the locality conditions (i)  and (ii) imply that the action of gauge transformations at $\st(e)$ and $\ta(e)$ on connections   $(k)_e$  define a $(K,K)$-bimodule structure on $K$. In analogy to the group case it is then 
natural  to impose that the action of these gauge transformations at $\st(e)$ and $\ta(e)$ is given by the left and right regular action of $K$ on itself from  Example \ref{ex:regacts}. This  implies by duality  that the action of gauge transformations on functions $(\alpha)_e$ with $\alpha\in K^*$ is given by the  left and right regular action of $K$ on $K^*$.  Summarising these conditions and  conclusions, we obtain the following definition of  a Hopf algebra gauge theory.

\begin{definition}\label{def:gtheory} Let $\Gamma$ be a ribbon graph with edge set $E$ and vertex set $V$ and $K$ a Hopf algebra. \linebreak
A {\bf Hopf algebra gauge theory} on $\Gamma$ with values in $K$  consists of the following data:
\begin{compactenum}
\item The vector space   $K\exoo{E}$  and  the Hopf algebra $K\exoo{V}$.\\[-1ex]
\item   An algebra structure  on  the vector space ${K^*}\exoo{E}$ with unit $1\exoo{E}$ such that:\\[-2ex]
\begin{compactenum}[(i)] 
\item $(\alpha)_e\cdot (\beta)_e\in \iota_e(K^*)$, $(\alpha)_e\cdot(\beta)_f\in \iota_{ef}(K^*\oo K^*)$ for all $\alpha,\beta\in K^*$ and $e,f\in E$.\\[-2ex]
\item  For all $\alpha,\beta\in K^*$ and edges $e,f\in E$  with $\{\st(e),\ta(e)\}\cap\{\st(f),\ta(f)\}=\emptyset$:\\
$(\alpha)_e\cdot( \beta)_f=(\beta)_f\cdot(\alpha)_e=(\alpha\oo\beta)_{ef}$.\\[-1ex]
\end{compactenum}
\item 
A $K\exoo{V}$-left module structure $\rhd:  K\exoo{V}\oo K\exoo{E}\to K\exoo{E}$  and the dual  $K\exoo{V}$-right module structure
$\lhd^*: {K^*}\exoo{E}\oo K\exoo{V} \to {K^*}\exoo{E}$ 
 such that:\\[-2ex]
\begin{compactenum}[(i)]
\item $\lhd^*$ gives ${K^*}\exoo{E}$ the structure of a $K\exoo{V}$-right module algebra,\\[-2ex]
\item For any $e\in E$ with $\st(e)\neq \ta(e)$, $v\in V\setminus\{\st(e), \ta(e)\}$ and $h,k\in K$ 
\begin{align}\label{eq:gkact}
&\pi_e((h)_v\rhd (k)_e)=\epsilon(h) k & &\pi_e((h)_{\ta(e)}\rhd (k)_e)=hk & &\pi_e((h)_{\st(e)}\rhd (k)_e)=kS(h).
\end{align}

\end{compactenum}
\end{compactenum}
The vector space ${K^*}\exoo{E}$ with this algebra structure is denoted $\mathcal A^*_\Gamma$ or $\mathcal A^*$.
Elements of  $K\exoo{E}$  are called {\bf  connections} or {\bf gauge fields}, elements of the algebra  $\mathcal A^*_\Gamma$ are called {\bf  functions} and elements of the Hopf algebra  $K\exoo{V}$ are called  {\bf gauge transformations}. A  
 function  $\alpha\in\mathcal A^*$ is called {\bf gauge invariant} or {\bf observable} if  $\alpha\lhd^*h=\epsilon(h)\,\alpha$ for all $h\in K\exoo{V}$. 
\end{definition}

By applying Lemma \ref{lem:project} to the $K\exoo{V}$-module algebra $\mathcal A^*$ one finds that  the observables of a Hopf algebra gauge theory form a subalgebra   $\mathcal A^*_{inv}\subset \mathcal A^*$. Moreover, if 
$\ell\in K$ is a Haar integral for $K$, then  $\ell\exoo{V}$ is a Haar integral for  $K\exoo{V}$ and defines a projector on $\mathcal A^*_{inv}$.

\begin{corollary} \label{lem:ginvproj} In any  Hopf algebra gauge theory,  the linear subspace  $\mathcal A^*_{inv}\subset\mathcal A^*$ of gauge invariant functions is a subalgebra. If $K$ is equipped with a Haar integral $\ell\in K$, then  the projector on  $\mathcal A^*_{inv}$ is given by  $\Pi: \mathcal A^*\to\mathcal A^*$, $\alpha\mapsto \alpha \lhd^* \ell\exoo{V}$.
\end{corollary}

The locality conditions in the definition of  a Hopf algebra gauge theory---i.e.~that gauge transformations at different vertices commute, that gauge transformations at a vertex $v$ act only on the edges incident at $v$, and that functions of the form $(\alpha)_e$, $(\beta)_f$ for two edges $e,f\in E$ commute if $e$ and $f$ do not have a vertex in common---suggests that a Hopf algebra gauge theory could be built up from  Hopf algebra gauge theories on  the  vertex neighbourhoods $\Gamma_v$ from Definition \ref{def:vertex_nb}.  Let $v\in V$ be an $n$-valent vertex. 
If one does not associate gauge transformations to  the univalent vertices  in $\Gamma_v$, a 
 Hopf algebra gauge theory on  $\Gamma_v$ is given by the vector space $K^{\oo n}$ of connections,  the Hopf algebra $K$ of gauge transformations and a $K$-module algebra structure $\mathcal A^*_v$ on $K^{*\oo n}$ such that the axioms of Definition \ref{def:gtheory} are satisfied. 

From a collection of Hopf algebra gauge theories on the vertex neighbourhoods $\Gamma_v$  one obtains 
a Hopf algebra gauge theory on  $\coprod_{v\in V}\Gamma_v$ by  taking as the algebra of functions the 
 tensor product   $\oo_{v\in V}\mathcal A^*_v$  with the induced $K\exoo{V}$-module structure.   As the ribbon graph $\Gamma$ is obtained by gluing the vertex neighbourhoods $\Gamma_v$ at their univalent vertices,
 one expects  that connections on $\Gamma$ are obtained from connections 
 on $\coprod_{v\in V}\Gamma_v$ via a linear map $G: \oo_{v\in V} K\exoo{v}\to K\exoo{E}$. 
  This map should be a module homomorphism with respect to the $K\exoo{V}$-module structures on $\oo_{v\in V}K\exoo{v}$ and on $K\exoo{E}$. It should 
send a connection  supported on the edge ends $s(e)$, $t(e)$ to a connection  supported on $e$ and hence take the form $G=\oo_{e\in E}G_e$ with linear maps $G_e: K\oo K\to K$ that  satisfy  $G_e\circ \iota_{s(e)t(e)}(K\oo K)\subset\iota_e(K)$ and are module maps with respect to gauge transformations at the vertices $\st(e)$ and $\ta(e)$. From condition 3.  in  Definition \ref{def:gtheory}, applied to the edge ends $s(e)$, $t(e)$ in $E(\coprod_{v\in V}\Gamma_v)$ and to the edge $e\in E(\Gamma)$, one finds that the only natural  candidate is   $G_e: K\oo K\to K$, $(k\oo k')_{s(e)t(e)}\mapsto (k' k)_e$. This yields a pair of dual  linear maps
\begin{align}\label{eq:dualemb}
&G=\oo_{e\in E} G_e: \,\oo_{v\in V} K\exoo{v}\to 
K\exoo{E},\\
&(k^1\oo k'^1\oo\ldots\oo k^E\oo k'^E)_{t(e_1)s(e_1)\ldots t(e_E)s(e_E)}\mapsto (k^1k'^1\oo\ldots\oo k^Ek'^E )_{e_1\ldots e_E}\nonumber\\
&G^*=\oo_{e\in E}G^*_e: \,
{K^*}\exoo{E}\to\oo_{v\in V}{K^*}\exoo{v}\nonumber\\
&(\alpha^1\oo \ldots \oo\alpha^E)_{e_1\ldots e_E}\mapsto  (\alpha^1_{(1)}\oo \alpha^1_{(2)}\oo\ldots \oo \alpha^E_{(1)}\oo \alpha^E_{(2)})_{t(e_1)s(e_1)\ldots t(e_E)s(e_E)}.\nonumber
\end{align}
Given a $K$-valued  Hopf algebra gauge theory on $\Gamma$ together with $K$-valued local Hopf algebra gauge theory on each vertex neighbourhood $\Gamma_v$, it is then natural to demand  that the maps in \eqref{eq:dualemb}  are  module homomorphisms with respect to the $K\exoo{V}$-module structures and that 
the linear map   $G^*$  in \eqref{eq:dualemb} is an injective algebra homomorphism.

\begin{definition} \label{def:local_gt} Let $\Gamma$ be a ribbon graph and $K$ a finite-dimensional Hopf algebra. A $K$-valued Hopf algebra gauge theory on $\Gamma$  is called {\bf  local} if there are $K$-valued Hopf algebra gauge theories on each vertex neighbourhood $\Gamma_v$ such that  the map $G^*:{K^*}\exoo{E}\to\oo_{v\in V}{K^*}\exoo{v}$ from \eqref{eq:dualemb} is a homomorphism of $K\exoo{V}$-module algebras. 
\end{definition}

Definition \ref{def:local_gt} embeds the algebra $\mathcal A^*$ of  functions of  a Hopf algebra gauge theory on $\Gamma$ into the  tensor product $\otimes_{v\in V} \mathcal A^*_v$ of the algebras of  functions on the vertex neighbourhoods $\Gamma_v$.
 If  $K$ is semisimple, then its Haar integral defines a projector on the image of $G^*$.

\begin{lemma}\label{lem:1proj}  If  $K$ is semisimple with Haar integral $\ell\in K$ then a projector on the subalgebra $G^*(\mathcal A^*)\subset \oo_{v\in V} \mathcal A^*_v$ is given by
$$\Pi: \otimes_{v\in V} \mathcal A^*_v\to \oo_{v\in V} \mathcal A^*_v,\quad (\alpha\oo\beta)_{s(e)t(e)} \mapsto \langle S(\low\alpha 1)\low\beta 2,\ell\rangle \, (\low\alpha 2\oo\low\beta 1)_{s(e)t(e)}.
$$
\end{lemma}
\begin{proof}
Applying the  axioms in Definition \ref{def:gtheory} to each vertex neighbourhood $\Gamma_v$ shows  that the linear map  
$\rhd: K\exoo{E}\oo( \oo_{v\in V}\mathcal A^*_v)\to\oo_{v\in V}\mathcal A^*_v$ given by
\begin{align}\label{eq:kleft}
&(x)_e\rhd(\alpha\oo \beta)_{s(e)t(e)}=\langle S(\low\alpha 1)\low\beta 2, x\rangle\, (\low\alpha 2\oo\low\beta 1)_{s(e)t(e)}\qquad  
&(x)_e\rhd(\alpha)_{f}=\epsilon(x)\, (\alpha)_f\
\end{align}
for all  $e,f\in  E$, $e\neq f$
defines a $K\exoo{E}$-left module structure on ${K^*}\exoo{2E}$. Consequently, by Lemma \ref{lem:ginvproj},  the map $\Pi$ is a projector on 
 $
 {K^*}\exoo{2E}_{inv}
 $. 
That $G^*(\mathcal A^*)\subset \mathrm{Im}(\Pi)={K^*}\exoo{E}_{inv}$ follows  from the identity $\epsilon(\ell)=1$ since for all $e\in E$
\begin{align*}
\Pi( (\low\alpha 2\oo\low\alpha 1)_{s(e)t(e)})=\langle S(\alpha_{(2)(1)})\alpha_{(1)(2)},\ell\rangle\; (\alpha_{(2)(2)}\oo \alpha_{(1)(1)})_{s(e)t(e)}=\epsilon(\ell)\; (\low \alpha 2\oo \low\alpha 1)_{s(e)t(e)}.
\end{align*}
To show that $ \mathrm{Im}(\Pi)= G^*(\mathcal A^*)$,  it is sufficient  to consider a  ribbon graph with a single edge $e\in E(\Gamma)$ and to show that in this case $\dim(\mathrm{Im}(\Pi))=\dim K$. 
For this, note that the
 linear map $(S\oo\id): K^*\oo K^*\to K^*\oo K^*$ is a bijective $K$-left module homomorphism between the $K$-left module structure in \eqref{eq:kleft}  and 
the $K$-module structure from Example \ref{ex:moduledim}. The claim then follows from Example \ref{ex:moduledim}. 
\end{proof}

\subsection{Hopf algebra gauge theory on a vertex neighbourhood}
\label{subsec:vertex_nb}

We can now investigate the implication of the axioms in Definition \ref{def:gtheory} and \ref{def:local_gt} and construct local Hopf algebra gauge theories on a ribbon graph $\Gamma$.  By Definition \ref{def:local_gt}, such  Hopf algebra gauge theories  are determined uniquely by  a  Hopf  algebra
gauge theory  on each vertex neighbourhood $\Gamma_v$. Hence, the first step is a construction of a Hopf algebra gauge theory on  $\Gamma_v$. 
If  we associate gauge transformations only with the vertices of $\Gamma$,  a $K$-valued Hopf algebra gauge theory on the vertex neighbourhood $\Gamma_v$ of an
$n$-valent vertex $v\in V$  involves the vector space $K^{\oo n}$ of connections, the Hopf algebra $K$ of gauge transformations at $v$ and a $K$-module algebra structure on  $K^{*\oo n}$  that satisfies the locality axioms in Definition \ref{def:gtheory}. Note that the choice of a  $K$-module algebra structure on $K^{*\oo n}$ is equivalent to the choice of a $K$-module coalgebra structure on $K^{\oo n}$ and that it is 
 constrained by the locality requirements in  equation \eqref{eq:gkact}, which imply 
\begin{align}\label{eq:intgt}
(\alpha)_e\lhd^* h= \langle \low\alpha 1, h\rangle\, (\low\alpha 2)_e\;\text{if } v=\ta(e),\quad
(\alpha)_e\lhd^*h=  \langle \low\alpha 2, S(h)\rangle\, (\low\alpha 1)_e \;\text{if } v=\st(e).
\end{align}
If all edge ends at $v$ are incoming, \eqref{eq:intgt} coincides with the right regular action of $K$ on $K^*$ from Example \ref{ex:regacts}, 4. which defines a $K$-right module algebra  structure on $K^*$. The construction of a Hopf algebra gauge theory on the vertex neighbourhood  $\Gamma_v$ then amounts to construction of a $K$-module algebra structure on the $n$-fold tensor product $K^{*\oo n}$ that induces the $K$-right module algebra structure  from Example \ref{ex:regacts}, 4.~on each copy of $K^*$.

  It is then reasonable to impose one further locality requirement, namely that the $K$-right module algebra structure on the vector space $K^{*\oo n}$ for an $n$-valent vertex is determined uniquely by the $K$-right module algebra structure on $K^*$ for a vertex with a single edge end via a gluing procedure.  This amounts to a functor $\oo_{\text{Alg}}: \text{Alg}_{\text{Mod-}K}\times\text{Alg}_{\text{Mod-}K}\to \text{Alg}_{\text{Mod-}K}$, where $\text{Alg}_{\text{Mod-}K}$ is the category of $K$-right module algebras.  For uniqueness of this construction, this functor should satisfy an associativity condition  and the underlying field $\mathbb F$ with the trivial $K$-right module structure should act as a unit for $\oo_{\text{Alg}}$. Moreover, forgetting the algebra structure by applying the forgetful functor $F: \text{Alg}_{\text{Mod-}K}\to \text{Mod-}K$ to the category of $K$-right modules should yield the usual tensor product $\oo_{\text{Mod-}K}$ in $\text{Mod-}K$ defined by the comultiplication and counit: $F\circ \oo_{\text{Alg}}=\oo_{\text{Mod-}K}\circ (F\times F)$.
We therefore  impose one additional axiom.

{\bf Iterative construction of vertex algebras:}
The $K$-module algebra structure on $K^{*\oo n}$ for an $n$-valent vertex is obtained from a $K$-right module algebra structure on $K^*$  by taking an $n$-fold tensor product in the category $\text{Alg}_{\text{Mod-}K}$ of $K$-right module algebras. The $K$-right module structure in this tensor product is the usual tensor product in the category $\text{Mod-}K$ of $K$-right modules.

The problem of defining a tensor product in the category of $K$-right module algebras  is well-known\footnote{C.~M.~thanks  Simon Lentner,  Hamburg University, for helpful remarks and  discussions.}, and its solution is due to Majid  \cite{majidbraid1, majidbraid}.
If $K$ is a Hopf algebra  and $A$, $B$ are  $K$-right module algebras, then the tensor product  $A\oo B$  has a canonical  algebra structure and a canonical $K$-right module structure $\lhd:  A\oo B\oo K\to A\oo B$ with
$ (a\oo b)\lhd k=(a\lhd_A \low k 1)\oo (b\lhd_B \low k 2)$. However,  in general they {\em do  not} define a  $K$-right module algebra structure on $A\oo B$  since  the identity $a\oo b=(a\oo 1)(1\oo b)=(1\oo b)(a\oo 1)$  would  imply
$(a\oo b)\lhd k=(a\lhd_A \low k 1)\oo (b\lhd_B \low k 2) 
=(a\lhd_A \low k 2)\oo (b\lhd_B \low k 1)$
for all $k\in K$, $a\in A$, $b\in B$.  This cannot not  hold for  general $K$-right module algebras $A,B$ unless  $K$ is cocommutative. 

To  obtain a  $K$-right module algebra structure on  $A\oo B$ for general $K$-right module algebras $A$, $B$, one needs additional structure on $K$ that relates the comultiplication of $K$ to the opposite comultiplication. A natural candidate for this is the $R$-matrix of a quasitriangular Hopf algebra, which can be used to deform the multiplication relations of  $A\oo B$. 
That  this indeed yields a $K$-right module algebra structure on the vector space $A\oo B$  was first shown by Majid, who considered this `braided tensor product'  in 
 the context of braided Hopf algebras  \cite{majidbraid1, majidbraid}. 
Adapted to our setting and notation, this $K$-right module algebra structure on $A\oo B$ is given as follows.

\begin{lemma}   \label{lem:braided}  (\cite{majidbraid} Prop. 4.1.)
 Let $(K,R)$ be a  quasitriangular  Hopf algebra and $(A,\lhd_A)$,  $(B,\lhd_B)$ $K$-right module algebras. Equip the vector space $A\oo B$ with the multiplication
\begin{align*}
(a\oo b)\cdot (a'\oo b')=a(a'\lhd_A \low R 1)\oo  (b\lhd_B \low R 2)b'
\end{align*}
Then $(A\oo B,\cdot)$ is a $K$-module algebra with respect to the induced $K$-right module structure  on $A\oo B$.
\end{lemma}

\bigskip
\begin{theorem}\label{th:braided}(\cite{majidbraid1, majidbraid}) The tensor product of  $K$-right module algebras in Lemma \ref{lem:braided} is associative and unital and gives the category $\text{Alg}_{\text{Mod-}K}$ of $K$-right module algebras the structure of a tensor category.
\end{theorem}

 In fact, this is the only known canonical tensor category structure on the category of $K$-right module algebras over a quasitriangular Hopf algebra $(K,R)$.
Lemma \ref{lem:braided} and Theorem \ref{th:braided} allow one to define a $K$-module algebra structure on $K^{*\oo n}$ as the $n$-fold braided tensor product of the $K$-module algebra from Example  \ref{ex:regacts}, 4. and will give rise to a Hopf algebra gauge theory on $\Gamma_v$.  
 Clearly, this  algebra structure depends on the ordering of the factors in the tensor product, which must reflect an ordering of the incident  edge ends. As they are already equipped with a cyclic ordering from the ribbon graph structure, this amounts to choosing a cilium at $v$.  
We conclude that  the natural data for  a  Hopf algebra gauge theory on a vertex neighbourhood  of $\Gamma$ is a {\em quasitriangular} Hopf algebra $K$ and a {\em ciliated ribbon graph} $\Gamma$. 

For  a vertex $v\in V(\Gamma)$ with $n$ incoming edge ends---ordered counterclockwise starting at the cilium as shown on p.~\pageref{fig:vertex_edgeends}---we  identify the edge ends with the different factors in the tensor product $K^{*\oo n}$ according to their ordering. 
To define a $K$-module algebra structure on  $\Gamma_v$, we then apply Proposition \ref{lem:braided} to the $n$-fold braided tensor product  $K^{*\oo n}$. Note, however that in this case the algebra structure from Lemma \ref{lem:braided} is not  unique. The product in Lemma \ref{lem:braided} can be modified  by letting the  $R$-matrix   act  on $K^{*\oo n}$ via the left regular action from Example \ref{ex:regacts}, 3 and this yields another $K$-right module algebra for the same module structure.  It will turn out that this non-uniqueness disappears if one requires that the Hopf algebra gauge theories on the vertex neighbourhoods induce  a Hopf algebra gauge theory on $\Gamma$  (see Remark \ref{rem:sigg}).

\begin{proposition} \label{th:algebra}  Let  $K$ be a finite-dimensional Hopf algebra and $R$  an $R$-matrix for $K$. Then for  $n\in\NN$ and any map  $\sigma:\{1,...,n\}\to \{0,1\}$  the following defines an algebra structure on $K^{*\oo n}$:
\begin{align}\label{eq:alg0}
&(\alpha)_i\cdot (\beta)_i=\langle \low\beta 1\oo\low\alpha 1, R\rangle \,(\low\beta 2\low\alpha 2)_i & &\sigma(i)=0\\
&(\alpha)_i\cdot (\beta)_i= (\alpha\beta)_i & &\sigma(i)=1\nonumber\\
&(\alpha)_i\cdot (\beta)_j=(\alpha\oo\beta)_{ij}    & &i<j\nonumber\\
&(\alpha)_i\cdot (\beta)_j=\langle \low\beta 1\oo\low\alpha 1, R\rangle \, (\low\alpha 2\oo \low\beta 2)_{ij}& &i>j.\nonumber
\end{align}
The linear map 
\begin{align}\label{eq:gt}
&\lhd^*:  K^{*\oo n}\oo K\to K^{*\oo n},\quad (\alpha^1\oo\ldots\oo\alpha^n)\lhd^*h=\langle \alpha^1_{(1)}\cdots \alpha^n_{(1)}, h\rangle\; \alpha^1_{(2)}\oo\ldots\oo\alpha^{n}_{(2)}
\end{align}
equips this algebra with the structure of a $K$-right module algebra,  and the dual $K$-module structure on $K^{\oo n}$ satisfies the conditions in Definition \ref{def:vertex_nb}, 3.
\end{proposition}

\begin{proof} This follows directly from Lemma \ref{lem:braided} and Theorem \ref{th:braided} applied to  $K^*$  with the $K$-right module structure $\lhd^*: K^*\oo K\to K^*$, $\alpha\lhd^* k=\langle k, \low\alpha 1\rangle \low\alpha 2$ and with either the usual algebra structure on $K^*$ (for $\sigma(i)=1$) or its braided  opposite with multiplication $\alpha\cdot\beta=\langle R,\low\beta 1\oo \low \alpha 1\rangle \low\beta 2\low\alpha 2$ (for $\sigma(i)=0$), see \cite{majidbraid1, majidbraid}. 
\end{proof}

There is another way to understand the  algebra structure on $K^{*\oo n}$ in Proposition \ref{th:algebra}, namely to note that  it is dual to a coalgebra structure on $K^{\oo n}$ obtained by twisting the comultiplication of the Hopf algebra $K^{\oo n}$.
Twists are known to give rise to module (co)algebra structures \cite[Section 4]{gelaki},  and hence it is  not surprising that the algebra associated to a vertex neighbourhood is of this type. For the definition of a twist and its properties, see Definition \ref{def:twist} and Lemma \ref{lem:twist}.

\begin{remark} \label{lem:twistlemma}The algebra structure in  \eqref{eq:alg0}  is dual to the comultiplication 
$\Delta^{F,G}=F\cdot \Delta\cdot G^\inv$ obtained by twisting  $K^{\oo n}$ with 
$G=\Pi_{\sigma^\inv(0)} R_{(n+i)i}^\inv$ and
$F=\Pi_{1\leq i<j\leq n} R_{(n+i) j}$, where the ordering of the factors in $F$ is  such that $R_{(n+i)j}$ is to the left of $R_{(n+k)l}$ if $i<k$, $j=l$ or $i=k$, $j>l$.
\end{remark}

\begin{corollary}\label{lem:semisimple} Let $K$ be a finite-dimensional semisimple  quasitriangular  Hopf algebra. Then for any ciliated ribbon graph  $\Gamma$ and any choice of the $R$-matrices,  the algebras $\mathcal A^*_v$ and $\oo_{v\in V}\mathcal A^*_{v}$ are semisimple.
\end{corollary}
\begin{proof} As $\text{char}(\FF)=0$ and $K$ is finite-dimensional,  $K$ is semisimple if and only if  it is cosemisimple  if and only if $S^2=\id$ \cite{LR}, and semisimplicity implies unimodularity.  As $\text{char}(\FF)=0$ and tensor products of semisimple Hopf algebras are semisimple (see for instance \cite[Corollary 2.37]{Kn}),  it follows that the Hopf algebra  $K^{\oo n}$ is semisimple and cosemisimple  for all $n\in\NN$.
It is shown in \cite[Theorem 3.13]{gelaki}  that for a cosemisimple  unimodular Hopf algebra $H$
any two-sided twist deformation $H_{F,G}$ obtained by replacing $\Delta\mapsto F\cdot \Delta\cdot G^\inv$, $\epsilon \mapsto \epsilon$ is a cosemisimple coalgebra. 
By Remark \ref{lem:twistlemma} the algebra $\mathcal A^*_v$ for a vertex neighbourhood $\Gamma_v$ is dual to such a two-sided twist of $K^{\oo n}$ and hence it is semisimple. As $\text{char}(\FF)=0$,  the same holds for the  tensor product $\oo_{v\in V}\mathcal A^*_v$.
\end{proof}

As for any $R$-matrix $R=\low R 1\oo\low R 2$ the element  $R_{21}^\inv=\low R 2\oo S(\low R 1)$ is another $R$-matrix for $K$, it is natural to ask how the algebra structures from Proposition \ref{th:algebra} for these $R$-matrices are related. It turns out that replacing   $R\to R_{21}^\inv$  corresponds to reversing  the edge ordering at this vertex and hence to  reversing  the orientation of the vertex neighbourhood. In particular,  for {\em triangular} Hopf algebras  the algebra structure from Proposition \ref{th:algebra} is orientation independent.

\begin{remark} \label{rem:orrev}Replacing the $R$-matrix $R$ in \eqref{eq:alg0} by the opposite $R$-matrix $R_{21}^\inv$ 
yields an algebra isomorphism  to the algebra \eqref{eq:alg0} with the opposite edge ordering. 
This follows because the algebra in  \eqref{eq:alg0}  is characterised uniquely up to isomorphism by the multiplication relations
\begin{align*}
&(\beta)_j\cdot (\alpha)_i=\langle \low\alpha 1\oo\low\beta 1, R\rangle\,(\low\alpha 2)_i \cdot (\low\beta 2)_j &   &i<j\\
&(\beta)_i\cdot (\alpha)_i=\langle \low\alpha 1\oo\low\beta 1, R\rangle\langle \low\alpha 3\oo\low\beta 3, R_{21}^\inv\rangle\,  (\low\alpha 2)_i\cdot (\low\beta 2)_i
& &\sigma(i)=0\\
& (\beta)_i\cdot (\alpha)_i=\langle \low\alpha 1\oo\low\beta 1, R\rangle\langle \low\alpha 3\oo\low\beta 3, R^\inv\rangle\, (\low\alpha 2)_i\cdot  (\low\beta 2)_i
& &\sigma(i)=1.
\end{align*} 
Due to the identity  $\Delta=R\cdot\Delta\cdot R^\inv=R_{21}^\inv\cdot \Delta\cdot R_{21}$, 
the last two relations are invariant under the substitution $R\to R_{21}^\inv$. The first  is mapped to 
$(\beta)_j\cdot  (\alpha)_i=\langle \low\beta 1\oo\low\alpha 1, R^\inv\rangle\,  (\low\alpha 2)_i\cdot (\low\beta 2)_i$, which is equivalent to the multiplication relation 
$(\alpha)_i\cdot  (\beta)_j=\langle \low\beta 1\oo\low\alpha 1, R\rangle\,  (\low\beta 2)_j\cdot (\low\alpha 2)_i$ for $i<j$.
\end{remark}

Proposition \ref{th:algebra} defines  the algebra and module structure of a Hopf algebra gauge theory for a vertex neighbourhood in which all edges are incoming. Generalising  it to  vertex neighbourhoods with outgoing edges requires an involution  $T^*: K^*\to K^*$ and its dual $T:K\to K$.  If the antipode of $K$
satisfies  $S^2=\id$, it is  natural to choose $T^*=S$.  
More generally, for  a finite-dimensional {\em ribbon} Hopf algebra $K$ one can consider the pair of dual involutions
\begin{align}\label{eq:invol}
 T:K\to K,\; k\mapsto g\cdot S(k),\qquad  T^*:K^*\to K^*, \;\alpha\mapsto \langle \low\alpha 1, g\rangle\, S(\low\alpha 2)\qquad
 \end{align}
 where $g$ is the grouplike element of $K$---see Remark \ref{rem:grouplike}.  Equivalently, one could work with 
the pair of dual involutions  $T': K\to K$, $k\mapsto S(k)g^\inv$ and $T'^*:\alpha\mapsto \langle \low\alpha 2, g^\inv\rangle\, S(\low\alpha 1)$.
If $K$ is semisimple and hence $S^2=\id$, the grouplike element is given by $g=1$ (see Lemma \ref{lem:ssimplerib}) and the two involutions  $T^*,T'^*$ coincide with $S$.

To obtain the $K$-module algebra structure  on a vertex neighbourhood with $n$ incident edge ends of general  orientation, we define a map $\tau: \{1,...,n\}\to\{0,1\}$ by $\tau(i)=0$ if the $i$th edge end is incoming and $\tau(i)=1$ if it is outgoing. The algebra  and module structure on $K^{*\oo n}$ is then determined by the condition  that $T^{*\tau(1)}\oo\ldots\oo T^{*\tau(n)}: K^{*\oo n}\to K^{*\oo n}$  is a morphism of module algebras.
With the properties of the antipode, the identities $(S\oo S)(R)=(S^\inv\oo S^\inv)(R)=R$ and the properties of the grouplike element $g$, one then obtains  the following   $K$-module algebra  structure on $K^{*\oo n}$.

\begin{corollary} \label{rem:flipalg} Let $(K,R)$ be a finite-dimensional  ribbon  Hopf algebra and  $\tau,\sigma:\{1,...,n\}\to\{0,1\}$ arbitrary maps. Then 
the multiplication 
\begin{align}\label{eq:flipalg}
&(\alpha)_i\cdot (\beta)_i=\begin{cases} 
\langle  \beta_{(1)}\oo\alpha_{(1)}, R\rangle \,(\low\beta 2\low\alpha 2)_i  &\sigma(i)=\tau(i)=0\\
\langle \beta_{(2)}\oo  \alpha_{(2)}, R\rangle \,(\low\alpha 1\low\beta 1)_i  &\sigma(i)=0,\tau(i)=1\\
 (\alpha\beta)_i & \sigma(i)=1, \tau(i)=0\\
 (\beta\alpha)_i & \sigma(i)=\tau(i)=1
\end{cases}\\
&(\alpha)_i\cdot (\beta)_j=\begin{cases}
\langle \beta_{(1+\tau(j))}\oo \alpha_{(1+\tau(i))}, (S^{\tau(i)}\oo S^{\tau(j)})(R)\rangle \, (\alpha_{(2-\tau(i))}\oo \beta_{(2-\tau(j))})_{ij}  &i>j\\
(\alpha\oo\beta)_{ij} & i<j,
\end{cases}
\nonumber
\end{align}
and the linear map $\lhd^*:  K^{*\oo n}\oo K\to K^{*\oo n}$ with
\begin{align}\label{eq:act_v}
&(\alpha^1\oo...\oo\alpha^n)\lhd^*h=\langle S^{\tau(1)}(\alpha^1_{(1+\tau(1))})\cdots S^{\tau(n)}(\alpha^n_{(n+\tau(n))}), h\rangle \; \alpha^1_{(2-\tau(1))}\oo\ldots \alpha^n_{(2-\tau(n))}
\end{align}
 define a $K$-right module algebra structure on $K^{*\oo n}$. The involution $T^{*\tau(1)}\oo\ldots\oo T^{*\tau(n)}: K^{*\oo n}\to K^{*\oo n}$ is an algebra homomorphism from the  algebra structure \eqref{eq:flipalg} to the one in \eqref{eq:alg0} and a homomorphism of $K$-right modules from  the module structure \eqref{eq:act_v} to the one  in \eqref{eq:gt}.
The  dual $K$-module structure on $K^{\oo n}$ satisfies the conditions in Definition \ref{def:vertex_nb}, 3.
\end{corollary}

\begin{theorem} \label{th:vertex_gt}Let  $(K,R)$ be  a  finite-dimensional ribbon Hopf algebra, $\Gamma$  a ciliated ribbon graph   and $v\in V$  a vertex with $n$ incident edge ends, ordered with respect to the cilium at $v$. Define $\tau: \{1,..., n\}\to\{0,1\}$ by $\tau(i)=0$ if the $i$th edge end is incoming,  $\tau(i)=1$ if it is outgoing and choose an arbitrary map $\sigma:\{1,...,n\}\to \{0,1\}$. 
Then the algebra structure  and  the $K$-right module structure from Corollary \ref{rem:flipalg}   define   a Hopf algebra gauge theory on $\Gamma_v$. 
\end{theorem}

\subsection{Hopf algebra gauge theory on a ribbon graph $\Gamma$}
\label{subsec:gribbongraph}

We will now combine the  Hopf algebra gauge theories on the vertex neighbourhoods $\Gamma_v$ into a local Hopf algebra gauge theory on the ciliated ribbon graph $\Gamma$. We consider  a  finite-dimensional ribbon Hopf algebra $K$ and
 assign to each 
 vertex $v\in V$
 an
 $R$-matrix $R_v$  and a map  $\tau_v$  as in Corollary \ref{rem:flipalg} and Theorem \ref{th:vertex_gt}.  
We also introduce a  map  $\rho: E(\Gamma)\to E( \Gamma_\circ)$  that selects for each edge $e\in E(\Gamma)$ one of the associated edge ends in $E(\Gamma_\circ)$, i.e.~either $\rho(e)=s(e)$ or $\rho(e)=t(e)$ for each edge $e\in E(\Gamma)$.  
For  each vertex $v\in V(\Gamma)$, we define the map  $\sigma_v: \{1,...,|v|\}\to\{0,1\}$ 
from Corollary \ref{rem:flipalg} by the condition that $\sigma_v(i)=0$ if 
the $i$th edge end at $v$ is  in the image of $\rho$ and $\sigma_v(i)=1$ else. 
 This data assigns to each
 vertex $v\in V(\Gamma)$  a   module  algebra  $\mathcal A^*_v$  as in 
Corollary \ref{rem:flipalg}.

The action of gauge transformations at vertices $v\in V$ equips the tensor product
$\oo_{v\in V}\mathcal A^*_v$ with the structure of a $K\exoo{V}$-right module algebra. 
By Definition \ref{def:local_gt}, this data induces a local Hopf algebra gauge theory on $\Gamma$ via the
 map $G^*:{K^*}\exoo{E}\to\otimes_{v\in V} \mathcal A^*_v$ from \eqref{eq:dualemb} if and only if 
\begin{compactenum}[(i)]
\item  $G^*({K^*}\exoo{E})$  is a subalgebra  of  the algebra $\oo_{v\in V}\mathcal A^*_v$,
\item   $G^*({K^*}\exoo{E})$ is a $K\exoo{V}$- submodule of the $K\exoo{V}$-right module  $\oo_{v\in V}\mathcal A^*_v$,
\item the induced $K\exoo{V}$-module structure on $G^*({K^*}\exoo{E})$ satisfies  axiom 3. in Definition
\ref{def:gtheory}.  
\end{compactenum}
The following two propositions show that these conditions are satisfied. Moreover,   the resulting algebra structure on ${K^*}\exoo{E}$ does not depend on the choice of the map $\rho: E(\Gamma)\to E(\Gamma_\circ)$ if the same $R$-matrix is assigned to all vertices $v\in V$.

\begin{proposition} \label{lem:edge_algebra} Let $K$ be a finite-dimensional   
ribbon Hopf algebra and $\Gamma$ a ciliated ribbon graph equipped with the data  above. Then:\\[-3ex]
\begin{compactenum}
\item The linear subspace $G^*({K^*}\exoo{E})\subset \otimes_{v\in V} \mathcal A^*_v$  is a subalgebra of  $\otimes_{v\in V} \mathcal A^*_v$.
\item If $R_v=R$ for all $v\in V$, the induced  algebra structure on ${K^*}\exoo{E}$ does not depend on $\rho$.
\item  If   $\ell\in K$ is a  Haar integral for $K$ then the  projector on $G^*({K^*}\exoo{E})\subset \otimes_{v\in V}\mathcal A^*_v$ is given by
$$\Pi: \otimes_{v\in V}\mathcal A^*_v\to \otimes_{v\in V}\mathcal A^*_v,\qquad (\alpha\oo \beta)_{s(e)t(e)}\mapsto \langle \low\alpha 1S(\low\beta 2), \ell\rangle\; (\low\alpha 2\oo \low\beta 1)_{s(e)t(e)}.$$ 
  \end{compactenum}
\end{proposition}
\begin{proof} 
 By definition of the algebra $\oo_{v\in V}\mathcal A^*_v$, every element $G^*( \alpha^1\oo...\oo \alpha^E)$ can be expressed in Sweedler notation as a product  of elements $(\alpha^i_{(1)})_{t(e_i)}$ and $(\alpha^i_{(2)})_{s(e_i)}$, where factors of the form $(\gamma)_e$ and $(\delta)_f$  for edge ends $e,f$ commute if $e,f$ are edge ends at different vertices, and the contributions of edge ends for a given vertex are ordered according to the  ordering  at this vertex.  Multiplying two elements $G^*( \alpha^1\oo...\oo \alpha^E)$ and $G^*( \beta^1\oo...\oo \beta^E)$ in $\oo_{v\in V}\mathcal A^*_v$  involves (i) reordering the factors  $(\alpha^i_{(k)})_e$, $(\beta^j_{(l)})_f$ for the different edge ends $e,f$ until all factors are again ordered according to the  ordering at each vertex and the two factors for each edge end are next to each other in the product and then (ii) multiplying the two factors   for each edge end.
To show that permuting the factors associated with two edge ends  in step (i) yields again products of factors $(\alpha'_{(1)})_{t(e)}$, $(\alpha'_{(2)})_{s(e)}$ and $(\beta'_{(1)})_{t(e)}$, $(\beta'_{(2)})_{s(e)}$ it is sufficient to note that for all $\alpha,\beta\in K^*$ end edges $e,f$ of $\Gamma$ the second relation in \eqref{eq:flipalg}, Corollary \ref{rem:flipalg} implies
\begin{align*}
&\big((\low \alpha 1)_{t(e)}\cdot (\low \beta 1)_{t(f)}\big) \oo \low\alpha 2\oo\low\beta 2=\langle R, \low\beta 1\oo \low\alpha 1\rangle\; \big((\low\beta 2 )_{t(f)}\cdot (\low\alpha 2)_{t(e)}\big)\oo \low\alpha 3\oo\low\beta 3\\ 
&=\langle R, \low\beta 1\oo \low\alpha 1\rangle\; \big((\beta_{(2)(1)} )_{t(f)}\cdot (\alpha_{(2)(1)})_{t(e)}\big)\oo \alpha_{(2)(2)}\oo\beta_{(2)(2)}\qquad\quad\; t(f)<t(e)\\[+1ex]
&\big((\low \alpha 2)_{s(e)}\cdot (\low \beta 1)_{t(f)}\big)\oo \low\alpha 1\oo\low\beta 2=\langle R, \low\beta 1\oo S(\low\alpha 3)\rangle\; \big((\low\beta 2 )_{t(f)}\cdot (\low\alpha 2)_{s(e)}\big)\oo \low\alpha 1\oo\low\beta 3 \\ 
&=\langle R, \low\beta 1\oo S(\low\alpha 2)\rangle\; \big((\beta_{(2)(1)} )_{t(f)}\cdot (\alpha_{(1)(2)})_{s(e)}\big)\oo \alpha_{(1)(1)}\oo\beta_{(2)(2)}\qquad t(f)<s(e)\\[+1ex]
&\big((\low \alpha 1)_{t(e)}\cdot (\low \beta 2)_{s(f)}\big)\oo \low\alpha 2\oo\low\beta 1=\langle R, S(\low\beta 3)\oo \low\alpha 1\rangle\; \big((\low\beta 2 )_{s(f)}\cdot (\low\alpha 2)_{t(e)}\big)\oo \low\alpha 3\oo\low\beta 1 \\ 
&=\langle R, S(\low\beta 2)\oo\low\alpha 1\rangle\; \big((\beta_{(1)(2)} )_{s(f)}\cdot (\alpha_{(1)(1)})_{t(e)}\big)\oo \alpha_{(1)(2)}\oo\beta_{(1)(1)}\qquad s(f)<t(e)\\[+1ex]
&\big((\low \alpha 2)_{s(e)}\cdot (\low \beta 2)_{s(f)}\big)\oo \low\alpha 1\oo\low\beta 1=\langle R, \low\beta 3\oo \low\alpha 3\rangle\; \big((\low\beta 2 )_{s(f)}\cdot (\low\alpha 2)_{t(e)}\big)\oo \low\alpha 1\oo\low\beta 1\\
&=\langle R, \low\beta 2\oo \low\alpha 2\rangle\; \big((\beta_{(1)(2)} )_{s(f)}\cdot (\alpha_{(1)(2)})_{s(e)}\big)\oo \alpha_{(1)(1)}\oo\beta_{(1)(1)}\qquad\quad\; s(f)<s(e).
\end{align*}
Note  that  the second and the third identity are also valid for the case  $e=f$. 
This shows that after the reordering, the two factors for the edges $e,f$ of $\Gamma$ are again of the form $(\alpha'_{(1)})_{t(e)}$, $(\alpha'_{(2)})_{s(e)}$ and $(\beta'_{(1)})_{t(f)}$, $(\beta'_{(2)})_{s(f)}$, for some $\alpha',\beta'\in K^*$.  After completing the reordering in step (i), one obtains linear combinations of  products of factors $(\alpha'^{i}_{(1)})_{t(e_i)}\cdot (\beta'^{i}_{(1)})_{t(e_i)}$ and $(\alpha'^i_{(2)})_{s(e_i)}\cdot (\beta'^i_{(2)})_{s(e_i)}$ for each edge $e_i$ of $\Gamma$, which are ordered according to the  ordering of the edge ends at each vertex, and in which the factors for edge ends at different vertices commute. To multiply the contributions for each edge end, recall  that by construction we either have $\sigma(t(e))=1-\sigma(s(e))=1$ or $\sigma(t(e))=1-\sigma(s(e))=0$ for each edge $e$ of $\Gamma$. 
In the first case we obtain
\begin{align}\label{eq:sampcalc}
&\big( (\low \alpha 1)_{t(e)}\cdot (\low\beta 1)_{t(e)}\big)\oo \big((\low\alpha 2)_{s(e)}\cdot (\low\beta 2)_{s(e)}\big)
= \langle R,\beta_{(1)(1)}\oo\alpha_{(1)(1)}\rangle\; (\beta_{(1)(2)}\alpha_{(1)(2)})_{t(e)}\oo (\low\beta 2\low\alpha 2)_{s(e)}\nonumber\\
&= \langle R,\beta_{(1)}\oo\alpha_{(1)}\rangle\; (  (\beta_{(2)}\alpha_{(2)})_{(1)})_{t(e)}\oo ( (\low\beta 2\low\alpha 2)_{(2)})_{s(e)}
\end{align}
and in the second case
\begin{align}\label{eq:samcomp}
&\big( (\low \alpha 1)_{t(e)}\cdot (\low\beta 1)_{t(e)}\big)\oo \big((\low\alpha 2)_{s(e)}\cdot (\low\beta 2)_{s(e)}\big)
=\langle R, \beta_{(2)(2)}\oo\alpha_{(2)(2)}\rangle (\low \alpha 1\low\beta 1)_{t(e)}\oo (\alpha_{(2)(1)}\beta_{(2)(1)})_{s(e)}\nonumber\\
&=\langle R, \beta_{(3)}\oo\alpha_{(3)}\rangle (\low \alpha 1\low\beta 1)_{t(e)}\oo (\alpha_{(2)}\beta_{(2)})_{s(e)}
= \langle R,\low\alpha 1\oo\low\beta 1\rangle\; (\low\beta 2\low\alpha 2)_{t(e)}\oo (\low\beta 3\low\alpha 3)_{s(e)}\nonumber\\
&= \langle R,\low\alpha 1\oo\low\beta 1\rangle\; ( (\low\beta 2\low\alpha 2)_{(1)})_{t(e)}\oo ( (\low\beta 2\low\alpha 2)_{(2)})_{s(e)}.
\end{align}
This shows that after reordering the factors and multiplying the two factors for each edge end, the product $G^*( \alpha^1\oo...\oo \alpha^E)\cdot G^*( \beta^1\oo...\oo \beta^E)$ becomes a linear combination of   products  of  the factors $(\gamma^i_{(1)})_{t(e_i)}$ and  $(\gamma^i_{(2)})_{s(e_i)}$ for the different edges $e_i$, ordered according to the ordering  at each vertex. All such products are in the image of $G^*$, and hence the image of $G^*$ is a subalgebra of $\oo_{v\in V}\mathcal A^*_v$.

To prove 2, suppose  that  the same $R$-matrix is assigned to each vertex of $\Gamma$ and note
that the identity $\Delta^{op}=R\cdot \Delta\cdot R^\inv$ implies    $\langle \low\alpha 1\oo\low\beta 1,R\rangle\, \low\alpha 2\low\beta 2=\langle \low\alpha 2\oo\low\beta 2, R\rangle \, \low\beta 1\low\alpha 1$ for all $\alpha,\beta\in K^*$. This shows that the expressions for $\rho(e)=s(e)$ and $\rho(e)=t(e)$ in equations \eqref{eq:sampcalc} and \eqref{eq:samcomp} agree.
That $\Pi$ is a projector   on $G^*({K^*}\exoo{E})$ 
then follows from Lemma  \ref{lem:1proj}.
\end{proof}

\medskip
\begin{remark}\label{rem:sigg}  Proposition \ref{lem:edge_algebra}
motivates the introduction of the maps $\sigma:\{1,...,n\}\to\{0,1\}$ in Proposition \ref{th:algebra} and Corollary \ref{rem:flipalg}  and of the map $\rho: E(\Gamma)\to E(\Gamma_\circ)$ at the beginning of this subsection.  
It is clear from  equations  \eqref{eq:sampcalc}  and \eqref{eq:samcomp} that  without them,  the image of   $G^*$ would not be a subalgebra of $\oo_{v\in V}{K^*}\exoo{v}$. 
This is because edge orientation is reversed in Corollary \ref{rem:flipalg} by applying  
 the involution $T^*: K^*\to K^*$, $\alpha\mapsto \langle \low\alpha 1,g\rangle\, S(\low\alpha 2)$, which is an {\em algebra anti-homomorphism}. Hence, for each edge $e$ the  algebra structures  for the  starting  end $s(e)$ and the target end  $t(e)$ are opposite if  
 one sets
 $\sigma(s(e))=\sigma(t(e))$. This mismatch between the two opposite algebras  prevents the image of $G^*$ from being a subalgebra. 
 To make the  algebra structures at the edge ends compatible, it is necessary to 
 modify the algebra structure at exactly one of these edge ends by introducing an $R$-matrix. The identity $\langle \low\alpha 1\oo\low\beta 1, R\rangle\, \low\alpha 2\low\beta 2=\langle \low\alpha 2\oo\low\beta 2, R\rangle\, \low\beta 1\low\alpha 1$ for $\alpha,\beta\in K^*$  then ensures  the compatibility of the algebra structures at the edge ends $s(e)$ and $t(e)$. \end{remark}

\medskip
\begin{proposition} \label{lem:gtrafolem} Let $K$ be a finite-dimensional 
ribbon 
Hopf algebra  and $\Gamma$ a ciliated ribbon graph  equipped with the data  above. Then:\\[-4ex]
\begin{compactenum}
\item $G^*({K^*}\exoo{E})\subset \oo_{v\in V}\mathcal A^*_v$  is a $K\exoo{V}$-submodule. 
\item The induced  $K\exoo{V}$-left module structure on $K\exoo{E}$  is given by
\begin{align}\label{eq:gtrafok}
&\pi_e((h)_v\rhd(k)_e)=\epsilon(h)\, k &  &\text{for }v\notin\{ \st(e), \ta(e)\},\\
&\pi_e((h\oo h')_{\st(e)\ta(e)}\rhd(k)_e)=h' k S(h) & &\text{for }\st(e)\neq \ta(e),\nonumber\\
&\pi_e((h)_{\st(e)}\rhd(k)_e)=\low h 1 k S(\low h 2) & &\text{for } t(e)<s(e),\nonumber\\
&\pi_e((h)_{\st(e)}\rhd(k)_e)=\low h 2 k S(\low h 1) &  &\text{for }s(e)<t(e).\nonumber
\end{align}
\end{compactenum}
\end{proposition}
\begin{proof} 

That $G^*({K^*}\exoo{E})\subset \oo_{v\in V}\mathcal A^*_v$ is a $K\exoo{V}$-submodule  follows by a direct computation 
from formulas  \eqref{eq:dualemb} and \eqref{eq:act_v}.  For each edge $e\in E$, they imply 
$G^*((\alpha)_e)\lhd^*(h)_u=\epsilon(h)\; G^*((\alpha)_e)$ for all  $u\notin\{ \st(e), \ta(e)\}$, $h\in K$, $\alpha\in K^*$. If  $\st(e)\neq \ta(e)$ one obtains 
\begin{align*}
& G^*((\alpha)_e)\lhd^*(h\oo h')_{\st(e)\ta(e)}= (\low\alpha 2\oo \low\alpha 1)_{s(e)t(e)}\lhd^*(h\oo h')_{\st(e)\ta(e)}\\
&
=\langle \alpha_{(2)(2)},S(h)\rangle\langle \alpha_{(1)(1)}, h'\rangle\; (\alpha_{(2)(1)}\oo \alpha_{(1)(2)})_{s(e)t(e)}
= \langle \low\alpha 3, S(h)\rangle\langle \low\alpha 1, h'\rangle\; G^*((\low\alpha 2)_e).
\end{align*}
For  a loop $e$ with $t(e)>s(e)$ formulas \eqref{eq:dualemb} and \eqref{eq:act_v} yield 
\begin{align*}
&G^*((\alpha)_e)\lhd^*(h)_{v}=((\low\alpha 2)_{s(e)}\lhd^*(\low h 1)_{v})\cdot (  (\low\alpha 1)_{t(e)}\lhd^*(\low h 2)_{v})\\
&=\langle \alpha_{(2)(2)}, S(\low h 1)\rangle \langle \alpha_{(1)(1)}, \low h 2\rangle\; (\alpha_{(2)(1)})_{s(e)}\cdot (\alpha_{(1)(2)})_{t(e)}=\langle \low\alpha 3,  S(\low h 1)\rangle \langle \low\alpha 1, \low h 2\rangle\; G^*((\low\alpha 2)_e),
\end{align*}
and the corresponding expression for a loop  with $t(e)<s(e)$  follows by applying the involution $T^*$. These identities imply \eqref{eq:gtrafok}  by duality. 
\end{proof}

Propositions \ref{lem:edge_algebra} and  \ref{lem:gtrafolem} allow one to pull back the algebra structure and module structure on 
 $\otimes_{v\in V}\mathcal A^*_v$ to ${K^*}\exoo{E}$ with the
 embedding $G^*$ from \eqref{eq:dualemb}. Proposition \ref{lem:gtrafolem} also shows that the resulting structures on $K^{*\oo E}$ satisfy  the axioms in Definition \ref{def:gtheory}.
By combining these two lemmas one then obtains a  local $K$-valued local Hopf algebra gauge theory on  $\Gamma$.

\begin{theorem} \label{def:algstruc} Let $K$ be a finite-dimensional ribbon Hopf algebra and  $\Gamma$ a ciliated ribbon graph.
Assign to each vertex $v$ of $\Gamma$  an  $R$-matrix $R_v$, maps $\tau_v,\sigma_v:\{1,...,|v|\}\to\{0,1\}$ as defined at the beginning 
of this subsection 
and the associated algebra $\mathcal A^*_v$ from Corollary \ref{rem:flipalg}. 
Then the  $K\exoo{V}$-module  algebra   structure on $\oo_{v\in V}\mathcal A^*_v$ defines a local $K$-valued Hopf algebra 
gauge theory on $\Gamma$ via  \eqref{eq:dualemb}. This algebra structure on ${K^*}\exoo{E}$ is denoted $\mathcal A^*$ in the following.
\end{theorem}

Clearly, the  algebra $\mathcal A^*$ from Theorem \ref{def:algstruc}  depends on the choice of  the cilium at each vertex of $\Gamma$. However,  one finds  that the
 algebra of observables $\mathcal A^*_{inv}\subset \mathcal A^*$ is largely independent of this choice and fully independent  of it
in the semisimple case.

\begin{proposition} \label{lem:invspace}  Let 
$\Gamma$ be a ciliated ribbon graph  and $\Gamma'$   obtained from  $\Gamma$ by moving the cilium at an $n$-valent vertex $v\in V(\Gamma)$ over the   $n$th edge end.  Then  the map 
 $$\phi_{\tau(n)}:K^{*\oo n}\to K^{*\oo n},\quad   
\alpha^1\oo ...\oo \alpha^n\mapsto   \langle g^{-1+2\tau(n)}, \alpha^n_{(1+\tau(n))}\rangle  \;   \alpha^1\oo...\oo\alpha^{n-1}\oo\alpha^n_{(2-\tau(n))}$$ 
induces an algebra isomorphism $\mathcal A^*_{inv}\to \mathcal A'^*_{inv}$. 
If  $K$ is semisimple, then  $\phi_{\tau(n)}$ reduces to a cyclic permutation of the tensor factors and $\mathcal A^*_{inv}\subset \mathcal A^*$ is independent of the choice of cilia. 
\end{proposition}
\begin{proof} 
As $T^*\circ \phi_{0}\circ T^*=\phi_1$, 
it is sufficient to consider the vertex neighbourhood of an $n$-valent vertex $v\in V$ at which all edges are incoming and to show that $\phi_0$ maps  the algebra $\mathcal A^*_{v\;inv}$ associated with $\Gamma$ to  the algebra $\mathcal A'^*_{v\;inv}$ associated with $\Gamma'$. 
As
the action of  a gauge transformation at $v$ on $\mathcal A^*_v$ and $\mathcal A'^*_v$ is given by
\begin{align}\label{eq:hmodref}
&(\alpha^1\oo...\oo\alpha^n)\lhd^*(h)_v=\langle \alpha^1_{(1)}\alpha^2_{(1)}\cdots\alpha^n_{(1)}, h\rangle\; \alpha^1_{(2)}\oo ...\oo\alpha^n_{(2)}\\
&(\alpha^1\oo...\oo\alpha^n)\lhd'^*(h)_v=\langle \alpha^n_{(1)}\alpha^1_{(1)}\cdots\alpha^{n-1}_{(1)}, h\rangle\; \alpha^1_{(2)}\oo ...\oo\alpha^n_{(2)},
\end{align}
one obtains  for all $h\in K$
\begin{align*}
&\phi_0(\alpha^1\oo...\oo\alpha^n)\lhd'^* (h)_v=\langle \alpha^n_{(1)}, g^\inv\rangle\, \langle \alpha^n_{(2)}\alpha^1_{(1)}\cdots\alpha^{n-1}_{(1)},h\rangle\; \alpha^1_{(2)}\oo...\oo\alpha^n_{(3)}\\
&=\langle \alpha^n_{(3)}, g^\inv\rangle\, \langle \alpha^n_{(4)}(\alpha^1_{(1)}\cdots\alpha^{n}_{(1)})S(\alpha^n_{(2)}),h\rangle\; \alpha^1_{(2)}\oo...\oo\alpha^{n-1}_{(2)}\oo\alpha^n_{(3)}\\
&=\langle \alpha^1_{(1)}\cdots\alpha^n_{(1)}, \low h 2\rangle \langle \alpha^n_{(2)}, S(\low h 3)g^\inv \low h 1\rangle\; \alpha^1_{(2)}\oo...\oo\alpha^{n-1}_{(2)}\oo\alpha^n_{(3)}\\
&=\langle \alpha^1_{(1)}\cdots\alpha^n_{(1)}, \low h 2\rangle \langle \alpha^n_{(2)}, g^\inv\rangle \langle \alpha^n_{(3)}, S^\inv(\low h 3) \low h 1\rangle\; \alpha^1_{(2)}\oo...\oo\alpha^{n-1}_{(2)}\oo\alpha^n_{(4)}\\
&=\phi_0((\alpha^1\oo...\oo\alpha^n)\lhd^* (\low  h 2 )_v)\lhd_{n} (S^\inv(\low h 3)\low h 1),
\end{align*}
where $\lhd_n: K^{*\oo n}\oo K\to K^{*\oo n}$ is the right regular action of $K$ on  the $n$th copy of $K^*$ in $K^{*\oo n}$. For  $\alpha\in A^*_{v\; inv}$ this yields
$
\phi_0(\alpha)\lhd'^* (h)_v
=\epsilon(\low h 2) \phi_0(\alpha)\lhd_{n}(S^\inv(\low h 3)\low h 1)
=\epsilon(h)\;\phi_0(\alpha)
$
and hence $\phi_0(\alpha)\in\mathcal A'^*_{v\;inv}$.
 To show that $\phi_0$ induces an algebra morphism on the invariants, we use the expression for the algebra structure in Proposition \ref{th:algebra}  and the identities 
$(\id\oo\Delta)(R)=R_{13}R_{12}$ and $(\Delta\oo\id)(R)=R_{13}R_{23}$ 
to compute  the product of two general elements of $K^{*\oo n}$
\begin{align}\label{eq:prodref0}
&(\alpha^1\oo...\oo\alpha^n)\cdot (\beta^1\oo...\oo \beta^n)\nonumber\\
=&\langle R, \beta^1_{(1)}\oo \alpha^1_{(1)}\cdots \alpha^n_{(1)}\rangle\langle R, \beta^2_{(1)}\oo \alpha^2_{(2)}\cdots\alpha^n_{(2)}\rangle\ldots \langle R, \beta^{n-1}_{(1)}\oo \alpha^{n-1}_{(n-1)}\alpha^n_{(n-1)}\rangle\langle R,\beta^n_{(1)}\oo \alpha^n_{(n)}\rangle\nonumber\\
&\beta^1_{(2)}\alpha^1_{(2)}\oo \beta^2_{(2)}\alpha^2_{(3)}\oo ...\oo \beta^{n-1}_{(2)}\alpha^{n-1}_{(n)}\oo \beta^n_{(2)}\alpha^n_{(n+1)}
\end{align}
and note that if  $\alpha^1\oo...\oo\alpha^n$ is an invariant with respect to the $K$-module structure $\lhd^*$ at the vertex $v$, then the first line in \eqref{eq:hmodref} implies
\begin{align}\label{eq:invrefcond}
(\alpha^1_{(1)}\alpha^2_{(1)}\cdots \alpha^n_{(1)})\oo \alpha^1_{(2)}\oo \ldots\oo \alpha^n_{(2)}=1\oo \alpha^1\oo....\oo\alpha^n.
\end{align}
Denoting by $\cdot$ and $\cdot'$ the multiplication maps for $\Gamma$ and $\Gamma'$ we then obtain
after a lengthy but direct computation using the properties of a quasitriangular Hopf algebra and of the grouplike element $g$
\begin{align*}
&\phi_0(\alpha^1\oo...\oo\alpha^n)\cdot' \phi_0(\beta^1\oo...\oo \beta^n)
=\phi_0\big(  (\alpha^1\oo...\oo\alpha^n)\cdot (\beta^1\oo...\oo \beta^n) \big).
\end{align*}
It also follows directly from the definitions that the image $G^*(K^{*\oo E})$ is invariant under the map $\phi_{\tau(n)}$. If the $n$th edge end at the vertex $v$ that is shifted 
$\oo_{v\in V}\mathcal A^*_v$  is the target end of $k$th edge of $\Gamma$, then we have from 
 \eqref{eq:dualemb} and  the definition of $\phi_{\tau(n)}$
\begin{align*}
&\phi_{\tau(n)}\circ G^*(\alpha^1\oo...\oo \alpha^k\oo....\oo \alpha^E)=G^*(\langle g^\inv,\alpha^k_{(1)}\rangle \alpha^1\oo...\oo\alpha^k_{(2)}\oo...\oo\alpha^E)
\end{align*}
  and if it is the starting end of the $k$th edge of $\Gamma$, we have
\begin{align*}
&\phi_{\tau(n)}\circ G^*(\alpha^1\oo...\oo \alpha^k\oo....\oo \alpha^E)=G^*(\langle g,\alpha^k_{(2)}\rangle \alpha^1\oo...\oo\alpha^k_{(1)}\oo...\oo\alpha^E).
\end{align*}
This shows that $\phi_{\tau(n)}$ induces an isomorphism of
$K^{\oo V}$-modules  $\phi_{\tau(n)}: K^{*\oo E}\to K^{*\oo E}$ and an algebra isomorphism between the subalgebras of invariants. 
 If $K$ is semisimple then $g=1$ and $\phi_0=\id$.
\end{proof}

 There is an also alternative description of the $K^{\oo V}$-module algebra  $\mathcal A^*_\Gamma$ associated with a ribbon graph $\Gamma$  in terms of the {\em braided tensor product} of the algebras $\mathcal A^*_v$. In this case, the braided tensor product is taken with respect to a $K$-right module structure at a bivalent vertex in the middle of each edge $e\in E$, and one must take the braided opposite for the algebra structure at each edge end. In this description, the $K^{\oo V}$-module algebra is then recovered as the subalgebra of invariants of the {\em braided} tensor product of the algebras $\mathcal A^*_v$ with respect to the resulting $K^{\oo E}$-module structure.

\begin{lemma}  Let $\Gamma$ be a ribbon graph and denote for each vertex $v\in V$ of $\Gamma$ by $\mathcal A^*_v$ the $K$-right module algebra from Corollary \ref{rem:flipalg} with the choice $\sigma(i)=1$ for all edges incident at $v$. Equip each edge end with a univalent vertex, corresponding to the middle  of the associated edge $e\in E$, and the associated $K$-module structure. Then:
\begin{compactenum}
\item The algebra $\mathcal A^*_v$ is a $K\oo K^{\oo |v|}$-right module algebra with respect to the $K$-right module structure at $v$ and the $K^{\oo |v|}$-right module structure defined by the middle vertices.\\[-2ex]
\item The braided tensor product  $\oo_{v\in V}^{br} \mathcal A^*_v$ with respect to the $K^{\oo E}$-right module structure at the middle vertices and with the ordering $t(e)<s(e)$ for all edges $e\in E$ is a $K^{\oo V}\oo K^{\oo E}$-right module algebra. Its subalgebra of invariants with respect to the $K^{\oo E}$-right module structure is the $K^{\oo V}$-module algebra $\mathcal A^*_\Gamma$. 
\end{compactenum}
\end{lemma}

\begin{proof}
1.~If $\sigma(i)=1$ for all edge ends $i$ incident at $v$, we have the algebra structure of the braided opposite $(\alpha)_i\cdot (\beta)_i=\langle R,\low\beta 1\oo\low\alpha 1\rangle (\low\beta 2\low\alpha 2)$ for each edge end $i$ at $v$. The $K$-right-module structure for the univalent vertex at the edge end $i$ is given by $(\alpha)\lhd^* h=\langle S(h), \low\alpha 2\rangle (\low\alpha 1)_i$ if $i$ is incoming at $v$ and by $(\alpha)_i\lhd^* h=\langle h, \low\alpha 1\rangle (\low\alpha 2)_i$ for $i$ outgoing at $v$. A direct computation shows that in both cases this  defines  a $K$-right module algebra structure on $K^*$ and that, together with the $K$-right module structure in Corollary \ref{rem:flipalg}  a $K\oo K$-right module algebra structure. As each of the $K$-right module algebra structures  at the univalent vertices affect only a single edge end, it follows directly that this defines a $K\oo K^{\oo|v|}$-right module algebra structure on $\mathcal A^*_v$.

2.~That $\oo_{v\in V}^{br} \mathcal A^*_v$  is a $K^{\oo V}\oo K^{\oo E}$-right module algebra then follows directly from the properties of the braided tensor product and the fact that the $K$-right module structures at the vertices $v\in V$ and at the middle vertices commute. For the middle vertex of an edge $e\in E$, the associated $K$-right module structure acts trivially on all edge ends $s(f)$, $t(f)$ for $f\in E\setminus\{e\}$ and
\begin{align*}
(\alpha\oo \beta)_{t(e)s(e)}\lhd^*(h)_e=\langle h, S(\low\alpha 2)\low\beta 1\rangle \, (\low\alpha 1\oo \low\beta 2)_{s(e)t(e)}.
\end{align*}
From this expression, one finds that its invariants are given by
\begin{align*}
(K^*\oo K^*)_{inv\, e}=\Delta(K^*)\subset K^*\oo K^*.
\end{align*}
and hence the subalgebra of invariants with respect to the $K^{\oo E}$-right module structure is precisely the image  of the map $G^*$ from  \eqref{eq:dualemb} with the induced $K^{\oo V}$-right module structure.  As  $\oo_{v\in V}^{br} \mathcal A^*_v$ is a $K^{\oo V}\oo K^{\oo E}$-right module algebra, the braided tensor products at the vertices $v\in V$ are not affected by the braidings with the braided tensor products at the middle vertices. It remains to show that for each edge $e\in E$, the braided algebra structure of $\oo_{v\in V}^{br} \mathcal A^*_v$ for $\sigma(s(e))=\sigma(t(e))=0$ induces the same algebra structure as $\bigotimes_{v\in V} \mathcal A^*_v$ for $\sigma(t(e))\neq \sigma(s(e))$.
This follows by a direct computation. If $e$ is an edge with $\st(e)\neq \ta(e)$, one obtains
\begin{align*}
&G^*((\alpha)_e)\cdot_{br} G^*((\beta)_e)=(\low\alpha 1\oo \low\alpha 2)_{s(e)t(e)}\cdot(\low\beta 1\oo \low\beta 2)_{t(e)s(e)}\\
&=\langle R, S(\beta_{(1)(3)})\oo \alpha_{(2)(1)}\rangle \langle R, \beta_{(1)(1)}\oo \alpha_{(1)(1)}\rangle \langle R, \beta_{(2)(1)}\oo \alpha_{(2)(2)}\rangle (\beta_{(1)(2)}\alpha_{(1)(2)}\oo \beta_{(2)(2)}\alpha_{(2)(3)})_{t(e)s(e)}\\
&=\langle R, S(\beta_{(3)})\beta_{(4)}\oo \alpha_{(3)}\rangle \langle R, \beta_{(1)}\oo \alpha_{(1)}\rangle  (\beta_{(2)}\alpha_{(2)}\oo \beta_{(5)}\alpha_{(4)})_{t(e)s(e)}\\
&= \langle R, \beta_{(1)}\oo \alpha_{(1)}\rangle  (\beta_{(2)}\alpha_{(2)}\oo \beta_{(3)}\alpha_{(3)})_{t(e)s(e)}=\langle R,\low\beta 1\oo \low\alpha 1\rangle G^*((\low \beta 2\low\alpha 2)_e) .
\end{align*}
A comparison with Corollary \ref{rem:flipalg} shows that taking the tensor product $\oo_{v\in V}\mathcal A^*_v$ with $\sigma(t(e))=1-\sigma(s(e))$ yields the same result (see also Proposition \ref{lem:algexplicit} (a)). The computation for a loop $e$ with $\st(e)=\ta(e)$ are analogous (see also Proposition \ref{lem:algexplicit} (b)).
\end{proof}

It is instructive to consider the case of a {\em cocommutative}  
Hopf algebra $K$. In this case $K$ is trivially 
ribbon  with universal $R$-matrix $R=1\oo 1$, ribbon element $\nu=u=1$, grouplike element $g=1$ and satisfies $S^2=\id$. The algebra structure  on the vertex neighbourhood from  Proposition \ref{th:algebra} 
reduces to the $n$-fold tensor product  $K^{*\oo n}$ of the {\em commutative} algebra $K^*$ with itself. Consequently, the subalgebra $\mathcal A^*\subset \oo_{v\in V}\mathcal A^*_v\cong {K^*}\exoo{E}$ is isomorphic  as an algebra to ${K^*}\exoo{E}$.  As $K$ is cocommutative, the  $K$-module structure on $K^{*\oo n}$ does not depend on the cyclic ordering of the incident edges at 
$v$, and the same holds for the $K\exoo{V}$-module structures on $K\exoo{E}$ and ${K^*}\exoo{E}$.  

Another instructive example is the case where  $K$ is the group algebra $\FF[G]$ of a finite group $G$.
 The Hopf algebra structure of $\FF[G]$ and its dual $\FF [G]^*=\mathrm{Fun}(G)$ are given in Example \ref{ex:finite_group}.
By applying the results and definitions above, 
one then finds that the  $K$-valued Hopf algebra gauge theory on  $\Gamma$ reduces to group gauge theory for  $G$.

\begin{example}  Let $G$ be a finite group and $\Gamma$ a ribbon graph. Then the Hopf algebra gauge theory on $\Gamma$ with values in $\FF[G]$ is given by the following:
\begin{compactenum}
\item  The space of connections is the vector space  $\FF[G]\exoo{E}\cong \FF[G\extimes{E}]$.
\item The algebra
 $\mathcal A^*$  of functions  is  the algebra  $\mathrm{Fun}(G\extimes{E})$ with the pointwise multiplication. 
 \item 
 The Hopf algebra of gauge transformations is the Hopf algebra $\FF[G]\exoo{E}\cong \FF[G\extimes{E}]$.
 \item  If all edge ends at $v\in V$ are incoming, a  gauge transformation at $v$  is given by\newline
 $\lhd^*\!\!:  \mathrm{Fun}(G\extimes{v})\oo \FF[G]\to\mathrm{Fun}(G\extimes{v})$, $( f\lhd^*h)(g_1,...,g_{|v|})
 =f(h g_1,...,h g_{|v|})$.
 \item Its action  on  an outgoing edge end $i$ is obtained by replacing $h g_i\to g_i h^\inv$.
 \item  The projector  on the gauge invariant subalgebra $\mathcal A^*_{inv}\subset A^*$ is given by $\Pi(f)=\Sigma_{h\in G\extimes{V}}\;   f\lhd^* h.$

 \end{compactenum}
\end{example}

\subsection{Explicit description of the algebra of functions}
\label{subsec:algfunc}

In this section, we derive an explicit description of the algebra $\mathcal A^*$ for a Hopf algebra gauge theory on a ciliated ribbon graph $\Gamma$ in terms of multiplication relations. This  allows us to relate it to the so-called {\em lattice algebra}  obtained in \cite{AGSI,AGSII, BR} via the combinatorial quantisation of Chern-Simons theory.  From the discussion in the previous subsections, it is obvious that a presentation of the algebra $\mathcal A^*$ for any ribbon graph can be obtained from 
 the multiplication relations in  Corollary \ref{rem:flipalg} together with formula \eqref{eq:dualemb}. This requires a straightforward but  lengthy computation which takes into account the relative ordering and orientation of edge ends  involved. Up to edge orientation, which can be reversed with the involution $T^*$ from \eqref{eq:invol}, one has to distinguish twelve edge constellations. 
 For simplicity, we restrict attention to the case where each vertex neighbourhood is equipped with the same  $R$-matrix. The multiplication relations for the  are then given by the following proposition.

\begin{proposition}\label{lem:algexplicit} Let $K$ be a finite-dimensional ribbon Hopf algebra. Consider the $K$-valued local Hopf algebra gauge theory 
from Theorem \ref{def:algstruc} and suppose that
each vertex is assigned the same $R$-matrix $R$. Then
the algebra $\mathcal A^*$  is characterised by the following multiplication relations on ${K^*}\exoo{E}$: \\[-2ex]

\begin{compactenum}[(a)]
\item For  $e\in E$ with  $\st(e)\neq \ta(e)$: 
\hfill\includegraphics[scale=.35]{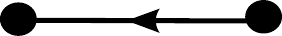}
\begin{flalign*}
&(\beta)_e\cdot(\alpha)_e=
\langle \low\alpha 1\oo\low\beta 1, R\rangle\, (\low\alpha 2\low\beta 2)_e &\\
\Rightarrow\; &(\beta)_e\cdot (\alpha)_e=\langle\low\alpha 1\oo\low\beta 1, R\rangle\langle \low\beta 3\oo\low\alpha 3, R^\inv\rangle\; (\low\alpha 2)_e\cdot (\low\beta 2)_e.
\end{flalign*}

\item  For a loop $e\in E$ with $t(e)<s(e)$: 
\hfill \begin{tikzpicture}
 \path[use as bounding box] (0,0) rectangle (2,0);
 \node [anchor=west] at (0,0)
 {\includegraphics[scale=.4]{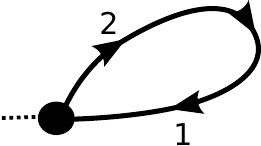}};
\end{tikzpicture}
\begin{flalign*}
& (\beta)_e\cdot (\alpha)_e=
\langle \low\alpha 1\oo S(\low\beta 3),  R\rangle\, \langle \low \alpha 2\oo\low\beta 1, R\rangle \; (\low\alpha 3\low\beta 2)_e.  &\\
\Rightarrow\;&(\beta)_e\cdot (\alpha)_e=\langle \low\alpha 1\oo S(\low\beta 5), R\rangle\langle \low\alpha 2\oo\low\beta 1, R\rangle \langle \low\beta 4\oo\low\alpha 5, R^\inv\rangle\langle \low\beta 2\oo \low\alpha 4, R\rangle \; (\low\alpha 3)_e\cdot  (\low\beta 3)_e.
\end{flalign*}

\item For $e,f\in E$ with no common vertex:
\hfill \begin{tikzpicture}
 \path[use as bounding box] (0,0) rectangle (2,0);
 \node [anchor=west] at (0,0)
 {\includegraphics[scale=.4]{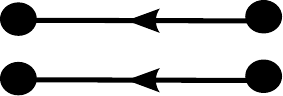}};
\end{tikzpicture}
\begin{flalign*}
&(\alpha)_e\cdot(\beta)_f=(\beta)_f\cdot(\alpha)_e=(\alpha\oo\beta)_{ef}. &
\end{flalign*}

\item  For  $e,f\in E$ with $\st(e)\neq \ta(e)$, $\st(f)\neq \ta(f)$, $\st(e)\neq \st(f)$ and 
$t(e)<t(f)$:
\hfill \begin{tikzpicture}
 \path[use as bounding box] (0,0) rectangle (2,0);
 \node [anchor=west] at (0,0)
 {\includegraphics[scale=.4]{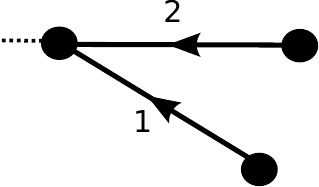}};
\end{tikzpicture}
\begin{flalign*}
&(\beta)_f\cdot (\alpha)_e=
\langle \low\alpha 1\oo\low\beta 1, R\rangle\; (\low\alpha 2\oo\low \beta 2)_{ef}\\
&(\alpha)_e\cdot (\beta)_f=(\alpha\oo \beta)_{ef} &\\
\Rightarrow\; &(\beta)_f\cdot (\alpha)_e=\langle\low\alpha 1\oo\low\beta 1,R\rangle\; (\low\alpha 2)_e\cdot (\low\beta 2)_f.
\end{flalign*}

\item For $e\in E$ with $\st(e)\neq \ta(e)$ and  a loop $f\in E$   with $t(e)<t(f)<s(f)$:
\hfill \begin{tikzpicture}
 \path[use as bounding box] (0,0) rectangle (2,0);
 \node [anchor=west] at (0,0)
 {\includegraphics[scale=.4]{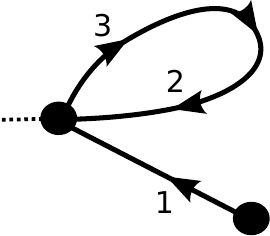}};
\end{tikzpicture}
\begin{flalign*}
& (\alpha)_f\cdot (\beta)_e=
\langle \low\beta 1\oo S(\low\alpha 3), R\rangle\, \langle \low\beta 2\oo \low\alpha 1, R\rangle\;   (\low\beta 3\oo\low\alpha 2)_{ef}\\
&(\beta)_e\cdot (\alpha)_f=(\beta\oo\alpha)_{ef} &\\
\Rightarrow\; &(\alpha)_f\cdot (\beta)_e=\langle \low\beta 1\oo S(\low\alpha 3),R\rangle\, \langle \low\beta 2\oo \low\alpha 1, R\rangle\;   (\low\beta 3)_e\cdot(\low\alpha 2)_f.
\end{flalign*}

\item For $e\in E$ with $\st(e)\neq \ta(e)$ and  a loop $f\in E$   with $t(f)<t(e)<s(f)$:
\hfill \begin{tikzpicture}
 \path[use as bounding box] (0,0) rectangle (2,0);
 \node [anchor=west] at (0,0)
 {\includegraphics[scale=.4]{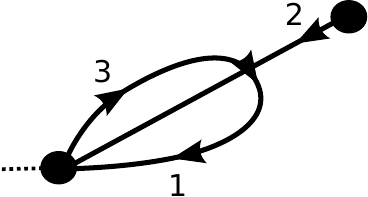}};
\end{tikzpicture}
\begin{flalign*}
& (\alpha)_f\cdot (\beta)_e
=\langle \low\beta 1\oo S( \low\alpha 2), R\rangle\; (\low \beta 2\oo\low \alpha 1)_{ef}\\
& (\beta)_e\cdot  (\alpha)_f
=\langle \low\alpha 1\oo \low\beta 1, R\rangle\;  (\low \beta 2\oo\low \alpha 2)_{ef} &\\
\Rightarrow\; &(\alpha)_f\cdot (\beta)_e=\langle\low\beta 1\oo S(\low\alpha 3), R\rangle\,\langle \low\alpha 1\oo\low\beta 2, R^\inv\rangle\;  (\low\beta 3)_e\cdot (\low\alpha 2)_f.
\end{flalign*}

\item For $e\in E$ with $\st(e)\neq \ta(e)$ and  a loop $f\in E$   with $t(f)<s(f)<t(e)$:
\hfill \begin{tikzpicture}
 \path[use as bounding box] (0,0) rectangle (2,0);
 \node [anchor=west] at (0,0)
 {\includegraphics[scale=.4]{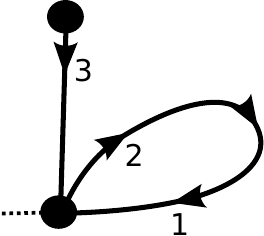}};
\end{tikzpicture}
\begin{flalign*}
& (\alpha)_f\cdot (\beta)_e=(\beta\oo\alpha)_{ef}\\
&  (\beta)_e\cdot (\alpha)_f
=\langle \low \alpha 1\oo\low \beta 1, R\rangle\langle S(\low \alpha 3)\oo  \low \beta 2, R\rangle\;   (\low \beta 3\oo\low \alpha 2)_{ef}. &\\
\Rightarrow\; &(\beta)_e\cdot(\alpha)_f= \langle \low\alpha 1\oo \low\beta 1 , R\rangle \langle S(\low\alpha 3) \oo\low\beta 2, R\rangle\; (\low\alpha 2)_f\cdot (\low\beta 3)_e.
\end{flalign*}

\item For two loops $e,f\in E$  with $t(e)<s(e)<t(f)<s(f)$:
\hfill \begin{tikzpicture}
 \path[use as bounding box] (0,-.2) rectangle (2,-.3);
 \node [anchor=west] at (0,0)
 {\includegraphics[scale=.4]{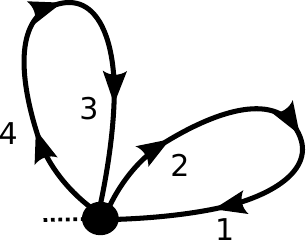}};
\end{tikzpicture}
\begin{flalign*}
& (\beta)_f\cdot (\alpha)_e
=\langle \low\alpha 1\oo S(\low\beta 5), R\rangle \langle\low\alpha 2\oo\low\beta 1, R\rangle 
\langle\low\alpha 5\oo\low \beta 4, R\rangle \langle \low\alpha 4\oo\low\beta 2, R^\inv\rangle\; (\low\alpha 3\oo\low\beta 3)_{ef}\\
& (\alpha)_e\cdot (\beta)_f=(\alpha\oo\beta)_{ef}\\
\Rightarrow\; &(\beta)_f\cdot (\alpha)_e=\langle \low\alpha 1\oo S(\low\beta 5), R\rangle \langle\low\alpha 2\oo\low\beta 1, R\rangle 
\langle\low\alpha 5\oo\low \beta 4, R\rangle \langle \low\alpha 4\oo\low\beta 2, R^\inv\rangle\; 
 (\low\alpha 3)_e \cdot  (\low\beta 3)_f. &
\end{flalign*}

\item  For two loops $e,f\in E$  with $t(e)<t(f)<s(e)<s(f)$:
\hfill \begin{tikzpicture}
 \path[use as bounding box] (0,.45) rectangle (1.3,.45);
 \node [anchor=west] at (0,0)
 {\includegraphics[scale=.4]{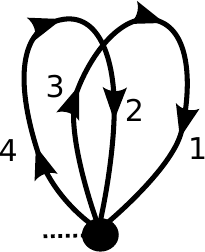}};
\end{tikzpicture}
\begin{flalign*}
&(\beta)_f\cdot (\alpha)_e
=\langle \low\alpha 1\oo S(\low\beta 4), R\rangle\langle \low\alpha 2\oo\low\beta 1, R\rangle \langle \low\alpha 4\oo\low\beta 3, R\rangle \;(\low\alpha 3\oo\low \beta 2)_{ef}\\
&(\alpha)_e\cdot (\beta)_f=\langle  S(\low \alpha 2)\oo \low\beta 1, R_{21}\rangle\; (\low\alpha 1\oo\low \beta 2)_{ef} &\\
\Rightarrow\; &(\beta)_f\cdot (\alpha)_e=\langle \low\alpha 1\oo S(\low\beta 5), R\rangle\langle \low\alpha 5\oo \low\beta 4, R\rangle \langle \low\alpha2\oo\low\beta 1, R\rangle \langle \low\alpha 4\oo\low\beta 2, R_{21}\rangle
\; (\low\alpha 3)_e\cdot (\low\beta 3)_f.
\end{flalign*}

\item  For two loops $e,f\in E$  with $t(e)<t(f)<s(f)<s(e)$:
\hfill \begin{tikzpicture}
 \path[use as bounding box] (0,.2) rectangle (2,.2);
 \node [anchor=west] at (0,0)
 {\includegraphics[scale=.4]{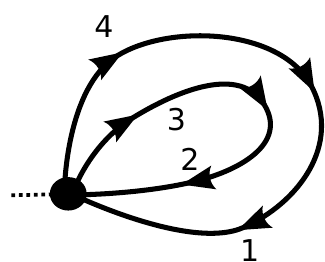}};
\end{tikzpicture}
\begin{flalign*}
& (\beta)_f\cdot (\alpha)_e
=\langle \low\alpha 1\oo S(\low\beta 3) , R\rangle \langle \low\alpha 2\oo \low\beta 1, R\rangle\; (\low\alpha 3\oo \low\beta 2)_{ef}\\
& (\alpha)_e\cdot (\beta)_f
=\langle S(\low\alpha 3)\oo\low\beta 1, R_{21}\rangle\; \langle \low\alpha 2\oo\low\beta 3, R_{21}\rangle\; (\low\alpha 1\oo\low \beta 2)_{ef} &\\
\Rightarrow\; &(\beta)_f\cdot (\alpha)_e=\langle \low\alpha 1\oo S(\low\beta 5),R\rangle\langle S(\low\alpha 5)\oo\low \beta 4, R_{21}\rangle\langle \low\alpha 2\oo\low\beta 1, R\rangle \langle \low\alpha 4\oo \low\beta 2, R_{21}\rangle
\; (\low\alpha 3)_e\cdot (\low\beta 3)_f.
\end{flalign*}

\item  For  $e,f\in E$ with $\st(e)\neq \ta(e)$, $\st(f)\neq \ta(f)$,  $t(e)<t(f)$ and $s(e)<s(f)$:
\hfill \begin{tikzpicture}
 \path[use as bounding box] (0,0) rectangle (2,0);
 \node [anchor=west] at (0,0)
 {\includegraphics[scale=.4]{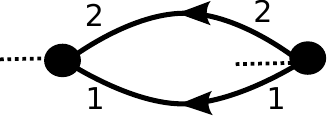}};
\end{tikzpicture}
\begin{flalign*}
&(\beta)_f\cdot (\alpha)_e
=\langle \low \alpha 1\oo\low\beta 1, R\rangle\langle \low\alpha 3\oo\low\beta 3, R\rangle\;  (\low\alpha 2\oo\low\beta 2)_{ef}\\
&(\alpha)_e\cdot (\beta)_f=(\alpha\oo\beta)_{ef} &\\
\Rightarrow\; &(\beta)_f\cdot (\alpha)_e=\langle \low \alpha 1\oo\low\beta 1, R\rangle\langle \low\alpha 3\oo\low\beta 3, R\rangle\;   (\low\alpha 2)_e\cdot  (\low\beta 2)_f.
\end{flalign*}

\item For $e,f\in E$ with $\st(e)\neq \ta(e)$, $\st(f)\neq \ta(f)$,  $t(e)<t(f)$ and $s(e)>s(f)$:
\hfill \begin{tikzpicture}
 \path[use as bounding box] (0,0) rectangle (2,0);
 \node [anchor=west] at (0,0)
 {\includegraphics[scale=.4]{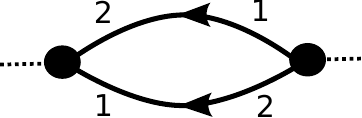}};
\end{tikzpicture}
\begin{flalign*}
&(\beta)_f\cdot (\alpha)_e 
=\langle \low\alpha 1\oo\low\beta 1, R\rangle\;  (\low \alpha 2\oo\low \beta 2)_{ef}\\
&(\alpha)_e\cdot(\beta)_f=
\,\langle \low\alpha 2\oo \low\beta 2, R_{21}\rangle\; (\low\alpha 1\oo\low\beta 1)_{ef}&\\
\Rightarrow\; &(\beta)_f\cdot (\alpha)_e= \langle\low\alpha 1\oo\low\beta 1, R\rangle\,\langle \low\alpha 3\oo\low\beta 3, R_{21}^\inv\rangle\; (\low\alpha 2)_e\cdot (\low\beta 2)_f.
\end{flalign*}
\end{compactenum}
The remaining  cases differ from the ones above only by edge orientation. They are obtained   from the ones above by applying the involution $T^*$ from \eqref{eq:invol}.\end{proposition}
\begin{proof} For case (c) this holds by definition. For the remaining cases
it  follows by a direct computation from
 \eqref{eq:flipalg} together with formula \eqref{eq:dualemb}.  Cases (a) and (b) were already  treated in equations \eqref{eq:sampcalc} and \eqref{eq:samcomp}. 
 We illustrate the other cases  by giving the computations for  case (d), since cases (e) to (l)
 are analogous.
Let $e,f\in E$  be edges with $\ta(e)=\ta(f)$, $\st(e)\neq \ta(e)$,  $\st(f)\notin\{\st(e), \ta(e)\}$ and  suppose that $t(e)<t(f)$.  Then one obtains from  \eqref{eq:flipalg}  and  \eqref{eq:dualemb}
\begin{align*}
& G^*((\beta)_f)\cdot G^*((\alpha)_e)=(\low\beta 2\oo\low\beta 1)_{s(f)t(f)}\cdot (\low\alpha 2\oo\low\alpha 1)_{s(e)t(e)}\\
&=(\low\alpha 2\oo\low\beta 2)_{s(e)s(f)}\cdot (\low\beta 2)_{t(f)}\cdot (\low\alpha 1)_{t(e)}\\
&=\langle \alpha_{(1)(1)}\oo\beta_{(1)(1)},R\rangle\;  (\low\alpha 2\oo\low\beta 2)_{s(e)s(f)}\cdot (\alpha_{(1)(2)})_{t(e)}\cdot  (\beta_{(1)(2)} )_{t(f)}\\
&=\langle\low\alpha 1\oo\low\beta 1, R\rangle\; (\low\alpha 3\oo \low\beta 3)_{s(e)s(f)}\cdot (\low \alpha 2)_{t(e)}\cdot  (\low\beta 2)_{t(f)}\\
&=\langle \low\alpha 1\oo \low\beta 1, R\rangle\, \cdot G^*((\low\alpha 2)_e) \cdot G^*((\low\beta 2)_f)=\langle R, \low \alpha 1\oo \low\beta 1\rangle\; G^*((\alpha\oo \beta)_{ef}). 
\end{align*}
\end{proof}

Proposition \ref{lem:algexplicit} gives an explicit description of the algebra structure that does not refer to the algebra structure on the vertex neighbourhoods. 
 In particular,  the multiplication relations in Proposition \ref{lem:algexplicit} (a), (b) show that  for each edge $e\in E$ the variables $(\alpha)_e$ with $\alpha\in K^*$,  form a subalgebra of $\mathcal A^*$. For a loop,  this algebra is related to $K^{op}$.

\begin{proposition}\label{lem:looplem} Let $K$ be a finite-dimensional ribbon Hopf algebra  and equip the vector space $K^*$ with the algebra structure from Proposition \ref{lem:algexplicit} (b)
$$
\beta \cdot' \alpha=\langle \low\alpha 1\oo S(\low\beta 3), R\rangle\langle \low\alpha 2\oo\low\beta 1, R\rangle \, \low\beta 2\low\alpha 3.
$$
Then the linear map $D: K^*\to K$, $D(\alpha)=\langle S^\inv(\alpha), \low Q 1\rangle \low Q 2$ with $Q=R_{21}R$ is an algebra homomorphism  from $(K^*,\cdot')$ to $K^{op}$. It is an isomorphism if and only if $K$ is factorisable.
\end{proposition}
\begin{proof} Note first that for all $\alpha,\beta\in K^*$, one has by definition
$$\langle D(\alpha),\beta\rangle
=\langle \low\alpha 2\oo\low\beta 1, R_{21}^\inv\rangle\langle S^\inv(\low\alpha 1)\oo\low\beta 2, R\rangle
=\langle D_{R_{21}^\inv}(\low\alpha 2) D_R(S^\inv(\low \alpha 1)), \beta\rangle,$$
where $D_R: K^*\to K$, $\alpha\mapsto \langle\alpha, \low R 1\rangle\,\low R 2$ and $D_{R_{21}^\inv}: K^*\to K$, $\alpha\mapsto \langle \alpha, S^\inv(\low R 2)\rangle\,\low R 1$ are the maps  from Lemma \ref{lem:quasprop}. 
 This implies 
$
D(\alpha\cdot \beta)=D_{R_{21}^\inv}(\low\alpha 2)\cdot D(\beta)\cdot D_R(S^\inv(\low\alpha 1))
$  
for all $\alpha,\beta\in K^*$.
As $\Delta^{op}=R\cdot \Delta\cdot R^\inv=R_{21}^\inv\cdot\Delta\cdot R_{21}$, one has 
$\Delta\cdot Q=Q\cdot \Delta$ and therefore for all $k\in K$, $\alpha\in K^*$
\begin{align*}
&(S^\inv(\low k 1)\oo 1)\cdot (S^\inv\oo\id)(Q)\cdot (1\oo\low k 2)=(1\oo\low k 2)\cdot (S^\inv \oo\id)(Q)\cdot (S^\inv(\low k 1)\oo 1),\\
&\langle \low \alpha 1, \low k 2\rangle\; D(\low \alpha 2)\cdot S(\low k 1)=\langle \low\alpha 2, \low k 2\rangle\; S(\low k 1)\cdot D(\low\alpha 1),
\end{align*}
where the second identity follows from the first by duality. This yields
\begin{align*}
&\langle \low\alpha 1\oo\low\beta 1, R\rangle\; D(\low \alpha 2\cdot \low \beta 2)=\langle \low\alpha 1\oo\low\beta 1, R\rangle \; D_{R_{21}^\inv}(\low\alpha 3) \cdot D(\low\beta 2)\cdot D_R(S^\inv(\low\alpha 2))\\
&=\langle \low\alpha 1\oo\low \beta 2, R\rangle\; D_{R_{21}^\inv}(\low\alpha 3)\cdot D_R(S^\inv(\low\alpha 2))\cdot D(\low\beta 1)=\langle\low\alpha 1\oo\low\beta 2, R\rangle\; D(\low\alpha 2)\cdot D(\low\beta 1),
\end{align*}
and it follows that
$D(\alpha)\cdot D(\beta)=\langle\low\alpha 1\oo S(\low\beta 3),R\rangle\langle\low\alpha 2\oo\low\beta 1, R\rangle\, D(\low\alpha 3\low\beta 2)$.
A comparison with Proposition \ref{lem:algexplicit} (b) proves that $D: (K^*,\cdot')\to K$ is an algebra homomorphism. By definition, the Hopf algebra $K$ is factorisable if and only if $(S\oo \id)\circ D:K^*\to K$   is a linear isomorphism, which is the case if and only if this holds for $D:K^*\to K$.
\end{proof}

For an edge $e$ with $\st(e)\neq \ta(e)$,  the characterisation of the algebra structure from Proposition \ref{lem:algexplicit} (a) is less immediate. If the 
 Hopf algebra $K$ is the {\em Drinfel'd double} $D(H)$ of a finite-dimensional  Hopf algebra $H$, one can show that this algebra is related to the Heisenberg double $\mathcal H_L(H)$  from Definition \ref{def:hdouble}. Note that in this case, the algebra in Proposition  \ref{lem:algexplicit} (b) is isomorphic to  $D(H)^{op}$ by Proposition \ref{lem:looplem} since the Drinfel'd double $D(H)$ of a
 finite-dimensional Hopf algebra $H$  is always factorisable.

\begin{proposition}\label{lem:edgelem} Let $K=D(H)$ be the Drinfel'd double of a finite-dimensional Hopf algebra $H$ and equip the vector space $K^*$ with the 
algebra  structure from Proposition \ref{lem:algexplicit} (a)
\begin{align}\label{eq:edgemult}
\beta\cdot'\alpha=\langle \low\alpha 1\oo\low\beta 1, R\rangle\;  \low\alpha 2\low\beta 2
\end{align}
where $\cdot$ denotes the multiplication of $K^*=H^{op}\oo H$ and $R=\Sigma_i1\oo x_i \oo \alpha^i\oo 1$  is the $R$-matrix of $D(H)$
from Theorem \ref{lem:ddouble2}.
Then the algebra  $(K,\cdot')$ is isomorphic to  $\mathcal H_L(H)^{op}$.
\end{proposition}

\begin{proof}  This follows by a direct computation.
From \eqref{eq:ddouble}, one finds that the comultiplication of $K^*=D(H)^*=H^{op}\oo H^*$ is given by
$ \Delta(x\oo\alpha)=\Sigma_{i,j}\;\low x 1\oo\alpha^i\low\alpha 1\alpha^j\oo S(x_j)\low x 2 x_i\oo \low\alpha 2$
 for all $x\in H$, $\alpha\in H^*$.  
 Inserting this expression into \eqref{eq:edgemult} together with the expression for $R$, one obtains
 for all $x,y\in H$ and $\alpha,\beta\in H^*$
\begin{align*}
& (y\oo\beta)\cdot' (x\oo\alpha)\\
&=\Sigma_{i,j,k,l,u}\; \langle \alpha^u, \low y 1\rangle\langle  \alpha^k \low\alpha 1\alpha^l, x_u\rangle\epsilon(\alpha^i\low\beta 1\alpha^j)\epsilon(\low x 1)\, (S(x_l)\low x 2 x_k\oo \low\alpha 2)\cdot (S^\inv(x_j)\low y 2 x_i\oo \low\beta 2)\\
&=\Sigma_{k,l}\langle  \alpha^k\low\alpha 1\alpha^l, \low y 1\rangle\, ( S^\inv(x_l)xx_k \oo\low\alpha 2)\cdot (\low y 2\oo\beta)\\
&=\langle \low\alpha 1,\low y 2\rangle\, ( S^\inv(\low y 3)x \low y 1\oo\low \alpha 2)\cdot (\low y 4\oo \low\beta 2)= \langle \low\alpha 1, \low y 2\rangle\, x\low y 1\oo\alpha\low\beta 2.
\end{align*}
Comparing the last expression with the first formula in Definition \ref{def:hdouble}, one finds that the flip map $H^*\oo H\to H\oo H^*$, $\alpha\oo x\mapsto x\oo\alpha$
defines an anti-algebra isomorphism from the left Heisenberg double $\mathcal H_L(H)$ in to the algebra structure in Proposition \ref{lem:algexplicit} (a).
\end{proof}

Finally, we can use Proposition \ref{lem:algexplicit} to show that a Hopf algebra gauge theory for a ribbon graph $\Gamma$ 
is related to the algebra
obtained from the combinatorial quantisation formalism for Chern-Simons gauge theory in  \cite{AGSI,AGSII,BR}.

 \begin{proposition} \label{lem:agslemma}Let $K$ be a finite-dimensional  ribbon Hopf algebra. Consider a $K$-valued Hopf algebra gauge theory on a ciliated ribbon graph $\Gamma$ in which  each vertex $v\in V$ is assigned the same $R$-matrix. Then the $K\exoo{V}$-right module  algebra structure from  Theorem \ref{def:algstruc} and Proposition \ref{lem:algexplicit} coincides with the one derived in \cite{AGSI,AGSII,BR}. 
 \end{proposition}
 \begin{proof}

As  the algebra structure in \cite{AGSI,AGSII,BR} is given  in terms of matrix elements of $K$ in its irreducible representations, we reformulate  Proposition \ref{lem:algexplicit} in terms of matrix elements. For an irreducible representation $\rho_I: K\to\mathrm{End}(V_I)$ of $K$ on a finite-dimensional $\FF$-vector space $V_I$, the matrix elements in terms of a  basis $\{v^I_a\}$ of $V_I$ are given by
 $\rho_I(k)v^I_a=\Sigma_{b} M_I(k)^b_{\;\;a}\,v^I_b$.  The associated elements of $K^*$  are defined by
  $\langle M_{I\;a}^b, k\rangle=M_I(k)^a_{\;\;b}$ for all $k\in K$, which implies $\Delta( M_{I\;a}^b)=\Sigma_c\,  M_{I\; a}^c\oo M_{I\;\;c}^b$.   Similarly,  the action of the $R$-matrix
on the tensor product of  two irreducible modules $V_I$, $V_J$ is characterised in terms of  matrix elements by
    $(\rho_I\oo \rho_J)(R)(v_a^I\oo v^J_b)=\Sigma_{c,d} R_{IJ\;ab}^{cd} v^I_c\oo v^J_c$. Using the notation    
 $M_I[e]^b_{\;\;a}, M_J[f]^b_{\;\;a},...$ for the elements $(M_{I\;a}^b)_e, (M_{J\; a}^b)_f,...\in {K^*}\exoo{E}$ and combining the matrix elements into matrices, we can then  rewrite the formulas from Proposition \ref{lem:algexplicit} in matrix notation  and obtain:
\begin{compactitem}
\item For an edge $e$ with $\st(e)\neq \ta(e)$, the expression in Proposition \ref{lem:algexplicit} (a)  takes the form
$M_J[e]M_{I}[e]=R_{IJ}M_I[e]M_J[e]R^\inv_{JI}$,
which agrees with  formula (2.46) in \cite{AGSI}, formula (43) in \cite{BR}.\\[-1ex]

\item For a loop $e$ with $ t(e)<s(e)$ we derive from Proposition \ref{lem:algexplicit} (b)
the relation 
$M_J[e]R_{IJ}M_{I}[e]=R_{IJ}M_I[e]R_{JI}M_{J}[e]R_{JI}^\inv$. This is the formula obtained by combining (2.6),  (2.7) and (2.17) in \cite{AGSII}, see also the first three formulas in Definition 12 in \cite{AGSII} and formula (46) in \cite{BR}. \\[-1ex]

\item For two edges $e$, $f$ which have nor vertex in common, Proposition \ref{lem:algexplicit} (c) reads in matrix notation
$M_I[e]M_J[f]=M_J[f]M_I[e]$, which coincides with formula (2.45) in \cite{AGSI},  formula (2.19) \cite{AGSII} and formula (45) in \cite{BR}.\\[-1ex]

\item For case (d) in Proposition \ref{lem:algexplicit}, we obtain in matrix notation
$M_J[f]M_I[e]=R_{IJ}M_I[e]M_J[f]$, which coincides with formula (2.47), (2.51) in \cite{AGSI}, formula (2.20) in \cite{AGSII}   and formulas (40) to (42) in \cite{BR} if the choice of orientation there is reversed.\\[-1ex]

\item For two loops $e,f$ with $t(e)<s(e)<t(f)<s(f)$ we obtain from Proposition  \ref{lem:algexplicit}  (h) the relation 
$R^\inv_{IJ}M_J[f]R_{IJ}M_I[e]=M_I[e]R_{IJ}^\inv M_J[f]R_{IJ}$. This coincides with the 5th to 11th equations of Definition 12 in \cite{AGSII}  if the first and second argument there are replaced by $e$ and $f$ and the choice of edge orientation there is taken into account.\\[-1ex]

\item For two loops $e,f$ with $t(e)<t(f)<s(e)<s(f)$ Proposition \ref{lem:algexplicit}), (i) yields the relation
$R^\inv_{IJ}M_J[f]R_{IJ}M_I[e]= M_I[e] R_{JI} M_J[f] R_{IJ}$. This coincides with the 4th formula in Definition 12 \cite{AGSII} if the arguments $a_i$, $b_i$ are replaced by $e$ and $f$ and the choice of edge orientation and ordering there is taken into account.\\[-1ex]

\item The formulas for the action of gauge transformations on the edge variables in equation  \eqref{eq:gkact}  and Proposition \ref{lem:gtrafolem} agree with the corresponding formulas  (2.49) in \cite{AGSI} and (2) in \cite{BR}.
\end{compactitem}
The multiplication relations  in Proposition  \ref{lem:algexplicit} (h) and (i) are not described in \cite{BR} because that paper restricts attention to 3-valent graphs without loops. The remaining edge constellations in Proposition \ref{lem:algexplicit} are not described explicitly in \cite{AGSI, AGSII,BR}, but the relations above are sufficient to establish equivalence. This follows in particular from Corollary \ref{cor:edgesubdiv} which allows one to restrict attention to bouquets or to ribbon graphs without loops or multiple edges. 
\end{proof}

Proposition \ref{lem:agslemma} shows that the $K\exoo{V}$-module algebra $\mathcal A^*_\Gamma$, the algebra of functions  in a local Hopf algebra gauge theory on $\Gamma$, agrees with the one obtained from the combinatorial quantisation procedure in  \cite{AGSI, AGSII,BR}.  The representation theory of the resulting algebra was investigated further in \cite{AS} and \cite{BR2}, which also relate  this algebra to Reshetikhin-Turaev invariants \cite{RT}.

However,  the two approaches that lead to this  module algebra structure are very different. 
 While  \cite{AGSI, AGSII,BR}  obtain it  by  {\em  canonically quantising}  the Poisson structure in \cite{FR} and \cite{AM} via the correspondence between quasitriangular Hopf algebras and quasitriangular Poisson-Lie groups, in the present article this algebra structure is {\em derived} from a number of simple axioms. 
 This addresses the question about the uniqueness of this algebra structure and quantisation approach. 
 
 It also exhibits clearly the mathematical structures associated with the notion of a Hopf algebra gauge theory, namely {\em module algebras} over a Hopf algebra and their {\em braided tensor products}.
 Moreover, it allows one to obtain the algebra structure in \cite{AGSI,AGSII,BR} from a  basic building block---the Hopf algebra gauge theory on a vertex neighbourhood $\Gamma_v$---which arises from a simple twist deformation of the Hopf algebra $K\exoo{v}$.

\section{Graph transformations and topological invariance}
\label{sec:graphops}

In this section, we prove that the operations on ciliated ribbon graphs from Section \ref{subsec:graphs} give rise to algebra and module homomorphisms between  the associated Hopf algebra gauge theories. We also show how  these algebra and module homomorphisms   can be described in terms of maps between the Hopf algebra gauge theories on the vertex neighbourhoods. 

In the following  let $K$ be a finite-dimensional  ribbon algebra.
Consider  ciliated  ribbon graphs $\Gamma$, $\Gamma'$  such that  $\Gamma'$ is obtained from $\Gamma$ by one of the graph transformations in Definition \ref{def:graphtrafos} and denote by $E,F,V$ and $E',F',V'$, respectively, their sets of edges, faces and vertices. Recall  from Definition \ref{def:graph_functor} and Proposition \ref{lem:graphtrafo_vertnb} that each of the graph transformations in Definition \ref{def:graphtrafos}
is associated with a  map $g_V: V'\to V$, which are
are  inclusion maps or identity maps. We use the notation $v=g_V(v')$ for  $v'\in V'$.
Similarly,  $\mathcal A^*_\Gamma$ and $\mathcal A^*_{\Gamma'}$ denote the  the  $K\exoo{V}$- and $K\exoo{V'}$-module algebra
structures on ${K^*}\exoo{E}$ and ${K^*}\exoo{E'}$  from Theorem \ref{def:algstruc}, $\mathcal A^*_v$ and $\mathcal A'^*_{v'}$ the algebras for the vertex neighbourhoods $\Gamma_v$  and $\Gamma_{v'}$  from Corollary \ref{rem:flipalg} 
and  $\cdot$, $\cdot'$ their multiplication maps. 
We suppose that  all vertices $v\in V$  and $v'\in V'$  are assigned the same $R$-matrices and the maps $\sigma_v$   from Corollary \ref{rem:flipalg} coincide for all vertices of  $\Gamma$ and $\Gamma'$   that are unaffected by the graph transformation. 

We construct linear maps $F^*:\mathcal A^*_{\Gamma'}\to \mathcal A^*_\Gamma$ associated with each graph transformation in Definition \ref{def:graphtrafos} in three steps:
\begin{enumerate}
\item To each graph operation  in Definition \ref{def:graphtrafos} we 
 assign a family of linear maps  $f^*_{v'}:  \mathcal A'^*_{v'}\to\oo_{v\in V} \mathcal A_v^*$, indexed by the vertices $v'\in V'$.

\item We show that the linear maps $f^*_{v'}$ induce   linear maps $f^*: \bigotimes_{v'\in V'}\mathcal A'^*_{v'}\to\oo_{v\in V} \mathcal A_v^*$.

\item We show that the maps $f^*$ induce unique  linear maps $F^*:\mathcal A^*_{\Gamma'}\to\mathcal A^*_\Gamma$ with $f^*\circ G^*_{\Gamma'}=G^*_\Gamma\circ F^*$
 for the maps $G^*_{\Gamma'}: \mathcal A^*_{\Gamma'}\to\bigotimes_{v'\in V'}\mathcal A'^*_{v'}$ and  $G^*_{\Gamma}: \mathcal A^*_{\Gamma}\to\bigotimes_{v\in V}\mathcal A^*_{v}$ from  \eqref{eq:dualemb} and that the maps $F^*$ are module and algebra homomorphisms. 
 \end{enumerate}

For  step 1, recall that the algebra $\mathcal A^*_v$ from Corollary \ref{rem:flipalg} for a vertex with $n$ incident edge ends and the ordering $1<2<...<n$ is the vector space $K^{*\oo n}$ with the multiplicative relation
 $$
\alpha^1\oo...\oo\alpha^n=(\alpha^1)_1\cdot (\alpha^2)_2\cdots (\alpha^n)_n
$$
 for  $\alpha^1,...,\alpha^n\in K^*$, and with further  multiplicative relations from  Corollary \ref{rem:flipalg}. It follows that any algebra homomorphism 
 $f_v: \mathcal A^*_v\to B$ into an associative unital algebra $B$  satisfies
\begin{align}\label{eq:gdef}
f_v: \mathcal A^*_v\to B, \quad \alpha^1\oo...\oo \alpha^n\mapsto g((\alpha^1)_1)\cdot g((\alpha^2)_2)\cdots g((\alpha^n)_n).
\end{align}
and hence is determined uniquely by its values on the elements $(\alpha)_e$ for edge ends $e$ incident at $v$. 
Conversely, the  linear map $f_v:\mathcal A^*_v\to B$ defined by \eqref{eq:gdef} is an algebra morphism if and only if  
 the remaining relations in equation
 Corollary  \ref{rem:flipalg} are satisfied.

\bigskip
\begin{definition} \label{def:edgeopmaps} Let $\Gamma'$ be obtained from $\Gamma$ by one of the graph operations from Definition \ref{def:graphtrafos}. 
Denote for each edge end $f\in E(\Gamma_\circ)$ that is not affected by this operation by $f'\in E(\Gamma'_\circ)$ the associated edge  end of $\Gamma'$ and label the remaining edges  as in Figure \ref{fig:graph_ops}. 
 Then the linear maps  $f^*_{v'}:\mathcal A'^*_{v'}\to \oo_{v\in V} \mathcal A^*_v$ induced by these graph transformations are defined by 
\eqref{eq:gdef} and their value on the variables $(\alpha)_{h'}$, $h'\in E(\Gamma'_\circ)$, as follows
\begin{compactenum}[(a)] 
\item {\bf Deleting  an edge $e$:} 
\begin{align}\label{eq:deleting}d^*_e: (\alpha)_{h'}\mapsto (\alpha)_h\quad\forall h'\in E(\Gamma'_\circ).\qquad\qquad\qquad\qquad\qquad\qquad\qquad\qquad\qquad\qquad\quad\;
\end{align}

\item {\bf Contracting an edge $e$ towards $\st(e)$:}  
\begin{align}\label{eq:contractstarting}
c^*_{\st(e)}: (\alpha)_{h'}\mapsto \begin{cases}  (\alpha)_h &  \ta(e)\notin\{\st(h), \ta(h)\}\\
 (\low\alpha 3\oo\low\alpha 2\oo\low\alpha 1)_{s(e)t(e)h} & 
\st(h)=\ta(e), h<t(e)\\
\langle g, \low\alpha 2\rangle (\low\alpha 4\oo\low\alpha 3\oo\low\alpha 1)_{s(e) t(e)h} & 
\st(h)=\ta(e), h>t(e)\\
(S^\inv(\low\alpha 1)\oo S^\inv(\low\alpha 2)\oo\low\alpha 3)_{s(e)t(e)h} & 
\ta(h)=\ta(e), h<t(e)\\
\langle g,\low\alpha 3\rangle  (S^\inv(\low\alpha 1)\oo S^\inv(\low\alpha 2)\oo \low\alpha 4)_{s(e)t(e)h} & 
\ta(h)=\ta(e), h>t(e).
\end{cases}
\end{align}
\item {\bf Contracting an edge $e$ towards $\ta(e)$:}  
\begin{align}\label{eq:contracttarget}
c^*_{\ta(e)}: (\alpha)_{h'}\mapsto \begin{cases}  (\alpha)_h & 
\st(e)\notin\{\st(h), \ta(h)\}\\
\langle \low\alpha 2, g^\inv\rangle\, (S(\low\alpha 4)\oo S(\low\alpha 3)\oo\low\alpha 1)_{s(e)t(e)h} & 
\st(h)=\st(e), h<s(e)\\
(S(\low\alpha 2)\oo S(\low\alpha 3)\oo\low\alpha 1)_{s(e)t(e)h} & 
\st(h)=\st(e), h>s(e)\\
\langle \low\alpha 3, g^\inv\rangle\, (\low\alpha 2\oo\low\alpha 1\oo\low\alpha 4)_{s(e)t(e)h} & 
\ta(h)=\st(e), h<s(e)\\
(\low\alpha 2\oo\low\alpha 1\oo\low\alpha 3)_{s(e)t(e)h} & 
\ta(h)=\st(e), h>s(e).
\end{cases}
\end{align}
\item {\bf Adding a loop $e''$ at $v$:} 
\begin{align}\label{eq:adding}
a^*_v: (\alpha)_{h'}\mapsto \begin{cases}
(\alpha)_h &h'\notin\{s(e''),t(e'')\}\\ 
\epsilon(\alpha) \, 1\exoo{2E} & h'\in\{s(e''),t(e'')\}
\end{cases}\qquad\qquad\qquad\qquad\qquad\qquad\qquad\qquad\;
\end{align}

\item {\bf Detaching adjacent edge ends  $e_1, e_2$ from $v$:}  
\begin{align}\label{eq:removing}
&w^*_{e_1e_2}: (\alpha)_{h'}\mapsto \begin{cases} (\alpha)_h & h'\neq s(e')\\
(\low\alpha 3\oo\low\alpha 2\oo\low\alpha 1)_{s(e_1)t(e_1)s(e_2)} & h'=s(e'), s(e_2)<t(e_1)\\
\langle \low\alpha 2, g^\inv\rangle\; (\low\alpha 4\oo\low\alpha 3\oo\low\alpha 1)_{s(e_1)t(e_1)s(e_2)} & h'=s(e'), s(e_2)>t(e_1)
\end{cases}
\end{align}

\item {\bf Doubling the edge $e$:}  
\begin{align}\label{eq:doubling}
do^*_{e}: (\alpha)_{h'}\mapsto\begin{cases} 
(\alpha)_h & h'
\notin \{s(e'),t(e'), s(e''), t(e'')\}\\
(\alpha)_{t(e)} & h'\in\{t(e'), t(e'')\}\\
(\alpha)_{s(e)} & h'\in\{s(e'), s(e'')\}.
\end{cases}\qquad\qquad\qquad\qquad\qquad\qquad
\end{align}
\end{compactenum}
\end{definition}

\bigskip
\begin{lemma} \label{lem:alghomord}$\quad$
\begin{compactenum}
\item The  linear maps $f^*_{v'}:\mathcal A'^*_{v'}\to \bigotimes_{v\in V}\mathcal A^*_v$ from Definition \ref{def:edgeopmaps} (a), (b), (c), (e), (f) are  morphisms of $K$-modules with respect to the $K$-module structures at $v'$ and $v$ and satisfy
$$
f^*_{v'}(\alpha)\lhd(k)_{g_V(w')}=\epsilon(k)\, \alpha\qquad\forall \alpha\in\mathcal A'^*_{v'} , \quad w'\in V'\setminus\{v'\}.
$$
\item The  linear maps $f^*_{v'}:\mathcal A'^*_{v'}\to \bigotimes_{v\in V}\mathcal A^*_v$ in cases (a), (b), (c), (e) in Definition \ref{def:edgeopmaps}  are algebra homomorphisms.
\end{compactenum}
\end{lemma}

\begin{proof}  1.~By definition of the linear maps $f^*_{v'}$ in \eqref{eq:gdef} and because $\mathcal A'^*_{v'}$ and $\mathcal A^*_v$ are $K$-module algebras, it is sufficient to show that 
 for all  distinct $v',w'\in V'$, edge ends $e'$ incident at $v'$, $\alpha\in K^*$ and $k\in K$
\begin{align}\label{eq:modulehelppp}
f^*_{v'}((\alpha)_{e'}\lhd (k)_{v'})=f^*_{v'}((\alpha)_{e'})\lhd (k)_{g_V(v')}
\qquad f^*_{v'}((\alpha)_{e'})\lhd (k)_{g_V(w')}=\epsilon(k)\, f^*_{v'}((\alpha)_{e'}).
\end{align}
These identities are obvious in case (a).
As the map in case (b) in \eqref{eq:contractstarting} is obtained from the one for case (c) in \eqref{eq:contracttarget} by applying the involution from \eqref{eq:invol}, case (b) follows from case (c). In case (c), the identities \eqref{eq:modulehelppp} are obvious from  \eqref{eq:contracttarget} for all vertices $v', w'\in V'$ except  for  $v'=\st(e)'$.
 If  $g\in E(\Gamma_\circ)$ is incoming at the vertex $\st(e)$, we obtain with \eqref{eq:contracttarget}
\begin{align*}
&f^*_{\st(e)'}((\alpha)_{g'})\lhd^* (h)_{\st(e)}\\
&=\begin{cases} 
\langle \low \alpha 3, g^\inv\rangle\, (\low\alpha 4\oo \low\alpha 2\oo\low\alpha 1)_{gs(e)t(e)}\lhd^*(h)_{s(e)} & g<s(e)\\
(\low\alpha 3\oo \low\alpha 2\oo\low\alpha 1)_{g s(e)t(e)}\lhd^*(h)_{s(e)} & g>s(e)
\end{cases}\\
&=\begin{cases} 
\langle \low \alpha 3, g^\inv\rangle \langle h, \alpha_{(4)(1)} S(\alpha_{(2)(2)})\rangle\; (\alpha_{(4)(2)}\oo \alpha_{(2)(1)}\oo\low\alpha 1)_{gs(e)t(e)} & g<s(e)\\
 \langle h, S(\alpha_{(2)(2)})\alpha_{(3)(1)}\rangle\; (\alpha_{(3)(1)}\oo \alpha_{(2)(1)}\oo\low\alpha 1)_{gs(e)t(e)} & g>s(e)\\
\end{cases}\\
&=\begin{cases} 
\langle \low\alpha 3, S(\low h 2)g^\inv \low h 1\rangle\; (\alpha_{(4)}\oo \alpha_{(2)}\oo\low\alpha 1)_{gs(e)t(e)} & g<s(e)\\
\epsilon( h)\; (\alpha_{(3)}\oo \alpha_{(2)}\oo\low\alpha 1)_{gs(e)t(e)} & g>s(e)\\
\end{cases}\\
&=\begin{cases} 
\langle \low\alpha 3, g^\inv S^\inv(\low h 2) \low h 1\rangle\; (\alpha_{(4)}\oo \alpha_{(2)}\oo\low\alpha 1)_{gs(e)t(e)} & g<s(e)\\
\epsilon( h)\; (\alpha_{(3)}\oo \alpha_{(2)}\oo\low\alpha 1)_{gs(e)t(e)} & g>s(e)\\
\end{cases}\\
&=\epsilon(h)\, f^*_{\st(e)'}((\alpha)_{f'}),
\end{align*}
where we used the identity $g^\inv S(h)=S^\inv(h) g^\inv$ for all $h\in K$.  An analogous computation shows that this identity also holds for edge ends $g'\in E(\Gamma'_\circ)$ for which the associated edge end $g\in E(\Gamma)$ is outgoing at $\st(e)$. 
 In case (e) the identities \eqref{eq:modulehelppp} are obvious from \eqref{eq:removing}  for all vertices  and edge ends of $\Gamma'_\circ$ except for 
$v'=\st(e')$, $w'=g_V^\inv(\st(e_2))=g_V^\inv(\ta(e_1))$ and  $e'$. In this case, one has 
 \begin{align*}
&f^*_{\st(e')}((\alpha)_{s(e')})\lhd^*(h)_{\st(e_2)}   \\
&=\begin{cases} (\low\alpha 3\oo\low\alpha 2\oo\low\alpha 1)_{s(e_1)t(e_1)s(e_2)}\lhd^*(h)_{v} & s(e_2)<t(e_1)\\
\langle \low\alpha 2, g^\inv\rangle\, (\low\alpha 4\oo \low\alpha 3\oo\low\alpha 1)_{s(e_1)t(e_1)s(e_2)}\lhd^*(h)_{v} & s(e_2)>t(e_1)
\end {cases}\\
&=\begin{cases} \langle S(\low\alpha 2) \low\alpha 3,h\rangle\, (\low\alpha 5\oo\low\alpha 4\oo\low\alpha 1)_{s(e_1)t(e_1)s(e_2)} & s(e_2)<t(e_1)\\
\langle \low\alpha 3, g^\inv\rangle \langle  \low\alpha 4 S(\low\alpha 2),h\rangle\, (\low\alpha 6\oo \low\alpha 5\oo\low\alpha 1)_{s(e_1)t(e_1)s(e_2)} & s(e_2)>t(e_1)
\end {cases}\\
&=\begin{cases} \epsilon(h) (\low\alpha 4\oo\low\alpha 3\oo\low\alpha 2\oo\low\alpha 1)_{s(e_1)t(e_1)s(e_2)t(e_2)} & s(e_2)<t(e_1)\\
\langle \low\alpha 3, S(\low h 2)g^\inv \low h 1\rangle \, (\low\alpha 5\oo\low\alpha 4\oo \low\alpha 2\oo\low\alpha 1)_{s(e_1)t(e_1)s(e_2)t(e_2)} & s(e_2)>t(e_1)
\end {cases}\\
&=\begin{cases} \epsilon(h) (\low\alpha 3\oo\low\alpha 2\oo\low\alpha 1)_{s(e_1)t(e_1)s(e_2)} & s(e_2)<t(e_1)\\
\langle \low\alpha 2, g^\inv S^\inv(\low h 2)\low h 1\rangle \, (\low\alpha 4\oo \low\alpha 3\oo\low\alpha 1)_{s(e_1)t(e_1)s(e_2)} & s(e_2)>t(e_1)
\end {cases}\\
&=\epsilon(h)\, f^*_{\st(e')}((\alpha)_{s(e')}),
 \end{align*}
where we used again the identity $g^\inv S(h)=S^\inv(h) g^\inv$ for all $h\in K$. Finally, in case (f) the identities \eqref{eq:modulehelppp} are obvious from \eqref{eq:doubling} for all $v',w'\in V'$ and edge ends $e'\in E(\Gamma'_\circ)$.

2.~To check that the map $f^*_{v'}:\mathcal A'^*_\Gamma\to \bigotimes_{v\in V} \mathcal A^*_v$ is an algebra morphism, it is sufficient to show that it satisfies the relations in 
Corollary  \ref{rem:flipalg}, which each involve only two generators $(\alpha)_{e'}$, $(\alpha)_{f'}$ for edge ends $e',f'$ at $v'$. 

(a) In case (a), this follows directly from \eqref{eq:deleting} and the identities $(\epsilon\oo \id)(R)=(\id\oo \epsilon)(R)=1=S(1)$.

(b) As the map in case (b) in \eqref{eq:contractstarting} is obtained from the one for case (c) in \eqref{eq:contracttarget} by applying the involution from \eqref{eq:invol}, which is an algebra homomorphism, case (b) follows from case (c).

(c) In case (c), it is obvious that the relations are satisfied except for the case $f',g'\in E(\Gamma'_\circ)$ that are both incident at the vertex $v'=\st(e)'=g_V^\inv(\st(e))=g_V^\inv(\ta(e))$. In this case, we have   the  cases
 \begin{compactenum}[ (i)]
\item $f'=g'\in E(\Gamma'_\circ)$ with  $f>s(e)$ 
\item $f'=g'\in E(\Gamma'_\circ)$ with  $f<s(e)$ 
\item $f',g'\in E(\Gamma'_\circ)$ with $g<t(e)$ and $f>s(e)$
\item $f',g'\in E(\Gamma'_\circ)$ with $g>t(e)$ and $f>s(e)$
\item $f',g'\in E(\Gamma'_\circ)$ with $g<t(e)$ and $f<s(e)$
\item $f',g'\in E(\Gamma'_\circ)$ with $g>t(e)$ and $f<s(e)$
\item $f',g'\in E(\Gamma'_\circ)$ with $s(e)<g<f$
\item $f',g'\in E(\Gamma'_\circ)$ with $g<f<s(e)$
\item $f',g'\in E(\Gamma'_\circ)$ with $g<s(e)<f$
\end{compactenum}
For this, we assume without loss of generality that $\sigma_v(f)=0$ if $f\in E(\Gamma_\circ)\cup E(\Gamma'_\circ)$ is incoming at  $v$ and that $\sigma_v(f)=1$ if $f\in E(\Gamma_\circ)\cup E(\Gamma'_\circ)$ is outgoing at $v$.   
Moreover, we can suppose without loss of generality  that all edge ends in $\Gamma$ that are
 incident at  $\st(e)$ or $\ta(e)$ except the edge end $s(e)$ are incoming since the corresponding expressions for outgoing edge ends are obtained by reversing the orientation  with the involution from \eqref{eq:invol}. 

(i) If $f\in E(\Gamma_\circ)$ is incoming at $\st(e)$ with $f>s(e)$, then  $(\alpha)_{t(e)}$ commutes with $(\beta)_{s(e)}$ and $(\gamma)_{f}$  for all $\alpha,\beta,\gamma\in K^*$. Using the fact that $t(e)$,  and $f$ are incoming while $s(e)$ is outgoing
  one obtains  from \eqref{eq:flipalg} and   \eqref{eq:contracttarget}
 \begin{align*}
 &f^*_{\st(e)'}((\beta)_{f'})\cdot f^*_{\st(e)'}((\alpha)_{f'})=(\low\beta 1\oo\low\beta 2\oo\low\beta 3)_{t(e)s(e)f}\cdot  (\low\alpha 1\oo\low\alpha 2\oo\low\alpha 3)_{t(e)s(e)f}\\
 &=(\low\beta 1)_{t(e)}\cdot(\low\alpha 1)_{t(e)}\cdot (\low\beta 2)_{s(e)}\cdot(\low\beta 3)_{f}\cdot (\low\alpha 2)_{s(e)}\cdot(\low\alpha 3)_{f}\nonumber\\
 &=\langle S(\low \alpha 3)\oo\low\beta 3, R\rangle\;  (\low\beta 1\low\alpha 1)_{t(e)}\cdot (\low\beta 2)_{s(e)}\cdot  (\low\alpha 2)_{s(e)}\cdot (\low\beta 4)_{f}\cdot (\low\alpha 4)_{f}\nonumber\\
&=\langle S(\low \alpha 4)\oo\low\beta 4, R\rangle\langle \low\alpha 3\oo\low\beta 3, R\rangle\; (\low\beta 1\low\alpha 1)_{t(e)}\cdot (\low\beta 2\low\alpha 2)_{s(e)}\cdot (\low\beta 5)_{f}\cdot (\low\alpha 5)_{f}\\
&=(\low\beta 1\low\alpha 1)_{t(e)}\cdot (\low\beta 2\low\alpha 2)_{s(e)}\cdot (\low\beta 3)_{f}\cdot (\low\alpha 3)_{f}= f^*_{\st(e)'}((\beta)_{f'}\cdot'(\alpha)_{f'})
 \end{align*}
(ii) If $f\in E(\Gamma_\circ)$ is incoming at $\st(e)$ with $f<s(e)$, then by a similar computation one obtains  from \eqref{eq:flipalg} and   \eqref{eq:contracttarget}
 \begin{align*}
 &f^*_{\st(e)'}((\beta)_{f'})\cdot f^*_{\st(e)'}((\alpha)_{f'})\\
 &=\langle  \low\beta 3, g^\inv\rangle\langle \low\alpha 3, g^\inv\rangle\;
 (\low\beta 1\oo\low\beta 2\oo\low\beta 4)_{t(e)s(e)f}\cdot  (\low\alpha 1\oo\low\alpha 2\oo\low\alpha 4)_{t(e)s(e)f}\\
 &=\langle  \low\beta 3, g^\inv\rangle\langle \low\alpha 3, g^\inv\rangle\; 
  (\low\beta 1)_{t(e)}\cdot(\low\alpha 1)_{t(e)}\cdot (\low\beta 4)_{f}\cdot(\low\beta 2)_{s(e)}\cdot (\low\alpha 4)_{f}\cdot(\low\alpha 2)_{s(e)}\\
   &=\langle  \low\beta 4, g^\inv\rangle\langle \low\alpha 3, g^\inv\rangle\langle \low\alpha 4\oo S(\low\beta 3), R\rangle\; 
  (\low\beta 1\low\alpha 1)_{t(e)}\cdot (\low\beta 5)_{f}\cdot (\low\alpha 5)_{f}\cdot(\low\beta 2)_{s(e)}\cdot(\low\alpha 2)_{s(e)}\\
     &=\langle  \low\beta 4, g^\inv\rangle\langle \low\alpha 4, g^\inv\rangle\langle S(\low\alpha 3)\oo \low\beta 3, R\rangle\; 
  (\low\beta 1\low\alpha 1)_{t(e)}\cdot (\low\beta 5)_{f}\cdot (\low\alpha 5)_{f}\cdot(\low\beta 2)_{s(e)}\cdot(\low\alpha 2)_{s(e)}\\
       &=\langle  \low\beta 5 \low\alpha 5, g^\inv\rangle\langle S(\low\alpha 4)\oo \low\beta 4, R\rangle\langle \low\alpha 3\oo\low\beta 3, R\rangle\;
  (\low\beta 1\low\alpha 1)_{t(e)}\cdot (\low\beta 6)_{f}\cdot (\low\alpha 6)_{f}\cdot(\low\beta 2\low\alpha 2)_{s(e)}\\
         &=\langle  \low\beta 3 \low\alpha 3, g^\inv\rangle\;
  (\low\beta 1\low\alpha 1)_{t(e)}\cdot (\low\beta 4)_{f}\cdot (\low\alpha 4)_{f}\cdot(\low\beta 2\low\alpha 2)_{s(e)}\\
  &=f^*_{\st(e)'}((\beta)_{f'}\cdot'(\alpha)_{f'}),
 \end{align*} 
 where we used the identities $R(g^\inv\oo g^\inv)=(g^\inv\oo g^\inv)R$, $R\cdot\Delta(k)=\Delta^{op}(k)\cdot R$ and $g \cdot S(k)g^\inv=S^\inv(k)\cdot g$ for all $k\in K$.

\medskip
(iii) If $f,g\in E(\Gamma_\circ)$ satisfy $g<t(e)$ and $f>s(e)$, then $g'<f'$ in $\Gamma'_\circ$, and  $(\alpha)_{g}$  and $(\beta)_{t(e)}$ commute  with $(\gamma)_{s(e)}$ and $(\delta)_{f}$ for all $\alpha,\beta,\gamma,\delta\in K^*$. 
By definition of $f^*_{\st(e)'}$, we then have
\begin{align*}
 f^*_{\st(e)'}((\alpha\oo\beta)_{g'f'})=f^*_{\st(e)'}((\alpha)_{g'}\cdot'(\beta)_{f'})=f^*_{\st(e)'}((\alpha)_{g'}) \cdot f^*_{\st(e)'}((\beta)_{f'})=(\alpha)_{g}\cdot (\low\beta 1)_{t(e)} \cdot (\low\beta 2\oo\low\beta 3)_{s(e)f}.
\end{align*}
As $s(e)$ is outgoing and  $g$, $f$ and $t(e)$ are incoming,  we obtain for the opposite product
 \begin{align*}
 & f^*_{\st(e)'}((\beta)_{f'})\cdot f^*_{\st(e)'}((\alpha)_{g'})=(\low\beta 1\oo \low\beta 2\oo\low\beta 3)_{t(e)s(e)f}\cdot (\alpha)_{g}\\
 &=(\low\beta 1)_{t(e)}\cdot (\alpha)_{g} \cdot (\low\beta 2\oo\low\beta 3)_{s(e)f}
 =\langle \low\alpha 1\oo \low\beta 1, R\rangle\, (\low\alpha 1)_g\cdot (\low\beta 2)_{t(e)}\cdot (\low\beta 3\oo\low\beta 4)_{s(e)f}\\
  &=\langle \low\alpha 1\oo \low\beta 1, R\rangle\, (\low\alpha 2\oo \low\beta 2\oo\low\beta 3\oo\low\beta 4)_{gt(e)s(e)f}=\langle\low\alpha 1\oo\low\beta 1, R\rangle\; f^*_{\st(e)'}((\low\alpha 2\oo\low\beta 2)_{g'f'})\\
  &= f^*_{\st(e)'}((\beta)_{f'}\cdot'(\alpha)_{g'}).
\end{align*}
(iv)-(vi): The proofs for the cases (iv)-(vi) are analogous to the one for  (iii).

\medskip
 (vii)  For  $f,g\in E(\Gamma_\circ)$ with $s(e)<g<f$
one has $g'<f'$ in $\Gamma'_\circ$,
and $(\alpha)_{t(e)}$ commutes with $(\beta)_{g}$ and $(\gamma)_{f}$ for all $\alpha,\beta,\gamma\in K^*$. 
By definition of $f^*_{\st(e)'}$ one obtains
\begin{align*}
 & f^*_{\st(e)'}((\alpha\oo\beta)_{g'f'})=f^*_{\st(e)'}((\alpha)_{g'}\cdot'(\beta)_{f'})=f^*_{\st(e)'}((\alpha)_{g'})\cdot   f^*_{\st(e)'}((\beta)_{f'})\\
 &= (\low\alpha 1)_{t(e)}\cdot (\low\beta 1)_{t(e)}\cdot (\low\alpha 2)_{s(e)}\cdot (\low\alpha 3)_{g}\cdot (\low\beta 2)_{s(e)}\cdot (\low\beta 3)_{f}\\
&=(\low\alpha 1\low\beta 1)_{t(e)}  \cdot (\low\alpha 2\low\beta 2)_{s(e)}\cdot (\low\alpha 3)_{g}\cdot (\low\beta 3)_{f}.\qquad\qquad\qquad\qquad\qquad\qquad\qquad\qquad\qquad\qquad\qquad\qquad
\end{align*}
As $s(e)$ is outgoing and  $g$, $f$ and $t(e)$ are incoming, we obtain for the opposite product
\begin{align*}
 & f^*_{\st(e)'}((\beta)_{f'})\cdot f^*_{\st(e)'}((\alpha)_{g'})=(\low\beta 1\oo \low\beta 2\oo\low\beta 3)_{t(e)s(e)f}\cdot (\low\alpha 1\oo \low\alpha 2\oo\low\alpha 3)_{t(e)s(e)g}\\
 &=(\low\beta 1)_{t(e)}\cdot (\low\alpha 1)_{t(e)}\cdot (\low\beta 2)_{s(e)}\cdot (\low\beta 3)_{f}\cdot (\low\alpha 2)_{s(e)}\cdot (\low\alpha 3)_{g}\\
 &=\langle S(\low\alpha 3)\oo\low\beta 3, R\rangle\; (\low\beta 1\low\alpha 1)_{t(e)}  \cdot (\low\beta 2)_{s(e)} \cdot (\low\alpha 2)_{s(e)}\cdot (\low\beta 4)_{f}\cdot (\low\alpha 4)_{g}\\
  &=\langle S(\low\alpha 4)\oo\low\beta 4, R\rangle\langle \low\alpha 3\oo\low\beta 3,R\rangle\; (\low\beta 1\low\alpha 1)_{t(e)}  \cdot (\low\beta 2\low\alpha 2)_{s(e)}\cdot (\low\beta 5)_{f}\cdot (\low\alpha 5)_{g}\\
    &=(\low\beta 1\low\alpha 1)_{t(e)}  \cdot (\low\beta 2\low\alpha 2)_{s(e)}\cdot (\low\beta 3)_{f}\cdot (\low\alpha 3)_{g}\\
    &=\langle \low\alpha 3\oo\low\beta 3, R\rangle\; (\low\beta 1\low\alpha 1)_{t(e)}  \cdot (\low\beta 2\low\alpha 2)_{s(e)}\cdot (\low\alpha 4)_{g}\cdot (\low\beta 4)_{f}\\
        &=\langle \low\alpha 1\oo\low\beta 1, R\rangle\; (\low\alpha 2\low\beta 2)_{t(e)}  \cdot (\low\alpha 3\low\beta 3)_{s(e)}\cdot (\low\alpha 4)_{g}\cdot (\low\beta 4)_{f}\\
        &=\langle \low\alpha 1\oo\low\beta 1, R\rangle\; f^*_{\st(e)'}((\low\alpha 2\oo\low\beta 2)_{g'f'})=f^*_{\st(e)'}((\beta)_{f'}\cdot'(\alpha)_{g'}).\qquad\qquad\qquad\qquad\qquad\qquad\qquad\qquad
\end{align*}

\medskip
(viii) For  $f,g\in E(\Gamma_\circ)$ with $g<f<s(e)$
one has again $g'<f'$ in $\Gamma'_\circ$,
and $(\alpha)_{t(e)}$ commutes with $(\beta)_{g}$, $(\gamma)_{f}$ for all $\alpha,\beta,\gamma\in K^*$. This implies by definition of $f^*_{\st(e)'}$
\begin{align*}
 &f^*_{\st(e)'}((\alpha\oo\beta)_{g'f'})=f^*_{\st(e)'}((\alpha)_{g'}\cdot'(\beta)_{f'})=f^*_{\st(e)'}((\alpha)_{g'})\cdot  f^*_{\st(e)'}((\beta)_{f'})\\
 &=\langle \low\beta 3\low\alpha 3, g^\inv\rangle  (\low\beta 1\low\alpha 1)_{t(e)} \cdot(\low\alpha 4)_{g} \cdot(\low\beta 4)_{f}  \cdot (\low\beta 2\low\alpha 2)_{s(e)}.\qquad\qquad\qquad\qquad\qquad\qquad\qquad\qquad\qquad\qquad
\end{align*}
As $s(e)$ is outgoing, while  $g$, $f$ and $t(e)$ are incoming, the opposite product satisfies
\begin{align*}
 & f^*_{\st(e)'}((\beta)_{f'})\cdot f^*_{\st(e)'}((\alpha)_{g'})\\
 &=\langle \low\beta3,g^\inv\rangle\langle \low\alpha 3, g^\inv\rangle(\low\beta 1)_{t(e)}\cdot (\low\alpha 1)_{t(e)}\cdot  (\low\beta 4)_{f}\cdot (\low\beta 2)_{s(e)}\cdot (\low\alpha 4)_{g}\cdot  (\low\alpha 2)_{s(e)}\\
 &=\langle \low\beta3\low\alpha 4, g^\inv\rangle\langle \low\alpha 3\oo S(\low\beta 4), R\rangle\; (\low\beta 1\low\alpha 1)_{t(e)}  \cdot (\low\beta 5)_{f} \cdot (\low\alpha 5)_{g}\cdot (\low\beta 2)_{s(e)}\cdot (\low\alpha 2)_{s(e)}\\
 &=\langle \low\beta4\low\alpha 4, g^\inv\rangle\langle S(\low\alpha 3)\oo \low\beta 3, R\rangle\; (\low\beta 1\low\alpha 1)_{t(e)}  \cdot (\low\beta 5)_{f} \cdot (\low\alpha 5)_{g}\cdot (\low\beta 2)_{s(e)}\cdot (\low\alpha 2)_{s(e)}\\
  &=\langle \low\beta 5\low\alpha 5, g^\inv\rangle\langle S(\low\alpha 4)\oo \low\beta 4, R\rangle\langle \low\alpha 3\oo \low\beta 3, R\rangle \; (\low\beta 1\low\alpha 1)_{t(e)}  \cdot (\low\beta 6)_{f} \cdot (\low\alpha 6)_{g}\cdot (\low\beta 2\low\alpha 2)_{s(e)}\\
    &=\langle \low\beta 3\low\alpha 3, g^\inv\rangle (\low\beta 1\low\alpha 1)_{t(e)}  \cdot (\low\beta 4)_{f} \cdot (\low\alpha 4)_{g}\cdot (\low\beta 2\low\alpha 2)_{s(e)}\\
        &=\langle \low\beta 3\low\alpha 3, g^\inv\rangle\langle \low \alpha 4\oo\low\beta 4, R\rangle  (\low\beta 1\low\alpha 1)_{t(e)}  \cdot(\low\alpha 5)_{g} \cdot(\low\beta 5)_{f}  \cdot (\low\beta 2\low\alpha 2)_{s(e)}\\
                &=\langle \low\beta 4\low\alpha 4, g^\inv\rangle\langle \low \alpha 1\oo\low\beta 1, R\rangle  (\low\beta 2\low\alpha 2)_{t(e)}  \cdot(\low\alpha 5)_{g} \cdot(\low\beta 5)_{f}  \cdot (\low\beta 3\low\alpha 3)_{s(e)}\\
                &=\langle \low\alpha1\oo\low\beta 1, R\rangle\;  f^*_{\st(e)'}((\low\alpha 2\oo\low\beta 2)_{g'f'})=f^*_{\st(e)'}((\beta)_{f'}\cdot'(\alpha)_{g'}).\qquad\qquad\qquad\qquad\qquad\qquad\qquad\qquad\qquad\qquad\qquad\qquad\qquad\qquad
\end{align*}

(ix) For $g<s(e)<f$ one has $f'<g'$ in $\Gamma'_\circ$,
and $(\alpha)_{t(e)}$ commutes with $(\beta)_{g}$, $(\gamma)_{f}$ for all $\alpha,\beta,\gamma\in K^*$. This yields 
\begin{align*}
 &f^*_{\st(e)'}((\alpha\oo\beta)_{g'f'})=f^*_{\st(e)'}((\beta)_{f'}\cdot (\alpha)_{g'})= f^*_{\st(e)'}((\beta)_{f'})\cdot  f^*_{\st(e)'}((\alpha)_{g'})\\
 &=\langle \low\alpha 3, g^\inv\rangle(\low\beta 1)_{t(e)}\cdot (\low\alpha 1)_{t(e)}\cdot (\low\beta 2)_{s(e)}\cdot  (\low\beta 3)_{f} \cdot (\low\alpha 4)_{g}\cdot  (\low\alpha 2)_{s(e)}\\
  &=\langle \low\alpha 3, g^\inv\rangle \langle \low\alpha 1\oo\low\beta 1, R\rangle\;  (\low\alpha 2\low\beta 2)_{t(e)} \cdot (\low\alpha 4)_{g} \cdot (\low\beta 3)_{s(e)}\cdot (\low\beta 4)_{f}\cdot   (\low\alpha 3)_{s(e)}\\
  &=\langle \low\alpha 3, g^\inv\rangle \langle \low\alpha 1\oo\low\beta 1, R\rangle\langle S(\low\alpha 4)\oo \low\beta 4, R\rangle\;  (\low\alpha 2\low\beta 2)_{t(e)} \cdot (\low\alpha 5)_{g} \cdot (\low\alpha 3\low\beta 3)_{s(e)}\cdot (\low\beta 5)_{f}\\
    &=\langle \low\alpha 3, g^\inv\rangle \langle \low\alpha 1\oo\low\beta 1, R\rangle\langle S(\low\alpha 2)\oo \low\beta 2, R\rangle\;  (\low\beta 3\low\alpha 3)_{t(e)} \cdot (\low\alpha 5)_{g} \cdot (\low\beta 4\low\alpha 4)_{s(e)}\cdot (\low\beta 5)_{f}\\
        &=\langle \low\alpha 3, g^\inv\rangle  (\low\beta 1\low\alpha 1)_{t(e)} \cdot (\low\alpha 5)_{g} \cdot (\low\beta 2\low\alpha 2)_{s(e)}\cdot (\low\beta 5)_{f}.\qquad\qquad\qquad\qquad\qquad\qquad\qquad\qquad\qquad\qquad\qquad\qquad\qquad\qquad
\end{align*}
The opposite product is given by
\begin{align*}
 &f^*_{\st(e)}((\alpha)_{g'})\cdot  f^*_{\st(e)'}((\beta)_{f'})=\langle  \low\alpha 3, g^\inv\rangle (\low\alpha 1)_{t(e)}\cdot  (\low\beta 1)_{t(e)} \cdot (\low\alpha 4)_{g}\cdot  (\low\alpha 2)_{s(e)}\cdot (\low\beta 2)_{s(e)}\cdot  (\low\beta 3)_{f}\\
&= \langle  \low\alpha 3, g^\inv\rangle \langle \low\beta 1\oo\low\alpha 1, R\rangle\;(\low\beta 2\low\alpha 2)_{t(e)} \cdot (\low\alpha 4)_{g}\cdot  (\low\beta 3\low\alpha 3)_{s(e)}\cdot  (\low\beta 4)_{f}\\
  &= \langle \low\beta 1\oo\low\alpha 1, R\rangle\; f^*_{\st(e)'}((\low \alpha 2\oo\low \beta 2)_{g'f'})=f^*_{\st(e)'}((\alpha)_{g'}\cdot'(\beta)_{f'}).
\end{align*}
This proves the claim for case (v). By combining   (i)-(ix) we obtain 
 that $f^*_{\st(e)'}$  is an algebra homomorphism.

(e) In case (e)  it is clear from \eqref{eq:removing},  that $f^*_{v'}:\mathcal A^*_{v'}\to \bigotimes_{v\in V} \mathcal A^*_v$ is an algebra homomorphism if $v'\neq \st(e')$.
For $v'=\st(e')$ it is also clear that $f^*_{v'}((\alpha)_{f'}\cdot' (\beta)_{g'})=f^*_{v'}((\alpha)_{f'})\cdot f^*_{v'}((\beta)_{g'})$ if $f'\neq s(e')$ or $g'\neq s(e')$.  
It remains to consider the case $f'=g'=s(e')$. As $(\alpha)_{s(e_1)}$ comutes with $(\beta)_{t(e_1)}$ and $(\gamma)_{s(e_2)}$ in $\bigotimes_{v\in V}\mathcal A^*_v$, we obtain
if   $s(e_2)<t(e_1)$ 
\begin{align*}
&f^*_{\st(e')}((\alpha)_{\st(e')})\cdot f^*_{\st(e')}((\beta)_{s(e')})
=(\low\alpha 3\oo\low\alpha 2\oo\low\alpha 1)_{s(e_1)t(e_1)s(e_2)}\cdot (\low\beta 3\oo\low\beta 2\oo\low\alpha 1)_{s(e_1)t(e_1)s(e_2)}\\
&=(\low\alpha 3)_{s(e_1)}\cdot (\low\beta 3)_{s(e_1)}\cdot (\low\alpha 1)_{s(e_2)}\cdot (\low\alpha 2)_{t(e_1)}\cdot (\low\beta 1)_{s(e_2)}\cdot (\low\beta 2)_{t(e_1)}\\
&=\langle S(\low\beta 2)\oo\low\alpha 2, R\rangle\langle\low\beta 3\oo\low\alpha 3, R\rangle\; (\low\beta 3\low\alpha 3)_{s(e_1)}\cdot (\low\beta 1\low\alpha 1)_{s(e_2)}\cdot (\low\beta 4\low\alpha 4)_{t(e_1)}\\
&= (\low\beta 2\low\alpha 2)_{s(e_1)}\cdot (\low\beta 1\low\alpha 1)_{s(e_2)}\cdot (\low\beta 3\low\alpha 3)_{t(e_1)}
=f^*_{\st(e')}((\beta\alpha)_{s(e')})
\end{align*}
and for $s(e_2)>t(e_1)$
\begin{align*}
&f^*_{\st(e')}((\alpha)_{s(e')})\cdot f^*_{\st(e')}((\beta)_{s(e')})\\
&=\langle \low\alpha 2, g^\inv\rangle\langle\low\beta 2, g^\inv\rangle\; (\low\alpha 4\oo\low\alpha 3\oo\low\alpha 1)_{s(e_1)t(e_1)s(e_2)}\cdot (\low\beta 4\oo\low\beta 3\oo\low\alpha 1)_{s(e_1)t(e_1)s(e_2)}\\
&=\langle \low\alpha 2,g^\inv\rangle\langle\low\beta 2, g^\inv\rangle (\low\alpha 4)_{s(e_1)}\cdot (\low\beta 4)_{s(e_1)}\cdot (\low\alpha 3)_{t(e_1)}\cdot (\low\alpha 1)_{s(e_2)}\cdot (\low\beta 3)_{t(e_1)} \cdot (\low\beta 1)_{s(e_2)}\\
&=\langle \low\alpha 3, g^\inv\rangle\langle \low\beta 2, g^\inv\rangle\;   \langle \low\beta 3\oo S(\low\alpha 2), R\rangle\langle\low\beta 4\oo\low\alpha 4, R\rangle\; (\low\beta 6\low\alpha 6)_{s(e_1)}\cdot (\low\beta 5\low\alpha 5)_{t(e_1)}\cdot (\low\beta 1\low\alpha 1)_{s(e_2)}\\
&=\langle \low\alpha 2, g^\inv\rangle\langle \low\beta 2, g^\inv\rangle\;   \langle \low\beta 3\oo \low\alpha 3, R^\inv\rangle\langle\low\beta 4\oo\low\alpha 4, R\rangle\; (\low\beta 6\low\alpha 6)_{s(e_1)}\cdot (\low\beta 5\low\alpha 5)_{t(e_1)}\cdot (\low\beta 1\low\alpha 1)_{s(e_2)}\\
&=\langle \low\beta 2\low\alpha 2, g^\inv\rangle\;   (\low\beta 4\low\alpha 4)_{s(e_1)}\cdot (\low\beta 3\low\alpha 3)_{t(e_1)}\cdot (\low\beta 1\low\alpha 1)_{s(e_2)}=f^*_{\st(e')}((\beta\alpha)_{s(e')}).
\end{align*}
This proves that $f^*_{\st(e')}$ is an algebra homomorphism as well.
\end{proof}

We now proceed to step 2. 
To extend the linear maps $f^*_{v'}:  \mathcal A'^*_{v'}\to \bigotimes_{v\in V}\mathcal A^*_v$ from Definition \ref{def:edgeopmaps} to linear maps
$f^*: \bigotimes_{v'\in V'}\mathcal A'^*_{v'}\to \bigotimes_{v\in V}\mathcal A^*_v$, it is sufficient to show that elements in the images of the maps $f^*_{v'}$ and $f^*_{w'}$ for different vertices $v'\neq w'\in V'$ commute and to impose that $f^*$ agrees with $f^*_{v'}$ on the subalgebra $\mathcal A'^*_{v'}\subset \bigotimes_{v'\in V'}\mathcal A'^*_{v'}$.

\bigskip
\begin{lemma} \label{lem:commute}For distinct $v',w'\in V'$ the maps $f^*_{v'}:\mathcal A'^*_{v'}\to \bigotimes_{v\in V}\mathcal A^*_v$ from Definition \ref{def:edgeopmaps} satisfy 
\begin{align}
f^*_{v'}(\alpha)\cdot f^*_{w'}(\beta)=f^*_{w'}(\beta)\cdot f^*_{v'}(\alpha)\qquad\forall \alpha\in \mathcal A'^*_{v'},\beta\in \mathcal A'^*_{w'}.
\end{align}
\end{lemma}

\begin{proof}
By definition of the maps $f^*_{v'}$ in \eqref{eq:gdef}, it is sufficient to check that for all distinct $v',w'\in V'$, edge ends $f'$ at $v'$ and $g'$ at $w'$ and $\alpha,\beta\in K^*$
\begin{align}\label{eq:helpproddifv}
f^*_{v'}( (\alpha)_{f'})\cdot f^*_{w'}( (\beta)_{g'})=f^*_{w'}( (\beta)_{g'})\cdot f^*_{v'}( (\alpha){f'}).
\end{align}
For the graph transformations (a), (b), (c), (d) and (f) in Definition  \ref{def:edgeopmaps} this is obvious from equations \eqref{eq:deleting}, \eqref{eq:contractstarting}, \eqref{eq:contracttarget}, \eqref{eq:adding} and \eqref{eq:doubling}. It remains to check this in case (e), where it is obvious except for $v'=\st(e')$, $w'=g_V^\inv(\st(e_2))=g_V^\inv(\ta(e_1))$, $g'=s(e')$ and $f'$ incident at $g_V^\inv(\st(e_2))$.
Without loss of generality, we can suppose that $f'$ is incoming at $\st(e_2)'=g_V^\inv(\st(e_2))=g_V^\inv(\ta(e_1))$ and that $\sigma(s(e'))=\sigma(s(e_2))=1$, $\sigma(t(e_1))=0$.
As $t(e_1)$ and $s(e_2)$ are adjacent at $\st(e_2)$,  any  edge end  $f\in E(\Gamma_\circ)\setminus \{s(e_2), t(e_1)\}$ that is incident at $\st(e_2)$ satisfies either $f<s(e_2), t(e_1)$ or $f>s(e_2), t(e_1)$.
If $s(e_2)<t(e_1)$ one obtains for an incoming edge end $f\in E(\Gamma_\circ)$ at $\st(e_2)$ with
$f<s(e_2)<t(e_1)$
\begin{align*}
&
f^*_{\st(e')}((\alpha)_{s(e')})\cdot f^*_{\st(e_2)'}((\beta)_{f'})
=(\low\alpha 3\oo\low\alpha 2\oo\low\alpha 1)_{s(e_1)t(e_1)s(e_2)}\cdot (\beta)_f\\
&=(\low\alpha 3)_{s(e_1)}\cdot (\low\alpha 1)_{s(e_2)}\cdot (\low\alpha 2)_{t(e_1)}\cdot (\beta)_f\\
&=\langle \low\beta 1\oo \low\alpha 3, R\rangle\langle \low\beta 2\oo S(\low\alpha 2), R\rangle (\low\beta 3)_f\cdot (\low\alpha 5\oo\low\alpha 4\oo\low\alpha 1)_{s(e_1)t(e_1)s(e_2)}\\
&=(\beta)_f\cdot (\low\alpha 3\oo\low\alpha 2\oo\low\alpha 1)_{s(e_1)t(e_1)s(e_2)}=f^*_{\st(e_2)'}((\beta)_{f'})\cdot f^*_{\st(e')}((\alpha)_{s(e')})\qquad\qquad\qquad\qquad\qquad\qquad\qquad\qquad\qquad\qquad
\end{align*}
and for an incoming edge end $f$ at $\st(e_2)$ with
$s(e_2)<t(e_1)<f$
\begin{align*}
&f^*_{\st(e_2)'}((\beta)_{f'})\cdot f^*_{\st(e')}((\alpha)_{s(e')})\\
&= (\beta)_f\cdot (\low\alpha 3\oo\low\alpha 2\oo\low\alpha 1)_{s(e_1)t(e_1)s(e_2)}=(\beta)_f\cdot (\low\alpha 3)_{s(e_1)}\cdot (\low\alpha 1)_{s(e_2)}\cdot (\low\alpha 2)_{t(e_1)}\\
&=\langle  S(\low\alpha 2)\oo \low\beta 1, R\rangle\langle  \low\alpha 3\oo \low\beta 2, R\rangle  (\low\alpha 5\oo\low\alpha 4\oo\low\alpha 1)_{s(e_1)t(e_1)s(e_2)}\cdot (\low\beta 3)_f\\
&=(\low\alpha 3\oo\low\alpha 2\oo\low\alpha 1)_{s(e_1)t(e_1)s(e_2)}\cdot (\beta)_f= f^*_{\st(e')}((\alpha)_{s(e')})\cdot f^*_{\st(e')}((\beta)_{f'}).\qquad\qquad\qquad\qquad\qquad\qquad\qquad\qquad\qquad\qquad
\end{align*}
If $s(e_2)>t(e_1)$ one obtains for an incoming edge end $f$ at $\st(e_2)$ with
$f<t(e_1)<s(e_2)$
\begin{align*}
&f^*_{\st(e')}((\alpha)_{s(e')})\cdot f^*_{\st(e_2)'}((\beta)_{f'})\\
&=\langle \low\alpha 2, g^\inv\rangle\; (\low\alpha 4\oo\low\alpha 3\oo\low\alpha 1)_{s(e_1)t(e_1)s(e_2)}\cdot (\beta)_f\\
&=\langle \low\alpha 3, g^\inv\rangle\;  \langle \low\beta 1\oo S(\low\alpha 2), R\rangle\langle \low\beta 2\oo \low\alpha 4, R\rangle (\low\beta 3)_f\cdot (\low\alpha 6\oo\low\alpha 5\oo\low\alpha 1)_{s(e_1)t(e_1)s(e_2)}\\
&=\langle \low\alpha 2, g^\inv\rangle\;  (\beta)_f\cdot (\low\alpha 4\oo\low\alpha 3\oo\low\alpha 1)_{s(e_1)t(e_1)s(e_2)}=f^*_{\st(e_2)'}((\beta)_{f'})\cdot  f^*_{\st(e')}((\alpha)_{s(e')})\qquad\qquad\qquad\qquad\qquad\qquad\qquad\qquad\qquad\qquad\qquad\qquad
\end{align*}
and for an incoming edge end $f$ at $\st(e_2)$ with
$t(e_1)<s(e_2)<f$
\begin{align*}
&f^*_{\st(e_2)'}((\beta)_{f'})\cdot f^*_{\st(e')}((\alpha)_{s(e')})\\
&=\langle \low\alpha 2, g^\inv\rangle\;  (\beta)_f\cdot (\low\alpha 4\oo\low\alpha 3\oo\low\alpha 1)_{s(e_1)t(e_1)s(e_2)}\\
&=\langle \low\alpha 3, g^\inv\rangle\;  \langle \low\alpha 4\oo\low\beta 1, R\rangle\langle S(\low\alpha 2)\oo\low\beta 2, R\rangle (\low\alpha 6\oo\low\alpha 5\oo\low\alpha 1)_{s(e_1)t(e_1)s(e_2)}\cdot (\low\beta 3)_f\\
&=\langle \low\alpha 2, g^\inv\rangle\;  (\low\alpha 4\oo\low\alpha 3\oo\low\alpha 1)_{s(e_1)t(e_1)s(e_2)}\cdot (\beta)_f=f^*_{\st(e')}((\alpha)_{s(e')})\cdot f^*_{\st(e_2)'}((\beta)_{f'}).\qquad\qquad\qquad\qquad\qquad\qquad\qquad\qquad\qquad\qquad\qquad\qquad\qquad\qquad
\end{align*}
\end{proof}

With Lemma \ref{lem:commute} we can combine the linear maps $f^*_{v'}:  \mathcal A'^*_{v'}\to \bigotimes_{v\in V}\mathcal A^*_v$ from Definition \ref{def:edgeopmaps} into  a linear map
$f^*: \bigotimes_{v'\in V'}\mathcal A'^*_{v'}\to \bigotimes_{v\in V}\mathcal A^*_v$. It then follows directly that $f^*$ is an algebra homomorphism whenever all maps $f^*_{v'}$ are algebra morphisms. An analogous statement holds for the compatibility with the gauge transformations at the vertices. 
For this
recall  from Definition \ref{def:graph_functor} and Proposition \ref{lem:graphtrafo_vertnb} that each of the graph transformations in Definition \ref{def:graphtrafos}
is associated with a  map $g_V: V'\to V$, which are
are  inclusion maps or identity maps. Consequently, they induce injective Hopf algebra homomorphisms 
 $K\exoo{V'}\to K\exoo{V}$ and hence
a  $K\exoo{V'}$-module algebra structure on the $K\exoo{V}$-module algebra $\mathcal A^*_{\Gamma}$.  
Conversely, one obtains a $K\exoo{V}$-module structure  on $\mathcal A^*_{\Gamma'}$  by setting $\alpha\lhd'^*(h)_v=\epsilon(h)\, \alpha$ for all $h\in K$, $\alpha\in \mathcal A^*_{\Gamma'}$ and $v\in V'\setminus V$.   Combining Lemma \ref{lem:alghomord} and  \ref{lem:commute} then yields

\bigskip
\begin{corollary} \label{cor:inducedmaps}$\quad$
\begin{compactenum}
\item The maps $f^*_{v'}:\mathcal A'^*_{v'}\to \bigotimes_{v\in V} \mathcal A^*_v$ from Definition \ref{def:edgeopmaps} induce a linear map 
\begin{align*}
&f^*: \bigotimes_{v'\in V'} \mathcal A'^*_{v'} \to \bigotimes_{v\in V} \mathcal A^*_v,\quad \alpha^1\oo...\oo \alpha^{V'}\mapsto f^*_{v_1}(\alpha^1)\cdots f^*_{v'_{V'}}(\alpha^{V'})\quad\text{for}\quad \alpha^i\in \mathcal A'^*_{v_i}.
\end{align*} 

\item In cases (a), (b), (c), (e), (f) the map $f^*$ is a  homomorphism of $K^{\oo V}$-modules. 
\item In cases (a), (b), (c), (e) the map $f^*$ is an algebra homomorphism.
\end{compactenum}
\end{corollary}

We now proceed to step 3 and to show that the linear maps $f^*$ from Corollary \ref{cor:inducedmaps} induce linear maps 
$F^*:\mathcal A^*_{\Gamma'}\to \mathcal A^*_\Gamma$ via the injective homomorphisms of $K^{\oo V}$-module algebras $G^*_\Gamma$ from \eqref{eq:dualemb}.

\bigskip
\begin{theorem} \label{th:alghomfinal}Let $\Gamma,\Gamma'$ be ribbon graphs such that $\Gamma'$ is obtained from $\Gamma$ by one of the graph operations in Definition \ref{def:graphtrafos}. Then the associated  linear map 
$f^*: \oo_{v'\in V'}\mathcal A'^*_{v'}\to\oo_{v\in V}\mathcal A^*_v$
from Corollary \ref{cor:inducedmaps}  induces a unique linear map $F^*:\mathcal A^*_{\Gamma}\to\mathcal A^*_{\Gamma'}$  such that the following 
diagram commutes
\begin{align}\label{eq:diagcomm}
\xymatrix{
\mathcal A^*_{\Gamma'} \ar[d]_{G^*_{\Gamma'}} \ar[r]^{ F^*} &  \mathcal A^*_{\Gamma} \ar[d]^{G^*_{\Gamma}}\\
\oo_{v'\in V'}\mathcal A'^*_{v'}  \ar[r]_{f^*} & \oo_{v\in V}\mathcal A^*_v.}
\end{align}
The map $F^*$ is a homomorphism of $K^{\oo V}$-modules and an algebra homomorphism.
\end{theorem}

\begin{proof} 1.~To obtain a unique linear map  $F^*$ that makes the diagram \eqref{eq:diagcomm} commute, it is sufficient  to show that $f^*\circ G^*_{\Gamma'}(\mathcal A^*_{\Gamma'})\subset G^*_{\Gamma}(\mathcal A^*_{\Gamma})$, because the maps
$G^*_\Gamma, G^*_{\Gamma'}$ from \eqref{eq:dualemb} are injective.
For cases (a), (d) and (f)  this is obvious from formulas \eqref{eq:deleting}, \eqref{eq:adding},   and \eqref{eq:doubling}. 
 The claim for case (b) follows from case (c) by applying the involution   \eqref{eq:invol}. 

To prove the claim for case (c), recall that an element of $\bigotimes_{v'\in V'}\mathcal A'^*_{v'}$ is contained in image of $G^*_{\Gamma'}$ if and only if it is a product of elements of the form $(\low\alpha 1)_{t(e')}$ and $(\low\alpha 2)_{s(e')}$ for  $\alpha\in K^*$ and edges $e'\in E_{\Gamma'}$ in Sweedler notation, where the factors for each vertex $v'$ are ordered according to the cyclic ordering at $v'$. By formula \eqref{eq:contracttarget} and Corollary \ref{cor:inducedmaps},  applying the map $f^*$ to such an element leaves the contributions of the edges $e'\in E_{\Gamma'}$ for which the associated edge $e\in E_\Gamma$ is not incident at $\st(e)$ unchanged, and by Lemma \ref{lem:commute} the contributions of edge ends at different vertices of $\Gamma'$ commute. 
By Lemma \ref{lem:alghomord} - see in particular the expressions  in the proof of cases (i) to (ix) -  the map $f^*$  respects the ordering of the contributions at each vertex. It is therefore sufficient to show  that $f^*\circ G^*_{\Gamma} ((\alpha)_{h'})\in \mathrm{im}(G^*_\Gamma)$ for edges $h'\in E_{\Gamma'}$ with $h\in E_\Gamma$  incident at $\st(e)$.
For $h'\in E'$ with $\st(e)=\ta(h)\neq \st(h)$, one obtains
\begin{align*}
&f^*\circ G^*_{\Gamma'} ((\alpha)_{h'})=f^*((\low\alpha 2\oo\low\alpha 1)_{s(h')t(h')})  \\
&=\begin{cases}(\low\alpha 1\oo\low\alpha 2\oo\low\alpha 3\oo\low\alpha 4)_{t(e)s(e)t(h)s(h)} & t(h)>s(e) \\
\langle\low\alpha 3, g^\inv\rangle  (\low\alpha 1\oo\low\alpha 2\oo\low\alpha 4\oo\low\alpha 5)_{t(e)s(e)t(h)s(h)} & t(h)<s(e) \\
\end{cases}\\
&=\begin{cases}
G^*_\Gamma((\low\alpha 1\oo\low\alpha 2)_{eh}) & t(h)>s(e) \\
G^*_\Gamma(\langle \low\alpha 2, g^\inv\rangle\, (\low\alpha 1\oo \low\alpha 3)_{eh}) & t(h)<s(e),
\end{cases}\qquad\qquad\quad\qquad\qquad\qquad
\end{align*}
and for $h'\in E'$ with $\st(e)=\st(h)\neq \ta(h)$ 
\begin{align*}
&f^*\circ G^*_{\Gamma'} ((\alpha)_{h'})=f^*((\low\alpha 2\oo\low\alpha 1)_{s(h')t(h')})\\
&=\begin{cases}(S(\low\alpha 4)\oo S(\low\alpha 3)\oo\low\alpha 2\oo\low\alpha 1)_{t(e)s(e)s(h)t(h)} & s(h)>s(e) \\
\langle\low\alpha 3, g^\inv\rangle  ( S(\low\alpha 5)\oo S(\low\alpha 4)\oo\low\alpha 2\oo\low\alpha 1)_{t(e)s(e)s(h)t(h)} & s(h)<s(e) 
\end{cases}\\
&=\begin{cases}
G^*_\Gamma( (S(\low\alpha 2)\oo\low\alpha 1)_{he}) & s(h)>s(e) \\
G^*_\Gamma(\langle \low\alpha 2, g^\inv\rangle\, (S(\low\alpha 3)\oo\low\alpha 1)_{he}) & s(h)<s(e).
\end{cases}
\end{align*}
The proofs  for $h'\in E'$ with $\ta(h)=\st(h)=\st(e)$  are analogous.

In case (e), an analogous argument shows that it is sufficient to prove that 
$f^*\circ G^*_{\Gamma'}((\alpha)_{e'})\in \mathrm{im}(G^*_\Gamma)$ for all $\alpha\in K^*$. In this case,
 formula \eqref{eq:removing} and Corollary \ref{cor:inducedmaps} imply that $f^*$ does not affect the contributions of the edges $e'\in E_{\Gamma'}$  for which the associated edge $e\in E_\Gamma$ is not incident at $\st(e_2)=\ta(e_1)$.
By Lemma \ref{lem:commute} the contributions of edges at different vertices commute, which implies in particular that $f^*((\alpha)_{e'})$ commutes with $(\beta)_{h}$ for all edge ends $h$ at $\st(e_2)=\ta(e_1)$. Hence, the factors for edge ends at $\st(e_2)=\ta(e_1)$ in elements of $\mathrm{im}(f^*\circ G^*_{\Gamma'})$ are ordered according to the cyclic ordering, and it is sufficient to show that   $f^*\circ G^*_{\Gamma'}((\alpha)_{e'})\in \mathrm{im}(G^*_\Gamma)$. Equation  \eqref{eq:removing} yields:
\begin{align*}
&f^*\circ G^*_{\Gamma'} ((\alpha)_{e'})=f^*((\low\alpha 2\oo\low\alpha 1)_{s(e')t(e')})\\
&=\begin{cases}
(\low\alpha 4\oo\low\alpha 3\oo\low\alpha 2\oo\low\alpha 1)_{s(e_1)t(e_1)s(e_2)t(e_2)} & s(e_2)<t(e_1)\\
\langle \low\alpha 3, g^\inv\rangle\, (\low\alpha 5\oo\low\alpha 4\oo\low\alpha 2\oo\low\alpha 1)_{s(e_1)t(e_1)s(e_2)t(e_2)} & s(e_2)>t(e_1)
\end{cases}\\
&=\begin{cases}
G^*_\Gamma( (\low\alpha 2\oo\low\alpha 1)_{e_1e_2})   & s(e_2)<t(e_1)\\
G^*_\Gamma (\langle \low\alpha 2, g^\inv\rangle\, (\low\alpha 3\oo\low\alpha 1)_{e_1e_2}) & s(e_2)>t(e_1).
\end{cases}
\end{align*}
This proves that in all cases $\mathrm{im}(f^*\circ G^*_{\Gamma'})\subset\mathrm{im}(G^*_\Gamma)$, and by the injectivity of the maps $G^*_\Gamma$, we obtain a unique linear map $F^*:\mathcal A^*_{\Gamma'}\to \mathcal A^*_\Gamma$ that makes the diagram \eqref{eq:diagcomm} commute.

2.~We prove that $F^*$ is an algebra homomorphism and a morphism of $K^{\oo V}$ modules.
In cases (a), (b), (c), (e), the map $f^*$ is an algebra and a module morphism by Corollary \ref{cor:inducedmaps}, and the claim follows directly from
 the fact that $G^*_\Gamma$ and $G^*_{\Gamma'}$ are algebra and module morphisms. 
 
  Cases (d) and (f) have to be considered separately. 
In case (d), note that  the map $F^*$ affects only the contribution of the edge $e''\in E_{\Gamma'}$ and acts on the contribution of this edge via the counit.
That $F^*$ is an algebra homomorphism follows directly  by applying the counit of $K^*$ to the multiplication relations in Proposition  \ref{lem:algexplicit}  (b), (e), (g), (h) and (j) and using the identity $\langle\low\alpha 1\oo S(\low\beta 2), R\rangle\langle \low\alpha 2\oo\low \beta 1, R\rangle=\epsilon(\alpha)\epsilon(\beta)$.
 To show that $F^*$ is a morphism of $K^{\oo V}$-modules,  it is  sufficient to consider the contribution of the edge $e''$, for which one has
 from \eqref{eq:gtrafok} for $t(e'')<s(e'')$
\begin{align*}
&G^*_\Gamma\circ F^*((\alpha)_{e''}\lhd'^*(h)_v)=f^*\circ G^*_{\Gamma'}((\alpha)_{e''}\lhd'^*(h)_v)=f^*( G^*_{\Gamma'}((\alpha)_{e''})\lhd'^*(h)_v)\\
&=f^*((\low\alpha 2\oo\low\alpha 1)_{s(e'')t(e'')}\lhd'^*(h)_v)
=\langle \alpha_{(1)} S(\alpha_{(4)}) , h\rangle\, f^*((\alpha_{(3)}\oo\alpha_{(2)})_{s(e'')t(e'')})\\
&=\epsilon(\alpha_{(2)})\epsilon(\alpha_{(3)}) \langle \alpha_{(1)} S(\alpha_{(4)}) , h\rangle\,1\exoo{2E}=
\epsilon(h)\epsilon(\alpha)\,1^{\oo 2 E}=\epsilon(h)\,G^*_\Gamma\circ F^*((\alpha)_{e''})).
\end{align*}
The claim then follows by injectivity of $G^*_\Gamma$.

In case (f), the map $F^*$ is a morphism of $K^{\oo V}$-modules because  $f^*$ is a module morphism by Corollary \ref{cor:inducedmaps}. To show that $F^*$ is an algebra homomorphism,  note that by equation \eqref{eq:doubling} the associated map $f^*$  affects only the contributions of the edges $e'$ and $e''$ in $E_{\Gamma'}$ and satisfies
$$
f^*((\alpha)_{f'}\cdot' (\beta)_{g'})=f^*((\alpha)_{f'})\cdot f^*((\beta)_{g'})
$$
if at least one of the edge ends $f',g'\in E({\Gamma_\circ '})$ is not contained in $\{t(e'), t(e''), s(e'), s(e'')\}$. 
As $G^*_{\Gamma}$, $G^*_{\Gamma'}$  are injective algebra homomorphisms, it is therefore sufficient to show that 
$$f^*\circ G^*_{\Gamma'}((\alpha\oo\beta)_{e' e''})\cdot f^*\circ G^*_{\Gamma'}((\gamma\oo\delta)_{e'e''})=f^*\circ G^*_{\Gamma'}((\alpha\oo\beta)_{e'e''}\cdot' (\gamma\oo\delta)_{e'e''})$$
for all $\alpha,\beta,\gamma,\delta\in K^*$. If $e$ is an edge with 
 $\st(e)\neq \ta(e)$, we have $t(e')<t(e'')$ and $s(e')>s(e'')$. From equation \eqref{eq:doubling} (f)  and Figure \ref{fig:graph_ops} we then obtain
\begin{align*}
&f^*\circ G^*_{\Gamma'}((\alpha\oo\beta)_{e'e''}\cdot ' (\gamma\oo\delta)_{e'e''})\\
&= f^*((\low\alpha 2\oo\low\alpha 1\oo\low\beta 2\oo\low\beta 1)_{s(e')t(e')s(e'')t(e'')}\cdot' (\low\gamma 2\oo\low\gamma 1\oo\low\delta 2\oo\low\delta 1)_{s(e')t(e')s(e'')t(e'')}))\\
\intertext{}\\[-10ex]
&= f^*((\low\beta 2)_{s(e'')}\cdot '(\low\alpha 2)_{s(e')}\cdot' (\low\delta 2)_{s(e'')}\cdot' (\low\gamma 2)_{s(e')}\cdot' (\low\alpha 1)_{t(e')}\cdot' (\low\beta 1)_{t(e'')}\cdot' (\low\gamma 1)_{t(e')}\cdot' (\low\delta 1)_{t(e'')})\\
\intertext{}\\[-10ex]
&=\langle \low\delta 4\oo\low\alpha 4, R\rangle\langle\low\gamma 1\oo\low\beta 1, R\rangle\langle \low\gamma 2\oo\low\alpha 1, R\rangle\langle\low\delta 1\oo\low\beta 2,R\rangle
\\
&\quad  f^*((\low\delta 3\low\beta 4)_{s(e'')}\cdot '(\low\gamma 4\low\alpha 3)_{s(e')}\cdot' (\low\gamma 3\low\alpha 2)_{t(e')}\cdot'  (\low\delta 2\low\beta 3)_{t(e'')})\\
\intertext{}\\[-10ex]
&=\langle \low\delta 4\oo\low\alpha 4, R\rangle\langle\low\gamma 1\oo\low\beta 1, R\rangle\langle \low\gamma 2\oo\low\alpha 1, R\rangle\langle\low\delta 1\oo\low\beta 2,R\rangle\;
 (\low\gamma 4\low\alpha 3\low\delta 3\low\beta 4)_{s(e)}\cdot (\low\gamma 3\low\alpha 2\low\delta 2\low\beta 3)_{t(e)}\\
\intertext{}\\[-10ex]
 &=\langle \low\delta 2\oo\low\alpha 2, R\rangle\langle\low\gamma 1\oo\low\beta 1, R\rangle\langle \low\gamma 2\oo\low\alpha 1, R\rangle\langle\low\delta 1\oo\low\beta 2,R\rangle\;
 (\low\gamma 4\low\delta 4\low\alpha 4\low\beta 4)_{s(e)}\cdot (\low\gamma 3\low\delta 3\low\alpha 3\low\beta 3)_{t(e)}\\
   &=\langle\low\gamma 1\low \delta 1\oo\low\alpha 1\low\beta 1, R\rangle\;
 (\low\gamma 3\low\delta 3\low\alpha 3\low\beta 3)_{s(e)}\cdot (\low\gamma 2\low\delta 2\low\alpha 2\low\beta 2)_{t(e)}\\
 &=(\low\alpha 2\low\beta 2)_{s(e)}\cdot (\low\gamma 2\low\delta 2)_{s(e)}\cdot 
(\low\alpha 1\low\beta 1)_{t(e)}\cdot (\low\gamma 1\low\delta 1)_{t(e)}\\
&=(\low\alpha 2\low\beta 2\oo\low\alpha 1\low\beta 1)_{s(e)t(e)}\cdot (\low\gamma 2\low\delta 2\oo\low\gamma 1\low\delta 1)_{s(e)t(e)}\\
&= f^*((\low\alpha 2\oo\low\alpha 1\oo\low\beta 2\oo\low\beta 1)_{s(e')t(e')s(e'')t(e'')})\cdot  f^*((\low\gamma 2\oo\low\gamma 1\oo\low\delta 2\oo\low\delta 1)_{s(e')t(e')s(e'')t(e'')})\\
&= f^*\circ G^*_{\Gamma'}((\alpha\oo\beta)_{e'e''})\cdot  f^*\circ G^*_{\Gamma'}((\gamma\oo\delta)_{e'e''})
\end{align*}
where we used the identities $(\id\oo\Delta)(R)=R_{13}R_{12}$ und $(\Delta\oo\id)(R)=R_{13}R_{23}$. Similarly, if
 $e$ is a loop with $s(e)<t(e)$, we have $s(e'')<s(e')<t(e')<t(e'')$ and obtain 
\begin{align*}
& f^*\circ G^*_{\Gamma'}((\alpha\oo\beta)_{e'e''}\cdot ' (\gamma\oo\delta)_{e'e''})\\
&= f^*((\low\alpha 2\oo\low\alpha 1\oo\low\beta 2\oo\low\beta 1)_{s(e')t(e')s(e'')t(e'')}\cdot' (\low\gamma 2\oo\low\gamma 1\oo\low\delta 2\oo\low\delta 1)_{s(e')t(e')s(e'')t(e'')}))\\
&= f^*((\low\beta 2)_{s(e'')}\cdot '(\low\alpha 2)_{s(e')}\cdot' (\low\alpha 1)_{t(e')}\cdot' (\low\beta 1)_{t(e'')}\cdot' (\low\delta 2)_{s(e'')}\cdot' (\low\gamma 2)_{s(e')}\cdot' (\low\gamma 1)_{t(e')}\cdot' (\low\delta 1)_{t(e'')})\\
&=\langle \low\delta3\low\gamma 3\oo\low\alpha 1\low\beta 1, R\rangle\\
&\quad  f^*((\low\beta 3)_{s(e'')}\cdot '(\low\alpha 3)_{s(e')}\cdot' (\low\delta 2)_{s(e'')}\cdot' (\low\gamma 2)_{s(e')}\cdot' (\low\alpha 2)_{t(e')}\cdot' (\low\beta 2)_{t(e'')}\cdot' (\low\gamma 1)_{t(e')}\cdot' (\low\delta 1)_{t(e'')})\\
&=\langle \low\delta5\low\gamma 5\oo\low\alpha 1\low\beta 1, R\rangle \langle \low\delta 4\oo\low\alpha 5, R\rangle\langle\low\gamma 1\oo\low\beta 2, R\rangle\langle \low\gamma 2\oo\low\alpha 2, R\rangle\langle\low\delta 1\oo\low\beta 3,R\rangle
\\
&\quad f^*((\low\delta 3\low\beta 5)_{s(e'')}\cdot '(\low\gamma 4\low\alpha 4)_{s(e')}\cdot' (\low\gamma 3\low\alpha 3)_{t(e')}\cdot'  (\low\delta 2\low\beta 4)_{t(e'')})\\
&=\langle \low\delta5\low\gamma 5\oo\low\alpha 1\low\beta 1, R\rangle \langle \low\delta 4\oo\low\alpha 5, R\rangle\langle\low\gamma 1\oo\low\beta 2, R\rangle\langle \low\gamma 2\oo\low\alpha 2, R\rangle\langle\low\delta 1\oo\low\beta 3,R\rangle
\\
&\quad (\low\gamma 4\low\alpha 4\low\delta 3\low\beta 5)_{s(e)}\cdot (\low\gamma 3\low\alpha 3\low\delta 2\low\beta 4)_{t(e)}\\
&=\langle \low\delta5\low\gamma 5\oo\low\alpha 1\low\beta 1, R\rangle \langle \low\delta 2\oo\low\alpha 3, R\rangle\langle\low\gamma 1\oo\low\beta 2, R\rangle\langle \low\gamma 2\oo\low\alpha 2, R\rangle\langle\low\delta 1\oo\low\beta 3,R\rangle
\\
&\quad (\low\gamma 4\low\delta 4\low\alpha 5\low\beta 5)_{s(e)}\cdot (\low\gamma 3\low\delta 3\low\alpha 4\low\beta 4)_{t(e)}\\
&=\langle \low\delta4\low\gamma 4\oo\low\alpha 1\low\beta 1, R\rangle 
\langle \low\gamma 1\low\delta 1\oo\low\alpha 2\low\beta 2, R\rangle\; (\low\gamma 3\low\delta 3\low\alpha 4\low\beta 4)_{s(e)}\cdot (\low\gamma 2\low\delta 2\low\alpha 3\low\beta 3)_{t(e)}\\
&=\langle \low\delta3\low\gamma 3\oo\low\alpha 1\low\beta 1, R\rangle 
\; (\low\alpha 3\low\beta 3)_{s(e)}\cdot (\low\gamma 2\low\delta 2)_{s(e)}\cdot  (\low\alpha 2\low\beta 2)_{t(e)}\cdot (\low\gamma 1\low\delta 1)_{t(e)}\\
&= (\low\alpha 2\low\beta 2)_{s(e)}\cdot  (\low\alpha 1\low\beta 1)_{t(e)}\cdot (\low\gamma 2\low\delta 2)_{s(e)} \cdot (\low\gamma 1\low\delta 1)_{t(e)}\\
&= f^*((\low\alpha 2\oo\low\alpha 1\oo\low\beta 2\oo\low\beta 1)_{s(e')t(e')s(e'')t(e'')})\cdot  f^*((\low\gamma 2\oo\low\gamma 1\oo\low\delta 2\oo\low\delta 1)_{s(e')t(e')s(e'')t(e'')})\\
&= f^*\circ G^*_{\Gamma'}((\alpha\oo\beta)_{e'e''})\cdot  f^*\circ G^*_{\Gamma'}((\gamma\oo\delta)_{e'e''}).
\end{align*}
The computations for a loop with $t(e)<s(e)$ are analogous. This proves the claim for case (f).
\end{proof}

\bigskip
\begin{remark}\label{rem:subdivgrtrafo} As  $F^*:\mathcal A^*_{\Gamma'}\to \mathcal A^*_\Gamma$ is defined uniquely by  $f^*: \oo_{v\in V'}\mathcal A'^*_v\to\oo_{v\in V}\mathcal A^*_v$ via \eqref{eq:diagcomm}, it follows directly that the assignment $f^*\to F^*$ is functorial, i.e.~the following diagrams commute 
\begin{align*}
\xymatrix{
\mathcal A^*_{\Gamma''}  \ar[r]^{F'^*} \ar[d]_{G^*_{\Gamma''}} &\mathcal A^*_{\Gamma'} \ar[d]_{G^*_{\Gamma'}} \ar[r]^{ F^*} &  \mathcal A^*_{\Gamma} \ar[d]^{G^*_{\Gamma}}\\
\oo_{v\in V''}\mathcal A''^*_v \ar[r]_{f'^*} & \oo_{v\in V'}\mathcal A'^*_v  \ar[r]_{f^*} & \oo_{v\in V}\mathcal A^*_v.}\qquad
\xymatrix{
\mathcal A^*_{\Gamma}  \ar[r]^{\id} \ar[d]_{G^*_{\Gamma}} &\mathcal A^*_{\Gamma} \ar[d]_{G^*_{\Gamma}} \\
\oo_{v\in V}\mathcal A^*_v \ar[r]_{\id} & \oo_{v\in V}\mathcal A^*_v. 
}
\end{align*}
This is explored in more depth in Section \ref{subsec:algprop}.
\end{remark}

\smallskip
\begin{remark}\label{rem:graphopprops}  From  equations \eqref{eq:deleting} to \eqref{eq:doubling} in Definition \ref{def:edgeopmaps} and   Theorem \ref{th:alghomfinal} together with the fact that the maps $G^*_\Gamma, G^*_{\Gamma'}$ are injective, we also obtain:
\begin{compactenum}
\item  The maps $d^*_e, c^*_{\ta(e)}, c^*_{\st(e)}, w^*_{e_1e_2}:\oo_{v\in V'}\mathcal A'^*_v\to\oo_{v\in V}\mathcal A^*_v$
and the  maps $D^*_e, C^*_{\ta(e)}, C^*_{\st(e)}, W^*_{e_1,e_2}:\mathcal A^*_{\Gamma'}\to\mathcal A^*_\Gamma$ are injective.\\[-2ex]

\item The maps $a^*_v, do^*_e: \oo_{v\in V'}\mathcal A'^*_v\to\oo_{v\in V}\mathcal A^*_v$ and the  maps 
$A^*_v, Do^*_e: \mathcal A^*_{\Gamma'}\to\mathcal A^*_\Gamma$ are surjective.
\end{compactenum}
\end{remark}

\smallskip

\medskip
\begin{remark} The linear maps $f^*:\oo_{v\in V'}\mathcal A'^*_v\to \oo_{v\in V}\mathcal A^*_v$ from Corollary \ref{cor:inducedmaps}
  and the associated  $K\exoo{V}$-module algebra homomorphisms  $F^*:\mathcal A^*_{\Gamma'}\to \mathcal A^*_\Gamma$  induce linear maps  $f:\oo_{v\in V}K\exoo{v}\to\oo_{v\in V'} K\exoo{v}$ 
with $\langle f^*(\alpha), k\rangle=\langle \alpha, f(k)\rangle$ for all $\alpha\in {K^*}\exoo{2E'}$ and $k\in K\exoo{E}$ 
and linear maps $F: K\exoo{E}\to K\exoo{E'}$ with $\langle F^*(\alpha), k\rangle=\langle \alpha, F(k)\rangle$ for all $\alpha\in \mathcal A^*_{\Gamma'}$, $k\in K\exoo{E}$. As the 
 $K\exoo{V}$-left module structures on $K\exoo{E}$ and $K\exoo{E'}$ are dual to the $K\exoo{V}$-right module structures on $\mathcal A^*_\Gamma$ and $\mathcal A^*_{\Gamma'}$, it follows that the latter
 are module homomorphisms with respect to the $K\exoo{V}$-left module structure of $K\exoo{E}$ and $K\exoo{E'}$.
\end{remark}

Explicit expressions for the algebra homomorphisms $F^*:\mathcal A^*_{\Gamma'}\to\mathcal A^*_\Gamma$ from Theorem \ref{th:alghomfinal} can be derived  from 
Definition \ref{def:edgeopmaps},
 but these expressions depend on the edge orientations, and one has to distinguish the cases  $\st(g)=\ta(g)$ and $\st(g)\neq \ta(g)$ for all 
edges $g$ at the relevant vertices.  This requires a   case by case analysis and will not be considered here.
 However, it is instructive to consider two examples.

\bigskip
\begin{example}[Contracting a bivalent vertex]  \label{ex:bivalent_cont}
 Let $e_1,e_2\in E$ be edges that share a single bivalent vertex  $v=\st(e_2)=\ta(e_1)$. Let $\Gamma'$ be the ciliated ribbon graph obtained by contracting $e_2$ 
 towards $\ta(e_2)$ or $e_1$ towards $\st(e_1)$ and denote  by  $e'$ the edge of $\Gamma'$ corresponding to the edge $e_1$ or $e_2$ that is not contracted.
If the cilium at $v$  points to the right, viewed in the direction of  $e_1$ and $e_2$, then  we have $s(e_2)>t(e_1)$ and
  the contraction of $e_2$ towards $t(e_2)$
is given  by
  \begin{align}
&c_{\ta(e_2)}: (k)_f\to (k)_{f}, & & (k)_{s(e_1)}\mapsto (k)_{s(e')}, & &(k\oo k'\oo k'')_{t(e_2)s(e_2)t(e_1)}\mapsto (kk'k'')_{t(e')}\\
&c^*_{\ta(e_2)}:  (\alpha)_{g}\to (\alpha)_{g},   & &  (\alpha)_{s(e')}\mapsto (\alpha)_{s(e_1)}, & & (\alpha)_{t(e')}\mapsto (\low\alpha 1\oo\low\alpha 2\oo\low\alpha 3)_{t(e_2)s(e_2)t(e_1)}.\nonumber
 \end{align}
If the cilium at $v$ points to the left, viewed in the direction of  $e_1$ and $e_2$,  then we have $s(e_2)<t(e_1)$ and 
 the contraction of  $e_1$ towards $\st(e_1)$ acts  on  the vertex neighbourhoods of $\Gamma$ and $\Gamma'$ by 
 \begin{align}
&c_{\st(e_1)}: (k)_f\to (k)_{f}, & & (k)_{t(e_2)}\mapsto (k)_{t(e')}, & &(k\oo k'\oo k'')_{s(e_1)t(e_1)s(e_2)}\mapsto (kk'k'')_{s(e')}\\
&c^*_{\st(e_1)}:  (\alpha)_{g}\to (\alpha)_{g}, & & (\alpha)_{t(e')}\mapsto (\alpha)_{t(e_2)} &  &(\alpha)_{s(e')}\mapsto (\low\alpha 1\oo\low\alpha 2\oo\low\alpha 3)_{s(e_2)t(e_1)s(e_1)}.\nonumber
 \end{align}
  for all 
  $f\in E(\Gamma_\circ)\setminus\{s(e_1), t(e_1), s(e_2), t(e_2)\}$ and $g\in E(\Gamma'_\circ)\setminus\{s(e'), t(e')\}$.
In both cases, the associated linear map $\mathcal A^*_{\Gamma'}\to\mathcal A^*_{\Gamma}$ and its dual are  given by 
\begin{align}\label{eq:bivcont}
& C^*_{\ta(e_2)}=C^*_{\st(e_1)}: (\alpha)_{e'}\mapsto (\low\alpha 2\oo \low\alpha 1)_{e_1e_2} & &(\alpha)_{g'}\mapsto (\alpha)_g\;\forall g'\in E'\setminus\{e'\}.\\
&C_{\ta(e_2)}=C_{\st(e_1)}: (k\oo k')_{e_1e_2}\mapsto (k'k)_{e'} & &(k)_g\mapsto (k)_{g'} \;\forall g\in E'\setminus\{e_1,e_2\}.\nonumber
\end{align}
 \end{example}

 If we consider the edge subdivision $\Gamma_\circ$ of $\Gamma$,  we can perform an edge contraction as in Example \ref{ex:bivalent_cont} for each bivalent vertex $v\in V(\Gamma_\circ)\setminus V(\Gamma)$. If the cilium at the bivalent vertex $v$ points to the right, we contract the outgoing edge at $v$ towards its target vertex in $V(\Gamma)$,  otherwise the incoming edge at $v$ towards its starting vertex in $V(\Gamma)$.   It is then clear from Definition \ref{def:edgeopmaps}, Theorem \ref{th:alghomfinal}  and expression \eqref{eq:bivcont} that the resulting linear map $C^*:\mathcal A^*_{\Gamma_\circ}\to\mathcal A^*_\Gamma$ does not depend on the order in which these edge contractions  are performed and coincides with the map $G^*$  in \eqref{eq:dualemb}.

\bigskip
\begin{corollary}  \label{cor:gstar} Let $\Gamma$ be a ciliated ribbon graph and $\Gamma_\circ$ its edge subdivision.  Then  map $C^*:\mathcal A^*_\Gamma\to \mathcal A^*_{\Gamma_\circ}$ obtained by contracting  at each bivalent vertex $v\in V(\Gamma_\circ)\setminus V(\Gamma)$  the outgoing (incoming) end towards its target (starting) vertex if the cilium at $v$ points to the right (left) is the  map $G^*$ in \eqref{eq:dualemb}.
\end{corollary}

\begin{figure}
\centering
\includegraphics[scale=0.35]{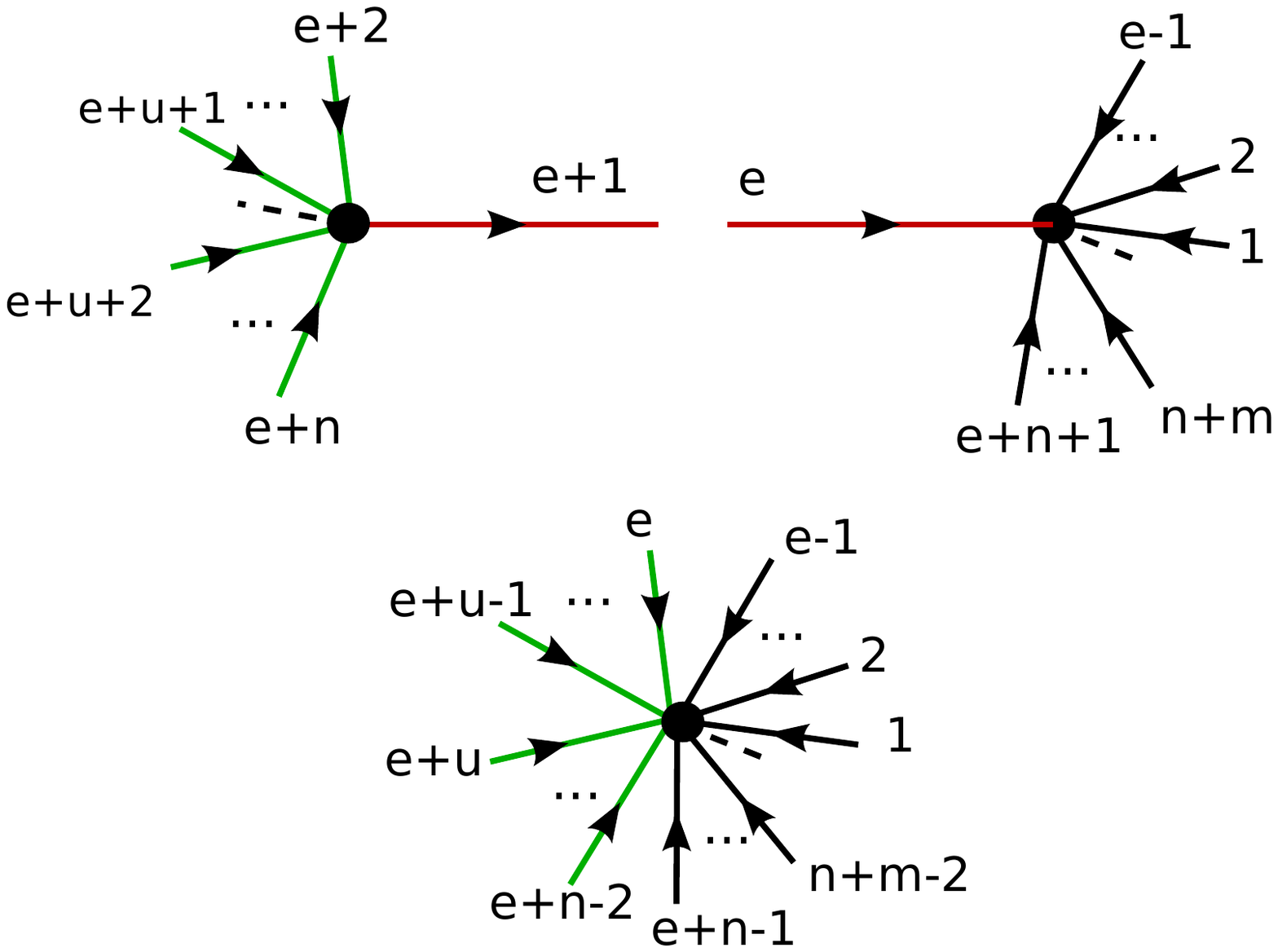}
\caption{Contracting an edge $e$ towards the target  vertex. Action of the graph operation on the vertex neighbourhoods $\Gamma_{\st(e)},\Gamma_{\ta(e)}$.}
\label{fig:edge_contracting2}
\end{figure}

\begin{example}[Contracting an edge towards the target vertex] \label{ex:contracting}

Let $e\in E$ be an edge with $\st(e)\neq \ta(e)$.  Suppose that $|\st(e)|=n$, $|\ta(e)|=m$, that all other edge ends  at $\st(e)$ and $\ta(e)$  are  incoming  and that  they  are  ordered as in Figure \ref{fig:edge_contracting2}, where the numbers  indicate the copy of $K^*$ in the tensor products $K^{*\oo (n+m)}$ and $K^{*\oo(n+m-2)}$. 
Then the restriction to $\mathcal A'^*_{\ta(e)}$ of the algebra homomorphism  $c^*_{\ta(e)}$  from Definition \ref{def:edgeopmaps}  (c) is the linear map $c^*_{\ta(e)}: K^{*\oo(n+m-2)}\to K^{*\oo(n+m)}$ given by
\begin{align}
c_{\ta(e)}^*:\;&\alpha^1\oo... \oo\alpha^{n+m-2}\mapsto \langle g^\inv, \alpha^{e+u}_{(3)}\cdots\alpha^{e+n-2}_{(3)}\rangle\;\alpha^1\oo ... \oo\alpha^{e-1}\oo
\alpha^{e}_{(1)}\cdots  \alpha^{e+n-2}_{(1)}\oo \nonumber\\
& \oo\alpha^{e}_{(2)}\cdots  \alpha^{e+n-2}_{(2)}\oo\alpha^e_{(3)}\oo...\oo\alpha^{e+u-1}_{(3)}\oo\alpha^{e+u}_{(4)}\oo... \oo \alpha^{e+n-2}_{(4)}\oo \alpha^{e+n-1}\oo...\oo \alpha^{n+m-2}.\nonumber
\end{align}
The dual map $c_{\ta(e)}: K^{\oo (n+m)}\to K^{\oo(n+m-2)}$ is given by 
\begin{align}
c_{\ta(e)}:\,&k^1\oo...\oo k^{n+m}\mapsto k^1\oo...\oo k^{e-1}\oo k^e_{(1)}k^{e+1}_{(1)}  k^{e+2}\oo...\oo k^e_{(u)}k^{e+1}_{(u)} k^{e+u+1}\oo\nonumber\\
&\oo  k^e_{(u+1)}k^{e+1}_{(u+1)} g^\inv k^{e+u+2} \oo ... \oo k^e_{(n)}k^{e+1}_{(n)}g^\inv k^{e+n}\oo k^{e+n+1} ...\oo k^{n+m}.\nonumber\qquad\qquad\qquad\quad
\end{align}
\end{example}

With  the explicit expression for the contraction map in Example \ref{ex:contracting} we can show that the injective algebra homomorphisms $C^*_{\ta(e)}:\mathcal A^*_{\Gamma'}
\to\mathcal A^*_\Gamma$, $C^*_{\ta(e)}:\mathcal A^*_{\Gamma'}
\to\mathcal A^*_\Gamma$ from Theorem \ref{th:alghomfinal} restrict to algebra isomorphisms between the algebras of gauge invariant functions for the ribbon graphs $\Gamma'$ and $\Gamma$.

\begin{theorem} \label{th:graphops} Let $\Gamma'$ be obtained from $\Gamma$ by contracting an edge towards its starting or target vertex. Then the associated $K\exoo{V}$-module algebra homomorphism $ C^*_{\st(e)}$ or $C^*_{\ta(e)}$ from Definition \ref{def:edgeopmaps} and Theorem \ref{th:alghomfinal} induces an isomorphism 
 $\mathcal A^*_{\Gamma'\;inv}\xrightarrow{\sim} \mathcal A^*_{\Gamma\; inv}$.  
\end{theorem}
\begin{proof} 
As $C^*_{\st(e)}$ is obtained from $C^*_{\ta(e)}$ by reversing the orientation of $e$ with the involution \eqref{eq:invol},  it   is  sufficient to prove this for $C^*_{\ta(e)}$. 
As all outgoing edge ends except $s(e)$ at the starting vertex  $\st(e)$  can also be reversed by applying the involution \eqref{eq:invol} it is sufficient to consider the case where all  other edge ends at $\st(e)$ are incoming.  This allows us to restrict attention to  the edge configuration from Example \ref{ex:contracting} and Figure \ref{fig:edge_contracting2}.  
We show that every element  $\theta \in K^{*\oo(n+m)}\cap G^*_\Gamma(\mathcal A^*_\Gamma)$ that is invariant under gauge transformations at $\st(e)$ is in the image of $c^*_{\st(e)}$. For the vertex neighbourhood in Figure \ref{fig:edge_contracting2}, one finds that any such element
 $\theta$ is of the form 
$$\theta=\Sigma_i \; \gamma^i\oo \beta^{i}_{(1)}\oo \beta^i_{(2)}\oo \alpha^{e+2,i}\oo ...\oo\alpha^{e+u+1,i}\oo\alpha^{e+u+2,i}\oo...\oo\alpha^{e+n,i}\oo \delta^i,$$ with $\beta^i,\alpha^{e+2,i},...,\alpha^{e+n,i}\in K^*$, $\gamma^i\in K^{*\oo (e-1)}$ and $\delta^i\in K^{*\oo (m-e)}$. 
Gauge invariance of $\theta$ under gauge transformations at $\st(e)$ implies for all $h\in K$
\begin{align*}
&\epsilon(h)\,\theta=\epsilon(h)\;\Sigma_i \; \gamma^i\oo \beta^{i}_{(1)}\oo \beta^i_{(2)}\oo \alpha^{e+2,i}\oo ....\oo\alpha^{e+n,i}\oo \delta^i=\theta\lhd^*(h)_{\st(e)}\\
&=\Sigma_i\; \langle \alpha^{e+u+2,i}_{(1)}\cdots \alpha^{e+n,i}_{(1)}S(\beta^i_{(3)})\alpha^{e+2,i}_{(1)}\cdots\alpha^{e+u,i}_{(1)}, h\rangle\; \gamma^i\oo \beta^{i}_{(1)}\oo \beta^i_{(2)}\oo \alpha^{e+2,i}_{(2)}\oo ...\oo\alpha^{e+n,i}_{(2)}\oo \delta^i.
\end{align*}
 This implies by duality 
\begin{align*}
&1\oo\theta=\Sigma_i\; 1\oo\gamma^i\oo \beta^{i}_{(1)}\oo \beta^i_{(2)}\oo \alpha^{e+2,i}\oo ...\oo\alpha^{e+u+1,i}\oo\alpha^{e+u+2,i}\oo...\oo\alpha^{e+n,i}\oo \delta^i\\
&=\Sigma_i\;  \alpha^{e+u+2,i}_{(1)}\cdots \alpha^{e+n,i}_{(1)}S(\beta^i_{(3)})\alpha^{e+2,i}_{(1)}\cdots\alpha^{e+u,i}_{(1)}\oo \gamma^i\oo \beta^{i}_{(1)}\oo \beta^i_{(2)}\oo \alpha^{e+2,i}_{(2)}\oo ...\oo\alpha^{e+n,i}_{(2)}\oo \delta^i.
\end{align*}
By applying the comultiplication to the tensor factors between $\beta^i_{(2)}$ and $\delta^i$
 in the last equation, then applying the antipode or its inverse to them and multiplying the resulting factors from the left and  right to the first factor in the tensor product, one then has
 \begin{align*}
&\Sigma_i\; S^\inv(\alpha^{e+n,i}_{(1)})\cdots S^\inv(\alpha^{e+u+2,i}_{(1)}) S(\alpha^{e+u,i}_{(1)})\cdots S(\alpha^{e+2,i}_{(1)})\oo\gamma^i\oo \beta^{i}_{(1)}\oo \beta^i_{(2)}\oo \alpha^{e+2,i}_{(2)}\oo ...\oo\alpha^{e+n,i}_{(2)}\oo \delta^i\\
&=\Sigma_i\; S(\beta^i_{(3)})\oo \gamma^i\oo \beta^{i}_{(1)}\oo \beta^i_{(2)}\oo \alpha^{e+2,i}\oo ...\oo\alpha^{e+n,i}\oo \delta^i.
\end{align*}
Eliminating the tensor factors $\beta^i_{(1)}$ and $\beta^i_{(2)}$ with the counit
and applying  the inverse of the antipode to the first factor in the tensor product  on both sides yields
 \begin{align*}
 &\Sigma_i\; \beta^i\oo \gamma^i\oo  \alpha^{e+2,i}\oo ...\oo\alpha^{e+n,i}\oo \delta^i\\
=&\Sigma_i\; \epsilon(\beta^i)  \alpha^{e+2,i}_{(1)}\cdots \alpha^{e+u,i}_{(1)} S^{-2}(\alpha^{e+u+2,i}_{(1)})\cdots S^{-2}(\alpha^{e+n,i}_{(1)})
\oo\gamma^i\oo  \alpha^{e+2,i}_{(2)}\oo ...\oo\alpha^{e+n,i}_{(2)}\oo \delta^i\\
=& \Sigma_i\; \epsilon(\beta^i)  \langle g, \alpha^{e+u+2,i}_{(1)}\cdots \alpha^{e+n,i}_{(1)}\rangle\langle g^\inv, \alpha^{e+u+2,i}_{(3)}\cdots \alpha^{e+n,i}_{(3)}\rangle\\
&\quad\alpha^{e+2,i}_{(1)}\cdots \alpha^{e+u,i}_{(1)} \alpha^{e+u+2,i}_{(2)}\cdots \alpha^{e+n,i}_{(2)}\oo \gamma^i\oo  \alpha^{e+2,i}_{(2)}\oo ...\oo \alpha^{e+u,i}_{(2)}\oo\alpha^{e+u+2,i}_{(4)}...\oo\alpha^{e+n,i}_{(4)}\oo \delta^i
\end{align*}
Applying the comultiplication to the first factor in the tensor product and permuting the factors in the tensor product then yields
\begin{align*}
\theta=&\Sigma_i\; \gamma^i\oo\beta^i_{(1)}\oo\beta^i_{(2)}\oo \alpha^{e+2,i}\oo\ldots\oo \alpha^{e+n,i}\oo\delta^i\\
=&\Sigma_i\;\epsilon(\beta^i)\;  \langle g, \alpha^{e+u+2,i}_{(1)}\cdots \alpha^{e+n,i}_{(1)}\rangle\langle g^\inv, \alpha^{e+u+2,i}_{(4)}\cdots \alpha^{e+n,i}_{(4)}\rangle\;\gamma^i\oo \alpha^{e+2,i}_{(1)}\cdots \alpha^{e+u,i}_{(1)} \alpha^{e+u+2,i}_{(2)}\cdots \alpha^{e+n,i}_{(2)}\\
&\oo  \alpha^{e+2,i}_{(2)}\cdots \alpha^{e+u,i}_{(2)} \alpha^{e+u+2,i}_{(3)}\cdots \alpha^{e+n,i}_{(3)}
\oo  \alpha^{e+2,i}_{(3)}\oo ...\oo \alpha^{e+u,i}_{(3)}\oo \alpha^{e+u+2,i}_{(5)}\oo...\oo\alpha^{e+n,i}_{(5)}\oo \delta^i\\
=&c^*_{\ta(e)}(\Sigma_i \; \epsilon(\beta^i)\langle \alpha^{e+u+2,i}_{(1)}\cdots \alpha^{e+n,i}_{(1)}, g \rangle\; \gamma^i\oo \alpha^{e+2,i}\oo\cdots\alpha^{e+u,i} \alpha^{e+u+2,i}_{(2)}\cdots \alpha^{e+n,i}_{(2)}\oo\delta^i)\\[-8ex]
\end{align*}
\end{proof}

Theorem \ref{th:graphops} has important implications for  topological invariance, i.e.~how the algebra of observables $\mathcal A^*_{\Gamma\, inv}$ of a Hopf algebra gauge theory depends on the choice of the ciliated ribbon graph $\Gamma$.

\begin{corollary}\label{cor:edgesubdiv} Let $K$ be a  finite-dimensional ribbon Hopf algebra. Then the  associated $K$-valued local Hopf algebra gauge theories have the following properties:
\begin{compactenum}
\item  For each ribbon graph $\Gamma$, the
 algebras $\mathcal A^*_{\Gamma\, inv}$ and $\mathcal A^*_{\Gamma_\circ\, inv}$ are isomorphic.  
\item   For each ribbon graph $\Gamma$, there is a  ribbon graph $\Gamma'$ without loops or multiple edges such that $\mathcal A^*_{\Gamma'\,inv}\cong\mathcal A^*_{\Gamma\,inv}$.

\item For a connected  ribbon graph $\Gamma$,  there is  ribbon graph $\Gamma'$ with a single vertex such that  $\mathcal A^*_{\Gamma'\,inv}\cong\mathcal A^*_{\Gamma\,inv}$.
\item Let   $\Gamma$, $\Gamma'$  be ribbon graphs  and $\Sigma$, $\Sigma'$  the surfaces obtained by gluing annuli to each face of $\Gamma$ or $\Gamma'$. If $\Sigma'$ and $\Sigma'$ are homeomorphic, then $\mathcal A^*_{\Gamma\,inv}\cong\mathcal A^*_{\Gamma'\,inv}$.
\end{compactenum}
\end{corollary}
\begin{proof} 
The ribbon graph $\Gamma$ is obtained from $\Gamma_\circ$ by contracting for each edge $e\in E(\Gamma)$ either the edge end $s(e)\in E(\Gamma_\circ)$ towards its starting vertex or the edge end $t(e)\in E(\Gamma_\circ)$ towards its target vertex. The associated linear map
$C^*:\mathcal A^*_{\Gamma}\to\mathcal A^*_{\Gamma_\circ}$ is given in Example \ref{ex:bivalent_cont} and induces 
an algebra isomorphism  $\mathcal A^*_{\Gamma}\xrightarrow{\sim} \mathcal A^*_{\Gamma_\circ\, inv}$ by
Theorem \ref{th:graphops}.  The second claim follows since for each ribbon graph $\Gamma$ is double edge subdivision $\Gamma_{\circ\circ}$ is a ribbon graph without loops or multiple edges.  The third claim holds because every connected ciliated ribbon graph $\Gamma$ can be transformed into a  ribbon graph with a single vertex by contracting the edges of a  maximal rooted tree $T\subset\Gamma$. By Corollary \ref{cor:gstar}, this algebra isomorphism does not depend on the order in which the edge contractions are performed, and by Theorem \ref{th:graphops} it induces an
algebra isomorphism $\mathcal A^*_{\Gamma\, inv}\xrightarrow{\sim}\mathcal A_{\Gamma'\,inv}$. For the fourth claim, note that if $\Gamma$, $\Gamma'$ are ribbon graphs such that gluing annuli  to their faces yields homeomorphic surfaces  $\Sigma$ and $\Sigma'$, then by Proposition \ref{prop:addcont}  $\Gamma$ and $\Gamma'$ are related by a sequence of edge contractions.  The claim then follows from Theorem \ref{th:graphops}.
\end{proof}

The graph transformations of    contracting a bivalent vertex, contracting an edge towards its starting or target vertex   and erasing an edge  were also considered in   \cite{AGSII}. 
In Propositions 8 to 10  in \cite{AGSII}
it is shown that they give rise  rise to algebra (iso)morphisms between the  counterparts of the  algebras $\mathcal A^*_\Gamma$ and $\mathcal A^*_{\Gamma'}$ considered there.   The Poisson-Lie analogues of these graph transformations were first considered in \cite{FR}, where it was shown that they give rise to Poisson maps between certain Poisson algebras associated with the underlying  ciliated ribbon graphs. However, 
both publications consider  only    the counterparts of the  maps $D^*_e,A^*_{v,i}, C^*_{\ta(e)}$, $C^*_{\st(e)}:\mathcal A^*_{\Gamma'}\to\mathcal A^*_\Gamma$.
In contrast,  this article, the  maps $D^*_e,A^*_{v,i}, C^*_{\ta(e)}$, $C^*_{\st(e)}:\mathcal A^*_{\Gamma'}\to\mathcal A^*_\Gamma$   are induced by  the  {\em elementary} operations on {\em vertex neighbourhoods} in Definition \ref{def:edgeopmaps}.  This
simplifies their description considerably.  Note that some of these graph operations  on vertex neighbourhoods were were also considered in \cite{BFK}, where they were used to prove that the graphical calculus defined there gives rise to topological invariants.

\section{Holonomies and curvature}
\label{sec:holsec}

\subsection{Holonomies and curvatures in a Hopf algebra gauge theory}
\label{subsec:genhols}

In this section, we introduce a notion of holonomy for a Hopf algebra gauge theory with values in a finite-dimensional semisimple quasitriangular  Hopf algebra $K$.  In analogy to lattice gauge theory for a group, a holonomy in a Hopf algebra gauge theory should assign to each path $p$ in $\Gamma$ a linear map $K\exoo{E}\to K$ in such a way that this assignment is  compatible with
  the composition of paths,  trivial paths,   reversal of edge orientation and  the  defining relations of the path groupoid $\mathcal G(\Gamma)$.
 In other words, 
 one has to equip the vector space $\mathrm{Hom}_{\FF}(K\exoo{E}, K)$ with the structure of an $\FF$-linear category with a single object and to construct a functor $\mathrm{Hol}: \mathcal C(\Gamma)\to \mathrm{Hom}_{\FF}(K\exoo{E}, K)$ that induces a functor $\mathrm{Hol}: \mathcal G(\Gamma)\to \mathrm{Hom}_{\FF}(K\exoo{E}, K)$, where $\mathcal C(\Gamma)$ and $\mathcal G(\Gamma)$
are the path category and path groupoid of $\Gamma$  from Definition \ref{def:free_cat}.

Giving  $\mathrm{Hom}_{\FF}(K\exoo{E}, K)$ the structure of an $\FF$-linear category with a single object  amounts to choosing an associative, unital  algebra structure   on $\mathrm{Hom}_{\FF}(K\exoo{E}, K)$, where morphisms  are the elements of $\mathrm{Hom}_{\FF}(K\exoo{E}, K)$, the identity morphism is the unit  and the composition of morphisms is given by the multiplication.
In  a Hopf algebra gauge theory, it is  natural to construct the
multiplication of $\mathrm{Hom}_{\FF}(K\exoo{E}, K)$ from two ingredients, namely  an associative multiplication map $m: K\oo K\to K$ that allows one to compose the 
contributions of different edges $e,f\in E$ in a path and  a coassociative comultiplication $\Delta_\oo: K\exoo{E}\to K\exoo{E}\oo K\exoo{E}$ that allows one to distribute the variable $(k)_e$ 
for an edge $e\in E$ over the different occurrences of $e$ in a path.  The resulting algebra structure on $\mathrm{Hom}_{\FF}(K\exoo{E}, K)$ is the well-known convolution product, which is defined for any algebra $A$ and coalgebra $C$. Its associativity and unitality are immediate.

\begin{lemma}  \label{lem:catlem}Let $(A, m, 1)$ be an algebra and  $(C,\Delta,\epsilon)$ a coalgebra. Then  the convolution product
 \begin{align}\label{eq:convol}
 \phi\bullet \psi=m\circ (\phi\oo\psi)\circ \Delta 
 \end{align}
 gives $\mathrm{Hom}_{\FF}(C,A)$ the structure of an associative algebra with unit $\eta=\epsilon\, 1$.
\end{lemma}

In the case at hand we have $A=K$ and $C=K\exoo{E}$. We choose  for
 $\Delta$ and $\epsilon$ in Lemma \ref{lem:catlem} the opposite comultiplication
$\Delta^{op}_\oo: K\exoo{E}\to K\exoo{E}\oo K\exoo{E}$ and the counit $\epsilon\exoo{E}$ of the Hopf algebra $K\exoo{E}$ and for $m: K\oo K\to K$ the multiplication of $K$.
Let now be $\mathrm{Hom}_{\FF}(K\exoo{E}, K)$  equipped with the algebra structure from Lemma \ref{lem:catlem} for these choices
and viewed as a category with a single object.
Then a functor $\mathrm{Hol}: \mathcal G(\Gamma)\to \mathrm{Hom}_{\FF}(K\exoo{E}, K)$ satisfies $F(\emptyset_v)=\eta$ for all $v\in V$ and
 is determined  uniquely by the maps $\mathrm{Hol}({e^{\pm 1}}): K\exoo{E}\to K$ for $e\in E$, since $\mathcal G(\Gamma)$ is the free groupoid generated by $E(\Gamma)$. 
The maps $\mathrm{Hol}({e^{\pm 1}}): K\exoo{E}\to K$
should be local and compatible with the  reversal of the edge orientation  via the involution $T: K\to K$, $k\mapsto gS(k)$ from \eqref{eq:invol}. Hence we impose
\begin{align}\label{eq:holdef}  
\mathrm{Hol}(e)(k^1\oo...\oo k^{E})=\Pi_{f\in E\setminus\{e\} }\epsilon(k^f)\; k^e
, \qquad \mathrm{Hol}(e^\inv)=T\circ \mathrm{Hol}(e).
\end{align} 
The linear maps $\mathrm{Hol}(e^{\pm 1}): K\exoo{E}\to K$ then induce a functor $\mathrm{Hol}:\mathcal G(\Gamma)\to \mathrm{Hom}_{\FF}(K\exoo{E}, K)$ if and only if they respect the defining relations of the path groupoid: for all $e\in E$
 \begin{align}\label{eq:grpoidcond}
m\circ (\id\oo T)\circ (\mathrm{Hol}(e)\oo \mathrm{Hol}(e))\circ \Delta_\oo^{op}=m\circ (T\oo\id)\circ (\mathrm{Hol}(e)\circ \mathrm{Hol}(e))\circ\Delta_\oo^{op}=\epsilon\exoo{E}\, 1.
 \end{align}
In  a Hopf algebra gauge theory, the only canonical choice for  the map $m: K\oo K\to K$ in Lemma \ref{lem:catlem} is the multiplication of the algebra $K$. However, there are several candidates for the map $\Delta: K\exoo{E}\to K\exoo{E}\oo K\exoo{E}$ in Lemma \ref{lem:catlem}, namely  the comultiplication of the Hopf algebra $K\exoo{E}$ and the comultiplication dual to the algebra structure on ${K^*}\exoo{E}$. However,   the choice is restricted 
by  \eqref{eq:grpoidcond}. If one chooses the former,  one obtains 
\begin{align}\label{eq:firstexp}
&\mathrm{Hol}(e\circ e^\inv)\circ\iota_e(k)=\low k 2 S^\inv(\low k 1)g=\epsilon(k)\, g, &
&\mathrm{Hol}(e^\inv\circ e)\circ\iota_e(k)=g S(\low k 2) \low k 1,\end{align}
and if one chooses for $\Delta$   the  comultiplication dual to the opposite multiplication  on ${K^*}\exoo{E}$
\begin{align*}
\mathrm{Hol}(e\circ e^\inv)\circ\iota_e(k)&=(\id\oo g\cdot S)( R_{21}(\low k 2\oo\low k 1))=\low R 2 \low k 2 g S(\low k 1)S(\low R 1)
=\epsilon(k)ug=\epsilon(k)\nu\\
 \mathrm{Hol}(e^\inv\circ e)\circ\iota_e(k)&=( g\cdot S\oo\id )( R_{21}(\low k 2\oo\low k 1))=( g\cdot S\oo\id )( (\low k 1\oo\low k 2)R_{21})
\epsilon(k)gu=\epsilon(k)\nu,
\end{align*}
 where $\nu\in K$ is the ribbon element. If $K$ is semisimple, then the  expressions in \eqref{eq:firstexp} reduce to $g S(\low k 2)\low k 1=\epsilon(k)\, g=\epsilon(k)\, 1$. Hence if one takes for  $\Delta$ the comultiplication of the Hopf algebra $K\exoo{E}$, one obtains 
  a functor $\mathrm{Hol}:\mathcal G(\Gamma)\to \mathrm{Hom}_{\FF}(K\exoo{E}, K)$, while the choice of comultiplication dual to the multiplication of $\mathcal A^*$ yields  $\mathrm{Hol}(e\circ e^\inv)\circ\iota_e(k)=\mathrm{Hol}(e^\inv\circ e)\circ\iota_e(k)=\epsilon(k)\, u$. 
  
If $K$ is not semisimple, none of these relations is compatible with \eqref{eq:grpoidcond},
 unless one postulates additional structure. This is the approach chosen in  \cite{AGSII}, where the authors take for $\Delta^{op}_\oo$ the comultiplication dual to $\mathcal A^*$. To satisfy the defining relations of the path groupoid, they  postulate  a square root for the ribbon element in each irreducible representation  of $K$ and rescale the functors accordingly; see (2.7) and (3.1) in \cite{AGSII}.  
However,  as the resulting notion of holonomy is simpler and has nice geometrical properties, we  restrict attention to {\em semisimple} Hopf algebras  $K$  and define 
the functor $\mathrm{Hol}:\mathcal G(\Gamma)\to\mathrm{Hom}_{\FF}(K\exoo{E},K)$ by \eqref{eq:holdef} and \eqref{eq:convol} with the comultiplication of $K\exoo{E}$ for $\Delta$.

Note  that  a functor $\mathrm{Hol}:\mathcal G(\Gamma)\to\mathrm{Hom}_{\FF}(K\exoo{E}, K)$ induces  a contravariant functor
 $\mathrm{Hol}^*:\mathcal G(\Gamma)\to\text{Hom}_{\FF}(K^*, {K^*}\exoo{E})$, where  $\text{Hom}_{\FF}(K^*, {K^*}\exoo{E})$ is 
 viewed as a category with a single object and equipped with the algebra structure dual to the one in Lemma \ref{lem:catlem}. This algebra structure is given by  
 $(\phi\bullet \psi)^*=\psi^*\bullet \phi^*=m_*\circ (\psi^*\oo\phi^*)\circ\Delta_*^{op}$,
  where $m_*: {K^*}\exoo{E}\oo {K^*}\exoo{E} \to {K^*}\exoo{E}$ is the multiplication of the Hopf algebra ${K^*}\exoo{E}$ and $\Delta_*: K^*\to K^*\oo K^*$ the comultiplication of $K^*$. The contravariant functor $\mathrm{Hol}^*$ is then given  by
 $ \langle \mathrm{Hol}^*(p)(\alpha), k\rangle=\langle \alpha, \mathrm{Hol}(p)(k)\rangle$ for all $\alpha\in K^*$, $k\in K\exoo{E}$. 
Combining these results, we obtain the following definition of holonomy.

\begin{definition}  \label{def:hols} Let $K$ be a finite-dimensional semisimple quasitriangular Hopf algebra and $\Gamma$ a ribbon graph. Equip 
$\mathrm{Hom}_{\FF}(K\exoo{E}, K)$ with the  algebra structure \eqref{eq:convol} where $m$  
is the multiplication of $K$ and $\Delta_\oo$ 
the comultiplication of  $K\exoo{E}$. 
Then the  {\bf holonomy functors} of a $K$-valued Hopf algebra gauge theory on $\Gamma$ are the functor $\mathrm{Hol}:\mathcal G(\Gamma)\to \mathrm{Hom}_{\FF}(K\exoo{E}, K)$ defined by  \eqref{eq:holdef} and the associated contravariant functor  $\mathrm{Hol}^*:\mathcal G(\Gamma)\to\text{Hom}_{\FF}(K^*, {K^*}\exoo{E})$. 
For a path $p\in\mathcal G(\Gamma)$  we use the notation $\hol_p:=\mathrm{Hol}(p)$ and $\hol_p^*:=\mathrm{Hol}^*(p)$.
\end{definition}

In particular, we can consider the holonomies around faces of $\Gamma$. In a group-valued lattice gauge theory, the holonomies around faces are interpreted as (group-valued) curvatures of  the underlying connection  and a connection is called flat at a given face if the holonomy along the face is the unit element. This  has a direct analogue in a Hopf algebra gauge theory.

\begin{definition}\label{def:curvature} The holonomy map $\hol_f: K\exoo{E}\to K$  of a face path  $f\in \mathcal G(\Gamma)$ is called the {\bf curvature} of $f$. A connection $k\in K\exoo{E}$ is called {\bf flat at $f$} if $\hol_f(k)=\epsilon\exoo{E}(k)\, 1$ 
and {\bf flat} if it is flat for all  face paths  $f\in \mathcal G(\Gamma)$. We denote by $\mathcal A_f\subset K\exoo{E}$ the linear subspace of connections that are flat at $f$ and by   $\mathcal A_{flat}$ the linear subspace of flat connections.
\end{definition}

It is worth commenting on the notion of holonomy  used in \cite{AGSII} and \cite{BR}. 
As explained above,  the authors in  \cite{AGSII} to choose for $\Delta^{op}_\oo$ in \eqref{eq:convol} the  comultiplication dual to $\mathcal A^*$, postulate the existence of square roots of the ribbon element in each irreducible representation of $K$ and rescale the functor $\mathrm{Hol}:\mathcal C(\Gamma)\to \mathrm{Hom}_{\FF}(K\exoo{E}, K)$ to satisfy the defining relations of the path groupoid. 
In \cite{BR}, a different notion of holonomy is used. It is defined in terms of the  comultiplication dual to $\mathcal A^*$ but the resulting expression for the holonomy is modified by $R$-matrices inserted at each vertex, depending on the relative ordering of the edge ends. The relevant expressions  in Definition 5,  formulas (40) to (46),  (60) and (61) in \cite{BR} indicate that the resulting holonomies coincide with the ones from Definition \ref{def:hols} if the choices of conventions are taken into account. However,  these holonomies are not defined as  functors and the $R$-matrices in these formulas are introduced by hand.

Note also that choosing the comultiplication dual to  $\mathcal A^*$ for  $\Delta^{op}_\oo$ in Lemma \ref{lem:catlem} yields holonomies that take a different form depending on whether a path turns left or right at a vertex. This follows from formulas  \eqref{eq:flipalg} for the multiplication of $\mathcal A^*_v$.
However, for face paths that are compatible with the ciliation (see Definition \ref{def:cilpath}), the two notions of holonomy coincide.

\begin{lemma}\label{lem:specialpath} Let $K$ be a finite-dimensional ribbon Hopf algebra. Let  $f$ be a face path of $\Gamma$ that is compatible with the ciliation. Then the holonomies of $f$ from Definition \ref{def:hols} agree with the holonomies  obtained by taking for $\Delta_\oo^{op}:K\exoo{E}\to K\exoo{E}\oo K\exoo{E}$ the comultiplication dual to $\mathcal A^*$. 
\end{lemma}
\begin{proof} This is most easily seen for the functors $\mathrm{Hol}^*:\mathcal G(\Gamma)\to\mathrm{Hom}_{\FF}(K^*,{K^*}\exoo{E})$ by applying the formulas for the multiplication in $\mathcal A^*$.
Suppose  that  $f=e_n^{\epsilon_n}\circ ...\circ e_1^{\epsilon_1}$  satisfies the assumptions. Then we have
$s(e_{i+1}^{\epsilon_{i+1}})<t(e_i^{\epsilon_i})$ for all $i\in\{1,...,n\}$ and $s(e_1^{\epsilon_1})<t(e_n^{\epsilon_n})$.  By Corollary \ref{rem:flipalg}  for each path
$s(e_{i+1}^{\epsilon_{i+1}})\circ t(e_i^{\epsilon_i})$ in $\Gamma_v$,  the holonomies from Definition \ref{def:hols} agree with the ones obtained by taking for $\Delta_\oo^{op}:K\exoo{E}\to K\exoo{E}\oo K\exoo{E}$ the comultiplication dual to $\mathcal A^*$. Moreover, these holonomies commute with each other
and with the holonomies of $s(e_1^{\epsilon_1})$ and $t(e_n^{\epsilon_n})$ with respect to the multiplication of $\oo_{v\in V}\mathcal A^*_v$. This implies
that the two notions of holonomy coincide.
\end{proof}

Although the two notions of holonomy coincide for faces that are compatible with the ciliation, they 
 differ substantially for general paths.
In the following, we will not consider the notion of holonomy  obtained from the comultiplication dual to $\mathcal A^*$ further
 since the one in Definition \ref{def:hols} is more natural in the semisimple case,  and does not require any modifications or rescalings  to define a functor.

\subsection{Algebraic properties of the holonomies}
\label{subsec:algprop}

In this section, we investigate the algebraic properties of the holonomy functor, focussing on the contravariant functor $\mathrm{Hol}^*: \mathcal C(\Gamma)\to\mathrm{Hom}_{\FF}(K^*,{K^*}\exoo{E})$ and the associated linear maps $\hol_p^*=\mathrm{Hol}^*(p):K^*\to \mathcal A^*_\Gamma$ for paths $p\in \mathcal G(\Gamma)$.  As semisimplicity is required to obtain holonomies that satisfy the defining relations of the path  groupoid, we restrict attention to finite-dimensional semisimple quasitriangular Hopf algebras $K$ in the following. Recall that this implies that $K$ is ribbon with ribbon element $\nu=u$ and grouplike element $g=1$.
We start by analysing the behaviour of the holonomies of face paths that are compatible with the ciliation (cf. Definition \ref{def:cilpath}) with respect to  the graph operations.

\begin{theorem}\label{th:holtheorem} Let $\Gamma'$ be obtained from $\Gamma$ by one of the graph operations in Definition \ref{def:graphtrafos} and  denote by
 $F:\mathcal C(\Gamma')\to\mathcal C(\Gamma)$,  $F_\circ:\mathcal C(\Gamma'_\circ)\to\mathcal C(\Gamma_\circ)$, $G_\Gamma:\mathcal C(\Gamma)\to \mathcal C(\Gamma_\circ)$ and $G_{\Gamma'}:\mathcal C(\Gamma')\to \mathcal C(\Gamma'_\circ)$  the associated functors from Definition \ref{def:graph_functor} and   Proposition  \ref{lem:graphtrafo_vertnb}. Then the algebra  homomorphism  $f^*:\bigotimes_{v\in V'}\mathcal A'^*_v\to \bigotimes_{v\in V}\mathcal A^*_v$ from Definition \ref{def:edgeopmaps} and Corollary \ref{cor:inducedmaps},
  the algebra homomorphism $F^*: K^{*\oo E'}\to K^{*\oo E}$ from Theorem \ref{th:alghomfinal} and the subdivision maps $G^*_\Gamma: K^{*\oo E}\to \bigotimes_{v\in V} \mathcal A^*_v$ and $G^*_{\Gamma'}: K^{*\oo E'}\to \bigotimes_{v\in V'} \mathcal A'^*_v$  from \eqref{eq:dualemb} relate the holonomies of $\Gamma$, $\Gamma'$, $\Gamma_\circ$ and $\Gamma'_\circ$ as follows
\begin{align}\label{eq:pathidentity}
&F^*\circ \mathrm{Hol}^*_{p'}=\mathrm{Hol}^*_{F(p')} & 
&f^*\circ \mathrm{Hol}^*_{q'}=\mathrm{Hol}^*_{F_\circ(q')}\\
&G^*_\Gamma\circ \mathrm{Hol}^*_{p}=\mathrm{Hol}^*_{G_\Gamma(p)} & 
&G^*_{\Gamma'}\circ \mathrm{Hol}^*_{p'}=\mathrm{Hol}^*_{G_{\Gamma'}(p)}\nonumber
\end{align}
for all  face paths $p\in \mathcal C(\Gamma)$, $p'\in \mathcal C(\Gamma')$ and $q\in \mathcal C(\Gamma'_\circ)$  that are compatible with the ciliation. The same holds if the path categories are replaced by  path groupoids, and we obtain the commuting diagram 
\begin{align}\label{eq:diagg}
\xymatrix{ \mathcal A^*_{\Gamma'} \ar[dd]_{G^*_{\Gamma'}}\ar[rr]^{F^*} & & \mathcal A^*_\Gamma \ar[dd]^{G^*_\Gamma}\\
 & \ar[ld]_{\qquad\hol^*_{G'(p')}} K^* \ar[lu]_{\hol^*_{p'}} \ar[ru]^{\hol^*_{F(p')}} \ar[rd]^{\!\!\!\!\hol^*_{G(F(p'))}}&\\
  \oo_{v\in V'}\mathcal A'^*_v \ar[rr]_{f^*}&  & \oo_{v\in V}\mathcal A^*_v.
}
\end{align}
\end{theorem}

\begin{proof}  To prove the identities in  \eqref{eq:pathidentity}, recall  from Lemma \ref{lem:specialpath} that the holonomy functors agree with those defined by the multiplication of $\mathcal A^*_\Gamma$,  $\mathcal A^*_{\Gamma'}$,  $\oo_{v\in V}\mathcal A^*_v$ and $ \oo_{v\in V'}\mathcal A'^*_v$  for all  face paths compatible with the ciliation. We may thus assume that the holonomy functors are the ones defined by the algebra structures of $\mathcal A^*_\Gamma$,  $\mathcal A^*_{\Gamma'}$,  $\oo_{v\in V}\mathcal A^*_v$ and $ \oo_{v\in V'}\mathcal A'^*_v$. Recall also that the maps $F^*$, $f^*$, $G^*_\Gamma$ and $G^*_{\Gamma'}$ are algebra homomorphisms with respect to these algebra structures.

Suppose that $C$ is a coalgebra, $B_i$ are algebras,  $\mathcal D_i$ categories, and $\hol_i: \mathcal D_i\to \mathrm{Hom}_\FF(C,B_i)$ for $i=1,2$ are functors, where $\hol^i: D_i\to \mathrm{Hom}_\FF(C,B_i)$ is equipped with the convolution product  $\bullet_i$ defined by the comultiplication of $C$ and the multiplication of $B_i$. Then for all functors $F: \mathcal D_2\to \mathcal D_1$, algebra morphisms $f: B_2\to B_1$ and composable morphisms $p,q$ in $\mathcal D_2$ with $\hol^1_{F(p)}=f\circ \hol^2_p$ and $\hol^1_{F(q)}=f\circ \hol^2_q$ one has
$$
\hol^1_{F(p\circ q)}=\hol^1_{F(p)}\bullet_1 \hol^1_{F(q)}=(f\circ \hol^2_{p})\bullet_1 (f\circ \hol^2_q)=f\circ (\hol^2_p\bullet_2\hol^2_q)=f\circ \hol^2_{p\circ q}.
$$ 
By applying this statement to the categories and algebra homomorphisms under consideration, one sees that the identities in  \eqref{eq:pathidentity} follow by induction over the lengths of the paths. It is thus sufficient to prove them for paths involving only a single edge.

For edges $e\in E(\Gamma)$ and $e'\in E(\Gamma')$  the associated morphisms in $\mathcal C(\Gamma_\circ)$ and $\mathcal C(\Gamma'_\circ)$ are given by $t(e)\circ s(e)$ and $t(e')\circ s(e')$, respectively. The identities  in the second line of  \eqref{eq:pathidentity} then follow directly from \eqref{eq:dualemb}, which implies
\begin{align*}
&G^*_{\Gamma}((\alpha)_{e})=(\low\alpha 2\oo\low\alpha 1)_{s(e)t(e)}=\hol^*_{t(e)\circ s(e)}(\alpha)=\hol^*_{G_\Gamma(e)}(\alpha) & &\forall e\in E.
\end{align*}
and similarly for $e'\in E'$. 

 The second identity in the first line of \eqref{eq:pathidentity} for a single  edge end $g'\in E(\Gamma'_\circ)$ follows  directly from the expressions for the maps $f_*$  in \eqref{eq:deleting} to \eqref{eq:doubling} with $g=1$ and the expressions for the functors $F_\circ$ in Proposition \ref{lem:graphtrafo_vertnb}.
By Proposition \ref{lem:graphtrafo_vertnb} we have $F_\circ\ G_{\Gamma'}=G_{\Gamma} F$, and by Remark \ref{rem:subdivgrtrafo} we have 
$G^*_\Gamma\circ F^*=f^*\circ G^*_{\Gamma'}$. Together with the last three identities in \eqref{eq:pathidentity} this implies
\begin{align*}
G^*_\Gamma\circ F^*\circ \mathrm{Hol}^*_{p'}=f^*\circ G^*_{\Gamma'}\circ \mathrm{Hol}^*_{p'}=f^*\circ \mathrm{Hol}^*_{G_{\Gamma'}(p)}=\mathrm{Hol}^*_{F_\circ G_{\Gamma'}(p')}=\mathrm{Hol}^*_{G_\Gamma F(p')}=G^*_{\Gamma}\circ \mathrm{Hol}_{F(p')}
\end{align*}
for all edges $p'\in\mathcal C(\Gamma')$. As $G^*_\Gamma$ is injective, the first identity in \eqref{eq:pathidentity} follows.  

The identities in \eqref{eq:pathidentity} induce analogous identities on the associated path groupoids, and by  combining the four identities in \eqref{eq:pathidentity}, we obtain the commuting diagram \eqref{eq:diagg}.
\end{proof}

We will now consider the transformation of the holonomies  face paths compatible with the ciliation under gauge transformations and show that they  form a subalgebra of $\mathcal A^*$.

\begin{theorem} \label{lem:holtransform} Let $K$ be a finite-dimensional semisimple Hopf algebra and $p$ a  face path compatible with the ciliation.
Then the holonomy of $p$ satisfies 
\begin{align}\label{eq:holtrafo}
\hol^*_p(\beta)\cdot \hol^*_p(\alpha)=
\langle\low \alpha 3\oo  \low\beta 1, R^\inv \rangle\langle \low\alpha 1\oo\low\beta 2, R\rangle\; \hol_p^*(\low\alpha 2\low\beta 3) 
\end{align}
for all $\alpha,\beta\in K^*$. For any 
 gauge transformation $h=h^1\oo\ldots\oo h^{V}\in K\exoo{V}$ and $\alpha\in K^*$, one has 
\begin{align}\label{eq:holeq}
&\hol^*_p(\alpha\lhd^* h)=\Pi_{i\neq  u,w}\epsilon(h^i)
\langle S(\low\alpha 3)\low \alpha 1, h^u\rangle\, \hol^*_p(\low\alpha 2). 
\end{align} 
\end{theorem}
\begin{proof} 
We apply the graph operations from Definition \ref{def:graphtrafos} to simplify the path $p$. 
Let $\Gamma'$ be obtained from $\Gamma$ by a one of the graph operations in Definition \ref{def:graphtrafos},  denote by $F:\mathcal G(\Gamma')\to\mathcal G(\Gamma)$  the associated functor  from Definition \ref{def:graph_functor} 
and by  $F^*:\mathcal A^*_{\Gamma'}\to\mathcal A^*_{\Gamma}$ the associated algebra homomorphism from Theorem \ref{th:alghomfinal}.
Then by Theorem \ref{th:holtheorem}  one has $\hol^*_{F(p')}=F^*\circ \hol^*_{p'}$  for all  face paths $p'\in \mathcal G(\Gamma')$ compatible with the ciliation.  As $F^*:\mathcal A^*_{\Gamma'}\to\mathcal A^*_{\Gamma}$  is an algebra homomorphism and a module homomorphism with respect to the $K\exoo{V}$-module structure of $\mathcal A^*_\Gamma$ and $\mathcal A^*_{\Gamma'}$, this implies
\begin{align*}
&\hol^*_{F(p')}(\alpha)\cdot \hol^*_{F(q')}(\beta)=F^*(\hol^*_{p'}(\alpha)\cdot \hol^*_{q'}(\beta)) & &\forall \alpha,\beta\in K^*\\
&\hol^*_{F(p')}(\alpha)\lhd^* h=F^*(\hol^*_{p'}(\alpha))\lhd^* h=F^*(\hol^*_{p'}(\alpha)\lhd'^* h) & &\forall h\in K\exoo{V}.
\end{align*}
 for face paths $p', q'$ compatible with the ciliation.
 It is therefore sufficient to show that by applying a finite sequence graph operations from Definition \ref{def:graphtrafos}, one can transform $\Gamma$ into a ciliated ribbon graph $\Gamma'$ with a single edge $e'$ such that  $p=F(e')$, where $F: \mathcal G(\Gamma')\to \mathcal G(\Gamma)$ is the functor associated with the finite sequence  of graph operations.
 As   $p$  is regular and does not traverse any cilia, we can construct such a ciliated ribbon graph $\Gamma'$  by (i) deleting all edges that do not occur in $p$, (ii) doubling the edges that are traversed twice,  (iii) detaching adjacent edge ends from vertices in $p$ and (iv) contracting  the resulting ribbon graph  to a single edge. 
This yields a ribbon graph $\Gamma'$ with $E(\Gamma')=\{e'\}$ and $F(e')=p$. The claim then follows 
 from Proposition \ref{lem:algexplicit} (a), (b) and Proposition \ref{lem:gtrafolem}. 
\end{proof}

Theorem \ref{lem:holtransform}  shows that the holonomies of  a  face path compatible with the ciliation 
form a subalgebra of $\mathcal A^*$  and transform under gauge transformations according to formula \eqref{eq:holeq}. Hence, they behave in the same way as the holonomies of a single  loop.  The holonomy map $\hol^*_p$  is a module homomorphism with respect to the action of gauge transformations at $v$ and the  left coadjoint action of $K^{op}$ on $K^*$ from 
Example \ref{ex:regacts}: 
\begin{align*}
\hol^*_p(\alpha)\lhd^*(h)_v=\hol^*_p( S(h)\rhd^{*}_{ad}\alpha).
\end{align*}
A  direct computation shows that for semisimple $K$ one has $S\circ \lhd^*_{ad}\circ (S\oo\id)=\rhd^{*}_{ad}\circ(S\oo\id)$
and the submodules of invariants of the two module structures coincide. In both cases, the submodule of invariants is the character algebra 
$C(K)=\{\alpha\in K^*:\Delta(\alpha)=\Delta^{op}(\alpha)\}$
from Example \ref{ex:chars}.
This relates the projection of the holonomies on the gauge invariant subalgebra $\mathcal A^*_{inv}$ to the character algebra $C(K)$. In particular, it shows that it is invariant under cyclic permutations of the  path.

\begin{lemma} \label{lem:reghols} Let $K$ be finite-dimensional  semisimple quasitriangular  and $p$ a  face path in $\Gamma$ that is compatible with the ciliation. Denote by 
  $\Pi:\mathcal A^*\to\mathcal A^*$ the projector on $\mathcal A^*_{inv}$ and by $\pi: K^*\to K^*$ the projector on $C(K)$. Then:
\begin{compactenum}
\item  $\Pi\circ\hol^*_p=\hol^*_p\circ \pi$ is invariant under cyclic permutations of $p$.
\item 
 $\hol^*_p: C(K)\to \mathcal A^*_{inv}$ is an algebra anti-homomorphism.
\end{compactenum}
\end{lemma}

\begin{proof} By Proposition \ref{lem:invspace} the map  $\Pi\circ \hol^*_p:K^*\to \mathcal A^*_{inv}$ is independent of the choice of the cilia, and one can assume without loss of generality that $p$ and its  cyclic permutations  satisfy the assumptions of Theorem \ref{lem:holtransform}. That   $\Pi\circ\hol^*_{p}=\hol^*_p\circ\pi$ is invariant under cyclic permutations of $p$ follows from Definition \ref{def:hols} and the fact that
 $\Delta^{(n)}(\alpha)$ is invariant under cyclic permutations for any  $\alpha\in C(K)$ and $n\in\NN$.
 From this and equation \eqref{eq:holeq}, we obtain for  $\alpha,\beta\in C(K)$
\begin{align}\label{eq:reghols}
\hol^*_p(\beta)\cdot \hol^*_p(\alpha)&=\langle \low\alpha 3\oo\low\beta 1, R^\inv\rangle\langle \low\alpha 1\oo\low\beta 2,R\rangle\, \hol^*_p(\low\alpha 2\low\beta 3)\\
&=\langle \low\alpha 1\oo\low\beta 1, R^\inv\rangle\langle \low\alpha 2\oo\low\beta 2,R\rangle\, \hol^*_p(\low\alpha 3\low\beta 3)=\hol^*(\alpha\beta).\nonumber
\end{align}
\end{proof}

Another important implication of 
Theorem \ref{lem:holtransform} is that  the linear subspace of connections that are flat at a given face  is invariant under gauge transformations.
 By dualising  equations \eqref{eq:holtrafo}, one obtains for any  face path $p$ based at $v$  that  is compatible with the ciliation
$\hol_p((h)_v\rhd k)=\epsilon(h) \hol_p(k)$  whenever $\hol_p(k)$ is central in $K$. In particular, this holds for connections that are flat at $f$.

\begin{corollary}\label{cor:gtrafoflat} Let $K$ be finite-dimensional semisimple and quasitriangular and $f\in \mathcal G(\Gamma)$ a  face path compatible with the ciliation. 
Then the linear subspace $\mathcal A_f$ of connections that are flat at $f$ is  invariant under  gauge transformations: $\mathcal G\rhd \mathcal A_f\subset \mathcal A_f$.
\end{corollary}

\subsection{Curvature and flatness}
\label{subsec:curve}

In this section, we  focus on the curvatures of a Hopf algebra gauge theory for a finite-dimensional semisimple quasitriangular Hopf algebra $K$. 
We will show that for a ribbon graph $\Gamma$  that satisfies certain mild regularity conditions,
 the projection of a curvature on the gauge invariant subalgebra $\mathcal A^*_{inv}$ is central in $\mathcal A^*_{inv}$. We then construct a subalgebra $\mathcal A^{*flat}_{inv}$ which can be viewed as the Hopf algebra analogue of the  algebra of gauge invariant functions on the set of  flat connections.   This requires  some  technical results that describe the commutation relations of the holonomies with elements in $\mathcal A^*$.
The first step is to notice that the commutation relations of the holonomies  $\hol^*_p(\alpha)$ and $\hol^*_q(\beta)$ for paths $p,q\in\mathcal G(\Gamma)$ take a particularly 
simple form if $p$ and $q$ have no vertices in common or intersect only in their endpoints.

\begin{lemma}\label{lem:hols} Let $K$ be a finite-dimensional, semisimple quasitriangular Hopf algebra and $\Gamma$ a ciliated ribbon graph. Then the holonomy functions of the local Hopf algebra gauge theory satisfy: 
\begin{compactenum}
\item $\hol^*_p(\alpha)\cdot \hol^*_q(\beta)=\hol^*_q(\beta)\cdot \hol^*_p(\alpha)$ for paths $p,q\in \mathcal G(\Gamma)$  with no common vertex.\\[-1ex]
\item  If $p,q\in \mathcal G(\Gamma)$  intersect only in  their endpoints, the multiplication relations of the holonomies $\hol_p^*$ and $\hol_q^*$ are  given  by relations (c) to (l) in Proposition \ref{lem:algexplicit}.
\end{compactenum}
\end{lemma}
\begin{proof}
  The  identities in 1.~follow from the fact that for a path $p=e_n^{\epsilon_n}\circ...\circ e_1^{\epsilon_1}$ one has $\hol^*_p(K^*)\subset \iota_{e_1...e_n}(K^{*\oo n})$  together with  the identities $(\alpha)_e\cdot (\beta)_f=(\beta)_f\cdot(\alpha)_e$ for all edges $e,f$ without a common vertex. 
  The claim 2.~follows by
 induction over the length of the path. It holds by definition for all paths $p=e^{\pm}$ and $q=f^{\pm 1}$ with $e,f\in E$. Suppose it holds for all paths of length $\leq n$ and let $p,q$ be paths of length $\leq n+1$. Then we can decompose
$p=p_1\circ p_2$ and $q=q_1\circ q_2$ with paths $p_i$, $q_i$ of length $\leq n$. 
Suppose at first that $\ta(p)\notin\{ \st(p), \st(q)\}$, $ \st(q)\not\in\{\ta(q), \st(p)\}$ and $t(p)<t(q)$. Then only $p_1$ and $q_1$ have a vertices in common, namely their target vertices. Hence $\hol^*_{p_2}(K^*)$ commutes with $\hol^*_{q_1}(K^*)$, $\hol^*_{q_2}(K^*)$ and $\hol^*_{q_2}(K^*)$ commutes with $\hol^*_{p_1}(K^*)$ in  $\mathcal A^*$  by 1. As  $t(p_1)<t(q_1)$, the induction hypothesis implies  $\hol^*_{q_1}(\beta)\cdot \hol^*_{p_1}(\alpha)
=\langle \low\alpha 1\oo\low\beta 1, R\rangle\,  \hol^*_{p_1}(\low\alpha 2) \cdot \hol^*_{q_1}(\low\beta 2)$ and 
\begin{align*}
& \hol^*_q({\beta})\cdot \hol^*_p(\alpha)= m_*(\hol^*_{q_2}(\low\beta 2)\oo\hol^*_{q_1}(\low\beta 1))\cdot m_* (\hol^*_{p_2}(\low\alpha 2)\oo\hol^*_{p_1}(\low\alpha 1))\\
&=\langle \low\alpha 1\oo\low\beta 1, R\rangle\,  m_*(\hol^*_{p_2}(\low\alpha 3)\oo\hol^*_{p_1}(\low\alpha 2))\cdot m_* (\hol^*_{q_2}(\low\beta 3)\oo\hol^*_{q_1}(\low\beta 2))\\
&=\langle \low\alpha 1\oo\low\beta 1, R\rangle\, \hol^*_p(\low\alpha 2)\cdot \hol^*_q(\low\beta 2),
\end{align*}
where $m_*: {K^*}\exoo{E}\oo {K^*}\exoo{E}\to {K^*}\exoo{E}$ is the multiplication of the algebra ${K^*}\exoo{E}$.
The  claims for paths $p$, $q$  that share starting vertices, that  are loops or have  two common  endpoints  follow by analogous computations. 
\end{proof}

We can now use Lemma \ref{lem:hols} to explicitly determine the commutation relations of the holonomy variables of a  face path $f$ with elements of the algebra $\mathcal A^*$.  In  this, we restrict attention to face paths that are compatible  with the ciliation.

\begin{lemma}\label{lem:holcomp}  Let  $K$ be finite-dimensional semisimple quasitriangular Hopf algebra and $\Gamma$ a ciliated a ribbon graph without univalent vertices. If   $f$ is  a face path of $\Gamma$ that  is compatible with the ciliation, then $\hol^*_f(\alpha)\cdot (\beta)_e=(\beta)_e\cdot \hol^*_f(\beta)$   for all edges $e$ with $\st(f)\notin \{\st(e), \ta(e)\}$ and $\alpha,\beta\in K^*$.
For  edges $e$ with $\st(f)\in \{\st(e), \ta(e)\}$ 
 the commutation relations between the variables $\hol^*_f(\alpha)$ and $(\beta)_e$ are given by the expressions in Proposition \ref{lem:algexplicit} (b) and (e)-(j).
\end{lemma}
\begin{proof}
The idea of the proof is to apply the graph operations from Definition \ref{def:graphtrafos} to simplify the face path $f$. 
Let $\Gamma'$ be obtained from $\Gamma$ by one of the graph operations in Definition \ref{def:graphtrafos},  denote by $F:\mathcal G(\Gamma')\to\mathcal G(\Gamma)$  the associated functor  from Definition \ref{def:graph_functor} 
and by  $F^*:\mathcal A^*_{\Gamma'}\to\mathcal A^*_{\Gamma}$ the associated algebra homomorphism from Theorem \ref{th:alghomfinal}.
Then by  Theorem \ref{th:holtheorem} one has  $\hol^*_{F(f)}=F^*\circ \hol^*_{f}$.  As $F^*$  is an algebra homomorphism, this implies
$$
\hol^*_{F(f)}(\alpha)\cdot \hol^*_{F(e)}(\beta)=F^*(\hol^*_{f}(\alpha)\cdot \hol^*_{f}(\beta))\qquad\forall \alpha,\beta\in K^*
$$
for all edges $e$ in $\Gamma$. 
 It is therefore sufficient to show that by applying graph operations from Definition \ref{def:graphtrafos} one can transform $\Gamma$ into a ribbon graph $\Gamma'$ with face path $f'\in \mathcal G(\Gamma')$ and an edge $e'\in E(\Gamma')$  satisfying $f=F(f')$ and $e=F(e')$ such that  $\hol^*_{f'}(\alpha)$ and $(\beta)_{e'}$ satisfy the  commutation relations in the lemma. 
To construct such a graph $\Gamma'$, a face path $f'\in\mathcal G(\Gamma')$ and an edge $e'\in E(\Gamma')$, note 
that each
 edge 
 $e\in E(\Gamma)$ satisfies exactly one of the
following
\begin{compactenum}[(i)]
\item $e$ and $f$ have no vertex in common.
\item  $e$ is not contained in $f$,  shares at least one  vertex with $f$, but  not the  vertex  $\st(f)=\ta(f)$.
\item  $e$ is not contained in $f$ and shares the  vertex $\st(f)=\ta(f)$ with $f$.
\item $e$ is contained in $f$ but does not coincide with the first or last edge  of $f$.
\item $e$ is the first or last edge in $f$.
 \end{compactenum}
In case (i) the claim follows  from Lemma \ref{lem:hols}, 1. So we suppose that $e$ satisfies one of the assumptions (ii)-(v). 
We first delete all edges in $E(\Gamma)\setminus\{e\}$ that do not occur in $f$ and double all edges that are traversed twice by $f$.  Denote by $\Gamma'_1$ the resulting ciliated ribbon graph and by  $F_1:\mathcal G(\Gamma'_1)\to\mathcal G(\Gamma)$ the associated  functor  from Definition \ref{def:graph_functor}. As $f$ is a face path that is compatible with the ciliation, there is a face path $f'_1\in \mathcal G(\Gamma'_1)$  that is compatible with the ciliation  and traverses each edge of $\Gamma'_1$ at most once with $F_1(f'_1)=f$. If $e$ does not occur in $f$ or is traversed only once by $f$, there is a unique edge  $e'_1\in E(\Gamma'_1)$ with $F_1(e'_1)=e$. If $e$ is traversed twice by $f$, there are two distinct edges $e'_1,e''_1\in F(\Gamma'_1)$ with $F_1(e'_1)=F_1(e''_1)=e$. 
Suppose that  $f'_1\in \mathcal G(\Gamma'_1)$ is given by  $f'_1=e_n^{\epsilon_n}\circ\ldots\circ e_1^{\epsilon_1}$ with $e_1,..,e_n\in E(\Gamma'_1)$. 
As each edge of $\Gamma'_1$ is traversed at most once by $f'_1$ we can ensure that $\epsilon_1=...=\epsilon_n=1$  by reversing the edge orientation.
As $f'_1$ is a face path that is compatible with the ciliation, it follows that that any  two consecutive edge ends $s(e_{i+1})$ and $t(e_i)$ in
the associated path $G_{\Gamma'_1}(f'_1)=t(e_n)\circ s(e_n)\circ \ldots \circ t(e_1)\circ s(e_1)$  
  are adjacent with $s(e_{i+1})<t(e_i)$.   
This allows us to apply the operation of detaching adjacent edge ends to $s(e_{i+1})$ and $t(e_i)$ whenever 
 their shared vertex is of valence $\geq 3$.  In case (ii) and (iii) we apply this detaching operation to all vertices $\st(e_{i})$ with $i\in\{2,...,n\}$ that are of valence $\geq 3$. In cases (iv) and (v)  we apply them to all such vertices except $\st(e'_1)$ and $\ta(e'_1)$. This yields a ciliated ribbon graph $\Gamma'_2$ in which every vertex is at most bivalent. Denoting by $F_2:\mathcal G(\Gamma'_2)\to\mathcal G(\Gamma'_1)$ the associated  functor  from Definition \ref{def:graph_functor}, we find that there is a face path $f'_2\in\mathcal G(\Gamma'_2)$ that is compatible with the ciliation and an edge $e'_2\in E(\Gamma'_2)$ with $f'_1=F_2(f'_2)$, $e'_1=F_2(e'_2)$ and such that every edge in $f'_2$ is traversed exactly once by $f'_2$.

In case (ii)  $f'_2$ and $e'_2$  have no vertices in common and hence  $\hol^*_{f'_2}(\alpha)\cdot \hol^*_{e'_2}(\beta)=\hol^*_{e'_2}(\beta)\cdot \hol^*_{f'_2}(\alpha)$ for all $\alpha,\beta\in K^*$ by  Lemma \ref{lem:hols}, 1. This  proves the claim in case (ii).
 In case (iii) the paths $f'_2$ and $e'_2$ intersect only in their starting or target vertices. By Lemma \ref{lem:hols}, 2.
the commutation relations of the elements $\hol^*_{f'_2}(\alpha)$ and $(\beta)_{e'_2}$ in $\mathcal A^*_{\Gamma'}$ are then given  by the expressions in Proposition  \ref{lem:algexplicit} (e) to (j), which proves the claim in case (iii).   
In case (iv) $f'_2$ traverses each edge of $\Gamma'_2$  exactly once and $e'_2$ is not the first or last edge in $f'_2$. 
To compute the commutation relations of $\hol^*_{f'_2}(\alpha)$ and $(\beta)_{e'_2}$
suppose without restriction of generality  that
$f'_2=e_n\circ ...\circ e_1$ and $e'_2=e_i$ with $e_j\in E(\Gamma'_2)$ for $j\in\{1,...,n\}$
and $i\in\{2,...,n-1\}$. 
As $f'_2$ is a face path that is compatible with the ciliation, any  two consecutive edge ends $s(e_{i+1})$ and $t(e_i)$ in the associated path
 $G_{\Gamma'_1}(f'_2)=t(e_n)\circ s(e_n)\circ \ldots \circ t(e_1)\circ s(e_1)$  
  are adjacent with $s(e_{i+1})<t(e_i)$. 
As the vertices $\ta(e_i)$ and $\st(e_i)$ are  bivalent and $f'_2$ traverses the edge $e_i$ only once, it is then sufficient to consider the paths $q_\circ=s(e_{i+1})\circ t(e_i)\circ s(e_i)\circ t(e_{i-1})$ and  $t(e_i)\circ s(e_i)$ in $\Gamma'_{2\,\circ}$ and to show that
$\hol^*_{q_\circ}(\alpha)\cdot \hol^*_{t(e_i)\circ s(e_i)}(\beta)=\hol^*_{t(e_i)\circ s(e_i)}(\beta)\cdot \hol^*_{q_\circ}(\alpha)$.  
As $f'_2$ is  a face path  that is compatible with the ciliation, one has $t(e_{i-1})> s(e_i)$ and $t(e_i)>s(e_{i+1})$. If the vertices $s(e_i)$ and $\ta(e_i)$ are distinct,  we obtain  
 with the convention  $\sigma(t(e_i))=0$, $\sigma(s(e_i))=1$ 
\begin{align*}
&\hol^*_{q_\circ}(\alpha)\cdot \hol^*_{t(e_i)\circ s(e_i)}(\beta)=(\low\alpha 4\oo\low\alpha 3\oo\low\alpha 2\oo\low\alpha 1)_{t(e_{i-1})s(e_i)t(e_i)s(e_{i+1})}\cdot (\low\beta 2\oo\low\beta 1)_{s(e_i) t(e_i)}\\
&= (\low\alpha 4\oo\low\alpha 3)_{t(e_{i-1})s(e_i)}\cdot (\low\beta 2)_{s(e_i)}\cdot (\low\alpha 2\oo\low\alpha 1)_{t(e_i)s(e_{i+1})}\cdot (\low\beta 1)_{t(e_i)}\\
&=(\low\alpha 3)_{s(e_i)}\cdot (\low\alpha 4)_{t(e_{i-1})}\cdot  (\low\beta 2)_{s(e_i)}\cdot (\low\alpha 1)_{s(e_{i+1})}\cdot (\low\alpha 2)_{t(e_i)} \cdot (\low\beta 1)_{t(e_i)} \\
&=\langle S(\beta_{(4)})\oo\alpha_{(5)} ,R\rangle\langle \beta_{(1)}\oo\alpha_{(2)} ,R\rangle\; (\beta_{(3)}\low\alpha 4)_{s(e_i)}\cdot (\low\alpha 6)_{t(e_{i-1})}\cdot (\low\alpha 1)_{s(e_{i+1})}\cdot (\beta_{(2)}\alpha_{(3)})_{t(e_i)}\\
&=\langle S(\beta_{(3)})\oo\alpha_{(4)} ,R\rangle\langle \beta_{(2)}\oo\alpha_{(3)} ,R\rangle\; (\low\alpha 5\beta_{(4)})_{s(e_i)}\cdot (\low\alpha 6)_{t(e_{i-1})}\cdot (\low\alpha 1)_{s(e_{i+1})}\cdot (\alpha_{(2)}\beta_{(1)})_{t(e_i)}\\
&=(\low\alpha 3\beta_{(2)})_{s(e_i)}\cdot (\low\alpha 4)_{t(e_{i-1})}\cdot (\low\alpha 1)_{s(e_{i+1})}\cdot (\alpha_{(2)}\beta_{(1)})_{t(e_i)},\\
\intertext{}
& \hol^*_{t(e_i)\circ s(e_i)}(\beta)\cdot \hol^*_{q_\circ}(\alpha)=(\low\beta 2\oo\low\beta 1)_{s(e_i) t(e_i)}\cdot (\low\alpha 4\oo\low\alpha 3\oo\low\alpha 2\oo\low\alpha 1)_{t(e_{i-1})s(e_i)t(e_i)s(e_{i+1})}\\
&= (\low\beta 2)_{s(e_i)}\cdot  (\low\alpha 4\oo\low\alpha 3)_{t(e_{i-1})s(e_i)}\cdot (\low\beta 1)_{t(e_i)}\cdot  (\low\alpha 2\oo\low\alpha 1)_{t(e_i)s(e_{i+1})}\\
&= (\low\beta 2)_{s(e_i)}\cdot  (\low\alpha 3)_{s(e_i)}\cdot (\low\alpha 4)_{t(e_{i-1})}  \cdot (\low\beta 1)_{t(e_i)} 
 \cdot (\low\alpha 1)_{s(e_{i+1})}\cdot (\low\alpha 2)_{t(e_i)} \\
 &=\langle S(\alpha_{(2)})\oo\beta_{(1)}, R\rangle\langle \alpha_{(3)}\oo\beta_{(2)}, R\rangle\;    (\low\alpha 5\low\beta 4)_{s(e_i)}\cdot 
  (\low\alpha 6)_{t(e_{i-1})}  \cdot (\alpha_{(4)}\beta_{(3)})_{t(e_i)} 
 \cdot (\low\alpha 1)_{s(e_{i+1})} \\
 &= (\low\alpha 3\low\beta 2)_{s(e_i)}\cdot 
  (\low\alpha 4)_{t(e_{i-1})}  \cdot (\alpha_{(2)}\beta_{(1)})_{t(e_i)} 
 \cdot (\low\alpha 1)_{s(e_{i+1})} =\hol^*_{q_\circ}(\alpha)\cdot \hol^*_{t(e_i)\circ s(e_i)}(\beta).
\end{align*}
If $e_i$ is a loop, then the fact that $f'_2$ is a face path implies that $\hol^*_{s(e_{i+1})\circ t(e_i)}(\alpha)$ commutes with $(\beta)_{s(e_i)}$ and $(\beta)_{t(e_{i-1})}$ and 
$\hol^*_{s(e_{i})\circ t(e_{i-1})}(\alpha)$ commutes with $(\beta)_{s(e_{i+1})}$ and $(\beta)_{t(e_i)}$.  An analogous computation then yields the same result, and this  proves the claim in case (iv).

In case (v),  the fact that $f$  is compatible with the ciliation implies that  $f$ is cyclically reduced if $\Gamma$ does not have any univalent vertices. 
Hence  $f'_2$ is either of the form (a) $f'_2={e'_2}^{\pm 1}$, (b)  $f'_2={e_2'}^{\pm 1}\circ q$ or (c) $f'_2=q\circ {e'_2}^{\pm 1}$ such that $e'_2$ is not the first or last edge in $q$.  In case (a) the claim follows directly from  Proposition \ref{lem:algexplicit} (b). In cases (b) and (c), 
we can assume for simplicity that $f'_2$ is of the form (b) $f'_2=e'_2\circ q$ or (c) $f'_2=q\circ e'_2$ since the other cases are obtained by applying antipode to $e'_2$. As $f'_2$ is  compatible with the ciliation,  the ordering of the edge ends is given by  $s(e'_2)<t(q)$, $s(q)<t(e'_2)$ in case (b) and by $s(q)<t(e'_2)$, $s(e'_2)<t(q)$ in case (c). 
To compute the commutation relations, consider the elements $\hol^*_{f'_\circ}(\alpha)$ and $\hol^*_{t(e_2')\circ s(e_2')}(\beta)$ in $\mathcal A^*_{\Gamma'_{2\circ}}$ with $f'_\circ= t(e_2')\circ s(e_2')\circ t(q')\circ s(q')$ in  (b) and $f'_\circ=t(q')\circ s(q')\circ t(e_2')\circ s(e_2')$ in  (c).  In case (b), we obtain with   $\sigma(t(e'_2))=0$ and $\sigma(s(e'_2))=1$ 
 \begin{align*}
 &\hol^*_{f'_\circ}(\alpha)\cdot \hol^*_{t(e'_2)\circ s(e'_2)}(\beta)=(\low\alpha 1\oo\low\alpha 2\oo\low\alpha 3\oo\low\alpha 4)_{ t(e'_2)\circ s(e'_2)\circ t(q')\circ s(q')}\cdot  (\low\beta 1\oo\low\beta 2)_{ t(e'_2)\circ s(e'_2)}\\
 &=(\low\alpha 2)_{s(e'_2)}\cdot (\low\alpha 3)_{t(q')}\cdot (\low\beta 2)_{s(e'_2)}\cdot (\low\alpha 4)_{s(q')}\cdot (\low\alpha 1)_{t(e'_2)}\cdot (\low\beta 1)_{t(e'_2)}\\
 &=\langle \low\beta 1\oo \low\alpha 1, R\rangle \langle S(\low\beta 4)\oo\low\alpha 4, R\rangle \; (\low\beta 3\low\alpha 3)_{s(e'_2)}\cdot (\low\alpha 5)_{t(q')}\cdot (\low\alpha 6)_{s(q')}\cdot (\low\beta 2\low\alpha 2)_{t(e'_2)}\\
  &=\langle \low\beta 2\oo \low\alpha 2, R\rangle \langle S(\low\beta 3)\oo\low\alpha 3, R\rangle \; (\low\alpha 4\low\beta 4)_{s(e'_2)}\cdot (\low\alpha 5)_{t(q')}\cdot (\low\alpha 6)_{s(q')}\cdot (\low\alpha 1\low\beta 1)_{t(e'_2)}\\
  &=(\low\alpha 2\low\beta 2)_{s(e'_2)}\cdot (\low\alpha 3)_{t(q')}\cdot (\low\alpha 4)_{s(q')}\cdot (\low\alpha 1\low\beta 1)_{t(e'_2)}\\
 & \hol^*_{t(e'_2)\circ s(e'_2)}(\beta)\cdot \hol^*_{f'_\circ}(\alpha)=(\low\beta 1\oo\low\beta 2)_{ t(e'_2)\circ s(e'_2)}\cdot (\low\alpha 1\oo\low\alpha 2\oo\low\alpha 3\oo\low\alpha 4)_{ t(e'_2)\circ s(e'_2)\circ t(q')\circ s(q')}\\
 &= (\low\beta 2)_{s(e'_2)}\cdot (\low\alpha 2)_{s(e'_2)}\cdot (\low\alpha 3)_{t(q')}\cdot (\low\beta 1)_{t(e'_2)} \cdot (\low\alpha 4)_{s(q')}\cdot (\low\alpha 1)_{t(e'_2)}\\
 &=\langle S(\low\alpha 6)\oo\low\beta 1, R\rangle \langle \low\alpha 1\oo \low\beta 2, R\rangle \;   (\low\alpha 3\low\beta 4)_{s(e'_2)}\cdot (\low\alpha 4)_{t(q')}\cdot (\low\alpha 5)_{s(q')}\cdot (\low\alpha 2\low\beta 3)_{t(e'_2)}\\
  &=\langle S(\low\alpha 3)\oo\low\beta 1, R\rangle \langle \low\alpha 1\oo \low\beta 2, R\rangle \; \hol^*_{f'_\circ}(\low \alpha 2)\cdot \hol^*_{t(e'_2)\circ s(e'_2)}(\low \beta 3)  \end{align*}
This agrees with the formula in Proposition \ref{lem:algexplicit} (g) if we apply the antipode  to the element $\alpha$ in the formula in Proposition \ref{lem:algexplicit} (g)
to take into account   that the loop $f$ there has the opposite orientation. In case (c) we obtain with   $\sigma(t(e'_2))=0$, $\sigma(s(e'_2))=1$ 
 \begin{align*}
 & \hol^*_{t(e'_2)\circ s(e'_2)}(\beta)\cdot \hol^*_{f'_\circ}(\alpha)=(\low\beta 1\oo\low\beta 2)_{ t(e'_2)\circ s(e'_2)}\cdot(\low\alpha 1\oo\low\alpha 2\oo\low\alpha 3\oo\low\alpha 4)_{t(q')\circ s(q')\circ t(e'_2)\circ s(e'_2)}\\
 &=(\low\beta 2)_{s(e'_2)}\cdot(\low\alpha 4)_{s(e'_2)}\cdot (\low\alpha 1)_{t(q')}\cdot (\low\beta 1)_{t(e'_2)}\cdot  (\low\alpha 2)_{s(q')}\cdot (\low\alpha 3)_{t(e'_2)}\\
 &=\langle \low\beta 1\oo S(\low \alpha 3), R\rangle \langle  \low\beta 2\oo \low\alpha 4, R\rangle \; (\low\alpha 6\low\beta 4)_{s(e'_2)}\cdot (\low\alpha 1)_{t(q')}\cdot (\low\alpha 2)_{s(q')}\cdot (\low\alpha 5\low\beta 3)_{t(e'_2)}\\
 &= (\low\alpha 4\low\beta 2)_{s(e'_2)}\cdot (\low\alpha 1)_{t(q')}\cdot (\low\alpha 2)_{s(q')}\cdot (\low\alpha 3\low\beta 1)_{t(e'_2)}\\
 &\hol^*_{f'_\circ}(\alpha)\cdot \hol^*_{t(e'_2)\circ s(e'_2)}(\beta)=(\low\alpha 1\oo\low\alpha 2\oo\low\alpha 3\oo\low\alpha 4)_{t(q')\circ s(q')\circ t(e'_2)\circ s(e'_2)}\cdot  (\low\beta 1\oo\low\beta 2)_{ t(e'_2)\circ s(e'_2)}\\
 &=(\low\alpha 4)_{s(e'_2)}\cdot (\low\alpha 1)_{t(q')}\cdot (\low\beta 2)_{s(e'_2)}\cdot (\low\alpha 2)_{s(q')}\cdot (\low\alpha 3)_{t(e'_2)}\cdot (\low\beta 1)_{t(e'_2)}\\
 &=\langle \low\beta 1\oo \low \alpha 4, R\rangle \langle  S(\low\beta 4)\oo \low\alpha 1, R\rangle \; (\low\beta 3\low\alpha 6)_{s(e'_2)}\cdot (\low\alpha 2)_{t(q')}\cdot (\low\alpha 3)_{s(q')}\cdot (\low\beta 2\low\alpha 5)_{t(e'_2)}\\
  &=\langle \low\beta 3\oo \low \alpha 6, R\rangle \langle  S(\low\beta 4)\oo \low\alpha 1, R\rangle \; (\low\alpha 5\low\beta 2)_{s(e'_2)}\cdot (\low\alpha 2)_{t(q')}\cdot (\low\alpha 3)_{s(q')}\cdot (\low\alpha 4\low\beta 1)_{t(e'_2)}\\
  &=\langle\low\beta 2\oo\low \alpha 3, R\rangle\langle S(\low\beta 3)\oo\low\alpha 1, R\rangle\;  \hol^*_{t(e'_2)\circ s(e'_2)}(\low \beta 1)\cdot \hol^*_{f'_\circ}(\low \alpha 2)
 \end{align*}
This agrees with  Proposition \ref{lem:algexplicit} (e) if we apply the antipode  to the elements $\alpha$, $\beta$ in formula in Proposition  \ref{lem:algexplicit} (e) to take into account   that the loop $f$ and the edge $e$ in  Proposition \ref{lem:algexplicit} (e) have the opposite orientation.
This proves the claim in case (v) and concludes the proof.
\end{proof}

Lemma \ref{lem:holcomp} states that  the only  variables $(\beta)_e$  which  do not commute with the holonomy of the face path $f$ 
 are those of edges $e\in E$ incident at the  vertex $\st(f)=\ta(f)$. This suggests that  the commutation relations simplify
further if one imposes  gauge invariance at this vertex.  Indeed, one finds that the projection of the holonomies of such face paths are central in  $\mathcal A^*$.

\begin{proposition}\label{lem:facecentral} 
Let 
$f$  be a face path that is compatible with the ciliation.
Then  $\Pi\circ\hol^*_f(\alpha)$  is central in $\mathcal A^*$  for all $\alpha\in K^*$ and  invariant under cyclic permutations of $f$.
The map $\hol^*_f: K^*\to\mathcal A^*$  restricts to an  algebra homomorphism $\hol^*_f: C(K)\to Z(\mathcal A^*_{inv})$. 
\end{proposition}
\begin{proof} 
 It  
  follows with Lemma \ref{lem:holcomp} that
$\Pi\circ\hol^*_f$ commutes with all edges $e\in E$ with $\st(f)\notin\{\st(e),\ta(e)\}$.  
If $f$ is given by $f=e_n^{\epsilon_n}\circ ...\circ e_1^{\epsilon_1}$, then 
 $s(\epsilon_1^{\epsilon_1})<t(e_n^{\epsilon_n})$ 
 and the assumptions  imply that
   either $t(e^{\pm 1})<s(e_1^{\epsilon_1})$ or $t(e^{\pm 1})>t(e_n^{\epsilon_n})$ for all other edges $e$ incident at $\st(f)$. As $f$ is compatible with the ciliation, the assumptions of Lemma \ref{lem:holcomp} and Theorem \ref{lem:holtransform}   are satisfied. 
From  Theorem \ref{lem:holtransform}
one has  $\Pi\circ\hol^*_f(\alpha)\lhd (h)_{\st(f)}=\langle S(\low\alpha 3)\low\alpha 1, h\rangle \Pi\circ \hol^*_f(\low\alpha 2)=\epsilon(h)\,\Pi\circ \hol^*_f(\alpha)$ for all $\alpha\in K^*$, $h\in K$.
For incoming edges  at $\st(f)$  that are not loops one then obtains from
 Lemma \ref{lem:holcomp}
\begin{align*}
(\beta)_e\cdot \Pi\circ \hol^*_f(\alpha)&=
\begin{cases}
\langle \low\beta 1\oo \alpha_{(3)}, R\rangle\langle \low\beta 2\oo \alpha_{(1)}, R^\inv\rangle\,  \Pi\circ \hol^*_f(\low\alpha 2) \cdot (\low\beta 3)_e& t(e)\leq s(e_1^{\epsilon_1})<t(e_n^{\epsilon_n})\\
\langle \alpha_{(3)}\oo\low\beta 1, R^\inv\rangle\langle \alpha_{(1)}\oo\low\beta 2, R\rangle\, \Pi\circ \hol^*_f(\low\alpha 2)\cdot (\low\beta 3)_e &
s(e_1^{\epsilon_1})<t(e_n^{\epsilon_n})\leq t(e)
\end{cases}\\
&=\begin{cases}
\langle D_R(S(\low \beta 1)),  S(\low\alpha 3)\low\alpha 1 \rangle\, \Pi\circ \hol^*_f(\low\alpha 2) \cdot (\low\beta 2)_e & t(e)\leq s(e_1^{\epsilon_1})<t(e_n^{\epsilon_n})\\
 \langle D_{R_{21}^\inv}(S(\low \beta 1)),  S(\low\alpha 3)\low\alpha 1 \rangle\, \Pi\circ \hol^*_f(\low\alpha 2) \cdot (\low\beta 2)_e &
s(e_1^{\epsilon_1})<t(e_n^{\epsilon_n})\leq t(e).
\end{cases}\\
&=\epsilon(\low\beta 1)\, \Pi\circ \hol^*_f(\alpha) \cdot (\low\beta 2)_e=\Pi\circ\hol^*_f(\alpha) \cdot (\beta)_e,
\end{align*}
where $D_R:  K^*\to K$, $\alpha\mapsto \langle\alpha,\low R 1\rangle\, \low R 2$ is the map from Lemma \ref{lem:quasprop}. The proofs for outgoing edges $e$ and loops  are analogous. As the elements $(\beta)_e$ with $e\in E$, $\beta\in K^*$ generate $\mathcal A^*$, this proves that $\Pi(\hol^*_f(\alpha))$ is central in $\mathcal A^*$.
 It then follows from Lemma \ref{lem:reghols} that  $\Pi\circ\hol^*_f$ is invariant under cyclic permutations of $f$ and restricts to an algebra homomorphism $C(K)\to Z(\mathcal A^*_{inv})$.
 \end{proof}

Proposition \ref{lem:facecentral} still relies on the assumption that 
 the face path is compatible with the ciliation. If every vertex of $\Gamma$ is at least 3-valent and $\Gamma$ has at least two faces, this can be achieved by adjusting the cilia at the vertices in $f$.  As this does not affect the algebra structure  of the subalgebra  $\mathcal A^*_{inv}\subset \mathcal A^*$ by  Proposition \ref{lem:invspace}, we can then apply Lemma \ref{lem:reghols} and Proposition \ref{lem:facecentral} to obtain the following theorem.

\begin{theorem}\label{cor:face} 
Let $\Gamma$ be a ribbon graph that satisfies the conditions above and let 
  $f$ be a face path of $\Gamma$. Then 
$\Pi\circ\hol^*_f(\alpha)\in \mathcal A^*_{inv}$ is central in $\mathcal A^*_{inv}$ for all $\alpha\in K^*$ and depends only on the associated face. The map $\hol^*_f: K^*\to \mathcal A^*$ induces an 
algebra homomorphism $\hol^*_f: C(K)\to Z(A^*_{inv})$.
\end{theorem}
\begin{proof} 
As the algebra structure of  $\mathcal A^*_{inv}$ does not depend on the choice of the cilia of the vertices of $\Gamma$ by Proposition \ref{lem:invspace}, we can assume 
without restriction of generality  that the cilia  are chosen in such a way that 
$f$  satisfies the assumptions of Lemma \ref{lem:reghols} and Proposition \ref{lem:facecentral}.  It then follows from Lemma \ref{lem:reghols} that  $\Pi\circ\hol^*_f$ is invariant under cyclic permutations of $f$ and   induces an algebra anti-homomorphism $C(K)\to\mathcal A^*_{inv}$. Proposition \ref{lem:facecentral} implies that it takes values in  $Z(\mathcal A^*_{inv})$. 
\end{proof}

As $K$ is finite-dimensional semisimple and $\text{char}(\FF)=0$, the Hopf algebra  $K^*$ is equipped with a Haar integral $\eta\in K^*$.  This allows one to associate a projector to each face $f\in F$ that acts on $\mathcal A^*$ by multiplication with the curvature $\hol^*_f(\eta)$. For this, note that 
$\hol^*_f(\eta)\in Z(\mathcal A^*)$ for any face $f$ that is compatible with the ciliation and that  $\hol^*_f(\eta)\cdot \hol^*_f(\eta)=\hol^*_f(\eta)$.
Hence,  we obtain an algebra homomorphism and projector
\begin{align}\label{eq:project}
P^*_f: \mathcal A^*\to\mathcal A^*,\quad\alpha\mapsto \hol^*_f(\eta)\cdot \alpha
\end{align}
that restricts to an algebra homomorphism and projector $P^*_f:\mathcal A^*_{inv}\to\mathcal A^*_{inv}$.  If $f$ 
is not compatible with the ciliation, 
 it is not guaranteed that $P^*_f:\mathcal A^*\to\mathcal A^*$ is an algebra homomorphism, but this still holds for the restriction $P^*_f:\mathcal A^*_{inv}\to\mathcal A^*_{inv}$ if  $\Gamma$  satisfies the assumptions before Theorem \ref{cor:face}.

\begin{lemma} \label{lem:haarproj} If  $f$ is a face path  that is compatible with the ciliation,  then  $P^*_f: \mathcal A^*\to\mathcal A^*$
from \eqref{eq:project} is a projector with $P^*_f(\hol^*_f(\alpha))=\epsilon(\alpha)P^*_f(1)$ for all $\alpha\in K^*$. 
If $\Gamma$ satisfies the assumptions before Theorem \ref{cor:face}, then for any face $f$, the restriction $P^*_f:\mathcal A^*_{inv}\to\mathcal A^*_{inv}$  is a projector with $P^*_f(\hol^*_f(\alpha))=\epsilon(\alpha)P^*_f(1)$ for all $\alpha\in C(K)$.
\end{lemma}
\begin{proof} Suppose that  $f$  is a face path that is compatible with the ciliation.
Then it follows from Proposition \ref{lem:facecentral} that  $P^*_f$ is an algebra homomorphism and  $\hol^*_f(\eta)$ is central in $\mathcal A^*$  since $\eta\in C(K)$. That it is a projector follows from  the properties of the Haar integral and equation \eqref{eq:reghols} in the proof of Lemma \ref{lem:reghols}, which implies
$
P^*_f\circ P^*_f(\beta)=\hol^*_f(\eta)\cdot \hol^*_f(\eta)\cdot\beta
=\hol^*_f(\eta^2)\,\beta=\hol^*_f(\eta)\cdot\beta=P^*_f(\beta)
$ for all $\beta\in\mathcal A^*$. 
To compute  $P^*_f(\hol^*_f(\alpha))$ for $\alpha\in K^*$, note that \eqref{eq:reghols} holds already  if one of the two arguments $\alpha,\beta$ in \eqref{eq:reghols} is contained in $C(K)$.  This yields
 $P^*_f(\hol^*_f(\alpha))=\hol^*_f(\eta)\cdot \hol^*_f(\alpha)=\hol^*_f(\eta\cdot\alpha)=\epsilon(\alpha) \hol^*_f(\eta)=\epsilon(\alpha) P^*_f(1)$.
If $f$ is a general face and $\Gamma$ satisfies the assumptions in Theorem \ref{cor:face}, then by Theorem \ref{cor:face} one has $\hol^*_f(\eta)\in Z(\mathcal A^*_{inv})$ and hence the restriction of $P^*_f$ to $\mathcal A^*_{inv}$ is an algebra homomorphism  and a projector by the argument above.   The identity $P^*_f(\hol^*_f(\alpha))=\epsilon(\alpha)P^*_f(1)$ for $\alpha\in C(K)$  follows because $\hol^*_f(\alpha)\in Z(\mathcal A^*_{inv})$ for $\alpha\in C(K)$.
\end{proof}

For a face  path $f$ that is compatible with the ciliation, by duality the 
 projector in \eqref{eq:project}   induces a projector 
 $P_f: K\exoo{E}\to K\exoo{E}$  with  $\langle \alpha, P_f(k)\rangle=\langle P_f^*(\alpha), k\rangle$ for all $\alpha\in \mathcal A^*$, $k\in K\exoo{E}$. 
The properties of  $P^*_f$ then  imply
$\hol_f(P_f(k))=\langle \eta, \hol_f(k)\rangle\, 1=\epsilon\exoo{E}(P_f(k))\, 1$ and  hence
  $P_f(K\exoo{E})\subset \mathcal A_f$.  
Hence, $P_f:K\exoo{E}\to K\exoo{E}$ projects on a linear subspace of the space of connections that are flat at $f$.
This allows us to interpret $P^*_f: \mathcal A^*\to\mathcal A^*$ from \eqref{eq:project}  as a projector on its dual  $\mathcal A^*_f$, i.e.~as a projector onto a certain quotient of the space of functions on $\mathcal A_f$.

If  $\Gamma$ satisfies the conditions before Theorem \ref{cor:face}, 
then the projectors $P_f^*, P^*_{f'}:\mathcal A^*_{inv}\to\mathcal A^*_{inv}$ for different faces $f,f'$ commute since $\hol^*_f(\eta)\in Z(\mathcal A^*_{inv})$ for all $f\in F$.   Hence, one obtains a projector $P^*_{flat}=\Pi_{f\in F}P^*_f:\mathcal A^*_{inv}\to\mathcal A^*_{inv}$. 
As all projectors $P^*_f$ are algebra homomorphisms, this also holds for $P^*_{flat}$. Consequently,  the image of $P^*_{flat}$  is a subalgebra of $\mathcal A^*_{inv}$ and can be interpreted  as a quotient of the algebra of  functions on the subspace  $\mathcal A_{flat}=\cap_{f\in F}\mathcal A_f$ of {\em flat connections}.

\begin{definition} Let $\Gamma$ be a ribbon graph that satisfies the assumptions before Theorem \ref{cor:face}. Then the  subalgebra $\mathcal M_{\Gamma}=\mathrm{Im}(P^*_{flat})\subset \mathcal A^*_{inv}$ is called the {\bf quantum moduli algebra}.
\end{definition}

As this holds already for the algebra $\mathcal A^*_{inv}$ and the projectors $P^*_f$,  the quantum moduli  algebra  is independent of the choice of cilia at the vertices of $\Gamma$.
We will now show that  it is also largely independent of the choice of the ribbon graph $\Gamma$ and depends only on the homeomorphism class of the closed surface $\Sigma_\Gamma$ obtained by gluing discs to the faces of $\Gamma$.  
To prove this, recall 
 from the discussion following Definition \ref{def:graphtrafos} that 
 any two ribbon graphs $\Gamma$ and $\Gamma'$ for which the surfaces $\Sigma_\Gamma$ and $\Sigma_{\Gamma'}$  are homeomorphic
  can be transformed into each other by  contracting and expanding  finitely many edges  and adding and removing finitely many   loops.

To describe the effect of these graph transformations on faces, note that if $\Gamma'$ is obtained from $\Gamma$ by contracting an edge, then
the associated functor $F: \mathcal G(\Gamma')\to\mathcal G(\Gamma)$  from Definition \ref{def:graph_functor} induces a 
 bijection between the faces of $\Gamma'$ and of $\Gamma$. If  $\Gamma'$ is obtained from $\Gamma$ by adding a loop $l$ at $v\in V(\Gamma)$, 
 then  $\Gamma'$ has one additional face, namely the added loop.  In this case one has
 $F(l)=\emptyset_v$ and $F$ induces a bijection between $F(\Gamma')\setminus\{l\}$ and  $F(\Gamma)$.
Taking into account this relation between the faces of $\Gamma$ and $\Gamma'$, we can prove that the moduli algebra $\mathcal M_\Gamma$ is a topological invariant.

\begin{theorem} Let $\Gamma$, $\Gamma'$ be ribbon graphs that satisfy the assumptions before Theorem \ref{cor:face}. Let  $\Sigma_\Gamma$, $\Sigma_{\Gamma'}$ be the surfaces obtained by gluing discs to the faces of $\Gamma$ and $\Gamma'$. If $\Sigma_\Gamma$ and $\Sigma_{\Gamma'}$ are homeomorphic, then  the moduli algebras $\mathcal M_\Gamma$ and $\mathcal M_{\Gamma'}$ are isomorphic.
\end{theorem}
\begin{proof}
If $\Sigma_\Gamma$ and $\Sigma_{\Gamma'}$ are homeomorphic, then $\Gamma$ and $\Gamma'$ can be transformed into each other by the contracting and expanding a finite number of edges and adding or removing a finite number of  loops. It is therefore sufficient to show that algebra homomorphisms $C^*_{\ta(e)}, C^*_{\st(e)}$ and $A^*_{v}: \mathcal A^*_{\Gamma'}\to\mathcal A^*_\Gamma$ from 
Definition \ref{def:edgeopmaps}  and 
Theorem \ref{th:alghomfinal} induce isomorphisms between the associated quantum moduli  algebras. 
By Theorem \ref{th:alghomfinal} and Remark \ref{rem:graphopprops}
 the contraction maps $C^*_{\ta(e)}$ and $C^*_{\st(e)}$ induce algebra isomorphisms between $\mathcal A^*_{\Gamma}$ and $\mathcal A^*_{\Gamma'}$ and the maps $A^*_v: \mathcal A^*_{\Gamma'\,inv}\to\mathcal A^*_{\Gamma\,inv}$
are injective algebra homomorphisms. Hence, it is sufficient to show that
 $P^\Gamma_{flat}\circ F^*=F^*\circ P^{\Gamma'}_{flat}$  for $F^*\in\{C^*_{\ta(e)}, C^*_{\st(e)}, A^*_{v}\}$.
 Theorems \ref{th:alghomfinal} and \ref{th:holtheorem}  imply
$$F^*\circ P'_f (\alpha)=F^*(\hol^*_{f'}(\eta)\cdot \alpha) =F^*(\hol^*_{f'}(\eta))\cdot F^*(\alpha)=\hol^*_{f}(\eta)\cdot F^*(\alpha)=P_f\circ F^*(\alpha)$$
 for each face $f\in F$ and $\alpha\in A^*_{\Gamma'\,inv}$. 
This proves the claim for $F^*=C^*_{\ta(e)}$ and $F^*=C^*_{\st(e)}$.
If $\Gamma'$ is obtained from $\Gamma$ by adding a loop $l$, then   \eqref{eq:adding}  implies $ A^*_{v}\circ P'_l=\epsilon(\eta)\id=\id$. As $P_{flat}^{\Gamma}=\Pi_{f\in F} P_f$ and $P_{flat}^{\Gamma'}=P'_l\cdot \Pi_{f\in F} P'_f$, it follows that  $A_{v}^*\circ P^{\Gamma'}_{flat}=P^{\Gamma}_{flat}\circ A_{v}^*$.
\end{proof}

To conclude the discussion, we comment on the generalisation of the results on holonomies to the case of a non-semisimple ribbon Hopf algebra:
\begin{compactitem}
\item 
Theorem \ref{lem:holtransform} 
 holds for finite-dimensional ribbon Hopf algebras $K$ that are not  semisimple.  This follows because the proof requires semisimplicity only insofar as it relies on results about the transformations of holonomies under graph operations, while the graph operations themselves are defined also in the non-semisimple case. As the two notions of holonomy agree for face paths that are compatible with the ciliation, the results  of
Theorem \ref{lem:holtransform} extend to the non-semisimple case.
 It is proven in \cite{AGSII}, Propositions 2 and 3,  for the holonomies based on the comultiplication dual to $\mathcal A^*$ and with different methods,  that the holonomies of such face paths satisfy  and \eqref{eq:holtrafo}.\\[-1ex]
 
\item Similarly,  Lemma \ref{lem:reghols} holds in the non-semisimple   case.  Corollary \ref{cor:gtrafoflat} and Lemma \ref{lem:hols} also hold in the non-semisimple case, see also Proposition 2 and 3 in \cite{AGSII}.\\[-1ex]

\item  Lemmas \ref{lem:hols} and \ref{lem:holcomp},  Proposition \ref{lem:facecentral}  and Theorem \ref{cor:face}  hold in the case of a non-semisimple ribbon Hopf algebra, although more care is required in the proof since $T^*(C(K))\neq C(K)$, $S^2\neq \id$ and the weaker result in Proposition \ref{lem:invspace}. 
Results analogous to  Lemma \ref{lem:holcomp} and Proposition \ref{lem:facecentral} and Theorem \ref{cor:face}  were also derived  in Propositions 2, 3, 4 in \cite{AGSII}, for the holonomies based on the comultiplication dual to the multiplication of $\mathcal A^*$, but by very different methods. That 
these results hold for both notions of holonomies is unsurprising, since 
 the  holonomies from Definition \ref{def:hols} coincide with the ones  based on the comultiplication dual to  the multiplication of $\mathcal A^*$  for face paths  that are compatible with the ciliation. Note also that  the  proof of Lemma \ref{lem:holcomp} is mainly based on graph transformations, which are also defined in the non-semisimple case. \\[-1ex]

\item 
 The projectors $P_f^*:\mathcal A^*\to\mathcal A^*$ from \eqref{eq:project} and Lemma \ref{lem:haarproj} cannot be generalised directly to the non-semisimple case since they require a Haar integral. However, it is still possible to define the moduli algebra by invariance requirements (see \cite{AGSII,BR}).
 \end{compactitem}

\section*{Acknowledgements}
This work was supported by the Emmy Noether research grant ME 3425/1-3 of the
German Research Foundation (DFG). 
C.~M.~thanks Simon Lentner, Christoph Schweigert and Ingo Runkel, Hamburg University, for helpful discussions and Andreas Knauf, University of Erlangen-N\"urnberg, for comments on a draft on this paper. Both authors thank  John Baez, UC Riverside, for helpful  discussions  and St\'ephane Baseilhac, Universit\'e de Montpellier, for pointing out the work \cite{BFK}. 
The work was presented at the workshops {\em Quantum Spacetime '16}, Feb 12-16 2016, in Zakopane and {\em Workshop on Noncommutative Field Theory and Gravity}, Sept 21-25, 2015 in Corfu, both supported by the COST network MP1405 QSPACE. The authors thank the referees for their careful reading of the article and many suggestions that lead to improvements in the presentation.

%:refs

\begin{appendix}

\section{Some facts about Hopf algebras}
\label{sec:hopfalg}

In this appendix, we collect some definitions and results on Hopf algebras that are needed in the article. Unless  specific citations are given,  these definitions and results are standard and can be found in any textbook on Hopf algebras, for instance the books by Kassel \cite{K}, Majid \cite{majidbook}, Montgomery \cite{Mon} or Radford \cite{rbook}. 

\subsection{Semisimplicity and Haar integrals}

We start with the notion of  (semi)simplicity. Recall that a Hopf algebra $H$ is  (semi)simple if it is semisimple as an algebra and it is co(semi)simple if $H^*$ is (semi)simple.

\begin{theorem}[ \cite{LR}] \label{th:lr}Let $H$ be a finite-dimensional Hopf algebra over a field $\FF$ of characteristic zero.   Then  $H$ is semisimple if and only if $H^*$ is semisimple if and only if $S^2=\id_H$.
\end{theorem}

\begin{definition} Let $H$ be a Hopf algebra. A (normalised) {\bf Haar integral} in $H$ is an element $\ell\in H$ with $h\cdot\ell=\ell\cdot h=\epsilon(h)\,\ell$ for all $h\in H$ and $\epsilon(\ell)=1$.
\end{definition}

\begin{lemma} If $H$ is finite-dimensional and semisimple, then $H$ has a Haar integral.
\end{lemma}

\begin{lemma} Let $H$ be a Hopf algebra.$\quad$\\[-4ex]
\begin{compactenum}
\item If $\ell,\ell'\in H$ are Haar integrals, then $\ell=\ell'$.
\item If $\ell\in H$ is a Haar integral, then $\Delta^{(n)}(\ell)$ is invariant under cyclic permutations and $S(\ell)=\ell$.
\item If $\ell\in H$ is a Haar integral  then  the element $e=(\id\oo S)(\Delta(\ell))$ is a separability idempotent in $H$, i.e.~one has $m(e)=\low \ell 1S(\low \ell 2)=1$, $e\cdot e=e$ and
for all $h\in H$
$$  (h\oo 1)\cdot \Delta(\ell)=(1\oo S(h))\cdot \Delta(\ell)\qquad \Delta(\ell)(h\oo 1)=\Delta(\ell)(1\oo S(h)).$$

\item If $\ell\in H$ is a Haar integral, then $\kappa: H^*\oo H^*\to\FF$, $\kappa(\alpha\oo\beta)=\langle \alpha\cdot\beta,\ell\rangle$ is a Frobenius form.

\item If $\ell\in H$ is a Haar integral, then $\langle \low\alpha 1,\ell\rangle\, \low\alpha 2=\langle \low\alpha 2,\ell\rangle\, \low\alpha 1=\langle \ell,\alpha\rangle\, 1$ for all $\alpha\in H^*$.
\end{compactenum}
\end{lemma}

\begin{example}  \label{ex:finite_group}The Hopf algebra structure of the group algebra $\FF[G]$ of a finite group $G$ is given by
\begin{align*}
&g\cdot h=gh, & &1=e, & &\Delta(g)=g\oo g, & &\epsilon(g)=1, & &S(g)=g^\inv
\end{align*}
for all $g\in G$, where $e\in G$ denotes the unit element. The dual Hopf algebra is the set $\mathrm{Fun}(G)$ of functions $f:G\to\FF$ with the Hopf algebra structure 
\begin{align*}
\delta_g\cdot\delta_h=\delta_g(h)\,\delta_{h}, & &1=\Sigma_{g\in G}\delta_g, & &\Delta(\delta_g)=\Sigma_{uv=g}\delta_u\oo\delta_v, & &\epsilon(\delta_g)=\delta_g(e), & &S(\delta_g)=\delta_{g^\inv},
\end{align*}
where $\delta_g: G\to\FF$ is given by $\delta_g(g)=1$, $\delta_g(h)=0$ if $g\neq h$.
The Hopf algebra $\FF[G]$ is cocommutative and semisimple with Haar integral 
$\ell=\Sigma_{g\in G}\; g$. The Hopf algebra $\mathrm{Fun}(G)$ is commutative and semisimple with Haar integral $\eta=\delta_e$.
\end{example}

\subsection{Twisting}

\begin{definition}\label{def:twist}
Let $H$ be a bialgebra. A {\bf twist} for $H$ is an invertible element $F\in H\oo H$ that satisfies the conditions
\begin{align*}
F_{12}(\Delta\oo\id)(F)=F_{23}(\id\oo\Delta(F)) \qquad(\epsilon\oo\id)(F)=(\id\oo\epsilon)(F)=1.
\end{align*}
\end{definition}

\begin{lemma} \label{lem:twist}Let $(H,m, 1,\Delta,\epsilon)$ be a bialgebra  and $F ,G$  twists for $H$. Then:
\begin{compactenum}
\item  The comultiplication  $\Delta_{F,G}=F\cdot\Delta\cdot G^\inv$ and counit  $\epsilon$ define a coalgebra structure on $H$.
\item The comultiplication  $\Delta_{F,F}=F\cdot \Delta\cdot F^\inv$  and $\epsilon$ equip $(H, m,1)$ with the structure of a bialgebra.
\item  If $S:H\to H$ is an antipode for $H$, then an antipode for the bialgebra from 2.  is given by
 $S_F=\nu_F\cdot S\cdot \nu_F^\inv$   with $\nu_F=m\circ (\id\oo S)(F)$, $\nu_F^\inv=(S\oo\id)(F^\inv)$.
 \end{compactenum}
\end{lemma}

\subsection{Quasitriangular Hopf algebras and ribbon algebras}

\begin{definition}  
A  Hopf algebra $H$  is called {\bf quasitriangular} if there is  an invertible element $R=\low R 1\oo\low R 2\in H\oo H$,  the  {\bf $R$-matrix},  that satisfies $R\cdot \Delta(h)=\Delta^{op}(h)\cdot R$ for all $h\in H$, $(\Delta\oo\id)(R)= R_{13} \cdot R_{23}$ and $(\id\oo\Delta)(R)=R_{13} \cdot R_{12}$.
The element $Q=R_{21}\cdot R\in H\oo H$ is called the  {\bf monodromy element}. $H$ is  {\bf triangular} if it is quasitriangular and $R_{21}^\inv=\low R 2\oo S(\low R 1)=R$.
\end{definition}

\begin{lemma} \label{lem:quasprop}  Let $H$ be a finite-dimensional quasitriangular Hopf algebra over a field $\FF$ and $R=\low R 1\oo\low R 2$ an $R$-matrix. Then:
\begin{compactenum}
\item  $(\id\oo\epsilon)(R)=(\epsilon\oo\id)(R)=1$, $(S\oo\id)(R)=(\id\oo S^\inv)(R)=R^\inv$, $(S\oo S)(R)=R$.
\item   $R_{21}^\inv=\low R 2\oo S(\low R 1)$ is another $R$-matrix for $H$.
\item $R$ satisfies the Quantum Yang Baxter equation (QYBE) $R_{12}R_{13}R_{23}=R_{23}R_{13}R_{12}$.
\item  The linear map $D_R: K^*\to K$, $\alpha\mapsto \langle \alpha,\low R 1\rangle\, \low R 2$ is an algebra isomorphism, an anti-coalgebra isomorphism and satisfies $D_R\circ S=S^\inv\circ D_R$. 
\item The {\bf Drinfel'd element}   
 $u=m\circ (S\oo\id)(R_{21})$ is invertible  with inverse $u^\inv=m\circ (\id\oo S^2)(R_{21})$  and satisfies  $\epsilon(u)=1$, $\Delta(u)=(u\oo u)R^\inv R_{21}^\inv$ and   $S^2(h)=u \cdot h \cdot u^\inv$ for all $h\in H$ \cite{Dr}. 

\item If $\mathrm{char}(\FF)=0$, this implies by Theorem \ref{th:lr} that  $H$ is semisimple if and only if $u$ is central. 
\item The element  $uS(u)$ is called the {\bf quantum Casimir}. It is central  and satisfies  $S(uS(u))=uS(u)$, $\epsilon(uS(u))=1$ and $\Delta(uS(u))=R^\inv R_{21}^\inv(uS(u)\oo uS(u))$ \cite{Dr}. 
\end{compactenum}
\end{lemma}

\begin{definition} \label{def:ribbon}Let $H$ be a finite-dimensional quasitriangular Hopf algebra with $R$-matrix $R$.
Then  $H$ is called {\bf ribbon}, if there is a central  invertible element $\nu\in H$, the {\bf ribbon element},  with $uS(u)=\nu^2$, $\epsilon(\nu)=1$,  $S(\nu)=\nu$,  $\Delta(\nu)=R^\inv R_{21}^\inv(\nu\oo\nu)$. 
\end{definition}

\begin{remark}\label{rem:grouplike} If $H$ is ribbon with Drinfel'd element $u$ and ribbon element $\nu$, then the element $g=u^\inv\cdot\nu$ satisfies $g^\inv=S(g)$, $\Delta(g)=g\oo g$ and $g S(h)=S^{-1}(h)g$ for all $h\in H$. It is called the {\bf grouplike element}.
\end{remark}

\begin{lemma}[\cite{sgel}] \label{lem:ssimplerib}Let $H$ be a finite-dimensional semisimple quasitriangular Hopf algebra over a field of characteristic zero. Then  $H$ is ribbon with ribbon element  $\nu=u$.
\end{lemma}

\begin{theorem}[\cite{Dr}] \label{lem:ddouble2} Let $H$ be a finite-dimensional  Hopf algebra. Define for $h,h'\in H$, $\alpha,\alpha'\in H^*$ 
\begin{align}
&(\alpha\oo h)\cdot (\alpha'\oo h')=\langle \low {\alpha'} 3, \low h 1\rangle \langle \low{\alpha'} 1, S^\inv(\low {h} 3)\rangle\;  \alpha\low{\alpha'} 2\oo \low h 2h' & &1=1_{H^*}\oo 1_H\nonumber \\
&\Delta(\alpha\oo h)=\alpha_{(2)}\oo h_{(1)}\oo \alpha_{(1)}\oo h_{(2)} & &\epsilon(\alpha\oo h)=\epsilon_{H^*}(\alpha)\epsilon_H(h)\nonumber\\
&S(\alpha\oo h)=(1_{H^*}\oo S_H(h))\cdot (S_{H^*}(\alpha)\oo 1_H). \label{eq:ddouble}
\end{align}
 Then  $(H^*\oo H, \cdot,1,\Delta,\epsilon, 1, S)$  is a   quasitriangular Hopf algebra, the  {\bf Drinfel'd double} $D(H)$ of $H$.
For any  basis $\{x_i\}$ of $H$ with associated dual basis $\{\alpha^i\}$  of $H^*$,  the standard  $R$-matrix of $D(H)$ is given by
$R=\Sigma_i 1\oo x_i\oo\alpha^i\oo 1$.
\end{theorem}

\begin{example}[Finite group]\label{ex:ddouble_grp} If $G$ is a finite group,  it follows from Example \ref{ex:finite_group} that the quasitriangular Hopf algebra structure of the Drinfel'd double $\mathcal D(\FF[G])=\mathrm{Fun}(G)\oo \FF[G]$ is given by 
\begin{align*}
&(\delta_h\oo g)\cdot (\delta_{h'}\oo g')=\delta_{h}\delta_{gh'g^\inv}\oo gg'=\delta_{g^\inv h g}(h')\; \delta_h\oo gg' & &1=1\oo e\\
&\Delta(\delta_h\oo g)=\Sigma_{u,v\in G, uv=h}\,\delta_v\oo g\oo \delta_u\oo g & &\epsilon(\delta_h\oo g)=\delta_h(e)\\
&S(\delta_h\oo g)=(1\oo g^\inv)(\delta_{h^\inv})=\delta_{g^\inv h^\inv g}\oo g^\inv & &R=\Sigma_{g\in G}\, 1\oo g\oo \delta_g\oo e.
\end{align*}
\end{example}

\begin{lemma} Let $H$ be a finite-dimensional Hopf algebra  over a field of characteristic zero. Then $H$ is semisimple if and only if $H^*$ is semisimple if and only if $D(H)$ is semisimple \cite{R}. If $\ell\in H$ and $\eta\in H^*$ are Haar integrals then $\eta\oo \ell$ is a Haar integral for $D(H)$.
\end{lemma}

\begin{definition} Let $H$ be a finite-dimensional quasitriangular Hopf algebra with $R$-matrix $R$.
 $H$ is called {\bf factorisable}, if  the {\bf Drinfel'd map} 
$
D_Q: H^*\to H,\quad \alpha\mapsto (\id\oo\langle\alpha,\cdot\rangle)(Q)=\low Q 1\langle \alpha, \low Q 2 \rangle$ with $Q=R_{21}R$
 is an isomorphism of vector spaces.
\end{definition}

\begin{lemma}[\cite{Dr}] Let $H$ be a finite-dimensional quasitriangular Hopf algebra. Then the Drinfel'd double $D(H)$ is factorisable.
\end{lemma}

\section{Module algebras over Hopf algebras}

\label{sec:modulealg}

In this section we summarise  basic facts about  module (co)algebras over Hopf algebras that are needed in this article. A good reference on this topic is the textbook  \cite{majidbook} by Majid.

\begin{definition}\label{def:modalg} Let $H,K$ be Hopf algebras over $\FF$.\\[-4ex]

$\bullet$ An {\bf $H$-left  module algebra}  is an  algebra object  in the category $H$-Mod of  left $H$-modules,  i.e.~an associative, unital algebra $(A,\cdot, 1)$ together with an $H$-left module structure  $\rhd: H\oo A\to A$, $h\oo a\mapsto h\rhd a$  such that  for all $h,h'\in H$, $a,a'\in A$
\begin{align*}
&h\rhd (a\cdot a')=(h_{(1)}\rhd a)\cdot (h_{(2)}\rhd a') & &h\rhd 1_A=\epsilon(h)\, 1.
\end{align*}
An {\bf $H$-right module  algebra}   is a $H^{op}$-left module algebra, i.e.~an algebra object  in the category Mod-$H$=$H^{op}$-Mod of right $H$-modules.

$\bullet$ An {\bf $H$-left module  coalgebra}  is a coalgebra object  in  $H$-Mod, i.e.~a coassociative, counital coalgebra $(A,\Delta, \epsilon)$ together with an $H$ left-module structure  $\rhd: H\oo A\to A$, $h\oo a\mapsto h\rhd a$ such that  for all $h,h'\in H$, $a\in A$
\begin{align*}
& \Delta(h\rhd a)=(h_{(1)}\rhd a_{(1)})\oo (h_{(2)}\rhd a_{(2)}) & &\epsilon(h\rhd a)=\epsilon(h)\, \epsilon(a).
\end{align*}
An {\bf $H$-right module  coalgebra}  is a $H^{op}$-left module  coalgebra, i.e.~a coalgebra object  in  Mod-$H=H^{op}$-Mod.

$\bullet$ A {\bf $(H,K)$-bimodule (co)algebra} is a $(H\oo K^{op})$-left module (co)algebra. This is equivalent to a $H$-left module algebra  structure $\rhd: H\oo A\to A$ and an $K$-right module algebra structure $\lhd: A\oo K\to A$ such that $h\rhd(a\lhd k)=(h\rhd a)\lhd k$ for all $a\in A$, $h\in H$, $k\in K$.
\end{definition}

\begin{remark}\label{rem:bimodule_switch} Let $H,K$ be Hopf algebras and $A$ an $(H,K)$-bimodule (co)algebra. Then $A^{(c)op}$ becomes a $(K,H)$-bimodule (co)algebra  with module structure
$
k\rrhd a\llhd h:=S(h)\rhd a\lhd S(k)
$ for all $a\in A$, $k\in K$, $h\in H$. If $\phi: A\to A$ is an invertible (co)algebra anti-homomorphism, then  $A$ becomes  a $(K,H)$-bimodule algebra with module structure  $k\rrhd a\llhd h:=\phi^\inv(S(h)\rhd \phi(a)\lhd S(k))$.
\end{remark}

\begin{remark} \label{rem:dualstruct} Let $K$ be a finite-dimensional Hopf algebra with dual $K^*$ and $H$ a Hopf algebra. Then a $H$-left module structure  $\rhd: H\oo K\to K$ on $K$ induces a $H$-right module structure $\lhd^*: H\oo K^*\to K^*$  defined by $\langle \alpha\lhd^* h, k\rangle=\langle \alpha, h\rhd k\rangle$ for all $k\in K$, $\alpha\in K^*$ and $h\in H$. We call the $H$-right module structure $\lhd^*$  the $H$-module  structure dual to $\rhd$.
\end{remark}

\begin{example}\label{ex:regacts}

Let $H$ be a  Hopf algebra and $H^*$ its dual. Then:\\[-5ex]
\begin{enumerate}
\item The {\bf left regular action} of $H$ on itself $\rhd: H\oo H\to H$, $h\oo k\mapsto h\cdot k$ gives $H$ the structure of an $H$-left module coalgebra.
\item The {\bf right regular action} of $H$ on itself $\lhd: H\oo H\to H$, $k\oo h\mapsto k\cdot h$ gives $H$ the structure of an $H$-right module coalgebra.
\item The {\bf left regular action} of $H$ on $H^*$ $\rhd^*: H\oo H^*\to H^*$, $h\oo\alpha\mapsto \langle\alpha_{(2)}, h\rangle \,\alpha_{(1)}$ is dual to the right regular action of $H$ on itself and gives $H^*$ the structure of an left $H$-module algebra.
\item The {\bf right regular action} of $H$ on $H^*$ $\lhd^*: H^*\oo H\to H^*$, $\alpha\oo h\mapsto \langle \alpha_{(1)}, h\rangle\, \alpha_2$ is dual to the left regular action of $H$ on itself and gives $H^*$ the structure of an $H$-right module algebra.
\item The {\bf left adjoint action} of $H$ on itself $\rhd_{ad}: H\oo H\to H$, $h\oo k\mapsto h_{(1)}\cdot k\cdot S(h_{(2)})$ gives $H$ the structure of an $H$-left module algebra.
\item The {\bf right adjoint action} of $H$ on itself $\lhd_{ad}: H\oo H\to H$, $ k\oo h \mapsto S(h_{(1)})\cdot k\cdot h_{(2)}$ gives $H$ the structure of an $H$-right module algebra.

\item The {\bf left coadjoint action} $\rhd^*_{ad}: H\oo H^*\to H^*$, $h\oo\alpha\mapsto \langle S(\alpha_{(1)})\low\alpha 3, h\rangle \,\alpha_{(2)}$ 
is dual to the right adjoint action of $H$ on itself and gives $H^*$ the structure of an $H$-left comodule algebra.

\item The {\bf right coadjoint action} $\lhd^*_{ad}: H^*\oo H\to H^*$, $\alpha\oo h\mapsto \langle \alpha_{(1)}S(\low\alpha 3), h\rangle \,\alpha_{(2)}$ is dual to the left adjoint action of $H$ on itself  and gives $H^*$ the structure of an $H$-right comodule algebra.

\item The left and right regular action of $H$ on itself (on $H^*$) equip $H$ ($H^*$) with the structure of a $(H,H)$-bimodule algebra. 
\end{enumerate}
\end{example}

\begin{example}[Finite group]\label{ex:haction_grp} 
For a finite group $G$, the Hopf algebra structure of the group algebra $\FF[G]$ and its dual $\mathrm{Fun}(G)$ are given in Example \ref{ex:finite_group}. In this case,  the left and right regular
action of $\FF [G]$ on itself are given by the left and right multiplication of $G$.  The left and right adjoint action correspond to the action of $G$ on itself by conjugation.  The left and right regular action and the left coadjoint action of $\FF[G]$ on its dual  $\mathrm{Fun}(G)$ are given by
\begin{align*}
g\rhd \delta_h=\delta_{hg^\inv}\qquad \delta_h\lhd g=\delta_{g^\inv h}\qquad g\rhd_{ad^*}\delta_h=\delta_{ghg^\inv}\qquad\forall g,h\in G.
\end{align*}
\end{example}

\begin{definition}\label{def:smash} Let $H$ be a 
Hopf algebra, $A$ an $H$-left module algebra and $B$ an $H$-right module algebra.
The  {\bf left cross product}  or {\bf left smash product} $A\#_L H$   is the algebra   $ (A\oo H,\cdot)$ with  
\begin{align}
\label{eq:cross_lefthd}
&(a\oo h)\cdot (a'\oo h')=a(h_{(1)}\rhd a')\cdot h_{(2)}h'.
\end{align}
The {\bf right cross product} or {\bf right smash product}  $H\#_R B$  
is the algebra  $(H\oo B,\cdot)$ 
with
\begin{align}\label{eq:cross_rhd}
&(h\oo b)\cdot (h'\oo b')=h\low {h'} 1\oo (b\lhd \low {h'}2)b'.
\end{align}
\end{definition}

\begin{definition} \label{def:hdouble} Let $H$ be a finite-dimensional Hopf algebra with  dual
 $H^*$. The left and right {\bf Heisenberg double}  of $H$  are the cross products $\mathcal H_L(H)=H^*\#_LH$, $\mathcal H_R(H)=H \#_R H^*$  for  the left and right regular action of $H$ on $H^*$. Explicitly, their multiplication laws are given by:
\begin{align*}
&\mathcal H_L(H): \qquad(\alpha\oo h)\cdot (\alpha'\oo h')=\langle\alpha'_{(2)}, h_{(1)}\rangle\; \alpha \alpha'_{(1)}\oo h_{(2)} h'\\
&\mathcal H_R(H): \qquad( h\oo \alpha)\cdot (h'\oo \alpha')=\langle \alpha_{(1)}, h'_{(2)}\rangle\; hh'_{(1)}\oo \alpha_{(2)}\alpha'.
\end{align*}
\end{definition}

\begin{example}[Finite group]\label{ex:hdouble_grp} For a finite group $G$, the Hopf algebra structure of the group algebra $\FF[G]$ and its dual $\mathrm{Fun}(G)$ are given in Example \ref{ex:finite_group}. The
 left and right Heisenberg double of the group algebra $\FF[G]$ are given by
\begin{align*}
&\mathcal H_L(\FF[G])=\mathrm{Fun}(G)\oo \FF[G]: \qquad(\delta_h\oo g)\cdot (\delta_{h'}\oo g')=\delta_h\delta_{h'g^\inv}\oo gg'=\delta_{hg}(h')\; \delta_{h}\oo gg' \\
&\mathcal H_R(\FF[G])= \FF[G]\oo \mathrm{Fun}(G): \qquad(g\oo  \delta_h)\cdot (g'\oo \delta_{h'})=gg'\oo \delta_{h}\delta_{g h'}=\delta_{g^\inv h}(h')\; gg'\oo \delta_h.
\end{align*}
\end{example}

An essential feature of module algebras over a Hopf algebra is that the submodule of invariants is  a subalgebra. 
This is  well-known,  but we include the proof  for  convenience.

\begin{lemma}\label{lem:project} Let $H$ be a  Hopf algebra,  $M$ a $H$-left module  with respect to  $\rhd: H\oo M\to M$ and
$$M_{inv}=\{m\in M: h\rhd m=\epsilon(h)\, m\;\forall h\in H\}.$$
\begin{compactenum}
\item If $M$ is an $H$-module algebra, then  $M_{inv}$ is a subalgebra of $M$.
\item If $\ell\in H$ is a Haar integral,  then the projector on $M_{inv}$ is given by
 $\Pi: M\to M$, $m\mapsto \ell \rhd m$.
\end{compactenum}
\end{lemma}
\begin{proof} If $M$ is a $H$-module algebra, the properties of the counit imply  for  $m,m'\in M_{inv}$  and $h\in H$
$h\rhd(m\cdot m')=(\low h 1\rhd m)\cdot (\low h 2\rhd m')=\epsilon(\low h 1)\epsilon(\low h 2)\, m\cdot m'=\epsilon(h)\, m\cdot m'$ and hence $m\cdot m'\in M_{inv}$.
If $\ell\in H$ is a Haar integral, then
the identity $\ell\cdot\ell=\ell$ and the fact that $M$ is an $H$-module  ensure that  $\Pi$ is a projector:
$(\Pi\circ\Pi)(m)=\ell\rhd(\ell\rhd m)=(\ell\cdot \ell)\rhd m=\ell\rhd m =\Pi(m)$ for all $m\in M$.
The identity $h\cdot\ell=\epsilon(h)\,\ell$ for all $h\in H$ implies $\Pi(M)\subset M_{inv}$ since for all $h\in H$, $m\in M$
$h\rhd\Pi(m)=h\rhd(\ell\rhd m)=(h\cdot\ell)\rhd m=\epsilon(h) \, \ell\rhd m=\epsilon(h)\;\Pi(m)$.
The identity $\epsilon(\ell)= 1$  implies $m=\epsilon(\ell)\, m=\ell\rhd m=\Pi(m)$ for  $m\in M_{inv}$ and hence  $M_{inv}=\Pi(M)$. 
\end{proof}

\begin{example} \label{ex:moduledim} For the 
$H$-left module structure  on $H^*\oo H^*$ induced by the left regular action
 $$\rhd: H\oo H^*\oo H^*\to H^*\oo H^*, \quad h\rhd(\alpha\oo\beta)=\langle \low\alpha 2\low\beta 2, h\rangle\, \low\alpha 1\oo\low\beta 1$$ 
 one has 
 $
 (H^*\oo H^*)_{inv}=(\id\oo S)\circ\Delta(H^*). 
 $

 That $(\id\oo S)\circ\Delta(H^*)\subset  (H^*\oo H^*)_{inv}$ follows by a direct computation. 
For the converse, note that this $H$-module structure   is dual 
to the $H^*$-right-comodule structure on $H^*\oo H^*$ with the comultiplication $\Delta: H^*\to H^*\oo H^*$ as a comodule map and hence $(H^*\oo H^*)_{inv}=(H^*\oo H^*)_{coinv}$; see for instance \cite[Lemma 1.7.1]{Mon}.   
As the module map $\lhd: H^*\oo H^*\oo H^*\to H^*\oo H^*$,
$(\alpha\oo \beta)\lhd \gamma=\alpha\oo\beta\cdot\gamma$ gives $H^*$ the structure of a $H^*$-right Hopf module, it follows  from the fundamental theorem of Hopf modules---see for instance \cite[Theorem 1.4.9]{Mon}---that 
$H^*\oo H^*\cong (H^*\oo H^* )_{coinv}\oo H^*$ and hence $\mathrm{dim}(H^*\oo H^*)_{inv}=\mathrm{dim}(H^*\oo H^*)_{coinv}=\dim H^*$ As the map
 $(\id\oo S)\circ\Delta: H^*\to H^*\oo H^*$ is injective, this proves the claim.
\end{example}

The left  coadjoint action of $H$ on $H^*$ do not equip $H^*$ with the structure of a left or right module algebra. Nevertheless,
 the submodule of invariants  is a subalgebra, namely  the character algebra of $H$.  If $H$ is finite-dimensional semisimple, then $S^2=\id$, and the submodules of invariants for the left and right coadjoint action coincide and are both given by the character algebra.

\begin{example} \label{ex:chars} Let $H$ be finite-dimensional semisimple. Then for the left and right coadjoint action of $H$ on $H^*$ from Example \ref{ex:regacts} one has
\begin{align}
H^*_{inv}=C(H)=\{\alpha\in H^*: \Delta(\alpha)=\Delta^{op}(\alpha)\}=\{\alpha\in H^*:\, \langle \alpha, hk\rangle=\langle\alpha, kh\rangle\,\forall h,k\in H\}
\end{align}
If $\alpha\in C(H)$ then $\Delta^{(n)}(\alpha)$ is invariant under cyclic permutations for all $n\in\mathbb N$, and one obtains
\begin{align*}
&\alpha\lhd^*_{ad} h=\langle\low\alpha 1 S(\low\alpha 3), h\rangle\,\low\alpha 2=\langle\low\alpha 2 S(\low\alpha 1), h\rangle\,\low\alpha 3=\epsilon(\low\alpha 2)\epsilon(h)\low\alpha 3=\epsilon(h)\alpha\\
&h\rhd^*_{ad}\alpha=\langle S(\low\alpha 1)\low\alpha 3, h\rangle\,\low\alpha 2=\langle S(\low\alpha 2)\low\alpha 1, h\rangle\, \low\alpha 3=\epsilon(\low\alpha 1)\epsilon(h)\low\alpha 2=\epsilon(h)\alpha
\end{align*}
Conversely, if $\alpha\in H^*$ is invariant under the left or right coadjoint action of $H$ on $H^*$ one has
\begin{align*}
&\langle \alpha, hk\rangle =\langle \alpha,\low h 1 k \low h 3S(\low h 2)\rangle=\langle \low\alpha 1 S(\low\alpha 3), \low h 1\rangle \langle \low\alpha 2, k\low h 2\rangle=\epsilon(\low h 1)\langle\low\alpha, k\low h 2\rangle=\langle \alpha, kh\rangle\\
&\langle \alpha, hk\rangle =\langle \alpha, S(\low k 2)\low k 1 h\low k 3\rangle=\langle S(\low\alpha 1) \low\alpha 3, \low k 2\rangle \langle \low\alpha 2, \low k 1 h \rangle=\epsilon(\low k 2)\langle\alpha, \low k 1 h\rangle=\langle \alpha, kh\rangle.
\end{align*}
\end{example}

\end{appendix}

\end{document}